\documentclass[preprint]{elsarticle}

\usepackage{amsthm}
\usepackage{thmtools}
\usepackage{amsmath}
\usepackage{amsxtra}
\usepackage{amssymb}
\usepackage{comment}
\usepackage{enumitem}
\usepackage{mathrsfs}
\usepackage{lineno}
\usepackage[colorlinks=true,linkcolor=blue,citecolor=blue,urlcolor=blue]{hyperref}
\usepackage[all]{xy}
\usepackage{graphics}
\usepackage{graphicx}
\usepackage{tikz}
\usetikzlibrary{matrix,positioning,calc,arrows,arrows.meta}
\usetikzlibrary{decorations.pathmorphing,shapes}
\usepackage{undertilde}

\swapnumbers
\declaretheorem[name=Definition,style=definition,qed=$\dashv$,
numberwithin=section]{dfn}
\declaretheorem[name=Definition,style=definition,numbered=no,qed=$\dashv$]{dfn*}
\declaretheorem[name=Definition,style=definition,numbered=no]{dfnnoqed*}

\declaretheorem[name=Theorem,style=plain,sibling=dfn]{tm}
\declaretheorem[name=Theorem,style=plain,numbered=no]{tm*}

\declaretheorem[name=Lemma,style=plain,sibling=dfn]{lem}

\declaretheorem[name=Corollary,style=plain,numbered=no]{cor*}
\declaretheorem[name=Remark,style=definition,sibling=dfn]{rem}

\declaretheorem[name=Fact,style=definition,sibling=dfn]{fact}

\swapnumbers
\declaretheoremstyle[headfont=\scshape]{claimstyle}
\declaretheorem[name=Claim,style=claimstyle]{clm}

\declaretheorem[name=Claim,style=claimstyle]{clmthree}
\declaretheorem[name=Claim,style=claimstyle]{clmfour}

\declaretheorem[name=Claim,style=claimstyle]{clmseven}

\declaretheorem[name=Claim,style=claimstyle]{clmnine}

\declaretheorem[name=Claim,style=claimstyle,numbered=no]{clm*}

\declaretheorem[name=Subclaim,style=claimstyle,numberwithin=clmfour]{sclmfour}

\declaretheorem[name=Subclaim,style=claimstyle,numberwithin=clmseven]{sclmseven}

\declaretheorem[name=Subclaim,style=claimstyle,numberwithin=clmnine]{sclmnine}

\declaretheorem[name=Subclaim,style=claimstyle,numbered=no]{sclm*}

\declaretheorem[name=Subsubclaim,style=claimstyle,numbered=no]{ssclm*}

\declaretheoremstyle[headfont=\scshape]{casestyle}

\declaretheorem[name=Case,style=casestyle]{casetwo}
\declaretheorem[name=Case,style=casestyle]{casethree}
\declaretheorem[name=Case,style=casestyle]{casefour}
\declaretheorem[name=Case,style=casestyle]{casefive}
\declaretheorem[name=Case,style=casestyle]{casesix}
\declaretheorem[name=Case,style=casestyle]{caseseven}
\declaretheorem[name=Case,style=casestyle]{caseeight}
\declaretheorem[name=Case,style=casestyle]{casenine}
\declaretheorem[name=Case,style=casestyle]{caseten}
\declaretheorem[name=Case,style=casestyle]{caseeleven}

\declaretheorem[name=Subcase,style=casestyle,numberwithin=casetwo]{scasetwo}
\declaretheorem[name=Subcase,style=casestyle,numberwithin=casethree]{scasethree}
\declaretheorem[name=Subcase,style=casestyle,numberwithin=casefour]{scasefour}
\declaretheorem[name=Subcase,style=casestyle,numberwithin=casefive]{scasefive}
\declaretheorem[name=Subcase,style=casestyle,numberwithin=casesix]{scasesix}

\declaretheorem[name=Subcase,style=casestyle,numberwithin=casenine]{scasenine}
\declaretheorem[name=Subcase,style=casestyle,numberwithin=caseten]{scaseten}

\declaretheorem[name=Subsubcase,style=casestyle,numberwithin=scasenine]
{sscasenine}
\declaretheorem[name=Subsubcase,style=casestyle,numberwithin=scaseten]
{sscaseten}

\newcommand{\CC}{\mathbb C}

\newcommand{\sub}{\subseteq}
\newcommand{\cross}{\times}

\newcommand{\inter}{\cap}
\renewcommand{\int}{\inter}

\newcommand{\om}{\omega}
\newcommand{\pow}{\mathcal{P}}
\newcommand{\OR}{\mathrm{OR}}

\newcommand{\Hull}{\mathrm{Hull}}

\newcommand{\cut}{\backslash}

\newcommand{\Tt}{\mathcal{T}}

\newcommand{\Uu}{\mathcal{U}}
\newcommand{\Vv}{\mathcal{V}}

\newcommand{\Ll}{\mathcal{L}}

\newcommand{\rg}{\mathrm{rg}}
\newcommand{\dom}{\mathrm{dom}}
\newcommand{\cod}{\mathrm{cod}}

\newcommand{\ins}{\trianglelefteq}
\newcommand{\nins}{\ntrianglelefteq}
\newcommand{\pins}{\triangleleft}
\newcommand{\npins}{\ntriangleleft}
\newcommand{\crit}{\mathrm{cr}}

\newcommand{\union}{\cup}
\newcommand{\rest}{\!\upharpoonright\!}
\newcommand{\com}{\circ}

\newcommand{\lh}{\mathrm{lh}}
\newcommand{\Ult}{\mathrm{Ult}}

\newcommand{\Fbar}{{\bar{F}}}

\newcommand{\sats}{\models}
\newcommand{\elem}{\preccurlyeq}
\newcommand{\J}{\mathcal{J}}

\newcommand{\AC}{\mathsf{AC}}

\newcommand{\HOD}{\mathrm{HOD}}

\newcommand{\ZFC}{\mathsf{ZFC}}
\newcommand{\ZF}{\mathsf{ZF}}

\newcommand{\es}{\mathbb{E}}

\newcommand{\eps}{\varepsilon}

\newcommand{\Ubar}{{\bar{U}}}

\newcommand{\pitilde}{\tilde{\pi}}

\newcommand{\Nbar}{{\bar{N}}}

\newcommand{\core}{\mathfrak{C}}

\newcommand{\her}{\mathcal{H}}

\newcommand{\pred}{\mathrm{pred}}

\newcommand{\un}{\union}

\newcommand{\id}{\mathrm{id}}

\newcommand{\sq}{\mathrm{sq}}

\newcommand{\conc}{\ \widehat{\ }\ }

\newcommand{\rSigma}{\mathrm{r}\Sigma}

\newcommand{\rPi}{\mathrm{r}\Pi}

\DeclareMathOperator{\Th}{Th}

\DeclareMathOperator{\card}{card}
\DeclareMathOperator{\cof}{cof}

\DeclareMathOperator{\rank}{rank}

\newcommand{\gammavec}{\vec{\gamma}}

\newcommand{\bfrSigma}{\utilde{\rSigma}}

\newcommand{\psub}{\subsetneq}

\newcommand{\cHull}{\mathrm{cHull}}

\newcommand{\DD}{\mathbb{D}}
\newcommand{\unsq}{\mathrm{unsq}}

\newcommand{\Mbar}{\bar{\M}}

\newcommand{\lpole}{\left\lfloor}
\newcommand{\rpole}{\right\rfloor}
\newcommand{\lrgcrd}{\mathrm{lgcd}}

\newcommand{\univ}[1]{\lpole #1\rpole}

\newcommand{\tu}{\textup}

\renewcommand{\succ}{\mathrm{succ}}
\newcommand{\dropset}{\mathscr{D}}
\newcommand{\movin}{\mathrm{movin}}

\newcommand{\exit}{\mathrm{ex}}
\newcommand{\dfnemph}{\textbf}
\newcommand{\ms}{\mathrm{ms}}

\newcommand{\modelset}{\mathfrak{M}}

\newcommand{\playerI}{\mathrm{I}}
\newcommand{\playerII}{\mathrm{II}}
\newcommand{\Btilde}{\widetilde{B}}  
\newcommand{\rhotilde}{\widetilde{\rho}}
\newcommand{\Mtilde}{{\widetilde{M}}}
\newcommand{\Ntilde}{{\widetilde{N}}}
\newcommand{\Htilde}{{\widetilde{H}}}
\newcommand{\kappatilde}{\widetilde{\kappa}}
\newcommand{\qtilde}{\widetilde{q}}
\newcommand{\ptilde}{\widetilde{p}}
\newcommand{\Etilde}{\widetilde{E}}
\newcommand{\lex}{\mathrm{lex}}

\newcommand{\zetatilde}{\widetilde{\zeta}}
\newcommand{\ztilde}{\widetilde{z}}
\newcommand{\Ftilde}{\widetilde{F}}
\newcommand{\strength}{\mathrm{str}}

\newcommand{\prixi}{\widetilde{\xi}}
\newcommand{\priGamma}{\widetilde{\Gamma}}
\newcommand{\priQ}{\widetilde{Q}}
\newcommand{\priCC}{\widetilde{\CC}}

\newcommand{\pripi}{\widetilde{\pi}}

\newcommand{\xivec}{\vec{\xi}}
\newcommand{\zetavec}{\vec{\zeta}}
\newcommand{\prizeta}{\widetilde{\zeta}}
\newcommand{\priDD}{\widetilde{\DD}}
\newcommand{\priDelta}{\widetilde{\Delta}}
\newcommand{\curlyB}{\mathscr{B}}
\newcommand{\curlyD}{\mathscr{D}}
\newcommand{\curlyM}{\mathscr{M}}
\newcommand{\curlyN}{\mathscr{N}}
\newcommand{\curlyQ}{\mathscr{Q}}

\newcommand{\FF}{\mathbb{F}}
\newcommand{\priR}{\widetilde{R}}
\newcommand{\lgcd}{\mathrm{lgcd}}
\newcommand{\pvec}{\vec{p}}
\renewcommand{\Mbar}{\bar{M}}
\begin{document}
\title{A premouse inheriting strong cardinals from $V$}
\author[fs]{Farmer Schlutzenberg}

\ead{farmer.schlutzenberg@gmail.com}

\begin{abstract}
We identify a premouse inner model $L[\es]$, such 
that for any 
coarsely iterable background universe $R$ modelling $\ZFC$, $L[\es]^R$
is a proper class premouse of $R$ inheriting all strong and Woodin 
cardinals from $R$. Moreover, for each $\alpha\in\OR$, $L[\es]^R|\alpha$
is $(\om,\alpha)$-iterable, via iteration trees which lift to coarse 
iteration trees on $R$.

We prove that  $(k+1)$-condensation follows from $(k+1)$-solidity
together with
$(k,\om_1+1)$-iterability (that is, roughly, iterability with respect to normal 
trees).
We also 
prove that a slight weakening of $(k+1)$-condensation follows from
$(k,\om_1+1)$-iterability (without the $(k+1)$-solidity hypothesis).

The results depend on the theory of generalizations 
of bicephali, which we also develop.
\end{abstract}
\maketitle
\section{Introduction}\label{sec:intro}
Consider fully iterable, sound premice $M,N$ with $M|\rho=N|\rho$  
and $\rho_\om^M=\rho=\rho_\om^N$. Under what 
circumstances can we deduce that either $M\ins N$ or $N\ins M$? This conclusion 
follows 
if $\rho$ is a cutpoint of both models. By \cite[Lemma 
3.1]{stacking_mice},\footnote{The paper 
\cite{stacking_mice} literally deals with premice with Jensen indexing, whereas 
we deal with 
Mitchell-Steel indexing. However, the same result still holds.} the conclusion 
also follows if $\rho$ is 
a regular uncountable cardinal in $V$ and there is no premouse with a 
superstrong 
extender. We will 
show that if  $M||\rho^{+M}=N||\rho^{+N}$
and $M,N$ have a certain 
\emph{joint} iterability property, then $M=N$.

The joint iterability property and the proof that $M=N$, is motivated by the 
\emph{bicephalus} 
argument of \cite[\S9]{fsit}. Bicephali in \cite{fsit} are structures
$B=(P,E,F)$, where both $(P,E)$ and $(P,F)$ are active premice. If $B$ is an 
iterable bicephalus 
and there is no iterable superstrong premouse
then $E=F$ (see \cite[\S9]{fsit} and \cite{extmax}); the proof is by comparison 
of $B$ with itself.
In \S\ref{sec:cephals} we  consider a more general form of bicephali, 
including, for example, 
the structure $C=(\rho,M,N)$, where $\rho,M,N$ are as at the end of the 
previous paragraph.
If $C$ is 
iterable, a comparison of $C$ with itself will be used to show that  $M=N$ in 
this situation.

Hugh Woodin also noticed that generalizations of bicephali can be used in 
certain fine structural 
arguments, probably before the author did; see \cite{woodin_fs_suit}. The 
bicephali used in 
\cite{woodin_fs_suit} have more closure than those considered here, but 
of course, the premice of \cite{woodin_fs_suit} are long extender premice. So 
while there is some 
overlap, it seems that things are quite different.

We 
will also consider bicephali $(\rho',M',N')$ in which $M'$ or $N'$ might fail 
to 
be fully sound. 
However, we will assume that both $M',N'$ project $\leq\rho'$, are 
$\rho'$-sound, and $M',N'$ agree 
below their common value for $(\rho')^+$.
If such a bicehpalus is iterable, it might be that $M'\neq N'$, but we will see 
that in this 
situation, $M'$ is an ultrapower of some premouse by an extender in 
the extender sequence $\es_+^{N'}$ of $N'$ (here $\es_+^{N'}$
includes the active extender of $N'$) or vice versa.

We will also 
prove similar results regarding \emph{cephalanxes}, a blend of bicephali and 
phalanxes. The presence 
of superstrong premice makes cephalanxes somewhat more subtle than bicephali.

We give two applications of these results. First, in 
\S\ref{sec:condensation}, we 
consider proving condensation from normal iterability. Let 
$k<\om$, let $H,M$ be 
$k$-sound premice, $\pi:H\to M$ be a near $k$-embedding\footnote{Actually 
we 
will work with the 
more general class of \emph{$k$-lifting} embeddings; see \ref{dfn:k-lifting}.},
$\rho_{k+1}^H\leq\rho<\rho_k^H$, and suppose $H$ is $\rho$-sound and 
$\rho\leq\crit(\pi)$. We 
wish to prove the conclusion of $(k+1)$-condensation 
for this embedding.\footnote{Approximately, that is, the ``version 
\ldots with 
$\rho_{k+1}^H$ 
replacing $\rho_\om^H$'' in \cite[pp. 87--88]{fsit}, or \cite[Lemma 
1.3]{stacking_mice}, though
this uses Jensen indexing, or \cite[Theorem 9.3.2]{imlc}, though this uses 
Jensen indexing and $\Sigma^*$-fine structure.} The classical
(phalanx-based) proof of condensation 
uses the $(k,\om_1,\om_1+1)^*$-iterability of $M$ (roughly,
iterability for stacks of normal trees), 
through its appeal to  weak Dodd-Jensen. We would like to reduce this assumption 
to $(k,\om_1+1)$-iterability 
(roughly, iterability for normal trees). Given the latter, and also
assuming  $M$ is $(k+1)$-solid, we will deduce the usual 
conclusion of condensation. 
We will also prove a slight weakening of 
$(k+1)$-condensation 
from $(k,\om_1+1)$-iterability,
 without the extra solidity hypothesis. (As we 
are not assuming $(k,\om_1,\om_1+1)^*$-iterability, it is natural to consider 
the circumstance that $M$ 
fails to be $(k+1)$-solid; see 
\S\ref{sec:questions}.\footnote{\label{ftn:fs_from_norm}
Actually, the author has since shown that $(k+1)$-solidity
follows from $(k,\om_1+1)$-iterability. This result will appear
in \cite{fsfni}. So the present paper together with \cite{fsfni}
gives a complete proof of $(k+1)$-condensation from 
$(k,\om_1+1)$-iterability.} But note that 
the assumption that $H$ is $\rho$-sound entails that $(H,p_{k+1}^H\cut\rho)$ is 
$(k+1)$-solid.)
Our proof makes substitutes bicephali and cephalanxes for phalanxes,
and avoids 
(weak) Dodd-Jensen.\footnote{The way we have presented our proof, 
we do make use of the standard proof of condensation, in proving 
\ref{lem:fully_elem_condensation}, 
but in circumstances in 
which Dodd-Jensen is not required. This appeal to the standard proof can, 
however, be removed, by 
arranging things more inductively and using the main structure of the proof of 
\ref{thm:condensation} to prove 
\ref{lem:fully_elem_condensation}.}\footnote{Some of the key arguments
involved here, and extensions thereof regarding solidity and universality,
were presented by the author at the 3rd M\"unster conference
on inner model theory, the core model induction, and hod mice,
in July 2015.  Some notes of those talks, taken
by Schindler, can be seen in \cite{ralf_notes_solidity_talk}.}

Let $W\sats\ZFC$ be coarsely iterable.
Let $N$ be the output of a 
(standard) fully backgrounded 
$L[\es]$-con\-struct\-ion of $W$. 
Then $N$ inherits the Woodin cardinals of $W$, meaning that every 
Woodin cardinal of $W$ is Woodin in $N$. However, 
$\kappa$ can be strong in $W$, 
but not strong in $N$. For example, if $\kappa$
is strong in $W$ but $W$ has no measurable cardinal $\mu>\kappa$,
then $N$ has no measurable cardinal $\geq\kappa$
(see \ref{rem:strong});
in particular, $\kappa$ is not even measurable in $N$,
let alone strong.

In 
\cite{localKc}, assuming that $W$ is a (finely) iterable premouse
with no largest cardinal,
Steel defined the 
\emph{local $K^c$-construction} $K^W_{\mathrm{loc}}$ of $W$,
such that $K^W_{\mathrm{loc}}$ inherits both Woodin and strong cardinals
from $W$.
Along with requiring that $W$ be a premouse,
an important feature used in ensuring that strong cardinals are inherited
is that
the background extenders used to construct
$K^W_{\mathrm{loc}}$ do not have to be $W$-total. As a consequence, 
when one lifts 
iteration trees on 
$M$ to iteration trees $\Uu$ on $V$, the tree $\Uu$ might have drops.

In 
\S\ref{sec:ultra-stack}, working with background theory $\ZF$,
given any transitive class $W\sats\ZFC$ which is (sufficiently) coarsely 
iterable,
we identify a new form of $L[\es]$-construction $\CC$ of $W$.
Letting $L[\es]$ be 
the final model of 
$\CC$ (as computed in $W$), we show that (a) $L[\es]$ is a proper class premouse
of $W$, outright definable over $W$, (b) if $\kappa$ is 
strong 
(Woodin) in $W$, then $\kappa$ is strong (Woodin) in $L[\es]$,
as witnessed by $\es$, and (c)
noting that $W$ might be proper class,
if there is a (class) function $f:\OR^W\to W$
such that $f(\alpha)$ wellorders $V_\alpha^W$ for each $\alpha\in\OR^W$,
then $L[\es]$ is 
iterable, with 
iteration trees on $L[\es]$ lifting to (coarse, hence non-dropping) trees on 
$W$.
Thus, we achieve 
many of the 
properties of 
the the local $K^c$-construction, but with the advantages that $W$ need not be 
a 
premouse, and (even if $W$ is a premouse)
trees $\Uu$ on $W$ resulting from lifting 
trees on $L[\es]$ 
are coarse (and hence non-dropping).\footnote{The key
ideas of the construction were presented by the author at the
MAMLS 2014 meeting at Miami University.}

We finally remark that Steel's local $K^c$-construction 
seems to be more local than $\CC$, and hence as one extra feature 
that
it seems $\CC$ might not: $K^W_{\mathrm{loc}}$ also inherits
all $\lambda$-strong cardinals of $W$, whenever
$\lambda$ is a limit cardinal of $W$.

\subsection{Notation and terminology}\label{subsec:ntn}

\subsubsection{General}\label{subsubsec:general}
The universe $N$ of a first-order structure $M=(N,\ldots)$ is denoted 
$\univ{M}$.

We use the lexicographic order 
on $[\OR]^{<\om}$: $a<b$
iff $a\neq b$ and $\max(a\Delta b)\in b$. We sometimes identify elements of 
$[\OR]^{<\om}$ with 
strictly descending sequences of ordinals. Let 
$a\in[\OR]^{<\om}$ with $a=\{a_0,\ldots,a_{k-1}\}$
where $a_i>a_{i+1}$ for all $i+1<k$. We write $a\rest j$ for 
$\{a_0,\ldots,a_{j-1}\}$.

\subsubsection{Premice}\label{subsubsec:premice} We deal with premice and 
related structures with Mitchell-Steel 
indexing, but with 
extenders of superstrong type permitted on their extender sequence. That is, a 
\dfnemph{super-fine 
extender sequence} $\vec{E}$ is a sequence such that for each 
$\alpha\in\dom(\vec{E})$, $\vec{E}$ 
is acceptable at $\alpha$, and if $E_\alpha\neq\emptyset$ then either:
\begin{enumerate}[label=--]
 \item $E_\alpha$ is a $(\kappa,\alpha)$ pre-extender over 
$\J^{\vec{E}}_\alpha$ 
and $E_\alpha$ is 
the trivial completion of $E_\alpha\rest\nu(E_\alpha)$ and $E_\alpha$ is not 
type Z, or
 \item $\J^{\vec{E}}_\alpha$ has largest cardinal $\nu$ and $E_\alpha$ is a 
$(\kappa,\nu)$ 
pre-extender over $\J^{\vec{E}}_\alpha$ and 
$i_{E_\alpha}(\kappa)=\nu=\nu(E_\alpha)$,
\end{enumerate}
and further, properties 2 and 3 of \cite[Definition 2.4]{outline} hold. We then 
define 
\dfnemph{\tu{(}potential\tu{)} premice} in terms of super-fine extender 
sequences, in the 
usual manner,
with the caveat that we consider a (potential) premouse to be an amenable 
structure $P=(\J_\alpha^\es,\es,\widetilde{F})$,
where $\widetilde{F}$ is the amenable coding of the active extender $F$
of $P$, as described in \cite[2.9--2.10]{outline}. 
We may blur the 
distinction between $F$ and $\widetilde{F}$.
Likewise for 
related terms, such as \emph{segmented-premouse} (see \cite[\S5]{extmax}).
See \cite{operator_mice} for discussion of the modifications of 
the general theory 
needed to deal 
with these changes.\footnote{\label{ftn:superstrong_mod}The only significant 
difference in the basic
definitions (other than  \emph{super-fine extender sequence})  is that
for  \emph{$k$-maximal iteration tree $\Tt$},
one must replace the usual requirement that $\lh(E^\Tt_\alpha)<\lh(E^\Tt_\beta)$
for all $\alpha<\beta$, with the requirement that 
$\lh(E^\Tt_\alpha)\leq\lh(E^\Tt_\beta)$ for all $\alpha<\beta$. However,
we then get that $\lh(E^\Tt_\alpha)=\lh(E^\Tt_\beta)$
iff $\alpha+1=\beta$ and $E^\Tt_\alpha$ is superstrong
and $M^{*\Tt}_{\alpha+1}$ is active type 2 with largest cardinal 
$\crit(E^\Tt_\alpha)$; in this case $M^\Tt_{\alpha+1}$
is active type 2 with ordinal height $\lh(E^\Tt_\alpha)$,
and so $F(M^\Tt_{\alpha+1})$ is the only possibility for $E^\Tt_{\alpha+1}$.}
We sometimes abbreviate \emph{premouse} with \emph{pm}
and \emph{segmented-premouse} with \emph{seg-pm}.
A \dfnemph{premouse 
extender} is the active extender of some premouse.
\emph{ISC} abbreviates ``initial
segment condition''.

Let $P$ be a seg-pm. We write 
$F^P=F(P)$ for the active extender of $P$ (possibly $F^P=\emptyset$), 
$\es^P=\es(P)$ for the 
extender sequence of $P$, 
\emph{excluding} $F^P$, and $\es_+^P=\es_+(P)=\es^P\conc F^P$. If 
$F^P\neq\emptyset$ we write 
$\lh(F^P)=\OR^P$.
We write $Q\ins P$ iff $Q$ is an initial segment of $P$
(that is, $\OR^Q\leq\OR^P$ and $\es^Q_+=(\es^P_+)\rest(\OR^Q+1)$),
and $Q\pins P$ iff $Q\ins P$ but $Q\neq P$.
Given a limit $\alpha\leq\OR^P$, we write $P|\alpha$ for the $Q\ins P$
such that $\OR^Q=\alpha$,
and $P||\alpha$ for its passivization $(\univ{Q},\es^Q,\emptyset)$. (So 
$P||\alpha$ is passive, and $P|\alpha$
is active iff $(\es_+^P)_\alpha\neq\emptyset$.)
If $P$ has a largest cardinal $\delta$, $\lgcd(P)$ denotes $\delta$. If $P$ is 
active, then $\nu(P)=\nu(F^P)$ denote the natural length of $F^P$,
and $\iota(P)=\iota(F^P)$ denote $\max(\lgcd(P),\nu(F^P))$. So if 
$P$ is 
an active premouse 
then $\iota(P)=\nu(F^P)$.
Given also another seg-pm $R$ and an ordinal $\alpha\leq\min(\OR^P,\OR^R)$, 
we write
$(P\sim R)|\alpha$ iff $P|\alpha=R|\alpha$.
We also use such notation with more than two structures, and 
also with ``$||$'' replacing ``$|$''.
We use similar notation for cephals; see \ref{dfn:cephal}.

Let $P$ be an active seg-pm, $F=F^P$
and $i^P_F:P\to\Ult(P,F)$ the ultrapower map.
We say 
that $F$, or 
$P$, has \dfnemph{superstrong type} (or just is \dfnemph{superstrong}) iff 
$i^P_F(\crit(F))<\lh(F)$.
(So if $F$ has superstrong type then $i^P_F(\crit(F))$ is the largest cardinal 
of 
$P$, and then $P$ 
is a premouse iff the initial segment condition holds for $P$.) In 
\cite{extmax}, all premice are 
assumed to be below superstrong type, but certain results 
there (in particular, \cite[2.17, 2.20]{extmax}) hold in our context,
by the same proofs, and when we cite these results, we literally refer to these 
generalizations. This generalization will be covered more explicitly in
\cite{fsfni}.
(However, the proof of
\cite[Theorem 5.3]{extmax} does not go through at the superstrong 
level; Theorem 
\ref{thm:no_iterable_sound_bicephalus} here generalizes that 
result at the 
superstrong level.)\footnote{The proof of Dodd-solidity for $1$-sound, 
$(0,\om_1,\om_1+1)^*$-iterable premice (for Mitchell-Steel indexing) does not 
immediately generalize, although it can be adapted to the superstrong level 
with some further work;
recall that Zeman \cite{zeman_dodd} proves the analogous result for Jensen 
indexing,
which is at the superstrong level. However, in this paper we do not need to 
consider Dodd-solidity.}

\subsubsection{Fine structure}\label{subsubsec:fine_structure}
Let $M$ be a premouse.
As in \cite{outline}, $\core_0(M)$ denotes the squash $M^\sq$ of $M$ if $M$ is 
type 3,
and otherwise denotes $M$ (which is by  definition 
amenable).
If $M$ is non-type 3, we also define $M^\sq=M$, so in all cases, 
$\core_0(M)=M^\sq$. Also in general, $\core_0(M)^\unsq$ denotes $M$.
We will often blur the distinction between $M$ and $\core_0(M)$.

 The \dfnemph{\tu{(}fine structural\tu{)} pm language}
$\Ll$ is 
$\{\dot{\in},\dot{=},\dot{\es},\dot{\widetilde{F}},\dot{F}_\downarrow\}$.
We interpret $\Ll$ over 
$M$ (for seg-pms $M$) and over $\core_0(M)$ (for premice $M$) in the usual 
manner. Over $M$: $\dot{\es}^{M}=\es^{M}$,
$\dot{\widetilde{F}}^{M}=\widetilde{F}^M$,
if $M$ is type 2 then $\dot{F}_\downarrow^M$
is the trivial 
completion of the largest non-type Z initial segment of $F^M$,
and otherwise $\dot{F}_\downarrow^M=\emptyset$.
Over $\core_0(M)$:
as above if $M=\core_0(M)$ is non-type 3, so suppose $M$ is type 3.
Then $\dot{\es}^{\core_0(M)}=\es^{\core_0(M)}=\es^M\rest\nu(F^M)$,
and $\dot{\widetilde{F}}^{\core_0(M)}$
is the set of all restrictions $F^M\rest\alpha$ for 
$\alpha<\nu(F^M)$,
and $\dot{F}^{\core_0(M)}_\downarrow=\emptyset$.

The language for the definability classes $\rSigma_{0}^M$ and $\rSigma_1^M$
is $\Ll$, with these classes interpreted over $\core_0(M)$.
Of course, most of the time, for type 3 premice $M$, we deal with $\core_0(M)$,
but in special circumstances we need to deal directly with $M$ instead,
interpreting $\Ll$ over $M$ as above (in these circumstances we use 
\emph{simple} embeddings and ultrapowers, as discussed below).

We also define the \dfnemph{natural language} $\Ll_{\mathrm{nat}}^M$ of $M$:
if $M$ is passive, 
$\Ll_{\mathrm{nat}}^M=\Ll\cut\{\dot{\widetilde{F}},\dot{F}_\downarrow\}$;
if $M$ is type 1/2, $\Ll_{\mathrm{nat}}=\Ll$,
and if $M$ is type 3, 
$\Ll_{\mathrm{nat}}^M=\Ll\cut\{\dot{F}_\downarrow\}$.

For the basic fine structural notions (soundness, solidity, $\rSigma_{n+1}$, 
etc),
we follow Mitchell-Steel, as 
modified in 
\cite[\S5]{V=HODX}. This modification involves three things.
The first (and main one) is that we
drop the parameter $u_n$
of \cite[\S2]{fsit}, defining  $p_{n+1}$ without
reference to $u_n$.
(Recall $u_0^M=\emptyset$ an if
$n>0$ then $u^M_n=(p^M_n,w^M_n,\widetilde{\rho}^M_{n-1},u_{n-1}^M)$
where 
$w_n$ is the set of $n$-solidity witnesses (in the sense of \cite{fsit}),
 $\widetilde{\rho}^M_{n-1}=\rho_{n-1}^M$ if $\rho_{n-1}^M<\rho_0^M$,
and $\widetilde{\rho}_{n-1}^M=0$ otherwise.)
 The reader who prefers the original 
fine structure simply need put $u_n$ into
all $\rSigma_{n+1}$ hulls and $\rSigma_{n+1}$ theories. 
By \cite{V=HODX}, this change does not have any significant impact;
it just simplifies notation.
The second is that, in the terminology of \cite{fsit},
we use only pure theories, not generalized theories. Thus
(comparing with \cite[Definition 2.3.6]{fsit}), if $M$ is $n$-sound
and $\om<\rho_n^M$, we define 
the predicate $T_n^M$,
where $n\geq 1$, as the set of tuples $(\alpha,q,t)\in\core_0(M)$
with $\alpha<\rho_n^M$ and $q\in\core_0(M)$
and $t=\Th_{\rSigma_n}^M(\alpha\cup\{q\})$,
where this denotes the pure $\rSigma_n$ theory; see below.
This also has no significant impact, as explained in \cite[Lemma 2.10]{fsit}.
The third is just terminological: for the definition of 
\emph{$(n+1)$-solidity} for a structure $N$,
we follow \cite{imlc}, not \cite{fsit}; this is discussed below.

As described in \S\ref{sec:prelim}, we also use $n$-lifting embeddings
where weak $n$-embeddings are used classically (but this does not 
impact any basic definitions).

So, from now on we use the fine structural notions
defined as in \cite[\S5]{V=HODX}. 
Let $n<\om$ and let $M$ be an $n$-sound premouse.
For 
$i\leq n+1$ we write $\vec{p}_i^{M}=(p_1^M,\ldots,p_i^M)$.
Now suppose 
$\om<\rho_n^M$, and let 
$X\sub
\core_0(P)$. Almost as in \cite{extmax}, $\Hull_{n+1}^M(X)$ denotes the 
restriction of (the predicates of ) $\core_0(M)$
to the points $y\in\core_0(M)$ such that for some $\rSigma_{n+1}$ formula 
$\varphi$
and $\vec{x}\in X^{<\om}$, $y$ is the unique $z\in\core_0(M)$ such that
$\core_0(M)\sats\varphi(\vec{x},z)$.  (This is not exactly as in \cite{extmax}, 
because we do not automatically put $u_n^M$ into 
the hull.) And $\cHull_{n+1}^M(X)$ is the transitive 
collapse of this structure. 
Also, $\Th_{\rSigma_{n+1}}^M(X)$ denotes the  $\rSigma_{n+1}$ theory of 
$\core_0(M)$ in parameters in
$X$ (that is, all pairs $(\varphi,\vec{x})$
such that $\varphi$ is an $\rSigma_{n+1}$ formula
and $\vec{x}\in X^{<\om}$ and $\core_0(M)\sats\varphi(\vec{x})$).
This notation differs from \cite{extmax}, in that it denotes a pure theory, not 
a generalized theory, in the terminology of \cite{fsit}; we have no need for 
generalized theories. If $\om<\rho_{n+1}^M$,
we then  define $T_{n+1}^M$ as stated above, and define  $\rSigma_{n+2}$ from 
this
as in \cite{fsit}.

Given $q\in[\rho_0^M]^{<\om}$ and $\alpha\in q$, the \dfnemph{$(n+1)$-solidity 
witness for $(M,q,\alpha)$} (or just \dfnemph{for $(q,\alpha)$}) is 
$w^M_{q,\alpha}=\cHull_{n+1}^M(\alpha\cup\{q',\pvec_n^M\})$
where $q'=q\cut(\alpha+1)$.
A 
\dfnemph{generalized $(n+1)$-solidity witness
for $(M,q,\alpha)$} is  the obvious adapatation from  \cite[\S1.12]{imlc}.
A \dfnemph{\tu{(}generalized\tu{)} $(n+1)$-solidity witness} for $(M,q)$ is one 
for $(M,q,\alpha)$,
for some $\alpha\in q$; a witness \dfnemph{for $M$} is one for $(M,p_{n+1}^M)$. 
We say that $(M,q)$ is \dfnemph{$(n+1)$-solid}
iff $w^M_{q,\alpha}\in M$ for each $\alpha\in q$.
We say that $M$ is \dfnemph{$(n+1)$-solid}
iff $(M,p_{n+1}^M)$ is $(n+1)$-solid.
\footnote{Regarding the $(n+1)$-solidity of $M$,
this follows the (analogous) terminology
of Zeman \cite[p.44, Definition prior to Lemma 
1.12.5]{imlc}, but \emph{not} that of Mitchell-Steel \cite[Definition 
2.8.2]{fsit}
and \cite[Definition 2.15]{outline}; Mitchell-Steel demands
universality as one of the conditions for solidity, whereas Zeman does not
(and neither do we).
In \cite{imlc}, see also its Lemma 5.1.7(c), the paragraph following 
Corollary 5.1.8, Theorem 5.2.1, Lemma 9.2.14, and Theorem 9.3.1, which
treat the $(n+1)$-solidity of a structure $M$
and $(n+1)$-universality for $(M,p_{n+1}^M)$ separately.}
Given $\rho\in[\rho_{n+1}^M,\rho_n^M]$,
we say that $M$ is \dfnemph{$\rho$-solid}
iff $(M,p_{n+1}^M\cut\rho)$ is $(n+1)$-solid,
and that $M$ is \dfnemph{$\rho$-sound}
iff $M$ is $\rho$-solid and $M=\Hull_{n+1}^M(\rho\cup p_{n+1}^M)$.
We say $M$ is \dfnemph{$(n+1)$-sound}
iff $M$ is $\rho_{n+1}^M$-sound.\footnote{
The terminology \emph{$(n+1)$-sound}
follows  Mitchell-Steel,
not Zeman,
as Zeman does not incorporate $(n+1)$-solidity into $(n+1)$-soundness.}
For $\delta\in[\rho_{n+1}^M,\rho_0^M]$, the 
$\delta$-core of $M$ is
$\cHull_{n+1}^M(\delta\un\{\pvec_{n+1}^M\})$.

\subsubsection{Extenders and ultrapowers}\label{subsubsec:ext_and_ult}
Our use of the term 
\emph{extender} is  standard
for inner model theory. \emph{Extenders} need not be total 
over $V$, and need not yield wellfounded ultrapowers. We use the term
basically as in \cite{extmax} (see its introduction), except that in 
\S\ref{sec:ultra-stack} we must also consider extenders over coarse 
structures.

Given an extender $E$ over $M$, $U=\Ult(M,E)$ denotes
the ultrapower of $M$ by $E$, computed using functions in $M$;
this is formed directly, without any squashing, whatever $M$ is,
and if $M$ is an amenable structure, its predicates are shifted
piece-wise as usual. Such an ultrapower is called 
\dfnemph{simple}.
We write $i^M_E$ for the ultrapower embedding $i^M_E:M\to U$.
We may abbreviate $i^M_E$ by $i_E$.
We write $\ms(E)$ for the measure space of 
$E$; 
that 
is, the supremum of all $\kappa+1$ such that $\kappa\in\dom(i_E)$ and $(\sup 
i_E``\kappa)<\lh(E)$ ($E$ might be long).
If $E$ is 
short and $\kappa=\crit(E)$,
we say that $E$ is \dfnemph{weakly amenable \tu{(}to $M$\tu{)}}
iff $\pow(\kappa)\inter M=\pow(\kappa)\inter U$.
Note that if $M$ is an iterable premouse, then by condensation,
this is equivalent to saying that $M||\kappa^{+M}=U|\kappa^{+U}$.
Now suppose  $M$ is an $n$-sound premouse, $E$ is short
and $\crit(E)<\rho_n^M$. Then $\Ult_n(M,E)$ denotes the 
degree
$n$ fine structural ultrapower of $M$ by $E$, and $i^{M,n}_E$ the ultrapower 
map. 
We may abbreviate $i^{M,n}_E$ by $i^M_E$ or $i_E$.
Given $a\in[\lh(E)]^{<\om}$ and an $\bfrSigma_n^M$ function $f$,
$[a,f]^{M,n}_E$ denotes the object represented by the pair $(a,f)$ in 
$\Ult_n(M,E)$.
Recall that if $M$ is type 3, then $\Ult_n(M,E)=\Ult_n(M^\sq,E)^\unsq$.
If $M$ is type 3 then we let $\core_{-1}(M)=\core_0(M)$, and for an extender 
$E$ over $\core_0(M)$, let $\Ult_{-1}(M,E)=\Ult_0(M,E)$.

\subsubsection{Embeddings}\label{subsubsec:embeddings} Given structures $X,Y$, 
if context determines an obvious natural embedding
$i:X\to Y$, we sometimes write $i_{XY}$ for $i$.

Let $M,N$ be segmented-premice. A \dfnemph{simple} embedding
$\pi:M\to N$
is a function $\pi$ with $\dom(\pi)=\univ{M}$ and
$\cod(\pi)=\univ{N}$, such that $\pi$ is $\Sigma_0$-elementary with respect to 
 $\Ll$. (Note that if 
$M$ is active then
$\pi(\lrgcrd(M))=\lrgcrd(N)$, because the amenable predicates for $F^M$ and
$F^N$ determine the largest cardinal.)
If $M,N$ are type 3 premice, a \dfnemph{squashed} embedding $\pi:M\to N$ is, 
literally, an $\rSigma_0$-elementary function 
$\pi:\core_0(M)\to\core_0(N)$.

Let $\pi:M\to N$ be simple. If $M$ is passive then
then $\psi_\pi$ denotes $\pi$. If $M$ is active then
$\psi_\pi:\Ult(M,F^M)\to\Ult(N,F^N)$
denotes the  embedding induced by $\pi$ (via the Shift Lemma).
Now let $\pi:M\to N$ be squashed (so $M,N$ are type 3). Then
$\psi_\pi:\Ult_0(M,F^M)\to\Ult_0(N,F^N)$
denotes the  embedding induced by $\pi$.
So $\pi\sub\psi_\pi$ in both cases.

Let $\pi:M\to N$ be simple or squashed.
We say $\pi$ is (i) \dfnemph{$\nu$-preserving} iff 
$\pi$ is simple or $\psi_\pi(\nu(F^M))=\nu(F^N)$,
(ii) \dfnemph{$\nu$-high}
iff $\pi$ is squashed and $\psi_\pi(\nu(F^M))>\nu(F^N)$,
(iii) \dfnemph{$\nu$-low}
iff $\pi$ is squashed and $\psi_\pi(\nu(F^M))<\nu(F^N)$,
(iv) \dfnemph{$\iota$-preserving} iff either $M,N$ are passive or 
$\psi_\pi(\iota(F^M))=\iota(F^N)$, (v) \dfnemph{c-preserving} iff for 
all 
$\alpha$,  $\alpha$ is a cardinal of $M$ iff $\pi(\alpha)$ is a cardinal 
of 
$N$.

\begin{rem}\label{rem:Q'_from_nu-low}
 Let $\pi:M\to N$ be a squashed embedding (so $M,N$ are type 3).
 Easy elementarity considerations show that if $\pi$ is $\rSigma_1$-elementary 
then $\pi$ is non-$\nu$-low, and if $\rSigma_2$-elementary then 
$\nu$-preserving.
Suppose $\pi$ is $\nu$-low and let $\nu'=\psi_\pi(\nu(F^M))$.
 Then $\nu'<\nu(F^N)$ is an $N$-cardinal,
 so by  ISC, there is $N'\pins N$ with $F^{N'}\rest\nu'=F^N\rest\nu'$.
 Note that we get an $\rSigma_0$-elementary
 $\pi':\core_0(M)\to\core_0(N')$ with the same graph as $\pi$,
 and $\pi'$ is $\nu$-preserving.
\end{rem}

Let $\pi:\core_0(M)\to\core_0(N)$ be $\rSigma_0$-elementary where $M,N$ are 
$n$-sound.
We say $\pi$ is (i) \dfnemph{$p_{n+1}$-preserving} iff 
$\pi(p_{n+1}^M)=p_{n+1}^N$, (ii) \dfnemph{$\vec{p}_{n+1}$-preserving} iff 
$\pi(\vec{p}_{n+1}^M)=\vec{p}_{n+1}^N$, (iii) \dfnemph{$\rho_j$-preserving}
iff either  $\pi(\rho_j^M)=\rho_j^N$ or [$\rho_j^M=\rho_0^M$ and 
$\rho_j^N=\rho_0^N$].

\subsubsection{Iteration trees and iterability}\label{subsubsec:iteration_trees}

Let $\Tt$ be an iteration tree of length $\lh(\Tt)=\lambda$. The objects
associated to $\Tt$ we write as:
tree order ${<^\Tt}$, drop-set $\dropset^\Tt$ 
(the nodes where drops in model occur), models $M^\Tt_\alpha$, degrees 
$\deg^\Tt_\alpha$,
embeddings $i^\Tt_{\alpha\beta}$ and $i^{*\Tt}_{\alpha\beta}$ (defined where 
appropriate),
exit extenders $E^\Tt_\alpha$, exit models 
$\exit^\Tt_\alpha=M^\Tt_\alpha|\lh(E^\Tt_\alpha)$,
ultrapower domains $M^{*\Tt}_{\alpha+1}$ (so if $\Tt$ is fine structural and
$d=\deg^\Tt_{\alpha+1}$ then
$M^\Tt_{\alpha+1}=\Ult_d(M^{*\Tt}_{\alpha+1},E^\Tt_\alpha)$),
and associated ultrapower maps $i^{*\Tt}_{\alpha+1}:M^{*\Tt}_{\alpha+1}\to 
M^\Tt_{\alpha+1}$
(so $i^{*\Tt}_{\alpha+1,\beta}=i^\Tt_{\alpha+1,\beta}\com i^{*\Tt}_{\alpha+1}$),
$\kappa^\Tt_\alpha=\crit(E^\Tt_\alpha)$ and 
$\nu^\Tt_\alpha=\nu(E^\Tt_\alpha)$. We write $\pred^\Tt(\alpha+1)$
for the ${<^\Tt}$-predecessor of $\alpha+1$, and given $\alpha<^\Tt\beta$,
 $\succ^\Tt(\alpha,\beta)$ denotes the least $\gamma\in(\alpha,\beta]_\Tt$. If
$\lh(\Tt)=\theta+1$, then $M^\Tt_\infty=M^\Tt_\theta$, 
$b^\Tt=[0,\infty]_\Tt=[0,\theta]_\Tt$,
and if there is no drop along
$b^\Tt$ then $i^\Tt=i^\Tt_{0\infty}=i^\Tt_{0\theta}$, etc.

Let $n\leq\om$ and $M$ be an $n$-sound premouse.
The notion \dfnemph{$n$-maximal iteration tree $\Tt$ on $M$}
is defined basically as in
\cite[Definition 3.4]{outline}, or \cite[Definition 6.1.2]{fsit},
but we must adapt these definitions to superstrongs.
That is, $\deg^\Tt_0=n$
and for $\alpha+1<\lh(\Tt)$, letting
$\kappa=\kappa^\Tt_\alpha$, we have: 
(i) $\lh(E^\Tt_\beta)\leq\lh(E^\Tt_\alpha)$ for all 
$\beta<\alpha$; (ii) 
$\pred^\Tt(\alpha+1)$ is the least $\beta$ such that $\kappa<\nu^\Tt_\beta$; 
(iii) 
$M^{*\Tt}_{\alpha+1}$ is the largest $N\ins M^\Tt_\beta$ such that 
$E^\Tt_\alpha$ measures 
$\pow(\kappa)\inter N$; and (iv) $\deg^\Tt(\alpha+1)$ is the largest $d\leq\om$ 
such that 
$\kappa<\rho_d(M^{*\Tt}_{\alpha+1})$ and either $[0,\alpha+1]_\Tt$ 
drops or 
$d\leq n$. We will also extend the notion  \emph{$n$-maximal iteration tree} to 
trees on premouse-related 
structures. 
An iteration tree is 
\dfnemph{degree-maximal} if 
$n$-maximal for some $n\leq\om$.

For $\theta\leq\OR$, the 
notions 
\emph{$(n,\theta)$-iteration strategy for} and 
\emph{$(n,\theta)$-iterability of} $M$ are 
as in \cite[Definition 3.9]{outline}
(but using $n$-maximality defined as above).
For 
\emph{$(n,\alpha,\theta)^*$-iteration strategy} and 
\emph{$(n,\alpha,\theta)^*$-iterable} see \cite[p.~1202]{cmwmwc}.
\footnote{The relevant trees are stacks
of degree-maximal trees, each of length $\leq\theta$, starting with $n$-maximal.
The superscript-$*$ has the effect that
if in some round ${<\alpha}$, a degree-maximal tree is produced which has 
length $\theta$,
then the game stops there immediately, and (if it has wellfounded models)
player II has won.}

If $\Tt$ is padded, unless otherwise specified,
if $\beta=\pred^\Tt(\alpha+1)$ then $E^\Tt_\beta\neq\emptyset$.

\section{Fine structural preliminaries}\label{sec:prelim}

\begin{dfn}\label{dfn:k-lifting}
Let $H,M$ be $k$-sound premice with $\rho_k^H,\rho_k^M>\om$. 
We say an embedding
$\pi:\core_0(H)\to\core_0(M)$ is \dfnemph{$k$-lifting} iff
$\pi$ is $\rSigma_0$-elementary with respect to the natural language 
$\Ll^H_{\mathrm{nat}}$ of $H$ (see \S\ref{sec:intro}) and
if $k>0$ then 
$\pi``T_k^H\sub T_k^M$.
\end{dfn}

A $k$-lifting embedding is similar to a $\Sigma_0^{(k)}$-preserving embedding 
of 
\cite{imlc}.
Note that $H,M$ may have different natural languages; maybe 
$\Ll^H_{\mathrm{nat}}\psub\Ll^M_{\mathrm{nat}}$.

\begin{lem}\label{lem:k-lifting_facts}
Let $H,M,k$ be as in \ref{dfn:k-lifting} and let $\pi:\core_0(H)\to\core_0(
M)$.
Then:
\begin{enumerate}
\item\label{item:characterize_k-lifting} $\pi$ is $k$-lifting iff 
for every $\rSigma_{k+1}$ formula $\varphi\in\Ll^H_{\mathrm{nat}}$ and 
$x\in\core_0(H)$, if
$\core_0(H)\sats\varphi(x)$ then 
$\core_0(M)\sats\varphi(\pi(x))$.
\item If $\pi$ is $k$-lifting and 
$H,M$ have different types then $k=0$, $H$ is passive 
and $M$ is active.
\item\label{item:k-lifting_Sigma_k_elem} If $k>0$ and $\pi$ is  $k$-lifting 
then 
$\pi$ is 
$\rSigma_k$-elementary, $(k-1)$-lifting, c-preserving.
\item\label{item:k>1_pi_Sigma_k_elem} Suppose $k>1$ and $\pi$ is 
$\rSigma_k$-elementary. Then $\pi$ is 
$p_{k-2}$-preserving and $\rho_{k-2}$-preserving, and if 
$\rho_{k-1}^H<\rho_0^H$ 
then
$\pi(p_{k-1}^H)=p_{k-1}^M\cut\pi(\rho_{k-1}^H)$ and
$(\sup\pi``\rho_{k-1}^H)\leq\rho_{k-1}^M\leq\pi(\rho_{k-1}^H)$.
\item\label{item:pi(p_k^H)} If $k>0$ and $\pi$ is $\rSigma_k$ elementary and 
$p_{k-1}$-preserving then $\pi(p_k^H)\leq p_k^M$.
\item\label{item:shift_lemma} The Shift Lemma holds with \emph{weak $k$-} 
replaced by 
\emph{$k$-lifting},
or by \emph{$k$-lifting c-preserving}.
\end{enumerate}
\end{lem}
\begin{proof}
Parts \ref{item:characterize_k-lifting}--\ref{item:k-lifting_Sigma_k_elem} are 
straightforward.
For part \ref{item:k>1_pi_Sigma_k_elem}, use $(k-1)$-solidity witnesses 
for $p_{k-1}$. For part \ref{item:pi(p_k^H)} use the fact that if $t$ is a 
$k$-solidity 
witness for $(H,p_k^H)$, then $\pi(t)$ is a generalized $k$-solidity witness 
for 
$(M,\pi(p_k^H))$.

Part \ref{item:shift_lemma}: We adopt the notation of \cite[Lemma 5.2]{fsit} 
(with $n=k$). Let $\Fbar=F^\Nbar$ and 
$\Ubar=\Ult_k(\Mbar,\Fbar)$ and $U=\Ult_k(M,F^N)$. Define the map
$\sigma:\core_0(\Ubar)\to\core_0(U)$
as 
there. It is straightforward to see that $\sigma$ is $\rSigma_k$-elementary. 
Suppose $k>0$. Let us 
observe that $\sigma``T_k^\Ubar\sub T_k^U$. Let $t\in T_k^\Ubar$,
$x\in\Ubar$ and 
$\alpha<\rho_k^\Ubar$ with
$t = \Th_{\rSigma_k}^\Ubar(\alpha\un\{x\})$.
Let $y\in\Mbar$ and $a\in\nu(\Fbar)^{<\om}$ with
$x\in\Hull_k^\Ubar(i^{\Mbar}_\Fbar(y)\un a)$.
Let $\beta<\rho_k^{\Mbar}$ be such that $\beta\geq\crit(\Fbar)$ and 
$i^{\Mbar}_\Fbar(\beta)\geq\alpha$. Let
$u = \Th_{\rSigma_k}^{\Mbar}(\beta\un\{y\})$.
Then $t$ is easily computed from $u'=i^{\Mbar}_{\Fbar}(u)$, and by 
commutativity,
$\sigma(u')\in T_k^U$.
It follows that $\sigma(t)\in T_k^U$, as required.
\end{proof}

\begin{rem}\label{rem:k-lifting_copying}
Clearly for $k<\om$, any $\rSigma_{k+1}$-elementary embedding is $k$-lifting. 
The author 
does not know whether ``weak 
$k$-'' implies ``$k$-lifting'', or vice versa. We will not 
deal with weak $k$-embeddings in this paper.

Standard arguments show that the copying construction propagates $k$-lifting 
c-preserving 
embeddings.
(But this may be false for weak $k$-embeddings; see \cite{recon_res}.)
Almost standard arguments 
show 
that $k$-lifting embeddings are propagated. That is, suppose $\pi:H\to M$ is 
$k$-lifting, and let $\Tt$ be a $k$-maximal iteration tree on $H$. We can define
$\Uu=\pi\Tt$ as usual, assuming it has wellfounded models. Let 
$H_\alpha=M^\Tt_\alpha$ and 
$M_\alpha=M^\Uu_\alpha$. Using the Shift Lemma as usual, we get
$\pi_\alpha:H_\alpha\to M_\alpha$
for each $\alpha<\lh(\Tt)$, and $\pi_\alpha$ is $\deg^\Tt(\alpha)$-lifting, and 
if $\pi$ 
is c-preserving, then so is $\pi_\alpha$. Let us just mention the extra details 
when $\pi$ 
fails to be c-preserving. In this case, $k=0$ and $H$ 
is passive. Suppose that $E^\Tt_0$ is total over $H$, and let 
$\kappa=\crit(E^\Tt_0)$. Suppose that 
$(\kappa^+)^H<\OR^H$ but $\pi((\kappa^+)^H)$ is not a cardinal of $M$. Then 
$\Uu$ drops in model at 
$1$, but $\Tt$ does not. Note though that $\rg(\pi)\sub M^{*\Uu}_1$ and
$\pi:H\to M^{*\Uu}_1$
is $0$-lifting (even if $M^{*\Uu}_1$ is active). So we can still produce
$\pi_1:H_1\to M_1$
via the 
Shift Lemma. This situation generalizes to an arbitrary $\alpha$ in place of 
$0$, when 
$\Tt$ does not drop in model along $[0,\alpha+1]_\Tt$. The other details 
are as usual.
Moreover, if (i) $[0,\alpha]_\Tt$ drops in model or (ii) $\deg^\Tt(\alpha)\leq 
k-2$ or (iii) 
$\deg^\Tt(\alpha)=k-1$ and $\pi$ is $p_{k-1}$-preserving, then $\pi_\alpha$ is 
a 
near 
$\deg^\Tt(\alpha)$-embedding; this uses the argument in \cite{fs_tame}.
\end{rem}

\begin{lem}\label{lem:k-lifting_dichotomy}
Let $k\geq 0$, let $\pi:H\to M$ be $k$-lifting where $H,M$ have the 
same type, and let $\rho_{k+1}^H\leq\rho\leq\rho_k^H$.
Then:
\begin{enumerate}
 \item\label{item:pi_k-embedding} If $p_{k-1}^M,p_k^M\in\rg(\pi)$ and 
$\rho_k^M=\sup\pi``\rho_k^H$ 
then $\pi$ is a $k$-embedding.
 \item\label{item:pi_bounded} If $H$ is $\rho$-sound, $\pi\rest\rho\in M$ 
and $\pi$ is \emph{not} a 
$k$-embedding then $H,\pi\rest\rho_k^H\in M$.
\end{enumerate}
\end{lem}
\begin{proof}
Part \ref{item:pi_k-embedding}: This is fairly routine. By 
\ref{lem:k-lifting_facts}, we have $\pi(p_{k-1}^H)=p_{k-1}^M$.
The $\rSigma_{k+1}$-elementarity of $\pi$ follows from this, together with the 
facts 
that $\pi$ is $k$-lifting, $p_k^M\in\rg(\pi)$ and $\pi``\rho_k^H$ is unbounded 
in $\rho_k^M$. Now 
let $\pi(q)=p_k^M$. Then
$p_k^H\leq q$ by \ref{lem:k-lifting_facts}, and $q\leq p_k^H$ by $\rSigma_k$ 
elementarity. And $\pi$ is $\rho_{k-1}$-preserving
by \ref{lem:k-lifting_facts} and 
$\rSigma_{k+1}$-elementarity. So $\pi$ is a $k$-embedding.

Part \ref{item:pi_bounded}: Suppose $\sup(\pi``\rho_k^H)<\rho_k^M$.
We use a stratifaction of  $\rSigma_{k+1}$ truth like in 
\cite[\S2]{fsit}.
Assuming familiarity with this, here is a sketch. Let 
$\alpha=\sup\pi``\rho_k^H$. Then 
the theory
$t=\Th_{\rSigma_k}^M(\alpha\un\pi(\vec{p}_k^H))$
is in $M$. Moreover, for any $\rSigma_{k+1}$ formula $\varphi$ and 
$\gammavec\in\rho^{<\om}$,
we have $H\sats\varphi(\gammavec,\vec{p}_{k+1}^H)$
iff there is $\beta<\alpha$ such that
$t\rest(\beta\un\pi(\vec{p}_k^H))$
is ``above'' a witness to 
$\varphi(\pi(\gammavec),\pi(\vec{p}_{k+1}^H))$ (see \cite[\S2]{fsit}). So the 
theory $\Th_{\rSigma_{k+1}}^H(\rho\cup\pvec_{k+1}^H)$ is computable from $t$ 
and $\pi\rest\rho$, so $H\in M$.
A little more work 
gives  $\pi\rest\rho_k^H\in M$, as desired.

Suppose now $\pi(p_{k-1}^H)=p_{k-1}^M$ but $\pi(p_k^H)\neq p_k^M$. Then 
$\pi(p_k^H)< p_k^M$, by \ref{lem:k-lifting_facts}. Again $t\in M$ ($t$ as 
above), as $t$ is 
computed from a 
$k$-solidity witness. The rest is the same.

Now suppose that $k>1$ and $\pi(p_{k-1}^H)\neq p_{k-1}^M$. By
\ref{lem:k-lifting_facts}, 
then $\pi(p_{k-1}^H)<p_{k-1}^M$. 
We may assume $\alpha=\rho_k^M=\sup\pi``\rho_k^H$.

\begin{clm*}
Let $\varphi$ be an $\rSigma_k$ formula, let $x\in H$ and 
$\gammavec\in\alpha^{<\om}$. If
$M\sats\varphi(\pi(x),\gammavec)$ then there is $\eps<\rho_{k-1}^M$, with 
$\max(\gammavec)<\eps$, 
such that the theory
\[ \Th_{\rSigma_{k-1}}^M(\eps\un\{\pi(x,\vec{p}_{k-1}^H)\}) \]
is ``above'' a witness to $\varphi(\pi(x),\gammavec)$.
\end{clm*}
\begin{proof}
Let $\delta<\rho_k^H$ be such that $\pi(\delta)>\max(\gammavec)$. Let
$v = \Th_{\rSigma_k}^H(\delta\un\{x\}\un\vec{p}_{k-1}^H)$.
Note then that for all $\xivec\in\delta^{<\om}$,
\[ (\varphi,(\xivec,x,\vec{p}_{k-1}^H))\in v\ \implies\  
(\psi_\varphi,(\xivec,x,\vec{p}_{k-1}^H))\in v,\] where 
$\psi_\varphi(\xivec,x,\vec{p}_{k-1}^H)$ asserts `There is $\eps<\rho_{k-1}$, 
with 
$\max(\xivec)<\eps$, such that the $\rSigma_{k-1}$ theory in parameters 
$\eps\un\{x\}\un\vec{p}_{k-1}^H$ is ``above'' a witness to
$\varphi(\xivec,x,\vec{p}_{k-1}^H)$'.
(Here the assertion that $\varepsilon<\rho_{k-1}$ does not require
the \emph{parameter} $\rho_{k-1}^H$. For note that the 
assertion ``$\dot{u}<\rho_{k-1}$'',
in the free variable $\dot{u}$, is $\rSigma_k$ 
without parameters, because it is just the assertion
``$\dot{u}\in\OR$ and there is $t\in T_{k-1}$ such that $t$ is a theory
in parameters from $\dot{u}$''.)
But then the same fact holds regarding $\pi(v)$, and since $\pi$ is $k$-liftng,
this proves the claim.
\end{proof}

By $(k-1)$-solidity, 
$u=\Th_{\rSigma_{k-1}}^M(\rho_{k-1}^M\un\pi(\vec{p}_{k-1}^H))$ is in $M$.
Define $t$ as before. By the claim, from $u$ we  compute $t$, so 
$t\in M$. The rest is now  as before.
\end{proof}

\begin{dfn}
Let $Q$ be a $k$-sound premouse. Let $\widetilde{\core}_0(Q)=\core_0(Q)$, and 
for 
$k>0$, let 
$\widetilde{\core}_k(Q)=(Q||\rho_k(Q),T')$,
where $T=\Th_{\rSigma_k}^Q(\rho_k\un\vec{p}_k^{\: Q})$, and $T'$ is 
given from $T$ by 
substituting $\vec{p}_k^{\: Q}$ for a constant symbol $c$.
\end{dfn}

\begin{dfn}
Let $k\geq 0$. Let $Q$ be a $k$-sound premouse with $\rho_k^Q>\om$. 
We say that $(U,\sigma^*)$ is \dfnemph{$k$-suitable for $Q$} iff
(i) $U,\sigma^*\in Q||\rho_k^Q$,
(ii) $U$ is a $k$-sound premouse with $\rho_k^U>\om$ and
(iii) $\sigma^*:\widetilde{\core}_k(U)\to\widetilde{\core}_k(Q)$ is 
$\Sigma_0$-elementary.
\end{dfn}

\begin{rem}
Clearly, if $(U,\sigma^*)$ is $k$-suitable for $Q$ then $\sigma^*$ extends 
uniquely to a 
$\vec{p}_k$-preserving $k$-lifting $\sigma:U\to Q$, and moreover,
$\sup\sigma``\rho_k^U<\rho_k^Q$.
Conversely, if 
$\sigma:U\to Q$ is $\vec{p}_k$-preserving $k$-lifting and 
$\sup\sigma``\rho_k^U<\rho_k^Q$ and 
$\sigma^*=\sigma\rest(U||\rho_k^U)$ with $\sigma^*\in Q$, then $(U,\sigma^*)$ 
is 
$k$-suitable for 
$Q$.
\end{rem}

\begin{lem}\label{lem:definability_k_suitable_triple}
Let $k\geq 0$. Then there is an $\rSigma_{k+1}$ formula $\varphi_k$ such that 
for all $k$-sound 
premice $Q$ with $\om<\rho_k^Q$, and all $U,\sigma^*\in Q$,
we have
\[ Q\sats\varphi_k(U,\sigma^*,\vec{p}_k^{\: Q})\ \iff (U,\sigma^*)\text{ is
}k\text{-suitable for }Q.\]
\end{lem}
\begin{proof}
We assume $k>0$ and leave the other case to the reader.
The most complex clause of $\varphi_k$ says ``There is $\alpha<\rho_k^Q$ such 
that letting
$t=\Th_{\rSigma_k}^Q(\alpha\un\vec{p}_k^{\: Q})$,
then for each $\beta<\rho_k^U$, letting
$u=\Th_{\rSigma_k}^U(\beta\un\vec{p}_k^{\: U})$,
and letting $t',u'$ be 
given from $t,u$ by substituting $\vec{p}_k^{\: Q},\vec{p}_k^{\: U}$ for the 
constant $c$, we have 
$\sigma^*(u')\sub t'$''. This statement is $\rSigma_{k+1}$. The rest is 
clear.
\end{proof}
\begin{dfn}
Let $m\geq 0$ and let $M$ be a segmented-premouse. Then $M$ is 
\dfnemph{$m$-sound} iff either $m=0$ or $M$ is an $m$-sound premouse.
\end{dfn}

\begin{dfn}\label{dfn:1st_order_con}
Let $r\geq 0$ and let $R$ be an $r$-sound premouse. Then we 
say that \dfnemph{suitable 
condensation holds at $(R,r)$} iff for every $(H,\pi^*)$, if $(H,\pi^*)$ is
$r$-suitable for $R$, $H$ is $(r+1)$-sound and
$\crit(\pi)\geq\rho=\rho_{r+1}^H$,
then either $H\pins R$, or $R|\rho$ is active with extender $F$ and 
$H\pins\Ult(R|\rho,F)$.

Let $m\geq 0$ and let $M$ be an $m$-sound segmented-premouse. We say that 
\dfnemph{suitable 
condensation holds below $(M,m)$} iff for every $R\ins M$ and 
$r<\om$ such that either $R\pins M$ or 
$r<m$, suitable condensation holds at $(R,r)$.
We say that \dfnemph{suitable 
condensation holds through $(M,m)$} iff $M$ is a premouse\footnote{We could 
have 
formulated this 
more generally for segmented-premice, but doing so would have increased 
notational load, and we do 
not need such a generalization.} and suitable condensation holds below and at 
$(M,m)$.
\end{dfn}
The following lemma follows easily from 
\ref{lem:definability_k_suitable_triple}:

\begin{lem}\label{lem:suitable_condensation_def}
Let $m\geq 0$. Then there is an $\rPi_{\max(m,1)}$ formula $\Psi_m$ such that 
for all 
$m$-sound segmented-premice $M$, suitable condensation holds below
$(M,m)$ iff 
$M\sats\Psi_m(\vec{p}_{m-1}^M)$, where $\vec{p}_{-1}^M=\emptyset$. Moreover, if 
$M$ 
is 
a premouse,
then suitable 
condensation holds through $(M,m)$ iff 
$M\sats\Psi_{m+1}(\vec{p}_m^M)$.\footnote{This clause only 
adds something because we do not assume that $M$ is $(m+1)$-sound.}
\end{lem}

\begin{rem}
Our proof of condensation from normal iterability (\ref{thm:condensation}) 
will use our 
analysis of bicephali and cephalanxes (\S\ref{sec:cephals}). This analysis 
will 
depend on the premice involved satisfying enough condensation, at levels lower 
in 
model or degree. As we 
will only have normal iterability, we can't appeal to the 
standard 
condensation theorem for this. One could get arrange everything inductively,
proving condensation and analysing bicephali and cephalanxes simultaneously. 
But it is simpler 
to use the following lemmas,
which will be generalized by \ref{thm:condensation}.
\end{rem}

\begin{lem}[Condensation for $\om$-sound 
mice]\label{lem:fully_elem_condensation}
Let $h\leq m<\om$ and $H,M$ be premice. Suppose 
 $H$ is $(h+1)$-sound, $M$ is $(m+1)$-sound, $M$ is $(m,\om_1+1)$-iterable,
and either $\rho_{m+1}^M=\om$ or $m\geq 
h+5$. Suppose $\pi:H\to M$ is $h$-lifting $\vec{p}_h$-preserving and
$\crit(\pi)\geq\rho=\rho_{h+1}^H$.
Then either
$H\ins M$ or $M|\rho$ is active with $H\pins\Ult(M|\rho,G)$.
\end{lem}
\begin{proof}
Let $\pi$, etc, be a counterexample. Let 
$\pi^*=\pi\rest(H||\rho_h^H)$.

We claim $H\in M$.
Suppose not. By \ref{lem:k-lifting_dichotomy}, $\pi$ is an $h$-embedding, 
and
$\rho\geq\rho_{h+1}^M$. Note $\pi(p_{h+1}^H)\leq p_{h+1}^M\cut\rho$ (use 
generalized 
solidity witnesses). If $\pi(p_{h+1}^H)<p_{h+1}^M\cut\rho$  we are done. 
Otherwise $\rho_{h+1}^M\un p_{h+1}^M\sub\rg(\pi)$, so $H=M$, contradiction.

We may assume $\rho_{m+1}^M=\om$, 
by replacing $M$ 
with $\bar{M}=\cHull_{m+1}^{M}(\vec{p}_m)$: all relevant facts
pass to $\bar{M}$ as $H\in M$, $\crit(\pi)\geq\rho$ and by 
\ref{lem:k-lifting_facts}(\ref{item:characterize_k-lifting}). Now use almost 
the usual condensation proof, but  when comparing the phalanx 
$(M,H,\rho)$ 
with $M$,  use
an $(m,h)$-maximal tree on $(M,H,\rho)$,  $m$-maximal  on $M$. As
$H\in M$, and using fine structure in place of weak 
Dodd-Jensen, this gives contradiction.
\end{proof}

\begin{lem}[Suitable condensation]\label{lem:suitable_condensation} Let $M$ be 
an $m$-sound, 
$(m,\om_1+1)$-iterable segmented-premouse. Then suitable condensation holds 
below $(M,m)$, and if 
$M$ is a premouse, through $(M,m)$.
\end{lem}
\begin{proof}
If $M$ is not a pm, use \ref{lem:fully_elem_condensation}. 
Suppose 
$M$ is a pm. By \ref{lem:suitable_condensation_def}, we may assume 
$\rho_{m+1}^M=\om$, replacing $M$ 
with
$\cHull^M_{m+1}(\vec{p}_m^M)$. Now argue as in
\ref{lem:fully_elem_condensation}.
\end{proof}

\section{The bicephalus \& the cephalanx}\label{sec:cephals}
\begin{dfn}\label{dfn:bicephalus}
An \dfnemph{exact bicephalus} is a tuple $B=(\rho,M,N)$ such that:
\begin{enumerate}
 \item $M$ and $N$ are premice.
 \item $\rho<\min(\OR^M,\OR^N)$ and $\rho$ is a cardinal of both $M$ and $N$.
 \item\label{item:bicephalus_agmt} $M||\rho^{+M}=N||\rho^{+N}$.
 \item $M$ is $\rho$-sound and for some $m\in\{-1\}\un\om$, we have
$\rho_{m+1}^M\leq\rho$. Likewise for $N$ and $n\in\{-1\}\un\om$.
\end{enumerate}
We say $B$ is \textbf{non-trivial} iff $M\neq N$.
Write $\rho^B=\rho$ and $M^B=M$
and
$N^B=N$, and
$m^B,n^B$ for the least $m,n$ as above. Let $(\rho^+)^B$ be
$\rho^{+M}=\rho^{+N}$. We say $B$ has \dfnemph{degree}
$(m^B,n^B)$. We say that $B$ is \dfnemph{sound} iff $M$ is $m^B+1$-sound and
$N$ is $n^B+1$-sound.
\end{dfn}

From now on we 
just say \emph{bicephalus}
instead of \emph{exact bicephalus}.
In connection with bicephali of degree $(m,n)$ with $\min(m,n)=-1$, we need the 
following:

\begin{dfn}
The terminology/notation \dfnemph{(near) $(-1)$-embedding}, 
\dfnemph{$(-1)$-lifting embedding}, 
$\Ult_{-1}$, $\core_{-1}$, and \dfnemph{degree $(-1)$ iterability} are defined 
by replacing `$-1$' 
with `$0$'. For $n>-1$ and appropriate premice $M$, the core embedding 
$\core_{n}(M)\to\core_{-1}(M)$ is just the core embedding 
$\core_n(M)\to\core_0(M)$.
\end{dfn}

\begin{dfn}Let $q<\om$.
A \dfnemph{passive right half-cephalanx} of \dfnemph{degree} $q$ is a tuple 
$B=(\gamma,\rho,Q)$ such that:
\begin{enumerate}
\item $Q$ is a premouse,
\item $\gamma$ is a cardinal of $Q$ and $\gamma^{+Q}=\rho<\OR^Q$,
\item $Q$ is $\gamma$-sound,
\item $\rho_{q+1}^Q\leq\gamma<\rho_q^Q$.
\end{enumerate}

An \dfnemph{active right half-cephalanx} (of \dfnemph{degree} $q=0$) is a 
$B=(\gamma,\rho,Q)$ with:
\begin{enumerate}
\item $Q$ is an active segmented-premouse,
\item $\gamma$ is the largest cardinal of $Q$ and $\gamma<\rho=\OR^Q$.
\end{enumerate}

A \dfnemph{right half-cephalanx} $B$ is either a passive, or active, right 
half-cephalanx.
We write $\gamma^B,\rho^B,Q^B,q^B$ for $\gamma,\rho,Q,q$ as above. If $B$ is 
active, we write 
$S^B=R^B=\Ult(Q,F^Q)$. If $B$ is passive, we write $S^B=Q$.
\end{dfn}

Note that if $B=(\gamma,\rho,Q)$ is a right-half cephalanx, then $B$ is active 
iff $Q|\rho$ is 
active. So it might be that $B$ is passive but $Q$ is active with $\rho<\OR^Q$.

\begin{dfn}\label{dfn:cephalanx}
Let $m\in\{-1\}\un\om$ and $q<\om$. A \dfnemph{cephalanx} of \dfnemph{degree 
$(m,q)$} is a 
tuple $B=(\gamma,\rho,M,Q)$ such that, letting $B'=(\gamma,\rho,Q)$, we have:
\begin{enumerate}
 \item $(\gamma,\rho,Q)$ is a right-half cephalanx of degree $q$,
 \item $M$ is a premouse,
 \item $\rho=\gamma^{+M}<\OR^M$,
 \item $M||\rho^{+M}=S^{B'}||\rho^{+M}$,
 \item $M$ is $\rho$-sound,
 \item $\rho_{m+1}^M\leq\rho<\rho_m^M$.
\end{enumerate}

We say that $B$ is \dfnemph{active} (\dfnemph{passive}) iff $B'$ is active 
(passive). 
\footnote{Note that a passive cephalanx $(\gamma,\rho,M,Q)$ might be such that 
$M$ and/or $Q$ 
is/are active.}
We write $\gamma^B,\rho^B$, etc, for $\gamma,\rho$, etc. We write $R^B$ for 
$R^{B'}$, if it is 
defined, and $S^B$ for $S^{B'}$. We say $B$ is \dfnemph{exact} iff 
$(\rho^+)^{S^B}=\rho^{+M}$, and $B$ is \dfnemph{sound}
iff $M$ is $(m+1)$-sound.

Suppose $B$ is active. Let $R=R^B$. We say $B$ is \dfnemph{non-trivial} iff 
$M\npins R$. If $B$ 
is non-exact, let $N^B$ 
denote the $N\pins R$ such that $\rho^{+N}=\rho^{+M}$ and $\rho_\om^N=\rho$, 
and let $n^B$ denote
the $n\in\{-1\}\un\om$ such that $\rho_{n+1}^N=\rho<\rho_n^N$.

Now suppose $B$ is passive. We say $B$ is 
\dfnemph{non-trivial} iff $M\nins Q$. Let $N^B$ denote the $N\ins Q$ such that 
$\rho^{+N}=\rho^{+M}$ and $\rho_\om^N\leq\rho$. Let $n^B$ be the 
$n\in\{-1\}\un\om$ such that 
$\rho_{n+1}^{N^B}\leq\rho<\rho_n^{N^B}$.

A \textbf{pm-cephalanx} is a cephalanx $(\gamma,\rho,M,Q)$ such that $Q$ is a 
premouse.\qedhere
\end{dfn}

\begin{dfn}\label{dfn:cephal}
A \dfnemph{cephal} is either a bicephalus or a cephalanx.
Let $B$ be a cephal, and let $M=M^B$.

A short extender $E$ is \dfnemph{weakly amenable
to $B$} iff $\crit(E)<\rho^B$
and $E$ is weakly amenable to $M$.

For $\alpha\leq\rho^B$, let $B||\alpha=M||\alpha$, and for $\alpha<\rho^B$, let 
$B|\alpha=M|\alpha$ 
and $(\alpha^+)^B=(\alpha^+)^M$. We write $P\pins B$ iff $P\pins B||\rho^B$.
Let $C,\alpha$ be such that $\alpha\leq\rho^B$, and either $C$ is a 
segmented-premouse and 
$\alpha\leq\OR^C$, or $C$ is a cephal and $\alpha\leq\rho^C$. Then we write
$(B\sim C)||\alpha$ iff $B||\alpha=C||\alpha$.
If also $\alpha<\rho^B$ and either $C$ is a segmented-premouse or 
$\alpha<\rho^C$, we use the same 
notation with ``$|$'' replacing ``$||$''. We also use similar notation with 
more than two 
structures.

A structure with the first order properties of a cephal
or other related structures is \dfnemph{wellfounded}
if each of the constituent models are wellfounded.
\end{dfn}

We will consider ultrapowers and iterates of cephals,
and also other related structures, and hence the wellfoundedness of such 
iterates.
Because of the symmetry of bicephali and partial symmetry of cephalanxes, 
we often state 
facts for just one side of the symmetry,  even when they hold for both.
The proofs of the next two lemmas are routine and  omitted.
In \ref{lem:ult_comm_moving_crit_F^Q}--\ref{rem:apply_to_iteration_map} below, 
the extender $E$ 
might be long.

\begin{lem}\label{lem:ult_comm_moving_crit_F^Q}
Let $Q$ be an active segmented-premouse. Let $E$ be
an extender over $Q$ with $\ms(E)\leq\crit(F^Q)+1$. Let
$R=\Ult(Q,F^Q)$ and $Q'=\Ult(Q',E)$ and
$R'=\Ult(Q',F^{Q'})$. Then $R'=\Ult(R,E)$ and the ultrapower embeddings
commute. Moreover, $i^R_E=\psi_{i^Q_E}$.
\end{lem}

\begin{lem}\label{lem:ult_comm_not_moving_crit_F^Q}
Let $Q$ be an active segmented-premouse. Let $E$ be
an extender over $Q$ with $(\crit(F^Q)^+)^Q<\crit(E)$. Let
$R=\Ult(Q,F^Q)$ and $R^*=\Ult(R,E)$ and $Q'=\Ult(Q,E)$. Then
$\Ult(Q,F^{Q'})=R^*$ and the ultrapower embeddings commute.\footnote{Note that 
in the conclusion, 
it is
$\Ult(Q,F^{Q'})$, not $\Ult(Q',F^{Q'})$.} Let $\psi:R\to R^*$ be given by the
Shift Lemma \tu{(}applied to $\id:Q\to Q$
and $i^Q_E$\tu{)}. Then $i^R_E=\psi$.
\end{lem}

\begin{dfn}
Let $E$ be a (possibly long) extender.
For a seg-pm $M$,
$E$ is
\dfnemph{reasonable for $M$} iff $E$ is over $M$
and either $M$ is 
passive or
letting $\kappa=\crit(F^M)$, $i^M_E$ is continuous at
$(\kappa^+)^M$, and if $M\sats$``$\kappa^{++}$ exists'' then $i^M_E$ is 
continuous at 
$(\kappa^{++})^M$.
For a bicephalus $B=(\rho,M,N)$, $E$ is \dfnemph{reasonable for 
$B$} iff $E$ 
is over $B||\rho$, if $m^B\leq 0$ then $E$ is reasonable for $M$, and if 
$n^B\leq 0$ then $E$ 
is reasonable for $N$.
For a cephalanx $B=(\gamma,\rho,M,Q)$,  $E$ is \dfnemph{reasonable 
for $B$} iff $E$ is 
over 
$B||\rho$,\footnote{\label{ftn:measure_space_E_cephal}Hence $E$ is equivalent to
an extender $E'$ with $\ms(E')\leq\gamma+1$.} if $q^B\leq 0$ then $E$ is 
reasonable for $Q$, if $m^B\leq 0$ then $E$ is 
reasonable for $M$, and if 
$N^B$ is defined and $n^B\leq 0$ then $E$ is reasonable for $N^B$.
\end{dfn}

\begin{lem}\label{lem:ult_comm}
Let $Q$ be an active seg-pm, $E$ reasonable
for $Q$. Let $Q'=\Ult(Q,E)$, $R=\Ult(Q,F^Q)$, $R'=\Ult(Q',F^{Q'})$,
$R^*=\Ult(R,E)$. Let $\kappa=\crit(F^Q)$ and $\eta=(\kappa^{++})^Q$. Further:
\begin{enumerate}[label=--]
\item If 
$\eta<\OR^Q$, let
$\gamma=i_{QR}(\eta)$, $\gamma^*=i^R_E(\gamma)$, $\eta'=i^Q_E(\eta)$.
Then $\gamma^*=i_{Q'R'}(\eta')$.
\item If $\eta=\OR^Q$, let $\gamma=\OR^R$, 
$\gamma^*=\OR^{R^*}$ and 
$\eta'=\OR^{Q'}$.
\end{enumerate}
Then in either case,
$(R^*\sim R')|\gamma^*$ and
\[ i^R_E\com i_{QR}\rest(Q|\eta) =i_{Q'R'}\com i^Q_E\rest(Q|\eta).\]
Moreover, let
$\psi:R|\gamma\to R'|\gamma'$ be induced by the Shift Lemma with
$i_{QQ'}\rest (Q|\eta)$ and $i_{QQ'}$. Then $\psi=i^R_E\rest(R|\gamma)$.
\end{lem}

\begin{proof}
Let $G$ be the length $i_E(\kappa)$ extender derived from $E$. Let
$j:\Ult(Q,G)\to\Ult(Q,E)$
be the factor embedding. Then $\crit(j)>(i^Q_G(\kappa)^{++})^{U_G}$ since $E$ 
is 
reasonable. Apply
\ref{lem:ult_comm_moving_crit_F^Q} to $G$, and then
\ref{lem:ult_comm_not_moving_crit_F^Q} to the extender derived from
$j$.
\end{proof}

\begin{dfn}\label{dfn:expansion}
Let $M$ be a type 3 premouse. The \dfnemph{expansion} of $M$ is the active 
segmented-premouse 
$M_*$ such that $M_*|\crit(F^{M_*})=M|\crit(F^M)$, and $F^{M_*}$ is 
the Jensen-indexed version of $F^M$. That is, let $F=F^M$, let $\mu=\crit(F)$, 
let 
$\gamma=(\mu^+)^M$, let $\gamma'=i_{F}(\gamma)$, let 
$R=\Ult(M,F^M)$; then
$M_*||\OR(M_*)=R|\gamma'$, and $F^{M_*}$ is the length $i_F(\mu)$ extender 
derived from 
$i_F$.\end{dfn}

Combining \cite[\S9]{fsit} with a simple variant of 
\ref{lem:ult_comm} one gets:
 
\begin{fact}\label{fact:type_3_ult}
Let $Q$ be a type 3 premouse. Let $E$ be an extender
over $Q^\sq$, reasonable for $Q$. Let $Q_*$ be the expansion of $Q$, let 
$U_*=\Ult(Q_*,E)$ 
and $U=\Ult_0(Q,E)$. Suppose $U_*$ is 
wellfounded. Then $U$ is wellfounded and $U_*$ is its expansion. Let
$i_*:Q_*\to U_*$ and $i_0:Q\to U$ be the ultrapower
maps  \tu{(}literally, $\dom(i_*)=Q_*$ and $\dom(i_0)=Q^\sq$\tu{)}. 
Then 
$i_0=i_*\rest 
Q^\sq$, and $i_*=\psi_{i_0}\rest Q_*$.\end{fact}

\begin{rem}\label{rem:apply_to_iteration_map}
We will apply \ref{lem:ult_comm} and
\ref{fact:type_3_ult} when $E$
is the extender of  iteration map $i^\Tt_{\alpha\beta}$ or 
$i^{*\Tt}_{\alpha\beta}$, where
$(\alpha,\beta]_\Tt$ does not drop and $\deg^\Tt(\alpha)=0$.
\end{rem}

\begin{dfn}[Ultrapowers of bicephali]\label{dfn:biceph_ult} Let $B=(\rho,M,N)$
be a bicephalus of degree $(m,n)$ and $E$  an extender reasonable for 
$B$. 
We have the ultrapower map
$i=i_E^{M,m}:M\to\Ult_{m}(M,E)$,
and  $j=i_E^{N,n}$. Let
$\rho' = \sup i``\rho=\sup j``\rho$
and define
\[ \Ult(B,E) = (\rho',\Ult_{m}(M,E),\Ult_{n}(N,E)).\qedhere\]
\end{dfn}

\begin{dfn}\label{dfn:augmented_biceph}
Let $B$ be a bicephalus. The associated \dfnemph{augmented bicephalus} is the 
tuple
$B_*=(\rho,M,N,M_*,N_*)$
where if $m\geq 0$ then $M_*=M$, and otherwise $M_*$ is the expansion of $M$; 
likewise for $N_*$.
(Note that if $m=-1$ then $M$ is type 3 
and $\rho=\nu(F^M)$.)
Let $E$ reasonable for $B$. If $m\geq 0$ let $\widetilde{M}=\Ult_m(M,E)$; 
otherwise 
let 
$\widetilde{M}=\Ult(M_*,E)$. Likewise for $\widetilde{N}$. We define
$\Ult(B_*,E) = \Ult(B,E)\conc\left<\widetilde{M},\widetilde{N}\right>$.
\end{dfn}

\begin{lem}\label{lem:biceph_ult} Let $B=(\rho,M,N)$ be a bicephalus of degree 
$(m,n)$, $B_*=(\rho,M,N,M_*,N_*)$, $E$ reasonable for $B$, 
$U=\Ult(B,E)$ and
$\widetilde{U}=\Ult(B_*,E)=(\rho^U,M^U,N^U,\widetilde{M},\widetilde{N})$. 
Let $i=i^{M,m}_E$ and $j=j^{N,n}_E$.
If $m\geq 0$ let 
$i_*=i$;  otherwise let $i_*:M_*\to\widetilde{M}$
be the ultrapower map. Likewise $j_*$.
If
$\widetilde{U}$ is 
wellfounded then:
\begin{enumerate}[label=\tu{(}\arabic*\tu{)}]
\item\label{item:ult_is_aug_biceph} $U$ is a bicephalus of degree $(m,n)$ and 
$\widetilde{U}=U_*$.
\item\label{item:triviality} $U$ is trivial iff $B$ is trivial.
\item\label{item:param_above_rho_preserved} 
$i(p_{m+1}^{M}\cut\rho)=p_{m+1}^{M^U}\cut\rho^U$.
\item\label{item:biceph_i_E_agmt}
$i_*\rest(\rho^+)^B=j_*\rest(\rho^+)^B$ and $i_*,j_*$ are 
continuous/cofinal at $(\rho^+)^B$.\footnote{That is, if 
$(\rho^+)^B\in\dom(i_*)$ then $i_*$ is continuous there; if $m\geq 
0$ and 
$(\rho^+)^B=\rho_0^M$ then $\rho_0^{M^U}=\sup i``(\rho^+)^B$; if $m=-1$ 
and 
$(\rho^+)^B=\OR(M_*)$ then $\OR((M^U)_*)=\sup i_*``(\rho^+)^B$.}
\item\label{item:m=-1_psi_agmt} $i_*=\psi_{i}\rest M_*$.
\item\label{item:sound_equiv} Suppose $E$ is short and weakly amenable to $B$. 
Then 
$M^U$ is 
$(m+1)$-sound iff $M$ is
$(m+1)$-sound and $\crit(E)<\rho_{m+1}^M$. If $M^U$ is $(m+1)$-sound then
$\rho_{m+1}^{M^U}=\sup i``\rho_{m+1}^M$ and
$p_{m+1}^{M^U}=i(p_{m+1}^{M})$.
\end{enumerate}
Likewise regarding $N,n,E$.
\end{lem}
\begin{proof}
Part \ref{item:sound_equiv} is by \cite[2.20]{extmax}, 
\ref{item:param_above_rho_preserved} is a standard calculation using 
generalized 
solidity witnesses 
(see \cite{imlc}), and \ref{item:m=-1_psi_agmt} is by \ref{fact:type_3_ult} 
(\ref{item:m=-1_psi_agmt} is trivial when $m\geq 0$).

Part \ref{item:biceph_i_E_agmt}: Let
$W=\Ult(B||(\rho^+)^B,E)$
and $k:B||(\rho^+)^B\to W$ the ultrapower map
and $\widetilde{\rho}=k(\rho)$.
We claim that ($\dagger$):
$k=i_*\rest(\rho^+)^B$ and 
$\widetilde{M}||(\widetilde{\rho}^+)^{\widetilde{M}}=W$.

If $m\leq 0$ this is 
immediate. If $m>0$, then because $(\rho^+)^B\leq\rho_{m}^{M}$, by 
\cite[\S 6]{fsit}, all functions forming the
ultrapower $M^U$ with codomain $(\rho^+)^B$ are in fact in $B||(\rho^+)^B$, 
which gives ($\dagger$).
Now \ref{item:biceph_i_E_agmt} follows from ($\dagger$).

Part 
\ref{item:ult_is_aug_biceph}:
By \ref{fact:type_3_ult}, $\widetilde{M}$ is the expansion of $M^U$. We
have 
$\rho^U\leq\widetilde{\rho}$ and by ($\dagger$),
$\widetilde{M}||(\widetilde{\rho}^+)^{\widetilde{M}}=
\widetilde{N}||(\widetilde{\rho}^+)^{\widetilde{N}}$.
If $m\geq 0$ then 
$\widetilde{\rho}<\rho_m(M^U)$. The rest of 
\ref{item:ult_is_aug_biceph} is routine.

Part \ref{item:triviality}: Assume $M\neq N$ but $m=n$. 
By $\rho$-soundness, there is an $\rSigma_{m+1}$ formula
$\varphi$ and
$\alpha<\rho$ with
\[ M\sats\varphi(p_{m+1}^{M}\cut\rho,\pvec_m^M,\alpha)\ \ \ 
\iff\ \ \ N\sats\neg\varphi(p_{m+1}^{N}\cut\rho,\pvec_m^N,\alpha).\]
Now $i,j$ are $\rSigma_{m+1}$-elementary, and by ($\dagger$), 
$i(\alpha)=j(\alpha)$; let $\alpha'=i(\alpha)$. So by 
\ref{item:param_above_rho_preserved}, we get $M^U\neq N^U$ and in fact
\[ M^U\sats\varphi(p_{m+1}^{M^U}\cut\rho^U,\pvec_m^{M^U},\alpha')\ \ \ \iff\ \ 
\ 
N^U\sats\neg\varphi(p_{m+1}^{N^U}\cut\rho^U,\pvec_m^{N^U},\alpha').\qedhere\]
\end{proof}

\begin{dfn}[Ultrapowers of cephalanxes]\label{dfn:uceph_ult} Let 
$B=(\gamma,\rho,M,Q)$ be a
cephalanx of degree $(m,q)$ and let $E$ be reasonable for $B$. Let 
$M'=\Ult_m(M,E)$,
$\gamma'=i^{M,m}_E(\gamma)$ and $\rho'=\sup 
i^{M,m}_E``\rho$. If $B$ 
is 
active let $Q'=\Ult(Q,E)$
(recall the ultrapower $\Ult(Q,E)$ is simple; it might be that $Q$ is type 
3, and we could have $\crit(E)=\gamma$.)
If $B$ is passive let $Q'=\Ult_q(Q,E)$.
Then we define $\Ult(B,E)=(\gamma',\rho',M',Q')$.
\end{dfn}

\begin{lem}\label{lem:passive_cephalanx_ult} In the context of 
\ref{dfn:uceph_ult}, suppose that 
$B$ is passive, 
and that $U=\Ult(B,E)$ is wellfounded.  Let $i=i^{M,m}_E$ and $j=i^Q_E$. 
Then:
\begin{enumerate}[label=\tu{(}\arabic*\tu{)}]
\item\label{item:ult_ceph} $U$ is a passive cephalanx of degree $(m,q)$.
\item\label{item:i_E_agmt_rho} $i\rest\rho=j\rest\rho$.
\item If $\rho\in\core_0(M)$ then $\rho'=i(\rho)$; otherwise 
$\rho'=\rho_0(M^U)$. 
Likewise $Q,j,Q^U$.
\item\label{item:psi_i_E_continuity_rho} 
$\psi_{i}(\rho)=\psi_j(\rho)=\rho'$.
\item\label{item:psi_i_E_continuity_rho^+} If $\rho^{+M}\in\dom(\psi_{i})$ 
then $\psi_{i}$ 
is continuous at $\rho^{+M}$;
otherwise $M$ is passive, $\OR^M=\rho^{+M}$ and $\OR(M^U)=\sup i``\OR^M$.
\item\label{item:psi_i_E_agmt_rho^+} 
$\psi_{i}\rest\rho^{+M}=\psi_{j}\rest\rho^{+M}$.
\item\label{item:param_preserved} 
$i(p_{m+1}^{M}\cut\rho)=p_{m+1}^{M^U}\cut\rho'$.
\item\label{item:soundness_preserved} Suppose $E$ is short and weakly amenable 
to 
$B$. Then $M^U$ is 
$(m+1)$-sound iff $M$ is $(m+1)$-sound and $\crit(E)<\rho_{m+1}^M$. If $M^U$ is 
$(m+1)$-sound then 
$\rho_{m+1}^{M^U}=\sup i``\rho_{m+1}^M$ and 
$p_{m+1}^{M^U}=i(p_{m+1}^M)$.
\item\label{item:exactness_preserved} If $B$ is non-exact then $U$ is non-exact.
\item\label{item:exactness_lost} If $B$ is exact \tu{(}so $N^B=Q$\tu{)} but $U$ 
is not, then $0\leq 
n^B<q$.
\item\label{item:non_triviality_preserved} Suppose that $B$ is non-trivial and 
that
suitable condensation holds below $(Q,q)$. Let $N=N^B$ and $n=n^B$.
Then:
\begin{enumerate}[label=\tu{(}\roman*\tu{)}]
\item\label{item:U_non-trivial} $U$ is non-trivial,
\item\label{item:N^U=Ult_n(N,E)} $N^U=\Ult_{n^B}(N,E)$ and $n^U=n$,
\item\label{item:N_pres} Parts 
\ref{item:i_E_agmt_rho}--\ref{item:soundness_preserved} hold with 
`$M$' 
replaced by `$N$' and `$m$' by `$n$'.
\end{enumerate}
\end{enumerate}
\end{lem}
We also have $j(p_{q+1}^Q\cut\gamma)=p_{q+1}^{Q^U}\cut\gamma^U$, but we 
won't need this.
\begin{proof}
Parts 
\ref{item:i_E_agmt_rho}--\ref{item:soundness_preserved} are much as in 
\ref{lem:biceph_ult}. (For
\ref{item:psi_i_E_agmt_rho^+}, note that for $A\in\pow(\rho)\inter M$, the 
value 
of $\psi_{i}(A)$ is determined by the values of 
$\psi_{i}(A\inter\alpha)$ for $\alpha<\rho$;
likewise for $\psi_{j}(A)$.) So \ref{item:ult_ceph} follows.
Part \ref{item:exactness_preserved} follows from 
\ref{item:psi_i_E_continuity_rho^+} and \ref{item:psi_i_E_agmt_rho^+}; part 
\ref{item:exactness_lost} is easy.

Part \ref{item:non_triviality_preserved}: Consider
the case that $B$ is 
exact. Part 
\ref{item:N_pres} is  as for $M$, so consider 
\ref{item:U_non-trivial} and \ref{item:N^U=Ult_n(N,E)}.
As $B$ is exact, 
$N=Q$. By the proof of
\ref{lem:biceph_ult}, we have $\Ult_n(Q,E)\neq\Ult_m(M,E)$, so it suffices to 
see 
that
\[ U_n=\Ult_n(Q,E)\ins Q^U=\Ult_q(Q,E)=U_q,\]
so assume $n<q$. If $n=-1$ then $U_n=U_q$, so 
assume $n\geq 0$, so $\rho\in\core_0(Q)$ and
\[ \rho_{q+1}^Q\leq\gamma<\rho=\rho_q^Q=\rho_{n+1}^Q<\rho_n^Q. \]
Let $\sigma:U_n\to U_q$ be the natural factor map.
Let $j_n:Q\to U_n$ and $j_q:Q\to U_q$ be the ultrapower maps. Then $\sigma\com 
j_n=j_q$,
$\sigma$ is $\vec{p}_{n+1}$-preserving $n$-lifting and $\crit(\sigma)>\rho'$.
Also, $U_n,U_q$ are $(n+1)$-sound 
and $\rho_{n+1}^{U_n}=\rho'=\rho_{n+1}^{U_q}$.

Suppose $((\rho')^+)^{U_n}=((\rho')^+)^{U_q}<\crit(\sigma)$. Then
$\rho_n^{U_q}=\sup\sigma``\rho_n^{U_n}$,
since otherwise, using the previous paragraph and as in the proof of 
\ref{lem:k-lifting_dichotomy}, we get $U_n\in U_q$, collapsing 
$((\rho')^+)^{U_q}$ 
in $U_q$. So by \ref{lem:k-lifting_dichotomy}, $\sigma$ is an 
$n$-embedding. But
$\rho_{n+1}^{U_q}\un 
p_{n+1}^{U_q}\sub\rg(\sigma)$,
so then $U_n=U_q$, which suffices.

Now suppose  $((\rho')^+)^{U_n}<((\rho')^+)^{U_q}$. Then much as before,
$\rho_n^{U_q}>\sup\sigma``\rho_n^{U_n}$.
Let $\sigma^*=\sigma\rest(U_n||\rho_n^{U_n})$. By \ref{lem:k-lifting_dichotomy} 
we get 
$U_n,\sigma^*\in U_q$
and $(U_n,\sigma^*)$ is $n$-suitable for $U_q$. By suitable condensation 
below $(Q,q)$
and \ref{lem:suitable_condensation_def}
and since $U_q|\rho'$ is passive, therefore $U_n\pins U_q$, which 
suffices.

If $B$ is non-exact, so $N\pins Q$,  let
$U_n=\Ult_n(N,E)$, consider the factor embedding
$\sigma:U_n\to j(N)$
and show $U_n\ins j(N)$.
This completes the proof.
\end{proof}

\begin{lem}\label{lem:active_cephalanx_ult} In the context of 
\ref{dfn:uceph_ult}, suppose that $B$ 
is active, 
and that $U=\Ult(B,E)$ and $R^U$ are wellfounded. Let $i=i^{M,m}_E$ and 
$j=i^Q_E$. Then:
\begin{enumerate}[label=\tu{(}\arabic*\tu{)}]
\item\label{item:active_ult_ceph} $U$ is an active cephalanx of degree $(m,0)$. 
\item\label{item:active_i_E^M_continuity_rho} If $\rho\in\core_0(M)$ then 
$\rho'=i(\rho)$; 
otherwise $\rho'=\rho_0(M^U)$.
\item\label{item:active_i_E^M_preservation} 
\ref{lem:passive_cephalanx_ult}\ref{item:i_E_agmt_rho} and
\ref{lem:passive_cephalanx_ult}\ref{item:psi_i_E_continuity_rho}--
\ref{item:soundness_preserved}
 hold.
\item\label{item:active_exactness_preserved} $U$ is exact iff $B$ is exact.
\item\label{item:active_non_triviality_preserved} Suppose that $B$ is non-exact 
and non-trivial and 
that
suitable condensation holds below $(Q,0)$. Let $N=N^B$ and $n=n^B$.
Then:
\begin{enumerate}[label=\tu{(}\roman*\tu{)}]
\item\label{item:active_U_non-trivial} $U$ is non-trivial,
\item\label{item:active_N^U=Ult_n(N,E)} $N^U=\Ult_{n^B}(N,E)$ and $n^U=n$,
\item\label{item:active_N_pres} Parts 
\ref{item:active_i_E^M_continuity_rho}--\ref{item:active_i_E^M_preservation} 
hold with `$M$' 
replaced by `$N$' and `$m$' by `$n$'.
\end{enumerate}
\end{enumerate}
\end{lem}
\begin{proof}
This follows from \ref{lem:ult_comm}, \ref{fact:type_3_ult} and the proof of 
\ref{lem:passive_cephalanx_ult}.\footnote{In \ref{lem:ult_comm} we set 
$\eta=(\kappa^{++})^Q$, and 
the reader might wonder why we didn't just use $\eta=(\kappa^+)^Q$. We need the 
larger value here 
if $Q$ has superstrong type.}
\end{proof}

\begin{lem}\label{lem:ceph_it}
Let $C$ be a  degree $(m,k)$ cephal. If $C$ is a bicephalus let $B=C_*$; 
otherwise let 
$B=C$. Let
$\left<E_\alpha\right>_{\alpha<\lambda}$ be a
sequence of short extenders. Let $B_0=B$,
$B_{\alpha+1}=\Ult(B_\alpha,E_\alpha)$, and
 $B_\gamma$ be the direct limit at limit $\gamma$. Suppose for each
$\alpha\leq\lambda$, $B_\alpha$ is wellfounded and if $\alpha<\lambda$
then $E_\alpha$ is weakly amenable to
$B_\alpha$.

If $C$ is a bicephalus \tu{(}passive cephalanx, active cephalanx, 
respectively\tu{)} then the 
conclusions of \ref{lem:biceph_ult} \tu{(}\ref{lem:passive_cephalanx_ult}, 
\ref{lem:active_cephalanx_ult}, respectively\tu{)} apply to $B$ and 
$B_\lambda$, 
together 
with the associated iteration maps, after deleting the sentence ``Suppose 
$E$ is short 
and weakly amenable to $B$.'' and replacing the phrase
``$\crit(E)<\rho_{m+1}^M$'' with ``$\crit(E_\alpha)<\rho_{m+1}^{M_\alpha}$ for
each $\alpha<\gamma$''.
\end{lem}
\begin{proof}
If $C$ is a bicephalus, this mostly follows from \ref{lem:biceph_ult}, 
\cite[2.20]{extmax} and 
\ref{fact:type_3_ult} by induction.
At limit stages, use \cite[2.20]{extmax} directly to prove
\ref{lem:biceph_ult}\ref{item:sound_equiv}. To see 
\ref{lem:biceph_ult}\ref{item:biceph_i_E_agmt}, 
replace the iteration used to define $C_\gamma$ with a single (possibly long) 
extender $E$, and 
apply \ref{lem:biceph_ult}. The cephalanx cases are similar.
\end{proof}

\begin{dfn}[Iteration trees on bicephali]\label{dfn:bicephalus_tree}
Let $B=(\rho,M,N)$ be a bicephalus of degree $(m,n)$ and let 
$\eta\in\OR\cut\{0\}$. An 
\dfnemph{iteration tree on $B$}, of \dfnemph{length $\eta$}, is a pair
$\Tt=\left(<_\Tt,
\left<E_\alpha\right>_{\alpha+1<\eta}\right)$
such that there are sequences
\[ \left<B_\alpha,M_\alpha,N_\alpha\right>_{\alpha<\eta}\text{ and }
\left<B^*_{\alpha+1},M^*_{\alpha+1},N^*_{\alpha+1}\right>_{\alpha+1<\eta}, \]
of models, sequences of embeddings
\[ \left<i_{\alpha\beta},j_{\alpha\beta}\right>_{\alpha,\beta<\eta}\text{ and }
 \left<i^*_{\alpha+1},j^*_{\alpha+1}\right>_{\alpha+1<\eta},\]
sequences of ordinals
$\left<\rho_\alpha\right>_{\alpha<\eta}$ and
$\left<\crit_\alpha,\nu_\alpha,\lh_\alpha,\rho^*_{\alpha+1}\right>_{
\alpha+1<\eta}$,
sets $\curlyB,\curlyM,\curlyN\sub\eta$ (specifying types and origins of 
structures), a function 
``$\deg$'' with domain $\eta$ (specifying 
degrees), and
a set $\curlyD\sub\eta$ (specifying drops in model), such that:
\begin{enumerate}
 \item $<_\Tt$ is an iteration tree order on $\eta$, with the usual properties.
\item $B_0=(\rho_0,M_0,N_0)=B$ and $\deg(0)=(m,n)$ and $i_{00}=\id$ and 
$j_{00}=\id$.
 \item $\curlyB,\curlyM,\curlyN$ are disjoint and for each $\alpha<\eta$, either
 \begin{enumerate}
  \item $\alpha\in\curlyB$ and $B_\alpha=(\rho_\alpha,M_\alpha,N_\alpha)$ is a 
bicephalus of degree
$(m,n)=\deg(\alpha)$, or
\item $\alpha\in\curlyM$ and $B_\alpha=M_\alpha$ is a
segmented-premouse and $N_\alpha=\emptyset$, or
\item $\alpha\in\curlyN$ and
$B_\alpha=N_\alpha$ is a segmented-premouse and $M_\alpha=\emptyset$.
\end{enumerate}
 \item For each $\alpha+1<\eta$:
 \begin{enumerate}[label=\tu{(}\roman*\tu{)}]
 \item Either $E_\alpha\in\es_+(M_\alpha)$ or 
$E_\alpha\in\es_+(N_\alpha)$.
\item $\crit_\alpha=\crit(E_\alpha)$ and $\nu_\alpha=\nu(E_\alpha)$ 
and $\lh_\alpha=\lh(E_\alpha)$.
\item For all $\beta<\alpha$ we have $\lh_\beta\leq\lh_\alpha$.
\item $\pred^\Tt(\alpha+1)$ is the
least $\beta$ such that $\crit_\alpha<\nu_\beta$.
\end{enumerate}
\end{enumerate}
Fix $\alpha+1<\eta$ and $\beta=\pred^\Tt(\alpha+1)$ and
$\kappa=\crit_\alpha$.
\begin{enumerate}[resume*]
\item\label{item:B_alpha+1_biceph} Suppose $\beta\in\curlyB$ and
$\kappa<\rho_\beta$ and $E_\alpha$ is total over $B_\beta||\rho_\beta$. Then
$\deg(\alpha+1)=(m,n)$ and
\[ (\rho^*_{\alpha+1},M^*_{\alpha+1},N^*_{\alpha+1})=B^*_{\alpha+1}=B_\beta
 \text{
and }
B_{\alpha+1}=\Ult(B^*_{\alpha+1},E), \]
$i^*_{\alpha+1}:M^*_{\alpha+1}\to M_{\alpha+1}$
is the ultrapower map, $i_{\alpha+1,\alpha+1}=\id$ and
\[ i_{\gamma,\alpha+1}=i^*_{\alpha+1}\com i_{\gamma\beta}:M_\gamma\to 
M_{\alpha+1} \]
for $\gamma\leq_\Tt\beta$;
likewise for $j^*_{\alpha+1}$ etc.
 \item\label{item:B_alpha+1=M_alpha+1} Suppose that $E_\beta\in\es_+(M_\beta)$. 
Suppose that either
$\beta\notin\curlyB$ (so $\beta\in\curlyM$), or $\kappa<\rho_\beta$ and
$E_\alpha$ is not total over $B_\beta||\rho_\beta$. Then we put
$\alpha+1\in\curlyM$,
$N_{\alpha+1}=N^*_{\alpha+1}=\emptyset$, and 
$j^*_{\alpha+1}$, etc, are undefined. We set $M^*_{\alpha+1}\ins M_\beta$ and 
$\deg(\alpha+1)$, etc, in the manner for
degree-maximal trees. Let $k=\deg(\alpha+1)$. Then
\[ M_{\alpha+1}=\Ult_k(M^*_{\alpha+1},E_\alpha) \]
and $i^*_{\alpha+1}$, etc, are defined in the usual manner. We set 
$B^*_{\alpha+1}=M^*_{\alpha+1}$ 
and $B_{\alpha+1}=M_{\alpha+1}$.
\item Suppose that $E_\beta\notin\es_+(M_\beta)$ (so 
$E_\beta\in\es_+(N_\beta)$) 
and $B_{\alpha+1}$ 
is not defined through 
clause \ref{item:B_alpha+1_biceph}. Then we proceed 
symmetrically to clause \ref{item:B_alpha+1=M_alpha+1} (interchanging ``$M$'' 
with 
``$N$'').
\item $\alpha+1\in\curlyD$ iff either $\emptyset\neq M^*_{\alpha+1}\pins 
M_\beta$ or 
$\emptyset\neq N^*_{\alpha+1}\pins N_\beta$.
\item For every limit $\lambda<\eta$, $\curlyD\inter[0,\lambda)_\Tt$ is
finite, and
$\lambda\in\curlyB$ iff $[0,\lambda)_\Tt\sub\curlyB$; the 
models $M_\lambda$, etc, and embeddings $i_{\alpha,\lambda}$, etc, are
defined via direct limits, and 
$\deg(\lambda)$ is the common value of $\deg(\alpha)$ for
all sufficiently large $\alpha<_\Tt\lambda$.
\end{enumerate}

For $\alpha<\lh(\Tt)$, $\curlyB(\alpha)$ denotes 
$\max(\curlyB\inter[0,\alpha]_\Tt)$.
\end{dfn}

\begin{lem}\label{lem:biceph_it_fine_structure}
 Let $\Tt$ be an iteration tree on a bicephalus of degree $(m,n)$ and let
$\alpha<\lh(\Tt)$. We write $B_\alpha=B^\Tt_\alpha$, etc. Then:
 \begin{enumerate}
  \item\label{item:semiclose} If $\alpha+1<\lh(\Tt)$ then $E_\alpha$ is 
weakly amenable to
$B^*_{\alpha+1}$.
  \item\label{item:close} If $\alpha+1<\lh(\Tt)$ and $\alpha+1\notin\curlyB$ 
then
$E_\alpha$ is close to $B^*_{\alpha+1}$.
 \item\label{item:curlyB_in_seg} $\curlyB$ is closed downward under $<_\Tt$ and 
if 
$\alpha\in\curlyM$ then $\curlyN\inter[0,\alpha]_\Tt=\emptyset$.
 \item If\label{item:no_model_drop_m_geq_0} $\alpha\in\curlyM$ and 
$[0,\alpha]_\Tt\inter\curlyD=\emptyset$ then $m\geq 0$.
  \item\label{item:unsound_premouse_no_drop} If 
$\alpha\in\curlyM$, $[0,\alpha]_\Tt\inter\curlyD=\emptyset$,
$\deg(\alpha)=m$ and $\beta=\curlyB(\alpha)$ then:
\begin{enumerate}[label=--]
\item $M_\beta$ is $\rho_\beta$-sound, whereas $M_\alpha$ is 
$\rho_\beta$-solid but not $\rho_\beta$-sound,
\item $M_\beta$ is the $\rho_\beta$-core of
$M_\alpha$ and $i_{\beta\alpha}$ is the $\rho_\beta$-core embedding,
\item $\rho_{m+1}(M_\beta)=\rho_{m+1}(M_\alpha)$,
\item 
$i_{\beta\alpha}(p_{m+1}^{M_\beta}\cut\rho_\beta)=p_{m+1}^{M_\alpha}
\cut\rho_\beta$.
\end{enumerate}
 \item\label{item:unsound_premouse_drop} Suppose $\alpha\in\curlyM$ and
$[0,\alpha]_\Tt$ drops in model or degree. Let $k=\deg^\Tt(\alpha)$. Then the 
core
embedding $\core_{k+1}(M_\alpha)\to M_\alpha$ relates to
$\Tt$ in the manner usual for degree-maximal iteration trees.
\end{enumerate}
\end{lem}
\begin{proof}
Parts \ref{item:semiclose}, \ref{item:curlyB_in_seg} and 
\ref{item:no_model_drop_m_geq_0} are 
easy. For part \ref{item:close}, use essentially the proof of 
\cite[6.1.5]{fsit}, combined with the following simple observation. Let 
$\xi+1<\lh(\Tt)$ 
be such that $[0,\xi]_\Tt$ does not drop in model and $E_\xi=F(M_\xi)$. Let 
$\chi=\pred^\Tt(\xi+1)$. Then
$[0,\xi+1]_\Tt$ does not drop in model and $\chi$ is the least
$\chi'\in[0,\xi]_\Tt$ such that $\crit(F(M_{\chi'}))=\crit_\xi$. We
omit further details of the proof of part \ref{item:close}.

Parts \ref{item:unsound_premouse_no_drop} and
\ref{item:unsound_premouse_drop} now follow as usual.
\end{proof}

\begin{dfn}[Iteration trees on cephalanxes]\label{dfn:iter_tree_on_cephalanx}
Let $B$ be a cephalanx. The notion of an \dfnemph{iteration
tree $\Tt$ on $B$} is defined much as in \ref{dfn:bicephalus_tree}. The key
differences are as follows. The models of the tree are all either cephalanxes 
or 
segmented-premice\footnote{\label{ftn:cephalanx_premice}In fact, for the 
cephalanxes $B$ we will produce -- in the 
proof of \ref{thm:condensation}) -- the models of all trees on $B$ will be 
either 
cephalanxes or premice.}, and if $B$ is passive, then the models are all either 
cephalanxes or 
premice. We write $(M_\alpha,i_{\alpha\beta})$ and
$(Q_\alpha,k_{\alpha\beta})$, etc, for the models and embeddings above $M^B$ and
$Q^B$ respectively. We write 
$B_\alpha=(\gamma_\alpha,\rho_\alpha,M_\alpha,Q_\alpha)$ when 
$B_\alpha$ is
a cephalanx, and otherwise $B_\alpha=M_\alpha\neq\emptyset$ or
$B_\alpha=Q_\alpha\neq\emptyset$, and write $\curlyQ$ for the set of $\alpha$ 
such that $B_\alpha=Q_\alpha$. Let 
$\iota_\alpha=\iota(E_\alpha)$. Other notation is as in 
\ref{dfn:bicephalus_tree}.

Let $\alpha+1<\lh(\Tt)$. Then:
\begin{enumerate}[label=--]
 \item Either $E_\alpha\in\es_+(M_\alpha)$ or $E_\alpha\in\es_+(Q_\alpha)$.
\end{enumerate}
Let $\kappa=\crit_\alpha$. Then:
\begin{enumerate}[resume*]
\item $\pred^\Tt(\alpha+1)$ is the least $\beta$ such that $\kappa<\iota_\beta$.
\end{enumerate}
Suppose $\beta\in\curlyB$. Then:
\begin{enumerate}[resume*]
 \item If $E_\beta\in\es_+(M_\beta)$ and either $\rho_\beta<\kappa$ or 
$E_\alpha$ is not total 
over $M_\beta$ then  $\alpha+1\in\curlyM$ and
$M^*_{\alpha+1}\ins M_\beta$ and $Q_{\alpha+1}=\emptyset$.
 \item If $E_\beta\notin\es_+(M_\beta)$ and either 
$\rho_\beta<\kappa$ or $E_\alpha$ is not 
total over $Q_\beta$ then $\alpha+1\in\curlyQ$ and $Q^*_{\alpha+1}\ins Q_\beta$ 
and $M_{\alpha+1}=\emptyset$.
\end{enumerate}
Now suppose that $\kappa<\rho_\beta$ and $E_\alpha$ is total 
over $B_\beta||\rho_\beta$ (so $\kappa\leq\gamma_\beta$). Then:
\begin{enumerate}[resume*]
 \item Suppose either
$\kappa<\gamma_\beta$ or $E_\beta\in\es_+(M_\beta)$. Then
$B^*_{\alpha+1}=B_\beta$.\footnote{Here if $\kappa=\gamma_\beta$ (so 
$E_\beta\in\es_+(M_\beta)$), one might wonder why we do not just set 
$M^*_{\alpha+1}=\emptyset$ 
and $Q^*_{\alpha+1}=Q_\beta$. This might be made to work, but doing this, it 
seems that 
$E_\alpha$ might not be close to $Q^*_{\alpha+1}$.}
 \item If $\kappa=\gamma_\beta$ and $E_\beta\notin\es_+(M_\beta)$\footnote{When 
this situation 
arises with one of the active cephalanxes we will
produce, $Q$ and $Q_\beta$ must be type 2 premice.} then 
$M_{\alpha+1}=\emptyset$ and 
$Q^*_{\alpha+1}=Q_\beta$.\footnote{In this situation it 
would have been possible to set $B^*_{\alpha+1}=B_\beta$, and the reader might 
object that we are 
dropping information unnecessarily here. But for the cephalanxes we will 
produce, our proof of 
iterability would break down if we set $B^*_{\alpha+1}=B_\beta$, and it will 
turn out that we have 
in fact carried sufficient information (at least, for our present purposes).}
\end{enumerate}
The remaining details are like in \ref{dfn:bicephalus_tree}.
\end{dfn}

\begin{lem}\label{lem:cephalanx_it_fine_structure}
 Let $\Tt$ be an iteration tree on a cephalanx $B=(\gamma,\rho,M,Q)$ of degree 
$(m,q)$ and let
$\alpha+1<\lh(\Tt)$.
 Then parts \ref{item:semiclose}--\ref{item:unsound_premouse_drop}
of \ref{lem:biceph_it_fine_structure}, replacing `$\curlyN$' with `$\curlyQ$', 
hold. Parts \ref{item:curlyB_in_seg}--\ref{item:unsound_premouse_drop}, 
replacing `$\curlyM$' with `$\curlyQ$'
, `$M$' with `$Q$', `$m$' with 
`$q$', 
`$\rho$' with `$\gamma$', and
`$\curlyN$' with `$\curlyM$', hold.
\end{lem}
\begin{proof}
This is mostly like \ref{lem:biceph_it_fine_structure}.
Part \ref{item:no_model_drop_m_geq_0},
when replacing `$\curlyM$' with `$\curlyQ$', etc,
and when $q=-1$, follows easily from the iteration rules.
For consider this situation and suppose $\alpha+1\in\curlyQ$ and 
$\beta=\pred^\Tt(\alpha+1)\in\curlyB$
but $\alpha+1\notin\mathscr{D}^\Tt$.
Since $q=-1$, $B$ is active, so $\OR^Q=\rho$,
so $\OR^{Q_\beta}=\rho_\beta$, so $\kappa<\rho_\beta$,
and since $\alpha+1\notin\mathscr{D}^\Tt$, $E_\alpha$ is total over 
$B_\beta||\rho_\beta$, so $\kappa\leq\gamma_\beta$.
Therefore (since $\alpha+1\notin\curlyB$)
$\kappa=\gamma_\beta$ and $E_\beta\notin\es_+(M_\beta)$.
Therefore $E_\beta=F^{Q_\beta}$ (as $\rho_\beta=\OR^{Q_\beta}$
and $Q_\beta||\rho_\beta=M_\beta||\rho_\beta$).
But $\nu(F^{Q_\beta})\leq\gamma_\beta$,
so $\iota_\beta=\gamma_\beta$, but by the iteration rules,
$\kappa<\iota_\beta$, a contradiction.
\end{proof}
\begin{dfn}\label{dfn:exit^Tt_alpha}
Let $\Tt$ be an iteration tree on a cephal $B$ and $\alpha+1<\lh(\Tt)$.
We write $\exit^\Tt_\alpha$ for the active segmented-premouse $P$ such that 
$E^\Tt_\alpha=F^P$,
if $B$ is a bicephalus then $P\ins M^\Tt_\alpha$ or 
$P\ins N^\Tt_\alpha$, and  if $B$ is a cephalanx then $P\ins M^\Tt_\alpha$ or 
$P\ins 
Q^\Tt_\alpha$.
\end{dfn}

\begin{dfn}
Let $B$ be a cephal. A
\dfnemph{potential tree on $B$} is a tuple
\[
\Tt=\left(<_\Tt,
\left<E_\alpha\right>_{\alpha+1<\eta}\right),
\]
such that if $\eta$ is a limit then $\Tt$ is an iteration tree on $B$, and if 
$\eta=\gamma+1$ then 
$\Tt\rest\gamma$ is an iteration tree on $B$, and $\Tt$ satisfies all 
requirements of 
\ref{dfn:bicephalus_tree}, except that we drop the requirement that $B_\gamma$ 
be a cephal or 
premouse, and add the requirement that $M_\gamma$, $N_\gamma$, $Q_\gamma$,
$\Ult(M_\gamma,F^{M_\gamma})$, $\Ult(N_\gamma,F^{N_\gamma})$, and 
$\Ult(Q_\gamma,F^{Q_\gamma})$ are 
all wellfounded (if defined).
\end{dfn}

The next lemma is easy:

\begin{lem}\label{lem:ceph_it_tree}
Let $\Tt$ be a potential tree on a cephal $B$. Then $\Tt$ is an iteration tree.
Moreover, if $\alpha<\beta<\lh(\Tt)$ and $\beta\in\curlyB^\Tt$ then we can 
apply 
\ref{lem:ceph_it} 
to $B_\alpha,B_\beta$ and the sequence of extenders used along 
$(\alpha,\beta]_\Tt$.
Further, assume that if $B$ is an active cephalanx and 
$\lgcd(Q^B)<\nu(F^{Q^B})$ then $Q^B$ is a premouse. Then every model of $\Tt$ 
is 
either a cephal or 
a premouse. 
\end{lem}

\begin{dfn}[Iterability for cephals] Let $B$ be a bicephalus and
$\alpha\in\OR$. The \dfnemph{length $\theta$ iteration game for $B$} is
defined in the obvious way: given $\Tt\rest\alpha+1$ with $\alpha+1<\theta$, 
player $\playerI$ must 
choose an extender $E_\alpha$, and given $\Tt\rest\lambda$ for a limit 
$\lambda<\theta$, player 
$\playerII$ must choose $[0,\lambda]_\Tt$. The first player to 
break one of these rules or one of the conditions of \ref{dfn:bicephalus_tree} 
loses, and otherwise 
player $\playerII$ wins.

The iteration game for cephalanxes is defined similarly.

We say that a cephal $B$ is \dfnemph{$\alpha$-iterable} if there is a winning
strategy for player $\playerII$ in the length $\alpha$ iteration
game for $B$.\end{dfn}

\begin{lem}\label{lem:condensation_iterates}
 Let $B$ be an $(\om_1+1)$-iterable cephal of degree $(m,k)$. Let $\Tt$ be an 
iteration 
tree on $B$ and $\alpha<\lh(\Tt)$. Then:
\begin{enumerate}[label=--]
 \item Suppose $M^\Tt_\alpha\neq\emptyset$. If $\alpha\in\curlyB^\Tt$ let 
$d=m$; 
otherwise let 
$d=\deg^\Tt(\alpha)$. Then suitable condensation holds through 
$(M^\Tt_\alpha,\max(d,0))$.
 \item Suppose $B$ is a cephalanx and $Q^\Tt_\alpha\neq\emptyset$. If 
$\alpha\in\curlyB^\Tt$ let 
$d=k$; otherwise let $d=\deg^\Tt(\alpha)$. Then suitable condensation holds 
below 
$(Q^\Tt_\alpha,\max(d,0))$, and if either $[0,\alpha]_\Tt$ drops or 
$Q,Q^\Tt_\alpha$ are premice, 
then suitable condensation holds through $(Q^\Tt_\alpha,\max(d,0))$.
\end{enumerate}
\end{lem}
\begin{proof}
If $\Tt$ is trivial, use
\ref{lem:suitable_condensation} (for example, $M^B$ is $(m,\om_1+1)$-iterable).
This extends to longer trees $\Tt$ by \ref{lem:suitable_condensation_def} and 
the elementarity of 
the iteration maps.
\end{proof}

\section{Analysis of iterable cephals}

In this section we prove the main facts about iterable bicephali and 
cephalanxes, which establish strong fine structural restrictions on them.

\begin{dfn}
Let $m<\om$ and $M$ a $\rho$-sound premouse  with
$\rho_{m+1}^M\leq\rho\leq\rho_m^M$. Let 
$\kappa<\rho$, let
$H=\cHull_{m+1}^M(\kappa\un\pvec_{m+1}^M)$ and
$\pi:H\to M$ the uncollapse.

Then $M$ has an 
\dfnemph{$(m,\rho)$-good core at $\kappa$} iff
$H||(\kappa^+)^H=M||(\kappa^+)^M$,
$H$ is $\kappa$-sound,
$\crit(\pi)=\kappa$,
$\pi(\kappa)\geq\rho$ and
$\pi(p_{m+1}^H\cut\kappa)=p_{m+1}^M\cut\kappa$.
In this context, let $H^M_{m,\kappa}=H$ and let $G^M_{m,\kappa,\rho}$ be the 
length $\rho$ extender 
derived from $\pi$.
\end{dfn}

\begin{rem}
Note that if $M$ has an $(m,\rho)$-good core at $\kappa$ then, with 
$\pi,H$ as above, we have $\rho_{m+1}^M\leq\kappa$, $M$ is not $(m+1)$-sound, 
$G=G^M_{m,\kappa,\rho}$ is weakly amenable to $H$, $M=\Ult_m(H,G)$ and 
$i^{H,m}_G=\pi$.\end{rem}

\begin{tm}\label{thm:no_iterable_sound_bicephalus} Let $B=(\rho,M,N)$ be an 
$(\om_1+1)$-iterable 
non-trivial bicephalus. Then $B$ is not sound. Let $m=m^B$
and $n=n^B$. Then exactly one of the following holds:
\begin{enumerate}[label=\tu{(}\alph*\tu{)}]
 \item\label{item:it_biceph_a} $N$ is active type 1 or type 3 with largest 
cardinal $\rho$,
and letting $\kappa=\crit(F^N)$, then $m\geq 0$ and $M$ has an $(m,\rho)$-good 
core at $\kappa$, 
and 
$G^M_{m,\kappa,\rho}=F^N\rest\rho$.
 \item\label{item:it_biceph_b}Vice versa.
 \end{enumerate}
\end{tm}
\begin{proof} 
We may assume $\ZFC$, as we can
work in an inner model which contains $B$ and is closed under an iteration 
strategy $\Sigma$ for $B$,
such as $\HOD_{B,\Sigma}$ or $L[B,\Sigma]$.
So we
may also assume $B$ is countable.
We
mimic the self-comparison argument used in \cite[\S 9]{fsit}. Fix an
$(\om_1+1)$-iteration strategy $\Sigma$ for $B$. We form a pair of padded
iteration
trees $(\Tt,\Uu)$ on $B$, each via $\Sigma$, by comparison. We will
show that the comparison terminates, using the ISC and some more.
Examining the circumstances under which the comparison terminates,
we will show that $B$ is unsound, and the comparison produces
an iterate $B'$ of $B$, also a cephal,
such that $B'$ has  a good core.
A new feature of the proof (in contrast to the classical phalanx comparisons)
is that we 
then need to show that the iteration map from $B$ to $B'$
cannot introduce this property, so$B$ also has a good core.

The trees $(\Tt,\Uu)$ may be padded, but for each $\alpha$ we will have either
$E^\Tt_\alpha\neq\emptyset$ or $E^\Uu_\alpha\neq\emptyset$.
See \S\ref{subsubsec:iteration_trees} regarding padding and tree 
ordering.
At stage $\alpha$ of the comparison, given $\alpha\in\curlyB^\Tt$, we may 
set $E^\Tt_\alpha=\emptyset$, and simultaneously declare that, if $\Tt$ is 
to later use a non-empty extender, then letting $\beta>\alpha$ be least such 
that 
$E^\Tt_\beta\neq\emptyset$, we will have 
$E^\Tt_\beta\in\es_+(M^\Tt_\alpha)=\es_+(M^\Tt_\beta)$. Or 
instead, we may declare that $E^\Tt_\beta\in\es_+(N^\Tt_\alpha)$. Toward this, 
we 
define non-empty sets
\[\modelset^\Tt_\beta\sub\{M^\Tt_\beta,N^\Tt_\beta\}\cut\{\emptyset\}.\] We 
will 
require that if $E^\Tt_\beta\neq\emptyset$, then $E^\Tt_\beta\in\es_+(P)$ for 
some 
$P\in\modelset^\Tt_\beta$. All models in
$\modelset^\Tt_\beta$ will be non-empty.

We also define sets $S^\Tt_\beta\sub t^\Tt_\beta\sub\{0,1\}$ for convenience.
Let $0\in t^\Tt_\beta$ iff $M^\Tt_\beta\neq\emptyset$, and 
$1\in t^\Tt_\beta$ iff $N^\Tt_\beta\neq\emptyset$. Let $0\in S^\Tt_\beta$ iff 
$M^\Tt_\beta\in\modelset^\Tt_\beta$, and $1\in S^\Tt_\beta$ iff 
$N^\Tt_\beta\in\modelset^\Tt_\beta$. 
(We will explicitly define either $\modelset^\Tt_\beta$ or $S^\Tt_\beta$, 
implicitly defining the 
other.)

The preceding definitions also extend to $\Uu$.

We start with
$B^\Tt_0=B=B^\Uu_0$ and $S^\Tt_0=\{0,1\}=S^\Uu_0$.

Suppose we have defined $(\Tt,\Uu)\rest\lambda$ for some limit $\lambda$. Then 
$(\Tt,\Uu)\rest\lambda+1$ is determined by $\Sigma$, and
$S^\Tt_\lambda=\lim_{\alpha<_\Tt\lambda}S^\Tt_\alpha$, and $S^\Uu_\lambda$ is 
likewise.

Now suppose we have defined $(\Tt,\Uu)\rest\alpha+1$ and $S^\Tt_\alpha$ and
$S^\Uu_\alpha$; we determine what to do next (at \dfnemph{stage} $\alpha$).

Let $\widetilde{\nu}(F)=\nu(F)$ for $F$ an extender,
and  $\widetilde{\nu}(\emptyset)=\infty$ (with $\infty>\alpha$
for  $\alpha\in\OR$).

\begin{casetwo} There is $\xi\in\OR$ such that for some
$Y\in\modelset^\Tt_\alpha$ and
$Z\in\modelset^\Uu_\alpha$, we have $\xi\leq\OR^Y\inter\OR^Z$ and $Y|\xi\neq
Z|\xi$.

Let $\xi$ be least such and $\nu=$ the minimum 
value of $\min(\widetilde{\nu}(F^{Y|\xi}),\widetilde{\nu}(F^{Z|\xi}))$ over all 
choices 
of pairs $(Y,Z)$ witnessing the 
choice of $\xi$ (there are at most 4).

\begin{scasetwo} For some  $(Y,Z)$ witnessing the choice of $\xi$, $Y|\xi$ 
and $Z|\xi$ are
both active and $\nu(F^{Y|\xi})=\nu(F^{Z|\xi})=\nu$.

Fix such $Y,Z$. We set 
$E^\Tt_\alpha=F^{Y|\xi}$ and $E^\Uu_\alpha=F^{Z|\xi}$. This 
determines $(\Tt,\Uu)\rest\alpha+2$. Also set
$S^\Tt_{\alpha+1}=t^\Tt_{\alpha+1}$ and $S^\Uu_{\alpha+1}=t^\Uu_{\alpha+1}$.
\end{scasetwo}

\begin{scasetwo}\label{scase:two} Otherwise.

Then take $Y,Z$ witnessing the choice of $\xi$ and such that either (i)
$Y|\xi$ is active, $\nu(F^{Y|\xi})=\nu$, and if $Z|\xi$ is active then 
$\nu(F^{Z|\xi})>\nu$; or (ii)
 vice versa.

 Say $Y|\xi$ is active with $\nu(F^{Y|\xi})=\nu$. Then we
set $E^\Tt_\alpha=F^{Y|\xi}$ and $E^\Uu_\alpha=\emptyset$. This
determines $(\Tt,\Uu)\rest\alpha+2$. Set $S^\Tt_{\alpha+1}=t^\Tt_{\alpha+1}$. 
Now
suppose there is $X\in\modelset^\Uu_\alpha$ with $X|\xi$ active and 
$\nu(F^{X|\xi})=\nu$. Then 
$X|\xi=Y|\xi$, so we must 
avoid setting $E^\Uu_\beta=F^{X|\xi}$ at some $\beta>\alpha$. So we set 
$\modelset^\Uu_{\alpha+1}=\{Z\}$, and set $S^\Uu_{\alpha+1}$ accordingly. 
If there is no 
such $X$ then set $S^\Uu_{\alpha+1}=S^\Uu_\alpha$.
(In any case, later extenders used in $\Uu$ will be incompatible with 
$E^\Tt_\alpha$.) The 
remaining cases are covered by symmetry.
\end{scasetwo}
\end{casetwo}

\begin{casetwo}\label{case:terminate} Otherwise.
 
Then we stop the
comparison at stage $\alpha$.\end{casetwo}

This completes the definition of $(\Tt,\Uu)$. For $\alpha<\lh(\Tt,\Uu)$, let 
$S^\Tt(\alpha)$ be the 
largest $\beta\leq_\Tt\alpha$ such that $S^\Tt_\beta=\{0,1\}$; here if 
$\alpha\in\curlyB^\Tt$ then
$B^\Tt_\beta=B^\Tt_\alpha$. Let $S^\Uu(\alpha)$ be likewise.
For $\alpha+1<\lh(\Tt,\Uu)$, let $\lh_\alpha=\lh(E^\Tt_\alpha)$
and $\nu_\alpha=\nu(E^\Tt_\alpha)$
if $E^\Tt_\alpha\neq\emptyset$, and $\lh_\alpha=\lh(E^\Uu_\alpha)$
and $\nu_\alpha=\nu(E^\Uu_\alpha)$
otherwise. (Note that if $E^\Tt_\alpha\neq\emptyset\neq E^\Uu_\alpha$
then $\lh(E^\Tt_\alpha)=\lh(E^\Uu_\alpha)$ and 
$\nu(E^\Tt_\alpha)=\nu(E^\Uu_\alpha)$.)

\begin{clm}\label{clm:comp_ext_incompat} Let $\alpha+1,\beta+1<\lh(\Tt,\Uu)$. 
Then (i) if $\alpha<\beta$ then $\lh_\alpha\leq\lh_\beta$ 
and $\nu_\alpha<\nu_\beta$; and
(ii) if 
$E^\Tt_\alpha\neq\emptyset\neq 
E^\Uu_\beta$ then $E^\Tt_\alpha\rest\nu_\alpha\neq 
E^\Uu_\beta\rest\nu_\beta$.
\end{clm}
\begin{proof}
Part  (i) is by standard considerations together with the 
definition
 of $\modelset^\Uu_{\alpha+1}$ in Subcase \ref{scase:two} above
 (which prevents  having $E^\Uu_{\alpha+1}=E^\Tt_\alpha$, for example).
Part (ii): If $\alpha+1=\beta+1$, this is directly by 
construction,
and if $\alpha+1<\beta+1$, use part (i)
and the ISC as usual.
\end{proof}

\begin{clm}\label{clm:bicephalus_comparison_terminates} The comparison 
terminates at some countable stage.\end{clm}

\begin{proof}
By the proof 
that standard comparison terminates,  with
 Claim \ref{clm:comp_ext_incompat}.
\end{proof}

So let $\alpha$ be such that the comparison stops at stage $\alpha$.

\begin{clm} \label{clm:1_model} $\card(S^\Tt_\alpha)=\card(S^\Uu_\alpha)=1$ and
$\modelset^\Tt_\alpha=\modelset^\Uu_\alpha$.
\end{clm}
\begin{proof}
If $\alpha\in\curlyB^\Tt$ then $B^\Tt_\alpha$ is non-trivial, by 
\ref{lem:ceph_it_tree}; likewise 
for $\Uu$. So because Case \ref{case:terminate} attains at stage $\alpha$, we 
do 
not have 
$S^\Tt_\alpha=S^\Uu_\alpha=\{0,1\}$.

It is not true that ($\dagger$) $P\pins Q$ or $Q\pins P$  for some 
$P\in\modelset^\Tt_\alpha$ and $Q\in\modelset^\Uu_\alpha$. For suppose 
($\dagger$) holds; we may assume $Q\pins P$. Then $Q$ is 
sound, so by \ref{lem:biceph_it_fine_structure}, $\alpha\in\curlyB^\Uu$,
so by ($\dagger$) and 
Case \ref{case:terminate} hypothesis, $\card(S^\Uu_\alpha)=1$. Say 
$S^\Uu_\alpha=\{0\}$.
Let $\beta=S^\Uu(\alpha)$. Then $B^\Uu_\beta=B^\Uu_\alpha$ and 
$E^\Tt_\beta\in\es_+(N^\Uu_\beta)$ and $E^\Uu_\gamma=\emptyset$ for all 
$\gamma\in[\beta,\alpha)$. 
Let $\varrho=\rho^\Uu_\beta$. Then
$\lh^\Tt_\beta\geq(\varrho^+)^{B^\Uu_\beta}$. So
$\pow(\varrho)\inter P=\pow(\varrho)\inter B^\Uu_\beta$, contradicting
the fact that $M^\Uu_\beta=Q\pins P$.

Now suppose that $S^\Tt_\alpha=\{0,1\}$ but
$\card(S^\Uu_\alpha)=1$. Let $\delta$ be 
least
such that $M^\Tt_\alpha|\delta\neq N^\Tt_\alpha|\delta$. Let 
$Q\in\modelset^\Uu_\alpha$. Then
$Q\pins M^\Tt_\alpha||\delta=N^\Tt_\alpha||\delta$, so ($\dagger$) holds, 
contradiction. So 
$\card(S^\Tt_\alpha)=\card(S^\Uu_\alpha)=1$, and because ($\dagger$) fails, 
$\modelset^\Tt_\alpha=\modelset^\Uu_\alpha$.
\end{proof}

\begin{clm}\label{clm:not_both_drop} $\alpha\in\curlyB^\Tt\Delta\curlyB^\Uu$.
\end{clm}
\begin{proof}
Either $\Tt$ or $\Uu$ is non-padded cofinally in $\alpha$
(that is, if $\alpha=\beta+1$ then either $E^\Tt_\beta\neq\emptyset$
or $E^\Uu_\beta\neq\emptyset$,
and if $\alpha$ is a limit then either $E^\Tt_\beta\neq\emptyset$ for 
cofinally many $\beta<\alpha$, or $E^\Uu_\beta\neq\emptyset$ for cofinally 
many $\beta<\alpha$). By this and Claim \ref{clm:1_model}, we get
$\alpha\notin\curlyB^\Tt\inter\curlyB^\Uu$, so assume that 
$\alpha\notin\curlyB^\Tt\cup\curlyB^\Uu$. 
Then standard calculations using \ref{lem:biceph_it_fine_structure}
give that 
$\Tt,\Uu$ use 
compatible extenders, a contradiction.
\end{proof}

By the previous claims, we can assume  
$\alpha\in\curlyB^\Tt\cut\curlyB^\Uu$, 
$S^\Tt_\alpha=\{0\}$ and $S^\Uu_\alpha=\{1\}$, so
$\Btilde=B^\Tt_\alpha$ is a bicephalus, $\alpha\in\curlyN^\Uu$, and 
$M^\Tt_\alpha=N^\Uu_\alpha$; the other
cases are almost symmetric. We will show  conclusion (a) of the theorem 
holds; under 
symmetric assumptions (b) can hold instead. Let $\beta=S^\Tt(\alpha)$. Let 
$\rhotilde=\rho(\Btilde)$. Then $\Btilde=B^\Tt_\beta$ and for all 
$\gamma\in[\beta,\alpha)$, we have 
$E^\Tt_\gamma=\emptyset\neq E^\Uu_\gamma$ and 
$(\rhotilde^+)^{\Btilde}\leq\lh^\Uu_\gamma$.

\begin{clm}\label{clm:E^Uu_beta_type_1_3}
 $\alpha=\beta+1$ and $\lh^\Uu_\beta=(\rhotilde^+)^{\Btilde}$ and $E^\Uu_\beta$ 
is type 1 or 
type 3.
\end{clm}
\begin{proof}
Suppose not. Then by \ref{lem:biceph_it_fine_structure}, 
$N^\Uu_\alpha$ is not 
$\rhotilde$-sound (recall that if $\alpha>\beta+1$ and 
$\lh^\Uu_{\beta+1}=\lh^\Uu_\beta$ then 
$E^\Uu_{\beta+1}$ is type 2). But by \ref{lem:biceph_it_fine_structure}, 
$M^\Tt_\alpha$ is $\rhotilde$-sound. So $M^\Tt_\alpha\neq N^\Uu_\alpha$, 
contradiction.
\end{proof}

Let $\Btilde=(\rhotilde,\Mtilde,\Ntilde)=B^\Tt_\alpha=B^\Tt_\beta$. Since 
$E^\Uu_\beta\in\es_+(\Ntilde)$, and $\lh^\Uu_\beta=(\rhotilde^+)^{\Btilde}$,
$\Ntilde|(\rhotilde^+)^{\Btilde}$ projects to $\rhotilde$, so 
$\OR^{\Ntilde}=(\rhotilde^+)^{\Btilde}$ and $F^{\Ntilde}=E^\Uu_\beta$. Let 
$\Ftilde=F^{\Ntilde}$ 
and 
$\kappatilde=\crit(\Ftilde)$.
It follows that (a) of the theorem holds regarding $\Btilde$; using the 
iteration embeddings we will deduce that $B$ is not sound, and (a) holds 
regarding $B$.
Note that either $\OR(\Mtilde)>\OR(\Ntilde)$, or  $\OR(\Mtilde)=\OR(\Ntilde)$, 
$\Ntilde$ has 
superstrong type and $\Mtilde$ is type 2; in either case $m\geq 0$.
Also $\OR^N=(\rho^+)^B$ and $N$ is active with $F=F^N$,
a preimage of $\Ftilde$. Let $\kappa=\crit(F)$; so $\kappa<\rho$.

\begin{clm}
$M$ is not $m+1$-sound, so $B$ is not sound.
\end{clm}
\begin{proof} Suppose $M$ is $m+1$-sound. Let $z=z_{m+1}^M$ and 
$\zeta=\zeta_{m+1}^M$. By 
\cite[2.17]{extmax}, $z=p_{m+1}^M$ and 
$\zeta=\rho_{m+1}^M\leq\rho$. So
$\kappa\in\Hull_{m+1}^M(\zeta\un z\un\pvec_m^M)$.
Let $\ztilde=z_{m+1}^{\Mtilde}$ and $\zetatilde=\zeta_{m+1}^{\Mtilde}$. By 
\cite[2.20]{extmax}, $\ztilde=i^\Tt_{0\alpha}(z)$ and $\zetatilde=\sup 
i^\Tt_{0\alpha}``\zeta$, 
so $\zetatilde\leq\rhotilde$ and
\begin{equation}\label{eqn:kappatilde_in} 
i^\Tt_{0\alpha}(\kappa)\in\Hull_{m+1}^{\Mtilde}(\zetatilde\un\ztilde\un\pvec_m^{
\Mtilde}). 
\end{equation}
Let $\Htilde=N^{*\Uu}_{\alpha}$. Then 
$\Mtilde=N^\Uu_{\alpha}=\Ult_m(\Htilde,\Ftilde)$ and
$\zetatilde=\sup i^{\Htilde}_{\Ftilde}``\zeta^{\Htilde}$, and since 
$\zetatilde\leq\rhotilde$, 
therefore $\zetatilde\leq\kappatilde$. Also, 
$\ztilde=i^{\Htilde}_{\Ftilde}(z_{m+1}^{\Htilde})$. But
$\kappatilde\notin\rg(i^{\Htilde}_{\Ftilde})$, so
\begin{equation}\label{eqn:kappatilde_out} 
\kappatilde\notin\Hull_{m+1}^{\Mtilde}(\zetatilde\un\ztilde\un\pvec_m^{\Mtilde}
). \end{equation}
But $i^\Tt_{0\alpha}\rest\rho=j^\Tt_{0\alpha}\rest\rho$, so 
$i^\Tt_{0\alpha}(\kappa)=\kappatilde$, contradicting lines 
(\ref{eqn:kappatilde_in}) and 
(\ref{eqn:kappatilde_out}).
\end{proof}

We can now complete the proof:

\begin{clm}\label{clm:pull_back_facts_to_B}
Conclusion \tu{(}a\tu{)} of the theorem holds.\end{clm}
\begin{proof}
Suppose $N$ is type 1. Let $\ptilde=p_{m+1}^\Mtilde\cut\rhotilde$ and
$\Htilde=\cHull_{m+1}^\Mtilde(\kappatilde\un\ptilde\un\pvec_m^{\Mtilde})$
and 
$\pitilde:\Htilde\to\Mtilde$  the 
uncollapse. Then $\Htilde=N^{*\Uu}_\alpha$, $\pitilde=j^{*\Uu}_{\alpha}$, 
$\Htilde$ is 
$\kappatilde$-sound and letting $\qtilde=p_{m+1}^\Htilde\cut\kappatilde$, we 
have
$\pitilde(\qtilde)=\ptilde$ and
$\rho_m^\Mtilde=\sup\pitilde``\rho_m^\Htilde$ and
\[ 
\Htilde||(\kappatilde^+)^{\Htilde}=\Mtilde||(\kappatilde^+)^{\Mtilde}=
\Ntilde||(\kappatilde^+)^{\Ntilde}. \]

We have $\kappa,H,\pi$ as in (a); let $p=p_{m+1}^M\cut\rho$.
We show $(\sup\pi``\rho_m^H)=\rho_m^M$.
Let $\gamma<\rho_m^M$.
We have $(\sup\widetilde{\pi}``\rho_m^{\widetilde{H}})=\rho_m^{\widetilde{M}}$.
So letting $i=i^\Tt_{0\alpha}$,
\[\widetilde{M}\sats\text{``There is }\beta>i(\gamma)\text{ with 
}\beta\in\Hull_{m+1}(\widetilde{\kappa}\cup\{\widetilde{p},
\pvec_m^{\widetilde{M}}\})\text{''},\]
 an  $\rSigma_{m+1}$ assertion about 
$i(\gamma,
\kappa,p,p_m^M)$, which pulls back to $M$, which suffices.

So $\pi:H\to M$ is an $m$-embedding. Let $\pi(p^H)=p$.
Let 
$\left<H_{\gamma}\right>_{\gamma<\rho_m^H}$ be the natural 
stratification of $\Hull_{m+1}^H(\kappa\un 
\{p^H,\pvec_m^H\})$ (the uncollapsed hull), and 
$M_{\pi(\gamma)}=\pi``H_\gamma$ and
$\pi_\gamma:H_\gamma\to M_{\pi(\gamma)}$ be the restriction of $\pi$.
(For example if $m=0$ 
and $M$ is passive, 
\[ H_{\gamma}=\Hull_1^{H|\gamma}(\kappa\un\{p^H\}).\]
If $M$ 
is active or $m>0$ use the 
stratification of 
$\rSigma_{m+1}$ truth described in \cite[\S2]{fsit}.
Note  $H_\gamma$ need not be transitive.)
So $H=\bigcup_{\gamma<\rho_\gamma^H}H_\gamma$.
For $\gamma$ large enough we have $\kappa\in H_\gamma$
and $H_\gamma$ is transitive below $(\kappa^+)^{H_\gamma}$,
so $\pi_\gamma\rest\pow(\kappa)\sub\pi$,
and in particular $\kappa'=\pi(\kappa)=\pi_\gamma(\kappa)$.
For such $\gamma$, let
$E_\gamma$ 
be the (short) $(\kappa,\kappa')$-extender derived from 
$\pi_\gamma$. Then 
$H_\gamma,E_\gamma\in M$ (as $\sup(\pi_\gamma``\rho_m^H)<\rho_m^M$) and 
$(\kappa^+)^{H_\gamma}<(\kappa^+)^M$. Let 
$\pitilde_\gamma:\Htilde_\gamma\to\Mtilde_\gamma$ and $\kappatilde'$ and 
$\Etilde_\gamma$ be 
defined likewise over $\Mtilde$ (for large enough $\gamma<\rho_m^{\Htilde}$). 
We 
have $(\Htilde_\gamma\sim\Mtilde)||(\kappatilde^+)^{\Htilde_\gamma}$ and
$\Etilde_\gamma\rest\rhotilde\sub F^\Ntilde$ for each $\gamma$;
the former is because by \ref{lem:fully_elem_condensation}, 
$\Ntilde\sats$``Lemma 
\ref{lem:fully_elem_condensation} holds for my proper segments''.

Now $i``\rg(\pi)\sub\rg(\widetilde{\pi})$ since $i=i^\Tt_{0\alpha}$ is an 
$m$-embedding and 
$i(\kappa,p)=(\kappatilde,\ptilde)$.
So for all $\gamma<\rho_m^H$, we have
$i(\pi(\gamma))\in\rg(\widetilde{\pi})$.
And note that $i(\kappa')=\kappatilde'$ and 
 if $\gamma<\rho_m^H$ is sufficiently large and 
$\widetilde{\pi}(\widetilde{\gamma})=i(\pi(\gamma))$, then
$i(H_\gamma)=\Htilde_{\widetilde{\gamma}}$
and $i(E_\gamma)=\Etilde_{\widetilde{\gamma}}$.
Also $\rho_m^\Mtilde=\sup i``\rho_m^M$ and 
$\OR^{\Ntilde}=\sup j``\OR^N$ and $i,j$ are continuous at 
$(\kappa^+)^N=(\kappa^+)^M$ and $j``F^N\sub 
F^\Ntilde$. It follows easily that 
$(H_\gamma\sim M)||(\kappa^+)^{H_\gamma}$ and $E_\gamma\rest\rho\sub F^N$ for
all sufficiently large $\gamma<\rho_m^H$. Therefore 
$H||(\kappa^+)^H=M||(\kappa^+)^M$ and $F^N\rest\rho$ is derived from 
$\pi$.

So $F^N$ is weakly amenable to $H$, $M=\Ult_m(H,F^N)$, and 
$\pi=i^{M,m}_{F^N}$ (we can 
factor $\pi:H\to M$ through $\Ult_m(H,F^N)$, and 
$\nu(F^N)=\rho$). 
So by 
\cite{extmax}, $\pi(z_{m+1}^H)=z_{m+1}^M$, but 
$z_{m+1}^M\cut\rho=p_{m+1}^M\cut\rho$, and therefore 
$z_{m+1}^H\cut\kappa=p_{m+1}^H\cut\kappa$, so $H$ is $\kappa$-sound. This 
completes the proof 
assuming that $N$ is type 1.

If instead, $N$ is type 3, then almost the same argument works.
\end{proof}

This completes the proof of the theorem.
\end{proof}

We now move on to analogues of \ref{thm:no_iterable_sound_bicephalus} for 
cephalanxes.

\begin{dfn}\label{dfn:passive_ceph_good_core}
Let $B$ be a passive cephalanx of degree $(m,q)$ and let $N=N^B$. We say that 
$B$ has a 
\dfnemph{good core} iff $m\geq 0$ and $N$ is active and letting $F=F^N$, 
$\kappa=\crit(F)$ and 
$\nu=\nu(F)$, we have:
(i) $\OR^N=\rho^{+M}$, (ii) $N$ is type 1 or 3,
(iii) $M$ has an $(m,\nu)$-good core at $\kappa$,
(iv) $G^M_{m,\kappa,\nu}=F\rest\nu$, and
(v) if $N$ is type 1 then $H^M_{m,\kappa}=Q$ and $m=q$.
\end{dfn}

\begin{tm}\label{thm:no_iterable_sound_passive_cephalanx}\label{thm:nispc}
Let $B=(\gamma,\rho,M,Q)$ be an $(\om_1+1)$-iterable, non-trivial, passive 
cephalanx of degree $(m,q)$. Then $B$ has a good core, so 
$B$ is not sound \tu{(}that is, $M$ is not $(m+1)$-sound\tu{)}.
\end{tm}
\begin{proof}
The proof is based on that of \ref{thm:no_iterable_sound_bicephalus}. The main 
difference 
occurs in the rules guiding the comparison.
We may assume  $B$ is countable.
We define padded  trees $\Tt,\Uu$ on $B$, and sets 
$S^\Tt_\alpha,S^\Uu_\alpha,\modelset^\Tt_\alpha,\modelset^\Uu_\alpha$, much as 
before. We 
start with $S^\Tt_0=S^\Uu_0=\{0,1\}$. At limit stages, proceed as in 
\ref{thm:no_iterable_sound_bicephalus}. Suppose we have defined 
$(\Tt,\Uu)\rest\alpha+1$, 
$S^\Tt_\alpha$ and $S^\Uu_\alpha$ and if 
$\card(S^\Tt_\alpha)=\card(S^\Uu_\alpha)=1$ then $B^\Tt_\alpha\nins 
B^\Uu_\alpha\nins 
B^\Tt_\alpha$ (otherwise the comparison has already terminated).

\begin{casethree} $\card(S^\Tt_\alpha)=\card(S^\Uu_\alpha)=1$.

Choose extenders as usual (as in \ref{thm:no_iterable_sound_bicephalus}).
\end{casethree}

\begin{casethree} $S^\Tt_\alpha=\{0,1\}$ and if $S^\Uu_\alpha=\{0,1\}$ then 
$\rho^\Tt_\alpha\leq\rho^\Uu_\alpha$.
 
So $B^\Tt_\alpha$ is a cephalanx; let 
$B^\Tt=(\gamma^\Tt,\rho^\Tt,M^\Tt,Q^\Tt)=B^\Tt_\alpha$. Let 
$B^\Uu=B^\Uu_\alpha$. We will have by induction 
that for every $\beta<\alpha$, $\lh^\Tt_\beta\leq\rho^\Tt$ and 
$\lh^\Uu_\beta\leq\rho^\Tt$. Since $B$ is passive, $B^\Tt|\rho^\Tt$ and 
$B^\Uu|\rho^\Tt$ are well-defined premice.

\begin{scasethree}
$B^\Tt|\rho^\Tt\neq 
B^\Uu|\rho^\Tt$. 

Choose extenders as usual. 
\end{scasethree}

Suppose $B^\Tt|\rho^\Tt=B^\Uu|\rho^\Tt$.
We say \dfnemph{in $\Tt$ we move into 
$M^\Tt$} if
either  [$E^\Tt_\alpha\neq\emptyset$ and $E^\Tt_\alpha\in\es_+(M^\Tt)$] or
[$E^\Tt_\alpha=\emptyset$ and 
$S^\Tt_{\alpha+1}=\{0\}$]. Likewise \dfnemph{move into $Q^\Tt$}, and 
likewise 
with regard to $\Uu$ 
if $S^\Uu_\alpha=\{0,1\}$. In each case below we will move into some model in 
$\Tt$; we may do likewise for $\Uu$. These choices produce premice
$R,S$ 
from which to 
choose $E^\Tt_\alpha,E^\Uu_\alpha$, as in the proof of 
\ref{thm:no_iterable_sound_bicephalus},\footnote{That is, we also minimize on
$\nu(E)$,
so if $E^\Tt_\alpha\neq\emptyset\neq E^\Uu_\alpha$
then $\nu(E^\Tt_\alpha)=\nu_\alpha=\nu(E^\Uu_\alpha)$.}, given that $R\nins 
S\nins R$ (for example, 
if $S^\Uu_\alpha=\{1\}$ and in $\Tt$ we move into $M^\Tt$, then $R=M^\Tt$ and 
$S=Q^\Uu$). If $R\ins 
S$ or $S\ins R$ then we terminate the comparison, saying the comparison 
\dfnemph{terminates early}. If $B^\Uu$ is a cephalanx and we do not move into 
any 
model in $\Uu$ and $E^\Uu_\alpha=\emptyset$ then we set 
$S^\Uu_{\alpha+1}=\{0,1\}$.

\begin{scasethree}\label{scase:card_2_card_1}
$\card(S^\Uu_\alpha)=1$ and $B^\Tt|\rho^\Tt=B^\Uu|\rho^\Tt$.

Let $\{P\}=\modelset^\Uu_\alpha$. If $Q^\Tt\ins P$ move into 
$M^\Tt$ in $\Tt$;
if $Q^\Tt\not\ins P$ move into $Q^\Tt$.\end{scasethree}

\begin{scasethree}$S^\Tt_\alpha=S^\Uu_\alpha=\{0,1\}$ and
$B^\Tt|\rho^\Tt=B^\Uu|\rho^\Tt$.

Let $(\gamma^\Uu,\rho^\Uu,M^\Uu,Q^\Uu)=B^\Uu$. So
$\rho^\Tt\leq\rho^\Uu$. Then:
\begin{enumerate}[label=--]
 \item If $Q^\Tt=Q^\Uu$ and $\rho^\Tt=\rho^\Uu$: Move into
 $Q^\Tt$ in $\Tt$ and $M^\Uu$ in $\Uu$. 
\item If $Q^\Tt=Q^\Uu$ and $\rho^\Tt<\rho^\Uu$: Move into $M^\Tt$
in $\Tt$, and if also 
$M^\Tt|\rho^\Uu=B^\Uu|\rho^\Uu$ then move into $Q^\Uu$ in $\Uu$.
\item If $Q^\Tt\pins Q^\Uu$: Move into $M^\Tt$ in $\Tt$
(note here $\rho^\Tt<\rho^\Uu$ and 
$Q^\Tt\pins B^\Uu||\rho^\Uu$).
\item If $Q^\Uu\pins Q^\Tt$: Move into $Q^\Tt$ in $\Tt$
and $M^\Uu$ in $\Uu$ (note here
$\rho^\Tt\leq\gamma^\Uu<\rho^\Uu$).
\item If $Q^\Tt\nins Q^\Uu\nins Q^\Tt$: Move into $Q^\Tt$ in $\Tt$;
if also
$Q^\Tt|\rho^\Uu=B^\Uu|\rho^\Uu$, move into $Q^\Uu$ in $\Uu$.
\end{enumerate}
\end{scasethree}
\end{casethree}

The remaining cases are  by symmetry.
Define $\lh_\alpha$ and $\nu_\alpha$ as for
\ref{thm:no_iterable_sound_bicephalus}.

\begin{clmthree}\label{clm:lh_leq_rho_cephalanx_comparison}
Let $\alpha<\beta<\lh(\Tt,\Uu)$. 
Then (i) if $\beta+1<\lh(\Tt,\Uu)$ then $\lh_\alpha\leq\lh_\beta$ 
and $\nu_\alpha<\nu_\beta$; and
(ii) if $S^\Tt_\beta=\{0,1\}$ then $\lh_\alpha\leq\rho^\Tt_\beta$.
\end{clmthree}
\begin{proof} By induction.
Part (i) is as for  \ref{thm:no_iterable_sound_bicephalus}.
Part (ii): If there are cofinally many $\alpha'<\beta$
with $E^\Tt_{\alpha'}\neq\emptyset$, use part (i)
and rules of iteration trees. Otherwise,
fix $\alpha<\beta$ least with $E^\Tt_{\alpha'}=\emptyset$
for all $\alpha'\in[\alpha,\beta)$. Note $S^\Tt_\alpha=\{0,1\}$,
and if 
there is 
$\alpha'\in[\alpha,\beta)$ with $\rho^\Tt_\alpha<\lh_{\alpha'}$,
and $\alpha'$ is least such, we move into a model of 
$B^\Tt_{\alpha'}=B^\Tt_\alpha=B^\Tt_\beta$
in $\Tt$ at stage $\alpha'$, so $S^\Tt_\beta\neq\{0,1\}$, 
contradiction.\end{proof}

It follows as before that the comparison terminates.

\begin{clmthree}\label{clm:not_terminates_early} Let $\alpha<\lh(\Tt,\Uu)$. 
Then 
\tu{(}i\tu{)} the 
comparison does not terminate early at stage $\alpha$; \tu{(}ii\tu{)} if at 
stage $\alpha$, in 
$\Tt$ we move into $R$, then for every $\beta\in(\alpha,\lh(\Tt,\Uu))$, 
we have
$R\npins 
S$ for all 
$S\in\modelset^\Uu_\beta$.
\end{clmthree}
\begin{proof}
By induction on $\alpha$. Suppose for example that Subcase 
\ref{scase:card_2_card_1} 
attains at stage $\alpha$. We have $P\in\modelset^\Uu_\alpha$.

Suppose $Q^\Tt\ins P$, so in $\Tt$ we move into $M^\Tt$. We have 
$M^\Tt|\rho^\Tt=P|\rho^\Tt$ and $N^\Tt\ins Q^\Tt\ins P$ and $M^\Tt\neq N^\Tt$ 
and both $M^\Tt,N^\Tt$ project $\leq\rho^\Tt$ and
\[ M^\Tt||((\rho^\Tt)^+)^{M^\Tt}=N^\Tt||((\rho^\Tt)^+)^{N^\Tt}. \]
So $M^\Tt\nins P$ and taking 
$\lambda$ least 
with $M^\Tt|\lambda\neq N^\Tt|\lambda$, we have
$\rho^\Tt<\lambda\leq\min(\OR(M^\Tt),\OR(N^\Tt))$.
So the comparison does not terminate early at stage $\alpha$, and as $M^\Tt$ 
projects 
$\leq\rho^\Tt$, for no $\beta>\alpha$ is $M^\Tt\pins 
S\in\modelset^\Uu_\beta$.

Now suppose $Q^\Tt\nins P$, so in $\Tt$ we move into $Q^\Tt$. If 
$\alpha\notin\curlyB^\Uu$ then 
$P=B^\Uu$ is unsound. Otherwise there is $\delta<\alpha$ such 
that at stage $\delta$, in $\Uu$ we move into $P$. In either case (by 
induction 
in the latter), 
$P\npins Q^\Tt$. So the comparison does not terminate early at stage 
$\alpha$. Let 
$\lambda$ be least with $Q^\Tt|\lambda\neq P|\lambda$. Then 
$\rho^\Tt<\lambda$ and since 
$Q^\Tt$ projects $\leq\gamma^\Tt$, there is no $\beta>\alpha$ such that 
$Q^\Tt\pins 
S\in\modelset^\Uu_\beta$.

The proof is similar in the remaining subcases.
\end{proof}

Let $\alpha+1=\lh(\Tt,\Uu)$. As in the proof of
\ref{thm:no_iterable_sound_bicephalus},
and by Claim \ref{clm:not_terminates_early}, we have 
$\card(S^\Tt_\alpha)=\card(S^\Uu_\alpha)=1$ and 
$\alpha\in\curlyB^\Tt\Delta\curlyB^\Uu$. We may
assume $\alpha\in\curlyB^\Tt$, so $B^\Tt_\alpha=(\gamma',\rho',M',Q')$ is 
a 
cephalanx and 
$B^\Uu_\alpha$ is a non-sound pm. So  $P\ins B^\Uu_\alpha$ 
where $\{P\}=\modelset^\Tt_\alpha$. 
But by Claim \ref{clm:not_terminates_early}, $P\npins B^\Uu_\alpha$, so 
$P=B^\Uu_\alpha$.
Let $\beta=S^\Tt(\alpha)$.

\begin{clmthree} $S^\Tt_\alpha=\{0\}$.\end{clmthree}
\begin{proof}
Suppose $S^\Tt_\alpha=\{1\}$, so $Q'=P=B^\Uu_\alpha$ is $\gamma'$-sound. At 
stage $\beta$, in 
$\Tt$ we move into $Q'$. For all $\xi\in[\beta,\alpha)$, 
$E^\Tt_\xi=\emptyset$, so $E^\Uu_\xi\neq\emptyset$, and $\rho'<\lh_\xi$, 
because 
$B'|\rho'=B^\Uu_\beta|\rho'$, and therefore $\rho'\leq\nu^\Uu_\xi$, because 
$\rho'$ is a 
cardinal of $Q'$. But then $B^\Uu_\alpha$ is not 
$\gamma'$-sound, contradicting the fact that $Q'=B^\Uu_\alpha$.\end{proof}

So $M'=P=B^\Uu_\alpha$. Let $N'=N^\Tt_\alpha$.

\begin{clmthree}\label{clm:analyse_B'_passive_cephalanx} 
$\OR(N')=((\rho')^+)^{M'}$,
 $N'$ is active type 1 or type 3,
 $\alpha=\beta+1$,
 $E^\Uu_\beta=F^{N'}$,
and if $N'$ is type 1 then $B^{*\Uu}_\alpha=Q'$.
\end{clmthree}
\begin{proof}
Assume, for example, that Subcase \ref{scase:card_2_card_1} attains at stage 
$\beta$. 
So $N'\ins Q'\ins B^\Uu_\beta$. We have $M'\neq N'$, 
both $M',N'$ project $\leq\rho'$, and
\[ M'||((\rho')^+)^{M'}=N'||((\rho')^+)^{N'}. \]
We have $E^\Tt_\beta=\emptyset$, so $E^\Uu_\beta\neq\emptyset$ and note that
$E^\Uu_\beta\in\es_+(N')$ and $\lh^\Uu_\beta\geq((\rho')^+)^{M'}$. Since 
$M'=B^\Uu_\alpha$ is $\rho'$-sound it 
follows that
$\alpha=\beta+1$ and $\nu^\Uu_\beta=\rho'$, so $E^\Uu_\beta$ is type 1 or type 
3. Therefore 
$N'|\lh^\Uu_\beta$ projects to $\rho'$, so $\OR(N')=\lh^\Uu_\beta$.

Now suppose further that $N'$ is type 1; we want to see that 
$B^{*\Uu}_\alpha=Q'$. We have $Q'\ins 
B^\Uu_\beta$ and $\crit(F^{N'})=\gamma'$ and $\rho_\om(Q')\leq\gamma'$ and
\[ \pow(\gamma')\inter Q'=\pow(\gamma')\inter N'.\]
So it suffices to see that 
$\pred^\Uu(\alpha)=\beta$. We may assume that $\lh^\Uu_\delta=\rho'$ for some 
$\delta<\beta$. 
Then $\rho'$ is a cardinal of $B^\Uu_\beta$, so $Q'\npins B^\Uu_\beta$, so 
$Q'=B^\Uu_\beta$. 
So $B^\Uu_\beta$ is $\gamma'$-sound, so there is a unique $\delta$ such that 
$\lh^\Uu_\delta=\rho'$, and moreover, $E^\Uu_\delta$ is type 3 and 
$\beta=\delta+1$. 
Therefore $\pred^\Uu(\alpha)=\beta$, as required.
\end{proof}

To complete the proof, one can now argue like in Claim 
\ref{clm:pull_back_facts_to_B} of \ref{thm:no_iterable_sound_bicephalus}.
\end{proof}
\begin{rem}\label{rem:active_cephalanx_maybe_sound}
We next consider active 
cephalanxes $B=(\gamma,\rho,M,Q)$. Here things are more subtle, for two 
reasons. 
First, if $Q$ is 
type 3 then $Q^\Tt_\alpha$ can fail  the ISC; this 
complicates comparison termination. Second, if $Q$ is superstrong then 
comparison termination 
is complicated 
further, and more importantly, we do not see how to show $B$ has a
good core (\ref{dfn:good_core_active_cephalanx}), nor how to rule out the 
possibility that $B$ 
is 
exact and $M$ is 
sound with $\rho_{m+1}^M=\rho$.
It is easy enough to illustrate how the latter might occur. Let $Q$ be 
a sound 
superstrong pm and $\kappa=\crit(F^Q)$ and  $J$ be a sound pm 
with
$J||(\kappa^{++})^J=Q|(\kappa^{++})^Q$ and 
$\rho_{m+1}^J=(\kappa^+)^Q=(\kappa^+)^J<\rho_m^J$.
Let $M=\Ult_m(J,F^Q)$ and $B=(\gamma,\rho,M,Q)$, where $\rho=\OR^Q$ and 
$\gamma=\lgcd(Q)$. Suppose  $M$ is wellfounded. Then $B$ is an exact, sound 
cephalanx. 
(We have $\rho_{m+1}^M=\rho<\rho_m^M$ and $M$ is $(m+1)$-sound, and $B$ is 
exact 
because $i^{J,m}_{F^Q}$ 
and $i^Q_{F^Q}$ are continuous at $(\kappa^{++})^J$.) It seems 
$J,Q$ might arise as iterates of a single mouse, so it seems $B$ might 
be iterable.\end{rem}

\begin{dfn}
Let $\Tt$ be an iteration tree on an active cephalanx $B$ and 
$\alpha+1<\lh(\Tt)$.
We say $\alpha$ is 
\dfnemph{$\Tt$-special} iff $\alpha\in\curlyB^\Tt$ and 
$E^\Tt_\alpha=F(Q^\Tt_\alpha)$.
\end{dfn}

\begin{lem}\label{lem:special_facts} Let $\Tt$ be an iteration tree on an 
active 
cephalanx $B$ 
and $\alpha<\lh(\Tt)$. Then:
\begin{enumerate}[label=\tu{(}\alph*\tu{)}]
\item\label{item:superstrong_equiv} If $\alpha\in\curlyB^\Tt$
then $Q^\Tt_\alpha$ has superstrong type iff $Q$ does.
 \item If $\iota(Q^B)=\gamma^B$ then $\curlyQ=\emptyset$.
\end{enumerate}
Suppose also that $\alpha+1<\lh(\Tt)$. Then:
\begin{enumerate}[resume*]
\item\label{item:model_for_special_ext} If $\alpha$ is $\Tt$-special then 
$\alpha+1\in\curlyB^\Tt$ and $\pred^\Tt(\alpha+1)$ is the least 
$\eps\in[0,\alpha]_\Tt$
such that either $\eps=\alpha$ or 
$\crit(F(Q^\Tt_\eps))<\crit(i^\Tt_{\eps\alpha})$.
\item If $B$ is a pm-cephalanx and $\exit^\Tt_\alpha$ is not a premouse then 
$\alpha$ is 
$\Tt$-special \tu{(}so $\exit^\Tt_\alpha=Q^\Tt_\alpha$\tu{)} and $Q$ is type 3.
\end{enumerate}
\end{lem}
\begin{proof}
For \ref{item:superstrong_equiv}, recall that in $\Tt$, we only form simple 
ultrapowers of $Q^B$ and its images.
\end{proof}

\begin{lem}\label{lem:lh_becomes_card}
 Let $\Tt$ be an iteration tree on an active pm-cephalanx 
$B=(\gamma,\rho,M,Q)$. 
 Let $\alpha<\beta<\lh(\Tt)$. Let $\lambda=\lh^\Tt_\alpha$. Then either:
\begin{enumerate}
 \item\label{item:not_cephalanx} $\beta\notin\curlyB^\Tt$ and either 
\tu{(}i\tu{)} 
$\lambda<\OR(B^\Tt_\beta)$ and $\lambda$ is a cardinal of $B^\Tt_\beta$, or 
\tu{(}ii\tu{)} 
$\beta=\alpha+1$, $E^\Tt_\alpha$ has superstrong type, 
$\lambda=\OR(B^\Tt_\beta)$ and $B^\Tt_\beta$ 
is an active type 2 premouse; or
 \item\label{item:is_cephalanx} $\beta\in\curlyB^\Tt$ and either \tu{(}i\tu{)} 
$\lambda<\rho(B^\Tt_\beta)$ and $\lambda$ is a cardinal of $B^\Tt_\beta$, or 
\tu{(}ii\tu{)} 
$\beta=\alpha+1$, $E^\Tt_\alpha$ has superstrong type, 
$\lambda=\rho(B^\Tt_\beta)$,
and letting $\eps=\pred^\Tt(\beta)$, $\crit^\Tt_\alpha=\gamma(B^\Tt_\eps)$.
\end{enumerate}
\end{lem}
Therefore if $\lh^\Tt_\alpha<\lh^\Tt_\beta$ then $\lh^\Tt_\alpha$ is a cardinal 
of 
$\exit^\Tt_\beta$.
\begin{proof}
If $\beta=\alpha+1$ it is straightforward to prove the conclusion. Now suppose 
$\beta>\alpha+1$. If 
$\lambda<\lh^\Tt_{\alpha+1}$ it is straightforward, so suppose 
$\lambda=\lh^\Tt_{\alpha+1}$. Then since the lemma held for $\beta=\alpha+1$, 
either 
$E^\Tt_{\alpha+1}$ is type 2, in which case things are straightforward, or 
$\alpha+1$ is $\Tt$-special, so letting $\mu=\crit^\Tt_{\alpha+1}$ and 
$\chi=\pred^\Tt(\alpha+2)$, we have 
that $B^{*\Tt}_{\alpha+2}=B^\Tt_\chi$ is a cephalanx and 
$\mu<\gamma(B^\Tt_\chi)$, which implies 
that $\lambda<\rho(B^\Tt_{\alpha+2})$ and $\lambda$ is a cardinal of 
$B^\Tt_{\alpha+2}$. The rest 
is clear.
\end{proof}

\begin{dfn}
 Let $B=(\gamma,\rho,M,Q)$ be an active cephalanx of degree $(m,0)$. We say 
that 
$B$ is 
\dfnemph{exceptional} iff
(i) $B$ is exact,
(ii) $Q$ has superstrong type, and
(iii) either $\rho_{m+1}^M=\rho$ or $M$ is not $\gamma$-sound.\qedhere
\end{dfn}

\begin{lem}\label{lem:gamma-sound_characterize}
 Let $M$ be an $m$-sound premouse and let $\rho_{m+1}^M\leq\gamma<\rho_m^M$. 
Then $M$ is 
$\gamma$-sound iff $M=\Hull_{m+1}^M(\gamma\un z_{m+1}^M\un\pvec_m^M)$.
\end{lem}
\begin{proof}This follows from \cite[2.17]{extmax}.\end{proof}
\begin{lem}\label{lem:exceptional_equiv}
Let $B,B'$ be active cephalanxes such that $B'$ is an iterate of 
$B$. Then $B'$ is exceptional iff $B$ is exceptional.
\end{lem}
\begin{proof}
By \ref{lem:ceph_it_tree}, 
\ref{lem:special_facts}\ref{item:superstrong_equiv} and 
\ref{lem:gamma-sound_characterize} and 
\cite[2.20]{extmax}.
\end{proof}

\begin{dfn}\label{dfn:exceptional_core}
Let $B=(\gamma,\rho,M,Q)$ be an active cephalanx of degree $(m,0)$.
Then $B$ has an \dfnemph{exceptional core} iff $Q$ has superstrong 
type and the following holds.
Let $F=F^Q$, $\kappa=\crit(F)$,
$X=i^Q_F``(\kappa^+)^Q$, $m'=\max(m,0)$,
\[ H=\cHull^M_{m'+1}(X\un z_{m+1}^M\un\pvec_m^M), \]
$\pi:H\to M$  the uncollapse.
Then $\pi``(\kappa^+)^H=X$
and $H||(\kappa^{++})^H=M|(\kappa^{++})^M$.
\end{dfn}
\begin{lem}\label{lem:exceptional_core_facts}
Let $B=(\gamma,\rho,M,Q)$ be an active pm-cephalanx of degree $(m,0)$. Suppose 
$B$ has an 
exceptional core. Let $F,\kappa,m',H,\pi$ be as in 
\ref{dfn:exceptional_core}. Then:
\begin{enumerate}
\item\label{item:M_is_ult(H,F)} $M=\Ult_{m'}(H,F)$ and $\pi=i^{H,m'}_F$ is an 
$m'$-embedding.
\item\label{item:z_pres} $\pi(z_{m+1}^H)=z_{m+1}^M$ and 
$\pi(p_{m+1}^H\cut(\kappa^+)^H)=p_{m+1}^M\cut\rho$.
\item\label{item:H_soundness} $\rho_{m+1}^H\leq(\kappa^+)^H<\rho_m^H$ and $H$ 
is 
$(\kappa^+)^H$-sound.
 \item\label{item:rho_m+1=kappa^+} If $\rho_{m+1}^H=(\kappa^+)^H$ then 
$\rho_{m+1}^M=\rho$ and 
$H,M$ are $(m+1)$-sound.
\item\label{item:rho_m+1_leq_kappa} If $\rho_{m+1}^H\leq\kappa$ then 
$\rho_{m+1}^H=\rho_{m+1}^M$ and $M$ is not $(m+1)$-sound.
\item\label{item:bounded_gens} If $M=\Hull_{m'+1}^M(\alpha\un 
z_{m+1}^M\un\pvec_m^M)$ where 
$\alpha<\rho$ 
and $\alpha$ is least such, then $\alpha\in\rg(\pi)$.
\end{enumerate}
\end{lem}
\begin{proof}
Part \ref{item:M_is_ult(H,F)}:
If $m'=-1$ this is easy, so suppose $m=m'\geq 0$.
The uncollapse map 
$\pi:H\to M$ is a near $m$-embedding.
Let $\pi(\bar{z})=z_{m+1}^M$.
We have
\begin{equation}\label{eqn:H_is_hull_of_kappa^+} 
H=\Hull_{m+1}^H((\kappa^+)^H\cup\{\bar{z},\pvec_m^H\}),\end{equation}
and $H||(\kappa^{++})^H=M|(\kappa^{++})^M$,
so $H$ collapses $(\kappa^{++})^M$ and $H\notin M$.
It follows that $\pi$ is an $m$-embedding.
Now
\begin{equation}\label{eqn:M_is_hull_of_gamma^+} 
M=\Hull_{m+1}^M(\gamma^{+M}\cup\{z_{m+1}^M,\pvec_m^M\}), \end{equation}
Let $M'=\Ult_m(H,F)$ and $\pi'=i^{H,m}_F$ and
$\sigma:M'\to M$ the natural factor map. Then $\pi'$
is an $m$-embedding and $\sigma$ is $m$-lifting
and $\sigma\com\pi'=\pi$, and since $\pi$ is an $m$-embedding,
in fact so is $\sigma$. But $\crit(\sigma)\geq\lh(F)=\gamma^{+M}$
and $z_{m+1}^M\in\rg(\sigma)$, so by line (\ref{eqn:M_is_hull_of_gamma^+}),
$M'=M$ and $\sigma=\id$.

Parts \ref{item:z_pres}--\ref{item:rho_m+1=kappa^+}:
If $m=-1$ this is trivial, so suppose $m=m'\geq 0$.
By \cite{extmax} and part \ref{item:M_is_ult(H,F)},
$\bar{z}=z_{m+1}^H$
(where $\pi(\bar{z})=z_{m+1}^M$), so by \ref{lem:gamma-sound_characterize} 
and line (\ref{eqn:H_is_hull_of_kappa^+}) above, $H$ is 
$(\kappa^+)^H$-sound and $p_{m+1}^H\cut(\kappa^+)^H=\bar{z}\cut(\kappa^+)^H$,
so
\[ \pi(p_{m+1}^H\cut(\kappa^+)^H)=z_{m+1}^M\cut\rho=p_{m+1}^M\cut\rho,\]
since $M$ is $\rho$-sound. The rest follows from \cite{extmax}.

Part \ref{item:rho_m+1_leq_kappa}: Because $\rho_{m+1}^H\leq\kappa$, we have 
$m\geq 0$.
Since $Q$ is a type 3 premouse (as $B$
is a pm-cephalanx) and 
$M||(\kappa^{++})^M=H||(\kappa^{++})^H$, $F$ is close to $H$, so
$\rho_{m+1}^M=\rho_{m+1}^H\leq\kappa$. Suppose $M$ is $(m+1)$-sound, so
$M=\Hull_{m+1}^M(\kappa\un\pvec_{m+1}^M)$.
Then
$M=\Hull_{m+1}^M(\rg(\pi)\un q)$
some $q\in\gamma^{<\om}$. But the generators of $F$ are unbounded in $\gamma$, 
a 
contradiction.

Part \ref{item:bounded_gens}: Suppose there is $\alpha<\rho$ such that
\begin{equation}\label{eqn:alpha_gens_M} M = \Hull_{m'+1}^M(\alpha\un 
z_{m+1}^M\un\pvec_m^M). 
\end{equation}
Let $\alpha$ be least such.
Note that $\alpha\geq\gamma$,
since $Q$ is a premouse and $F$ has superstrong type.
Now if $\alpha>\gamma$ then $\alpha$ is a successor.
For if not, then since $\alpha<\rho$,
there is a surjection $f:\gamma\to\alpha$ in $M$,
so there is $\xi<\alpha$
with
\[ f\in\Hull_{m'+1}^M(\{\xi,z_{m+1}^M,\pvec_m^M\}),\]
but then $\max(\xi+1,\gamma)<\alpha$ suffices in place of $\alpha$,
a contradiction.

Since $\pi$ is cofinal in $\rho=\gamma^{+M}$ and $\pi$ is 
$\rSigma_{m'+1}$-elementary and 
$\pi(z_{m+1}^H)=z_{m+1}^M$, the existence
of $\alpha$ reflects to $H$, in that 
there is $\beta<(\kappa^+)^H$ 
such that
\[ H = \Hull_{m'+1}^H(\beta\un z_{m+1}^H\un\pvec_m^H). \]
Let $\beta$ be least such. As above,
either $\beta\leq\kappa$ or $\beta=\zeta+1$ for some 
$\zeta$. If $\beta\leq\kappa$ then note $\alpha=\gamma=\pi(\kappa)$.
So suppose $\beta=\zeta+1>\kappa$. We claim $\pi(\beta)=\alpha$.
For
\[ \zeta\notin\Hull_{m'+1}^H(\zeta\cup 
z_{m+1}^H\cup\pvec_m^H), \]
and this non-membership is an $\rPi_{m'+1}$ assertion in these parameters,
which therefore lifts to $M$ and $\pi(\zeta)$, etc,
so $\pi(\beta)\leq\alpha$.
Conversely, because
\[ H=\Hull_{m'+1}^H((\zeta+1)\cup z_{m+1}^H\cup\pvec_m^H),\]
we get
$\pi``H=
\Hull_{m'+1}^M(\pi``(\zeta+1)\cup z_{m+1}^M\cup\pvec_m^M)$  is unbounded
in $\rho=\gamma^{+M}$,
but $\gamma\leq\pi(\zeta)$, so
\[ \gamma^{+M}\sub\Hull_{m'+1}^M(\pi(\zeta+1)\cup 
z_{m+1}^M\cup\pvec_m^M),\]
so $\pi(\beta)\geq\alpha$.
\end{proof}

\begin{dfn}\label{dfn:good_core_active_cephalanx}
Let $B=(\gamma,\rho,M,Q)$ be an active cephalanx.
 We say 
$B$ has a 
\dfnemph{good core} iff
the following statements ((a)--(c)) hold:

\begin{enumerate}[label=\tu{(}\alph*\tu{)}]
 \item 
 $B$ has degree $(m,0)$ with $m\geq 0$.
\item Either
\begin{enumerate}[label=\tu{(}\roman*\tu{)}]
 \item $B$ is exact, and let $F=F^Q$; or
\item\label{item:B_non-exact_etc} $B$ is non-exact, and letting 
$N=N^B$, we have
$\OR^N=\rho^{+M}$ and $N$ is active type 1 or 3; let $F=F^N$.
\end{enumerate}
\item  Let $\kappa=\crit(F)$ and $\nu=\nu(F)$. Then
\begin{enumerate}[label=\tu{(}\roman*\tu{)}]
\item $M$ has an $(m,\nu)$-good core at $\kappa$, and 
$G^M_{m,\kappa,\nu}=F\rest\nu$.
\item Suppose case \ref{item:B_non-exact_etc} above holds and $N$ is type 1; so 
$\kappa=\gamma$. Then:
\begin{enumerate}[label=--]
\item If $Q$ is type 2 then $H^M_{m,\kappa}=Q$.
\item Suppose $Q$ is not type 2, nor superstrong. Let 
$\mu=\crit(F^Q)$. Then $M$ has an $(m,\gamma)$-good core at $\mu$, and 
$G^M_{m,\mu,\gamma}=F^Q$.\qedhere
\end{enumerate}
\end{enumerate}
\end{enumerate}
\end{dfn}

\begin{rem}
It seems that $B$ might have an exceptional non-good core.
\end{rem}

\begin{tm}\label{thm:no_iterable_sound_active_cephalanx} 
Let $B=(\gamma,\rho,M,Q)$ be an $(\om_1+1)$-iterable, non-trivial, active 
pm-cephalanx, of degree 
$(m,q)$. Then:
\begin{enumerate}[label=--]
\item If $B$ is 
exceptional 
then $B$ has an exceptional core \tu{(}see \ref{dfn:exceptional_core},
\ref{lem:exceptional_core_facts}\tu{)}.
\item If $B$ is non-exceptional then $m\geq 0$ and $B$ has a 
good core \tu{(}see \ref{dfn:good_core_active_cephalanx}\tu{)}, so $B$ 
is not sound
\tu{(}i.e. $M$ is not $(m+1)$-sound\tu{)}.
\end{enumerate}
\end{tm}
\begin{proof}
Suppose first that $B$ is exact and $Q$ is superstrong, but $B$ is not 
exceptional.
Then $\rho_{m+1}^M\leq\gamma\leq(\gamma^{++})^M=\rho^{+M}\leq\rho_m^M$ and $M$ 
is $\gamma$-sound, as is $Q$. Note then that $m\geq 0$,
since otherwise $\OR^M=\rho_{-1}^M=\rho_{m}^M=((\rho_{m+1}^M)^+)^M$. So 
$C=(\gamma,M,Q)$ is a 
non-trivial bicephalus, and note that $C$ is $(\om_1+1)$-iterable. So by 
\ref{thm:no_iterable_sound_bicephalus},
$M$ has an $(m,\gamma)$-good core corresponding to $F^Q$,
and because $B$ is exact, this implies that
$B$ has a good core (\ref{dfn:good_core_active_cephalanx}).
So we may now assume 
that:
\begin{equation}\label{eqn:non-excep_case} \text{If }B\text{ is exact and 
}Q\text{ is 
superstrong, then }B\text{ is exceptional.}\end{equation}

Under this assumption, the proof is based on that of 
\ref{thm:no_iterable_sound_passive_cephalanx}. 
The main differences 
occur in the rules guiding the comparison, the proof that the comparison 
terminates, and when 
$B$ is exceptional. We may assume $B$ is countable. 

We define $\Tt,\Uu$ on $B$ and sets
$S^\Tt_\alpha,S^\Uu_\alpha,\modelset^\Tt_\alpha,\modelset^\Uu_\alpha$, much as 
before. (But if $Q$ is type 1 or 3 then  $S^\Tt_\alpha\neq\{1\}\neq 
S^\Uu_\alpha$ for all $\alpha$.)
If $S^\Tt_\alpha=\{0,1\}$ and 
$B^\Tt_{\alpha+1}=B^\Tt_\alpha$ and $S^\Tt_{\alpha+1}=\{0\}$,
we say that (in $\Tt$) we \dfnemph{move into $M^\Tt_\alpha$ at stage $\alpha$}.
 We never \emph{move into $Q^\Tt_\alpha$};
that is, if $\alpha\in\curlyB^\Tt$ 
then $S^\Tt_\alpha\neq\{1\}$. Likewise for $\Uu$.
Also, we allow
$\alpha$ such that
$E^\Tt_\alpha=\emptyset=E^\Uu_\alpha$, but only when
we move into $M^\Tt_\alpha$ or $M^\Uu_\alpha$  at 
stage $\alpha$.
If we move into $M^\Tt_\alpha$ in at stage $\alpha$,
we will have $E^\Tt_\alpha=\emptyset$
and either $E^\Uu_\alpha=\emptyset$ or $E^\Uu_\alpha=F(Q^\Tt_\alpha)$.
Likewise for $\Uu$. We will not move into
both $M^\Tt_\alpha$ (in $\Tt$) and $M^\Uu_\alpha$ (in $\Uu$) at
stage $\alpha$. If $\alpha\in\curlyB^\Tt$
and $\lh(E^\Tt_\alpha)>\rho(B^\Tt_\alpha)$
(that is, 
$E^\Tt_\alpha\in\es_+(M^\Tt_\alpha)\cut\es_+(Q^\Tt_\alpha)$), then we will have
$S^\Tt_\alpha=\{0\}$, so there is some $\beta<\alpha$
such that $B^\Tt_\beta=B^\Tt_\alpha$ and
at stage $\beta$ we move into $M^\Tt_\beta=M^\Tt_\alpha$. Likewise for $\Uu$.

Suppose we 
have defined $(\Tt,\Uu)\rest\alpha+1$, $S^\Tt_\alpha$ and $S^\Uu_\alpha$.
Suppose there are $A\in\modelset^\Tt_\alpha$ and $B\in\modelset^\Uu_\alpha$
such that $A\nins B\nins A$; otherwise the comparison terminates
at stage $\alpha$.
We next determine what to do at stage $\alpha$.
In certain cases we
implicitly
specify two 
segmented-premice $A,B$, with $A\nins B\nins A$, from which to 
select $E^\Tt_\alpha,E^\Uu_\alpha$.
We then find the least disagreement
between $A,B$, and then minimize on $\iota(E)$, 
rather 
than $\nu(E)$, when 
selecting extenders. (For example, if 
$E^\Tt_\alpha\neq\emptyset\neq 
E^\Uu_\alpha$ then $\iota^\Tt_\alpha=\iota^\Uu_\alpha$.)

Let 
$B^\Tt=B^\Tt_\alpha$, 
$M^\Tt=M^\Tt_\alpha$, etc.
If $S^\Tt=\{0,1\}$, then we will have 
by induction that ($\dagger$) for every $\beta<\alpha$, if 
$E^\Tt_\beta\neq\emptyset$ then $\lh^\Tt_\beta\leq\rho^\Tt$ and 
$\iota^\Tt_\beta\leq\gamma^\Tt$, and 
if $E^\Uu_\beta\neq\emptyset$ then $\lh^\Uu_\beta\leq\rho^\Tt$.
Likewise regarding $\rho^\Uu,\gamma^\Uu$ if $S^\Uu=\{0,1\}$.
We leave the maintenance of ($\dagger$) to the
reader.

We  say that $\alpha$ is \dfnemph{$(\Tt,\Uu)$-unusual} iff $S^\Tt=\{0,1\}$ 
and either
\begin{enumerate}[label=\tu{(}\roman*\tu{)}]
\item there is $\xi<\alpha$ such that
$F(Q^\Tt)\rest\nu(F(Q^\Tt))=E^\Uu_\xi\rest\nu^\Uu_\xi$, or
\item there are $\xi_0<\xi_1<\alpha$ such that
\begin{enumerate}[label=--]
 \item $\alpha=\xi_1+1$,
\item $S^\Tt_{\xi_0}=\{0,1\}$ and $E^\Tt_{\xi_0}=\emptyset$ and 
$E^\Uu_{\xi_0}=F(Q^\Tt_{\xi_0})$ and 
$S^\Tt_{{\xi_0}+1}=\{0\}$,
 \item $S^\Uu_{\xi_1}=\{0,1\}$ and $E^\Uu_{\xi_1}=\emptyset$ and 
$E^\Tt_{\xi_1}=F(Q^\Uu_{\xi_1})$ and $S^\Uu_{{\xi_1}+1}=\{0\}$,
\item $\crit^\Tt_{\xi_1}=\gamma(B^\Tt_{\xi_0})$.
\end{enumerate}
\end{enumerate}
In case (i)/(ii) we say that $\alpha$ 
is \dfnemph{type} (i)/(ii). We define
\dfnemph{$(\Uu,\Tt)$-unusual} symmetrically.

\begin{casefour} $\alpha$ is both $(\Tt,\Uu)$- and $(\Uu,\Tt)$-unusual.

 We terminate the comparison here \dfnemph{with unusual failure 1}.
 (We will show that this can not occur.)
\end{casefour}

From now on, we assume that $\alpha$ is not both $(\Tt,\Uu)$- 
and $(\Uu,\Tt)$-unusual.

\begin{casefour} Either:
\begin{enumerate}[label=--]
 \item $\card(S^\Tt)=\card(S^\Uu)=1$, or
 \item $S^\Tt=\{0,1\}$ and $B^\Tt||\rho^\Tt\neq 
B^\Uu||\rho^\Tt$ and if $S^\Uu=\{0,1\}$ then $\rho^\Tt\leq\rho^\Uu$, or
 \item $S^\Uu=\{0,1\}$ and $B^\Uu||\rho^\Uu\neq 
B^\Tt||\rho^\Uu$ and if $S^\Tt=\{0,1\}$ then $\rho^\Uu\leq\rho^\Tt$.
\end{enumerate}

Select extenders by least disagreement
and minimization on $\iota(E)$ (there is no moving into models).
\end{casefour}

\begin{casefour}\label{case:Tt_2_models}
 $S^\Tt=\{0,1\}$, $B^\Tt||\rho^\Tt=B^\Uu||\rho^\Tt$ and
if $S^\Uu=\{0,1\}$ then 
$\rho^\Tt<\rho^\Uu$.

\begin{scasefour}\label{scase:alpha_T,U-unusual} $\alpha$ is 
$(\Tt,\Uu)$-unusual (hence not 
$(\Uu,\Tt)$-unusual).

If $B^\Uu|\rho^\Tt$ is 
active,
we terminate \dfnemph{with unusual failure 2} (we will show this cannot occur).
Otherwise, in $\Tt$ we move into $M^\Tt$, and we set 
$E^\Tt_\alpha=\emptyset=E^\Uu_\alpha$ and $S^\Uu_{\alpha+1}=S^\Uu_\alpha$.
\end{scasefour}

\begin{scasefour}\label{scase:unusual_failure_2} $\alpha$ is 
$(\Uu,\Tt)$-unusual (hence not 
$(\Tt,\Uu)$-unusual).

 We terminate \dfnemph{with unusual failure 3}.
 (We will show  this cannot occur.)
\end{scasefour}

\begin{scasefour}\label{case:card_2_card_1} $\alpha$ is neither $(\Tt,\Uu)$- 
nor $(\Uu,\Tt)$-unusual.

If $Q^\Tt\ins B^\Uu$ then in $\Tt$ we move into $M^\Tt$,
and we set $E^\Uu_\alpha=F(Q^\Tt)$.

If $Q^\Tt\nins B^\Uu$ then we select extenders from $Q^\Tt$
and $B^\Uu$.\footnote{It might be 
that $B^\Uu|\OR(Q^\Tt)$ is active with extender $E$ and 
$\iota(F(Q^\Tt))>\iota(E)$, 
in which case 
$E^\Tt_\alpha=\emptyset$ and $E^\Uu_\alpha=E$. In this case we keep 
$S^\Tt_{\alpha+1}=\{0,1\}$. 
If $E$ is superstrong, we could end up with $F(Q^\Tt)$ active on 
some 
$S\in\modelset^\Uu_{\alpha+1}$.}
\end{scasefour}

\end{casefour}

\begin{casefour}\label{case:Uu_2_models}
 $S^\Uu=\{0,1\}$, $B^\Uu||\rho^\Uu= 
B^\Tt||\rho^\Uu$ and
if $S^\Tt=\{0,1\}$
then $\rho^\Uu<\rho^\Tt$.

We have subcases 4.1--4.3
analogous to 3.1--3.3.

\end{casefour}

\begin{casefour}\label{case:both_B_matching}
 $S^\Uu=\{0,1\}=S^\Tt$,
 and $B^\Tt||\rho^\Tt= 
B^\Uu||\rho^\Uu$ (so $\rho^\Tt=\rho^\Uu$).

\begin{scasefour}\label{scase:both_B_matching_unusual}
$\alpha$ is either $(\Tt,\Uu)$-unusual or $(\Uu,\Tt)$-unusual.

 We terminate  \dfnemph{with unusual failure 4}.
 (We will show  this cannot occur.)
\end{scasefour}

\begin{scasefour} Otherwise.

If $Q^\Tt\neq Q^\Uu$, we select extenders from $Q^\Tt$ and 
$Q^\Uu$.\footnote{Here it would have been equivalent to set
$E^\Tt_\alpha=F(Q^\Tt)$ and $E^\Uu_\alpha=F(Q^\Uu)$.
We don't do this because if $Q$ is type 2, it seems it might break the 
rule that we
minimize on $\iota(E)$ before selecting extenders - albeit in a trivial manner.
(Suppose $Q$ is type 2. Then so are $Q^\Tt$ and $Q^\Uu$.
Suppose $\iota(Q^\Tt)=\nu(Q^\Tt)<\nu(Q^\Uu)=\iota(Q^\Uu)$, so 
$E^\Tt_\alpha=F(Q^\Tt)$
and $E^\Uu_\alpha=\emptyset$ and $S^\Uu_{\alpha+1}=\{0,1\}$.
Since $Q^\Tt$ is type 2, we have $B^\Tt_{\alpha+1}|\OR(Q^\Uu)$
is well-defined and is passive, so we end up with 
$E^\Tt_\alpha=\emptyset$ and $E^\Uu_{\alpha+1}=F(Q^\Uu)$.)}

If $Q^\Tt=Q^\Uu$, either
\begin{enumerate}[label=--]
 \item 
in $\Tt$ we move into $M^\Tt$,
and set $E^\Uu_\alpha=F(Q^\Tt)=F(Q^\Uu)$, or
\item in $\Uu$ we move into $M^\Uu$,
and set $E^\Tt_\alpha=F(Q^\Tt)=F(Q^\Uu)$.\footnote{We choose
a side randomly. We could have specified one,
but at a loss of symmetry.}
\end{enumerate}
\end{scasefour}
\end{casefour}

This completes the rules of comparison. Given $\alpha<\lh(\Tt,\Uu)$
such that $\alpha\in\curlyB^\Tt$ but  
$S^\Tt_\alpha=\{0\}$, we set $\movin^\Tt(\alpha)=$
the $\beta\leq_\Tt\alpha$ such that
$B^\Tt_\alpha=B^\Tt_\beta$ and at stage $\beta$,
in $\Tt$ we move into $M^\Tt_\beta=M^\Tt_\alpha$. Likewise for $\movin^\Uu$.

By Claim \ref{clm:T-unusual}(\ref{item:no_unusual_termination}) below,
the comparison
does not terminate unusually. By Claim 
\ref{clm:T-unusual}(\ref{item:no_exts_match}), no two 
extenders used in $\Tt$ and $\Uu$ are equivalent
to each other.
If $B$ is active and $Q$ type 3 then
$Q^\Tt_\alpha$ might fail the ISC,
so this needs an argument.

\begin{clmfour}\label{clm:T-unusual}
Let $\eta\leq\lh(\Tt,\Uu)$. Then:
\begin{enumerate}
 \item\label{item:type_1_3_S_not_1}
If $Q$ is type 1/3 then $S^\Tt_\alpha\neq\{1\}\neq 
S^\Uu_\alpha$ for all $\alpha<\eta$.

\item\label{item:type_1_2_all_premice} If $Q$ is type 1/2 then for 
every
$\alpha<\eta$, $M^\Tt_\alpha$, $Q^\Tt_\alpha$ are premice (or empty).

 \item\label{item:rules_moving_in} For all $\alpha+1<\eta$:
 \begin{enumerate}
  \item\label{item:when_both_exts_empty} 
 $E^\Tt_\alpha=\emptyset=E^\Uu_\alpha$ 
 $\Longrightarrow$ [$\alpha$ is either $(\Tt,\Uu)$- or $(\Uu,\Tt)$-unusual].
 \item\label{item:when_unusual}
 $E^\Tt_\alpha=\emptyset=E^\Uu_\alpha$ 
 $\Longleftarrow$ [$\alpha$ is either $(\Tt,\Uu)$- or $(\Uu,\Tt)$-unusual].
 \item\label{item:if_move_in} If we move into $M^\Tt_\alpha$ at stage
 $\alpha$ then $B^\Tt_\alpha||\rho^\Tt_\alpha=B^\Uu_\alpha||\rho^\Tt_\alpha$
 and $E^\Tt_\alpha=\emptyset$ and [either $E^\Uu_\alpha=\emptyset$
 or $E^\Uu_\alpha=F(Q^\Tt_\alpha)$] and if $S^\Uu_\alpha=\{0,1\}$
 then $\rho^\Tt_\alpha\leq\rho^\Uu_\alpha$.
 \item\label{item:when_ext_above_Q} If $\alpha\in\curlyB^\Tt$ and 
$E^\Tt_\alpha\in\es_+(M^\Tt_\alpha)\cut\es_+(Q^\Tt_\alpha)$
 then $S^\Tt_\alpha=\{0\}$.
 \end{enumerate}
\item\label{item:no_exts_match} For all $\alpha+1,\beta+1<\eta$, if 
$E^\Tt_\alpha\neq\emptyset\neq 
E^\Uu_\beta$ then
$E^\Tt_\alpha\rest\nu^\Tt_\alpha\neq E^\Uu_\beta\rest\nu^\Uu_\beta$.

\item\label{item:T,U-unusual} Let $\alpha<\eta$ be $(\Tt,\Uu)$-unusual. Then:
\begin{enumerate}[label=\tu{(}\roman*\tu{)}]
 \item\label{item:Q_type_3_and_Q_alpha_not_pm} $Q$ is type 3 and $Q^\Tt_\alpha$ 
is not a premouse.
 \item\label{item:alpha_not_U,T-unusual} $\alpha$ is not $(\Uu,\Tt)$-unusual.
 \item\label{item:Case_5_attains} Case \ref{case:Tt_2_models} of the comparison 
rules attains
 at stage $\alpha$ and
$B^\Tt_\alpha||\rho^\Tt_\alpha=B^\Uu_\alpha|\rho^\Tt_\alpha$ is passive.
 \item\label{item:gamma_not_later_crit_of_F^Q} For all 
$\beta\in[\alpha,\eta)$, if $\beta\in\curlyB^\Tt$ then
$\crit(F(Q^\Tt_\beta))\neq\gamma^\Tt_\alpha$,
and if $\beta\in\curlyB^\Uu$ then 
$\crit(F(Q^\Uu_\beta))\neq\gamma^\Tt_\alpha$.
 \item\label{item:type(i)} Suppose $\alpha$ is type \tu{(}i\tu{)}, as witnessed 
by $\xi$. Then:
\begin{enumerate}[label=\tu{(}\alph*\tu{)}]
\item\label{item:Q_not_superstrong} $Q$ is not superstrong, 
\item\label{item:alpha=xi+1,emptyext}  $E^\Tt_\xi=\emptyset$ and $\alpha=\xi+1$,
\item\label{item:lh^Uu_xi<gamma} $\lh^\Uu_\xi<\gamma^\Tt_\alpha$,
\item\label{item:triv_comp_of_E^Uu_xi_rest_is_pm_ext} the trivial completion of 
$E^\Uu_\xi\rest\nu^\Uu_\xi$ is a type 3 premouse extender,\footnote{Recall that 
a
premouse extender
is the active extender of some premouse.}
\end{enumerate}
\item\label{item:type(ii)} Suppose that $\alpha$ is type \tu{(}ii\tu{)}, as 
witnessed by ${\xi_0},{\xi_1}$. Then:
\begin{enumerate}[label=\tu{(}\alph*\tu{)}]
 \item\label{item:Q_is_superstrong} $Q$ is superstrong,
 \item\label{item:B_is_exact} $B$ is exact,
 \item\label{item:chi=pred^Tt(alpha)} 
$M^{*\Tt}_\alpha=B^\Tt_{\xi_0}$,
 \item\label{item:ext_is_composition} $F(Q^\Tt_\alpha)=E^\Tt_{\xi_1}\com 
E^\Uu_{\xi_0}$.
\item\label{item:double_successor} 
$M^\Tt_\alpha||((\rho^\Tt_\alpha)^+)^{M^\Tt_\alpha}=
M^\Uu_\alpha||((\rho^\Tt_\alpha)^+)^{M^\Uu_\alpha}$.
 \end{enumerate}
\end{enumerate}
\item\label{item:U,T-unusual} Likewise for $(\Uu,\Tt)$-unusual $\alpha<\eta$.
\item\label{item:no_unusual_termination} The comparison
does not terminate unusually at any stage $\alpha<\eta$.
\end{enumerate}
\end{clmfour}
\begin{proof}
We prove all parts together by simultaneous induction on $\eta$.

Parts \ref{item:type_1_3_S_not_1},
\ref{item:type_1_2_all_premice},
\ref{item:when_both_exts_empty},
\ref{item:if_move_in},
\ref{item:when_ext_above_Q}: by
the rules of comparison and for normal trees.

Part \ref{item:no_exts_match}: Suppose otherwise.
Then by part \ref{item:type_1_2_all_premice}
and the rules of comparison, $Q$ 
is type 3, so part \ref{item:type_1_3_S_not_1} applies.
Let $(\alpha,\beta)$ be 
the lexicographically least counterexample (with $\alpha+1,\beta+1<\eta$).
Let 
$\lambda=\lh^\Tt_\alpha$.

Suppose that 
$\lh^\Uu_\beta=\lambda$. So $E^\Tt_\alpha=E^\Uu_\beta$, so
by the rules of extender selection,  $\alpha\neq\beta$.
So suppose $\alpha<\beta$ (the other case is almost symmetric).
The rules give some $\delta\in[\alpha,\beta)$ 
with
$E^\Uu_\delta\neq\emptyset$; let 
$\delta$ be least such and let $G=E^\Uu_\delta$. Then 
$\lh(G)=\lambda=\lh^\Uu_\beta$, so $G$ has superstrong 
type, $\iota^\Tt_\alpha=\iota(G)$ and $\delta=\alpha$ but $E^\Tt_\alpha\neq G$. 
Let 
$\eps=\pred^\Uu(\alpha+1)$. By \ref{lem:lh_becomes_card}, 
$\alpha+1\in\curlyB^\Uu$ and $\crit(G)=\gamma(B^\Uu_\eps)$. So 
$\alpha$ is not $\Uu$-special, so $G$ is a premouse extender.
Note  $\beta=\alpha+1$ and 
$F(Q^\Uu_{\alpha+1})=E^\Uu_\beta=E^\Tt_\alpha$.
Standard extender factoring arguments 
(for example, 
see \cite[\S5]{extmax}) now show that 
there is $\alpha'<\alpha$ such that $E^\Tt_{\alpha'}=G$. But 
$(\alpha',\alpha)<_\lex(\alpha,\beta)$, 
contradiction.

So we may assume that $\lambda=\lh^\Tt_\alpha<\lh^\Uu_\beta$; so $\alpha<\beta$.
Then $\exit^\Uu_\beta$ is not a premouse because, letting 
$\nu=\nu^\Tt_\alpha$, we have $\nu<\lh^\Tt_\alpha$ and
$\lh^\Tt_\alpha$ is a cardinal of $\exit^\Uu_\beta$, but 
$E^\Uu_\beta\rest\nu\notin 
\exit^\Uu_\beta$. So 
$\beta$ is $\Uu$-special. But then
$\beta$ is $(\Uu,\Tt)$-unusual (of type (i)).
So by induction (with parts \ref{item:U,T-unusual}
and \ref{item:no_unusual_termination}, as $\beta+1<\eta$),
$\beta$ is not $(\Tt,\Uu)$-unusual.
Note then that, by induction,
in the rules of comparison, Subcase 4.1 of Case \ref{case:Uu_2_models}
attains at stage $\beta$,
so $E^\Uu_\beta=\emptyset$,
a contradiction. 

Part \ref{item:T,U-unusual}: Let $\alpha<\eta$ be $(\Tt,\Uu)$-unusual.
So $S^\Tt_\alpha=\{0,1\}$. Let 
$B^\Tt=B^\Tt_\alpha$, $M^\Tt=M^\Tt_\alpha$, etc.
Let 
$F=F(Q^\Tt)$ and $\mu=\crit(F)$.

\begin{caseeleven}
 $\alpha$ is $(\Tt,\Uu)$-unusual of type (i).

Parts
\ref{item:Q_type_3_and_Q_alpha_not_pm}, 
\ref{item:type(i)}\ref{item:lh^Uu_xi<gamma}:
Let us show $\lh^\Uu_\xi<\rho^\Tt=\lh(F)$.
Suppose not. Then $\lh^\Uu_\xi=\rho^\Tt=\lh(F)$ and 
$E^\Uu_\xi=F$. It follows that $E^\Tt_\delta\neq\emptyset$ for some 
$\delta\in[\xi,\alpha)$. Let 
$\delta$ be least such and  $G=E^\Tt_\delta$. As in  part 
\ref{item:no_exts_match}, $G$ is a superstrong premouse extender also used in 
$\Uu$, 
contradicting part
\ref{item:no_exts_match}.

Since $\lh^\Uu_\xi<\rho^\Tt$, $Q^\Tt$ is not a 
premouse, so $Q$ is type 3. It easily follows that 
$\lh^\Uu_\xi<\gamma^\Tt$, since if $\gamma$ is a successor cardinal in $Q$ then 
$Q^\Tt$ is a premouse.

Part \ref{item:type(i)}\ref{item:Q_not_superstrong} ($Q$ is not superstrong):
Suppose otherwise. Then because $Q^\Tt$ is not a premouse, there 
is 
$\delta<_\Tt\alpha$ 
such that $Q^\Tt_\delta$ is a premouse and 
$\crit(j^\Tt_{\delta\alpha})=\gamma(B^\Tt_\delta)$ (otherwise $j^\Tt_{0\alpha}$ 
is 
continuous at $\gamma^B$ and $Q^\Tt$ is a premouse). So $Q^\Tt$ 
fails the ISC. So $E^\Uu_\xi$ is not a premouse 
extender and $\xi$ is $\Uu$-special. But then 
$F(Q^\Uu_\xi)$ has superstrong type, so $\lh^\Uu_\xi=\rho^\Tt$, a 
contradiction.

Part \ref{item:type(i)}\ref{item:alpha=xi+1,emptyext} 
($E^\Tt_\xi=\emptyset$, $\alpha=\xi+1$): We have 
$i_F(\mu)>\rho^\Tt$,
so  $B^\Uu_{\xi+1}|\rho^\Tt=Q^\Tt||\rho^\Tt$. Now suppose there is 
$\delta\in[\xi,\alpha)$ such 
that 
$E^\Tt_\delta\neq\emptyset$. Fix the least such with 
$\delta+1\leq_\Tt\alpha$. Let 
$\eps=\pred^\Tt(\delta+1)$. So $\eps\in\curlyB^\Tt$ and 
$\kappa=\crit^\Tt_\delta\leq\gamma^\Tt_\eps$. If $\kappa<\nu(F(Q^\Tt_\eps))$ 
then 
$\nu(F)>\nu^\Uu_\xi$, contradiction. So $\kappa\geq\nu(F(Q^\Tt_\eps))$. But then
as before, $E^\Tt_\delta$ is a premouse extender used in 
both $\Tt,\Uu$, a contradiction.
The desired conclusions follow.

Part \ref{item:type(i)}\ref{item:triv_comp_of_E^Uu_xi_rest_is_pm_ext}:
By an extender factoring as before;
otherwise there is $\delta<\xi$
such that $E^\Uu_\delta$ is a premouse extender
also used in $\Tt$, a contradiction.

Part \ref{item:alpha_not_U,T-unusual} ($\alpha$ is not $(\Uu,\Tt)$-unusual):
Suppose otherwise.
Then $S^\Tt_\alpha=\{0,1\}=S^\Uu_\alpha$,
so $\alpha$ is type (i) with respect to both kinds of
unusualness
(directly by definition). But by part \ref{item:type(i)} and symmetry,
therefore
 $\alpha=\xi+1$ with $E^\Tt_\xi=\emptyset$ (by $(\Tt,\Uu)$-unusualness of type 
(i)), but  $E^\Tt_\xi\neq\emptyset$ (by $(\Uu,\Tt)$-unusualness of type (i)),
a contradiction.

Part
\ref{item:Case_5_attains} (Case \ref{case:Tt_2_models} 
attains and $B^\Tt||\rho^\Tt=B^\Uu|\rho^\Tt$):
This is because $\alpha$ is non-$(\Uu,\Tt)$-unusual,
$S^\Tt_\alpha=\{0,1\}$
and since by part \ref{item:type(i)} and its proof,
we have $Q^\Tt||\rho^\Tt=B^\Uu|\rho^\Tt$ 
and $\rho^\Tt<i_F(\mu)=i^{*\Uu}_\alpha(\mu)$.

Part \ref{item:gamma_not_later_crit_of_F^Q}:
We have $\mu=\crit(Q^\Tt_\alpha)<\gamma^\Tt_\alpha$.
If $\alpha\in\curlyB^\Uu$ then  
$\crit(F(Q^\Uu_\alpha))\neq\gamma^\Tt_\alpha$ since $\alpha=\xi+1$
and $\lh^\Uu_\xi<\rho^\Tt<i^{*\Uu}_{\alpha}(\mu)$.
Moreover, $\rho^\Tt$ is a cardinal in $M^\Uu_\alpha$,
so if $\beta\geq\alpha$ and $E^\Uu_\beta\neq\emptyset$
then $\rho^\Tt\leq\iota(E^\Uu_\beta)$, which
easily gives that if $\beta+1\in\curlyB^\Uu$
then $\crit(F(Q^\Uu_\beta))\neq\gamma^\Tt_\alpha$.
And at stage $\alpha$, in $\Tt$ we move into $M^\Tt$,
so the rest is similar.

\end{caseeleven}

\begin{caseeleven}
$\alpha$ is $(\Tt,\Uu)$-unusual of type (ii). 

Let 
$F_0=F(Q^\Tt_{\xi_0})=E^\Uu_{\xi_0}$, $\mu_0=\crit(F_0)$,
$F_1=F(Q^\Uu_{\xi_1})=E^\Tt_{\xi_1}$, $\mu_1=\crit(F_1)$.

Part \ref{item:alpha_not_U,T-unusual} ($\alpha$ is not $(\Uu,\Tt)$-unusual): By 
definition
of $(\Tt,\Uu)$-unusual type (ii),  $S^\Uu_\alpha=\{0\}$.

Part 
\ref{item:type(ii)}\ref{item:Q_is_superstrong} ($Q$ is superstrong): Suppose 
not.
Let $\delta\in[{\xi_0},{\xi_1})$ be least such that 
$\delta+1\leq_\Uu\xi_1$. Let 
$G=E^\Uu_\delta$ and 
$\theta=\crit(G)$. Then 
$\theta\leq{\mu_1}<i_G(\theta)$, so ${\mu_1}\notin\rg(j^\Uu_{0{\xi_1}})$, but 
${\mu_1}=\crit(F(Q^\Uu_{\xi_1}))$, so 
${\mu_1}\in\rg(j^\Uu_{0{\xi_1}})$, contradiction.

Part \ref{item:type(ii)}\ref{item:chi=pred^Tt(alpha)} 
($M^{*\Tt}_\alpha=B^\Tt_{\xi_0}$):
Let 
$\beta<\alpha$ with 
$E^\Tt_\beta\neq\emptyset$, so $\beta\neq{\xi_0}$. If $\beta<{\xi_0}$ 
then
$\iota^\Tt_\beta\leq\iota(F_0)$
by ($\dagger$), and $\iota(F_0)=\mu_1$ as $Q$ is superstrong. As
$S^\Tt_{{\xi_0}+1}=\{0\}$, 
$\rho^\Tt_{\xi_0}\leq\iota^\Tt_\beta$  if $\beta>{\xi_0}$. 
This 
suffices.

Parts \ref{item:Q_type_3_and_Q_alpha_not_pm}, 
\ref{item:Case_5_attains},
\ref{item:type(ii)}\ref{item:ext_is_composition}:
These are easy consequences of the fact that
$Q$ is  superstong  and $\alpha$ type (ii)
(in particular, $Q^\Tt_\alpha$ fails the ISC,
so is not a premouse).

Part \ref{item:gamma_not_later_crit_of_F^Q}:
much like in the type (i) case.

Part \ref{item:type(ii)}\ref{item:B_is_exact} ($B$ is exact):
We have ${\xi_1}\in\curlyB^\Uu$ and 
$\crit(F(Q^\Uu_{\xi_1}))={\mu_1}=\iota(E^\Uu_{\xi_0})$. Let 
$\eps=\pred^\Uu({\xi_0}+1)$. Since $\mu_1\in\rg(j^\Uu_{0\xi_1})$,
it is 
easy to see that either:
\begin{enumerate}[label=\tu{(}\arabic*\tu{)}]
 \item\label{item:correct} $\xi_0+1\leq_\Uu\xi_1$ (so 
$\eps,\xi_0+1\in\curlyB^\Uu$)
and $\crit(F(Q^\Uu_{\xi_0+1}))=\mu_1$ 
and $\crit(F(Q^\Uu_\eps))={\mu_0}=\crit(E^\Uu_{\xi_0})$, or
\item\label{item:incorrect} $\xi_0+2\leq_\Uu\xi_1$ (so $\xi_0+2\in\curlyB^\Uu$) 
and $\crit(F(Q^\Uu_{\xi_0+2}))=\mu_1$
and $\eps,\xi_0+1\in\curlyB^\Uu$ and 
$\gamma(B^\Uu_\eps)=\mu_0=\crit(E^\Uu_{\xi_0})$
and $\gamma(B^\Uu_{\xi_0+1})=\mu_1$
and $E^\Uu_{{\xi_0}+1}=F(Q^\Uu_{{\xi_0}+1})$.
\end{enumerate}
We claim that \ref{item:correct} holds,
so suppose \ref{item:incorrect} holds.
Now $E^\Uu_\varepsilon\neq\emptyset$ just by definition,
as $\varepsilon=\pred^\Uu({\xi_0}+1)$. We have
 $M^{*\Uu}_{\xi_0+1}=B^\Uu_\eps$ and by the normality rules,
 $\crit(E^\Uu_{\xi_0})=\gamma(B^\Uu_\eps)<\iota(E^\Uu_\eps)$,
 and since $Q$ is superstrong, therefore $E^\Uu_\eps\notin \es_+(Q^\Uu_\eps)$.
 So by part \ref{item:when_ext_above_Q}, $S^\Tt_\eps=\{0\}$,
 so there is 
$\eps'<\eps$ with
 $B^\Uu_{\eps'}=B^\Uu_\eps$, where at stage $\eps'$
 we move into $M^\Uu_{\eps'}=M^\Uu_\eps$ in $\Uu$.
Note that $\eps'$ is not $(\Uu,\Tt)$-unusual, by part 
\ref{item:gamma_not_later_crit_of_F^Q}
and since $\gamma(B^\Uu_\eps)={\mu_0}=\crit(F(Q^\Tt_{\xi_0}))$.
And $\eps'$ is not $(\Tt,\Uu)$-unusual (type (ii)) as $S^\Uu_{\eps'}=\{0,1\}$.
So by \ref{item:when_both_exts_empty} and \ref{item:if_move_in},
$E^\Tt_{\eps'}=F(Q^\Uu_\eps)$ and $E^\Uu_{\eps'}=\emptyset$.
But then ${\xi_0}+1$ is $(\Uu,\Tt)$-unusual, so 
$E^\Uu_{{\xi_0}+1}=\emptyset\neq 
F(Q^\Uu_{{\xi_0}+1})$, contradiction.

Since $E^\Uu_{\xi_0}$ is 
total 
over 
$B^\Uu_\eps$, 
$({\mu_0}^{++})^{Q^\Tt_{\xi_0}}\leq({\mu_0}^{++})^{B^\Uu_\eps}$ and
$(Q^\Tt_{\xi_0}\sim B^\Uu_\eps)||({\mu_0}^{++})^{Q^\Tt_{\xi_0}}$.
Let $U_0=\Ult(Q^\Tt_{\xi_0},F_0)$.
Since $k=i_{F_0}^{Q^\Tt_{\xi_0}}$ is continuous at 
$({\mu_0}^{++})^{Q^\Tt_{\xi_0}}$, then
$({\mu_1}^{++})^{U_0}\leq({\mu_1}^{++})^{B^\Tt_{{\xi_0}+1}}$
and
$(U_0\sim 
B^\Uu_{{\xi_0}+1})||({\mu_1}^{++})^{U_0}$.

Suppose $B$ is not exact. By \ref{lem:ceph_it_tree}, neither is 
$B^\Tt_{\xi_0}$. So
$({\mu_1}^{++})^{M^\Tt_{\xi_0}}<({\mu_1}^{++})^{U_0}$
and
$(M^\Tt_{\xi_0}\sim U_0\sim 
B^\Uu_{{\xi_0}+1})||({\mu_1}^{++})^{M^\Tt_{\xi_0}}$,
but by non-triviality, $M^\Tt_{\xi_0}\npins U_0$. So
$M^\Tt_{{\xi_0}}\npins B^\Uu_{{\xi_0}+1}$.
We have $E^\Tt_{\xi_0}=\emptyset$ and $S^\Tt_{{\xi_0}+1}=\{0\}$ and 
$M^\Tt_{{\xi_0}+1}=M^\Tt_{\xi_0}$. So
\[ 
({\mu_1}^{++})^{M^\Tt_{\xi_0}}\geq({\mu_1}^{++})^{\exit^\Tt_{\xi_1}}=({\mu_1}^{
++ } )^{ B^\Uu_{\xi_1}} =({\mu_1}^{
++} )^{B^\Uu_{{\xi_0}+1}}
>({\mu_1}^{++})^{M^\Tt_{\xi_0}}, \]
contradiction. 

Part \ref{item:type(ii)}\ref{item:double_successor}:
Arguing as just above, we have
$\mu_0''=({\mu_0}^{++})^{B^\Uu_\eps}=({\mu_0}^{++})^{Q^\Tt_{\xi_0}}$
and $(Q^\Tt_{\xi_0}\sim\exit^\Uu_{\xi_0}\sim B^\Uu_\eps)||\mu_0''$.
By exactness on both sides, this easily gives the claimed agreement
between $M^\Tt_\alpha$ and $M^\Uu_\alpha$.
\end{caseeleven}

Part \ref{item:U,T-unusual} is symmetric.

Part \ref{item:when_unusual}: Suppose
$\alpha+1<\eta$ and $\alpha$ is $(\Tt,\Uu)$-unusual.
By
part \ref{item:T,U-unusual},
Case \ref{case:Tt_2_models} attains,
$\alpha$ is non-$(\Uu,\Tt)$-unusual,
and $M^\Uu_\alpha|\rho^\Tt_\alpha$ is passive,
so Subcase \ref{scase:alpha_T,U-unusual} attains
and $E^\Tt_\alpha=\emptyset=E^\Uu_\alpha$.

Part \ref{item:no_unusual_termination} follows similarly from part 
\ref{item:T,U-unusual}.
\end{proof}

\begin{clmfour}\label{clm:implies_unusual_ii}
Suppose $Q$ is type 3,
 $S^\Uu_{\xi_1}=\{0,1\}$, 
 $E^\Tt_{\xi_1}=F(Q^\Uu_{\xi_1})$, $\xi_1+1\in\curlyB^\Tt$ and
$\crit^\Tt_{\xi_1}=\gamma^\Tt_{\eps}$ where 
${\eps}=\pred^\Tt({\xi_1}+1)$. Then 
${\xi_1}+1$ is 
$(\Tt,\Uu)$-unusual of type \tu{(}ii\tu{)}.
\end{clmfour}
\begin{proof}
Suppose not. Note that $E^\Tt_\eps\in\es_+(M^\Tt_\eps)\cut\es_+(Q^\Tt_\eps)$,
so $\xi_0=\movin^\Tt(\eps)$ exists. If $E^\Uu_{\xi_0}=F(Q^\Tt_{\xi_0})$
then $\xi_1+1$ is $(\Tt,\Uu)$-unusual type (ii).
But if $E^\Uu_{\xi_0}=\emptyset$ then, as $E^\Tt_{\xi_0}=\emptyset$,
$\xi_0$ is $(\Tt,\Uu)$-unusual, so by Claim
\ref{clm:T-unusual}(\ref{item:T,U-unusual})
\ref{item:gamma_not_later_crit_of_F^Q}
$\gamma^\Tt_{\xi_0}\neq\crit(F(Q^\Uu_{\xi_1}))$, contradiction.
\end{proof}

\begin{clmfour}\label{clm:comparison_terminates}
 The comparison terminates at some countable stage.
\end{clmfour}
\begin{proof}
We may assume that $Q$ is type 3, since otherwise every 
extender used 
in $(\Tt,\Uu)$ is a premouse extender and the usual argument works.

Suppose 
$(\Tt,\Uu)$ reaches length $\om_1+1$. Let $\eta\in\OR$ be large,
$\varrho:X\to V_\eta$ be elementary with $X$ countable transitive and 
everything relevant in 
$\rg(\varrho)$. Let $\mu=\crit(\varrho)$. Let 
$W=B^\Tt_{\om_1}||\om_1=B^\Uu_{\om_1}||\om_1$. As is routine, 
either 
$i^\Tt_{\mu\om_1}$ or $j^\Tt_{\mu\om_1}$ is defined,  if 
$i^\Tt_{\mu\om_1}$ is 
defined then $(\mu^+)^{M^\Tt_\mu}=(\mu^+)^W$ and
$(M^\Tt_{\mu}\sim W)||(\mu^+)^W$
and
$i^\Tt_{\mu\om_1}\sub\varrho$
and likewise if $j^\Tt_{\mu,\om_1}$ is defined. Likewise for $\Uu$.

Let $\xi_0$ be least such that
$E^\Tt_{\xi_0}\neq\emptyset$
and $\kappa<_\Tt\xi_0+1<_\Tt\om_1$,
and $\alpha$ likewise for $\Uu$.
Let us assume 
that ${\xi_0}\leq\alpha$; otherwise it is 
symmetric. Let
$\iota=\iota^\Tt_{\xi_0}\leq\iota^\Tt_\alpha$.
As usual, $E^\Tt_{\xi_0}\rest\iota=E^\Uu_\alpha\rest\iota$.

\begin{sclmfour}\label{sclm:E^Tt_alpha_E^Uu_beta_distinctions}
We have:
\begin{enumerate}[label=\tu{(}\alph*\tu{)}]
\item\label{item:pm_extender} The trivial completion of 
$E^\Tt_{\xi_0}\rest\nu^\Tt_{\xi_0}$ is a premouse extender.
 \item\label{item:distinctions} ${\xi_0}<\alpha$ and 
$\iota^\Tt_{\xi_0}<\iota^\Uu_\alpha$ 
and $\nu^\Tt_{\xi_0}<\nu^\Uu_\alpha$ and $\lh^\Tt_{\xi_0}<\lh^\Uu_\alpha$.
\item\label{item:beta_Uu-special} $E^\Uu_\alpha\rest\nu^\Tt_{\xi_0}\notin 
\exit^\Uu_\alpha$, so 
$\exit^\Uu_\alpha$ is not a premouse and $\alpha$ is $\Uu$-special.
\end{enumerate}
\end{sclmfour}
\begin{proof}
Part \ref{item:pm_extender}: We have $\nu^\Tt_{\xi_0}\leq\nu^\Uu_\alpha$ because
$\iota^\Tt_{\xi_0}\leq\iota^\Uu_\alpha$ and by compatibility.
So part \ref{item:pm_extender} follows from standard extender factoring 
(otherwise 
we get some 
premouse extender which factors into $E^\Tt_{\xi_0}$, used in both $\Tt,\Uu$; 
cf.~\cite[\S5]{extmax}).

Part \ref{item:distinctions}: If
$\iota^\Tt_{\xi_0}=\iota^\Uu_\alpha$ then
$E^\Tt_{\xi_0}=E^\Uu_\alpha$, contradicting Claim 
\ref{clm:T-unusual}(\ref{item:no_exts_match}). So 
$\iota^\Tt_{\xi_0}<\iota^\Uu_\alpha$, so
${\xi_0}<\alpha$.
We have $\nu^\Tt_{\xi_0}\leq\nu^\Uu_\alpha$. But
$\nu^\Tt_{\xi_0}\neq\nu^\Uu_\alpha$ by Claim 
\ref{clm:T-unusual}(\ref{item:no_exts_match}). So 
$\nu^\Tt_{\xi_0}<\nu^\Uu_\alpha$.
We have $\lh^\Tt_{\xi_0}\leq\lh^\Uu_\alpha$. Suppose 
$\lh^\Tt_{\xi_0}=\lambda=\lh^\Uu_\alpha$.
Let $P=\exit^\Uu_\alpha$ and $\delta=\lgcd(P)=\lgcd(\exit^\Tt_{\xi_0})$. Then 
$E^\Tt_{\xi_0}\notin P$. Since 
$\nu^\Tt_{\xi_0}<\nu^\Uu_\alpha$ and by part \ref{item:pm_extender}, therefore 
$P$ 
is not a premouse.
So $\alpha$ is $\Uu$-special, so $\iota^\Tt_{\xi_0}<\delta=\iota^\Uu_\alpha$. 
But 
$\iota^\Tt_{\xi_0}\geq\delta$ as $\delta=\lgcd(\exit^\Tt_{\xi_0})$, a 
contradiction. So 
$\lh^\Tt_{\xi_0}<\lh^\Uu_\alpha$.

Part \ref{item:beta_Uu-special}: $\lh^\Tt_{\xi_0}$ is a cardinal of 
$P=\exit^\Uu_\alpha$ and $E^\Tt_{\xi_0}\notin P$, by 
\ref{lem:lh_becomes_card} and agreement between models of 
$\Tt$ and $\Uu$. Since $\nu^\Tt_{\xi_0}<\nu^\Tt_\alpha$ and by 
part 
\ref{item:pm_extender}, $P$ fails the ISC, and so $\alpha$ is $\Uu$-special.
\end{proof}

Let $\eps<_\Uu\alpha$ be largest such that
$F(Q^\Uu_{\eps})\rest\nu(F(Q^\Uu_{\eps}))$ satisfies the ISC. Let 
${\xi_1}+1=\min((\eps,\alpha]_\Uu)$. So 
$E^\Uu_{\xi_1}\neq\emptyset$ but ${\xi_1}$ is not $\Uu$-special. Let 
$F_0=F(Q^\Uu_\eps)$; then
$E^\Tt_{\xi_0}\rest\nu^\Tt_{\xi_0}=F_0\rest\nu(F_0)$.
Let $\delta$ be least such that
$E^\Tt_{\delta}\neq\emptyset$
and $\xi_0+1<_\Tt\delta+1<_\Tt\om_1$.
Let $\iota_1=\min(\iota^\Uu_{\xi_1},\iota^\Tt_{\delta})$. Extender factoring 
gives 
$E^\Tt_{\delta}\rest\iota_1=E^\Uu_{\xi_1}\rest\iota_1$.

\begin{sclmfour}\label{sclm:E^Tt_pi_E^Uu_delta_distinctions} 
(a) $\exit^\Uu_{\xi_1}$ is a premouse, (b) $\delta>{\xi_1}$ and 
$\iota^\Tt_{\delta}>\iota^\Uu_{\xi_1}$ and 
 $\nu^\Tt_{\delta}>\nu^\Uu_{\xi_1}$
 and $\lh^\Tt_{\delta}>\lh^\Uu_{\xi_1}$, (c)
 $E^\Tt_{\delta}\rest\nu^\Uu_{\xi_1}\notin \exit^\Tt_{\delta}$, so 
$\exit^\Tt_{\delta}$ 
fails the ISC and ${\delta}$ is $\Tt$-special.
\end{sclmfour}
\begin{proof}
Like Subclaim 
\ref{sclm:E^Tt_alpha_E^Uu_beta_distinctions} and  because ${\xi_1}$ is not 
$\Uu$-special.\end{proof}

\begin{sclmfour}\label{sclm:Q^Uu_alpha=Q^Uu_xi+1} 
$Q^\Uu_\alpha=Q^\Uu_{\xi_1+1}$,
so $E^\Uu_\alpha$ is equivalent to $E^\Uu_{\xi_1}\com E^\Tt_{\xi_0}$.
\end{sclmfour}
\begin{proof}
Suppose not. Fix  $\gamma'$ least with
$E^\Uu_{\gamma'}\neq\emptyset$
and $\xi_1+1<_\Uu\gamma'+1\leq_\Uu\alpha$.
Fix $\gamma$ least with $E^\Tt_{\gamma}\neq\emptyset$
and $\gamma+1<_\Tt\delta$ and $F(Q^\Tt_{\gamma+1})\rest\nu(F(Q^\Tt_{\gamma+1}))$
fails the ISC. Then both $\exit^\Tt_{\gamma}$
and $\exit^\Uu_{\gamma'}$ are premice, and
extender factoring gives $\exit^\Tt_\gamma=\exit^\Tt_{\gamma'}$,
contradiction.
\end{proof}

\begin{sclmfour}\label{sclm:Q_superstrong_and_alpha_unusual}$Q$ is superstrong 
and $\xi_1+1$ is $(\Uu,\Tt)$-unusual
type (ii).\end{sclmfour}
\begin{proof}
Let $\mu_0=\mu=\crit(E^\Tt_{\xi_0})=\crit(Q^\Uu_\eps)$.
Recall $F_0=F(Q^\Uu_\eps)$.
Let $\mu_1=\crit(E^\Tt_\delta)=\crit(E^\Uu_{\xi_1})$.
As $M^{*\Uu}_{\xi_1+1}=B^\Uu_{\eps}$, we have
$\mu_1\leq\gamma^\Uu_\eps\leq i_{F_0}(\mu_0)$.
And as $\delta$ is $\Tt$-special, we 
have $\xi_0+1=\pred^\Tt(\delta+1)\leq_\Tt\delta$ and
\[ 
\mu_1=\crit(F(Q^\Tt_\delta))=i^\Tt_{0\delta}(\crit(F^Q))=i^\Tt_{0,\xi_0+1}
(\crit(F^Q)),\]
so
$\mu_1\in\rg(i^\Tt_{\mu_0,\xi_0+1})$.
But $\mu_1>\mu_0$, so $\mu_1\geq i^\Tt_{\mu_0,\xi_0+1}(\mu_0)$.
But $E^\Tt_{\xi_0}$ is equivalent to $F_0$, so putting things together,
$i^\Tt_{\mu_0,\xi_0+1}(\mu_0)=\mu_1=\gamma^\Uu_\eps=i_{F_0}
(\mu_0)$.

It follows that
$Q^\Uu_\eps$ and $Q$ are superstrong, $Q^\Uu_\eps$ is therefore a premouse,
 $\nu(F_0)=\mu_1=\nu(E^\Tt_{\xi_0})$ and 
$\exit^\Tt_{\xi_0}=Q^\Uu_\eps=Q^\Uu_{\xi_0}$
(and $B^\Uu_\eps=B^\Uu_{\xi_0}$, though maybe $\eps>\xi_0$).
Therefore in $\Uu$, we move into $M^\Uu_{\xi_0}$
at stage $\xi_0$, and $E^\Tt_{\xi_0}=F(Q^\Uu_{\xi_0})$,
and $\gamma^\Uu_{\xi_0}=\mu_1$.

Let $\varphi<_\Tt\delta$ be largest such that 
$F(Q^\Tt_\varphi)\rest\nu(Q^\Tt_\varphi)$ satisfies the ISC.
Since $Q$ is superstrong,
$Q^\Tt_\varphi$ is a superstrong premouse,
so $\nu(Q^\Tt_\varphi)=\gamma^\Tt_\varphi$. By the claims above,
$\gamma^\Tt_\varphi=\nu(E^\Uu_{\xi_1})$ and in fact
$Q^\Tt_\varphi=\exit^\Uu_{\xi_1}$, and hence
$Q^\Tt_{\xi_1}=\exit^\Uu_{\xi_1}$
(and $B^\Tt_\varphi=B^\Tt_{\xi_1}$, though maybe $\varphi>\xi_1$), and in $\Tt$,
we move into $M^\Tt_{\xi_1}$ at stage $\xi_1$.

We have therefore established that $\xi_1+1$ is $(\Uu,\Tt)$-unusual
of type (ii), as witnessed by $\xi_0,\xi_1$.
\end{proof}

By the last subclaim and rules of comparison, 
$E^\Tt_{\xi_1+1}=\emptyset=E^\Uu_{\xi_1+1}$,
and in $\Uu$, we move into $M^\Uu_{\xi_1+1}$ at stage $\xi_1+1$.
So $\xi_1+2\in\curlyM^\Uu$ and
$M^\Uu_{\xi_1+2}=M^\Uu_{\xi_1+1}$,
so note
 $\OR(Q^\Uu_{\xi_1+1})<\zeta=\lh(E^\Tt_{\xi_1+2})$ or 
$\OR(Q^\Uu_{\xi_1+1})<\zeta=\lh(E^\Uu_{\xi_1+2})$, whichever extender is 
defined.
But  $E^\Uu_\alpha=F(Q^\Uu_\alpha)$
and $\alpha\geq\xi_1+1$,
so  $\alpha>\xi_1+1$, but by Subclaim \ref{sclm:Q^Uu_alpha=Q^Uu_xi+1}, 
$Q^\Uu_\alpha=Q^\Uu_{\xi_1+1}$,
so $\lh(E^\Uu_\alpha)<\zeta$, a contradiction,
completing the proof that comparison terminates.
\end{proof}

We now analyse the manner in which 
the comparison terminates.
Let $\alpha+1=\lh(\Tt,\Uu)$. Let $B^\Tt=B^\Tt_\alpha$, etc.
We say that the comparison \dfnemph{terminates early}
if $\alpha=\beta+1$ for some $\beta$ and $E^\Tt_\beta=\emptyset=E^\Uu_\beta$.
We begin with the non-exceptional case:

\begin{clmfour}\label{clm:B_non-excep_good_core}
Suppose that $B$ is non-exceptional.
Then:
\begin{enumerate}[label=--]
\item $\alpha\in\curlyB^\Tt\Delta\curlyB^\Uu$ and $\card(S^\Tt)=\card(S^\Uu)=1$ 
and $\modelset^\Tt=\modelset^\Uu$.
\item $m\geq 0$ and the cephalanx $C\in\{B^\Tt,B^\Uu\}$ has a good core.
\end{enumerate}
\end{clmfour}
\begin{proof}
Note that ($*$)
either $B$
is non-exact or 
$Q$ is non-superstrong,
because $B$ is non-exceptional and by line (\ref{eqn:non-excep_case}).
So if any $\beta$ is $(\Tt,\Uu)$- or $(\Uu,\Tt)$-unusual,
by Claim \ref{clm:T-unusual}, it is type (i).

Suppose that $\card(S^\Uu)=2$, so $\modelset^\Tt=\{Z\}$
with $Z\pins B^\Uu$. So $Z$ is sound, $\alpha\in\curlyB^\Tt$,
 $\beta=\movin^\Tt(\alpha)$ is defined and $M^\Tt=Z=M^\Tt_\beta$.
So either (a) $E^\Uu_\beta=F(Q^\Tt_\beta)$ or (b) [$\beta=\xi+1$ is 
$(\Tt,\Uu)$-unusual
type (i) and $E^\Uu_\xi$ is equivalent to $F(Q^\Tt_\beta)$].
If $B,B^\Tt_\beta$ are non-exact then since they are non-trivial,
$M^\Tt_\beta\neq N^{B^\Tt_\beta}\pins B^\Uu$,
contradicting that $M^\Tt_\beta=M^\Tt\pins 
B^\Uu$. If $B,B^\Tt_\beta$ are exact,  so $Q,Q^\Tt_\beta$  
non-superstrong, then note that we get enough agreement
that 
$((\rho^\Tt_\beta)^+)^{M^\Tt_\beta}=((\rho^\Tt_\beta)^+)^{B^\Uu_{\beta+1}}$,
which again gives a contradiction. Likewise $\card(S^\Tt)=1$.

So  $\card(S^\Tt)=\card(S^\Uu)=1$. Likewise
$\modelset^\Tt=\modelset^\Uu$. We  have 
$\alpha\in\curlyB^\Tt\cup\curlyB^\Uu$  as usual.
Suppose $\alpha\in\curlyB^\Tt\inter\curlyB^\Uu$.
Let $\beta^\Tt=\movin^\Tt(\alpha)$ and $\beta^\Uu=\movin^\Uu(\alpha)$.
Then $\beta^\Tt\neq\beta^\Uu$, so suppose $\beta^\Tt<\beta^\Uu$.
Then $E^\Tt_\gamma=\emptyset$ for all $\gamma\geq\beta^\Tt$,
and hence $\beta^\Uu$ is $(\Uu,\Tt)$-unusual, hence type (i) (see
above), so $\beta^\Uu=\xi+1$ and $E^\Tt_{\xi}\neq\emptyset$, 
contradiction.

So $\alpha\in\curlyB^\Tt\Delta\curlyB^\Uu$
and we may assume $\alpha\in\curlyB^\Tt\cut\curlyB^\Uu$.
So $S^\Tt=\{0\}$ and letting 
$\modelset^\Tt=\{Z\}=\modelset^\Uu$, $Z=M^\Tt=B^\Uu$ is 
unsound.
We need to show that $m\geq 0$ and $B^\Tt$ has a 
good core. Let 
$\beta=\movin^\Tt(\alpha)$.
Basically as above,
$E^\Tt_\gamma=\emptyset\neq E^\Uu_\gamma$ for all
$\gamma>\beta$.

\renewcommand{\thecasefive}{\Roman{casefive}}
\begin{casefive}$\beta$ is $(\Tt,\Uu)$-unsual (equivalently,
$E^\Tt_\beta=\emptyset=E^\Uu_\beta$).
 
So $\beta=\xi+1$ 
is type 
(i), $E^\Uu_\xi$ is equivalent to $F(Q^\Tt)$, and $Q$ is type 
3 but not 
superstrong. 
We have $B^\Uu_\beta=B^\Uu_{\beta+1}$ and $S^\Uu_\beta=S^\Uu_{\beta+1}$
and
$(M^\Tt\sim B^\Uu_{\beta+1})|\rho^\Tt$.

\begin{scasefive} $M^\Tt\neq B^\Uu_{\beta+1}$.

Then $E^\Uu_{\beta+1}\neq\emptyset$ and
$\rho^\Tt<\lh^\Uu_{\beta+1}$ and $\rho^\Tt$ is a successor cardinal of 
$B^\Uu_{\beta+1}$.
So $\rho^\Tt\leq\iota^\Uu_{\beta+1}$. Since $M^\Tt$ is 
$\rho^\Tt$-sound, it follows that there is 
exactly one ordinal $\delta$ such that $\delta\geq\beta+1$ and 
$\delta+1\leq_\Uu\alpha$, and in fact
$\delta+1=\alpha$. So $\exit^\Uu_\delta$ is a premouse, as
$\alpha\notin\curlyB^\Uu$. Since $M^\Tt$ is $\rho^\Tt$-sound, 
therefore $\delta=\beta+1$ and $E^\Uu_{\beta+1}$ is type 1 or type 3, with 
$\lh^\Uu_{\beta+1}=((\rho^\Tt)^+)^{M^\Tt}$. It follows that $m\geq 0$ and 
$B^\Tt$ is 
non-exact, and 
letting 
$F^*=F(N^{B^\Tt})$ and $\kappa^*=\crit(F^*)$, we have $E^\Uu_{\beta+1}=F^*$, 
and 
$M^\Tt$ has an $(m,\rho^\Tt)$-good core at $\kappa^*$, and 
$G^{M^\Tt}_{m,\kappa^*,\rho^\Tt}=F^*\rest\rho^\Tt$ and 
$H^{M^\Tt}_{m,\kappa^*}=B^{*\Uu}_{\alpha}$. Also, if $F^*$ is type 1 then 
$\pred^\Uu(\alpha)=\beta+1$ and $\beta+1\notin\curlyB^\Uu$ and 
$B^{*\Uu}_\alpha=B^\Uu_{\beta+1}$, and 
$M^\Tt$ has an $(m,\gamma^\Tt)$-good core at $\crit(F(Q^\Tt))$, etc. So $B^\Tt$ 
has
a good core.
\end{scasefive}
\begin{scasefive} $M^\Tt=Z=B^\Uu_{\beta+1}$.

This is a simplification of the previous case, but here, the comparison 
terminates early (so $\alpha=\beta+1$), and $B^\Tt,B$ are 
exact.
\end{scasefive}

\end{casefive}

\begin{casefive} $\beta$ is not $(\Tt,\Uu)$-unusual (so $E^\Uu_\beta=F'$ where 
$F'=F(Q^\Tt)$).

\begin{scasefive} $Q$ is not superstrong.
 
So $F'$ does not have superstrong type. Things work much as in 
the previous case, but there are a couple more possibilities, which we just 
outline. If $B$ is 
exact then 
$\alpha=\beta+1$, and $F'$ is the last extender used in $\Uu$. If $B$ is 
non-exact then 
$\alpha=\beta+2$ and like above, $F^*=F(N^{B^\Tt})$ is type 1 or type 3 and is 
the last extender 
used in $\Uu$. Here if $B$ is non-exact with $N^{B^\Tt}$ is type 1 and 
$Q^\Tt$ type 
2 then 
$Q^\Tt=B^{*\Uu}_\alpha$.
\end{scasefive}

\begin{scasefive} $Q$ is superstrong.
 
So $F'$ has superstrong type, so by 
($*$) above, $B$ is 
non-exact. Things work much 
as before, but there are some extra details, which we just illustrate
in an example 
case. Let
$\eps=\pred^\Uu(\beta+1)$. Note first that if
$\beta+1\in\curlyB^\Uu$ then $\rho^\Tt<\rho^\Uu_{\beta+1}$, 
by Claim \ref{clm:implies_unusual_ii} and non-exactness.
Now
$\kappa'=\crit(F')<\iota^\Uu_\eps$ and
$((\kappa')^+)^{Q^\Tt}\leq\lh^\Uu_\eps$.
Suppose for 
example that $((\kappa')^+)^{Q^\Tt}=\lh^\Uu_\eps$. Then $E^\Uu_\eps$ is type 2 
and 
$B^{*\Uu}_{\beta+1}=\exit^\Uu_\eps$, and $\OR(B^\Uu_{\beta+1})=\OR(Q^\Tt)$ and 
$B^\Uu_{\beta+1}$ is active type 2, so $E^\Uu_{\beta+1}=F(B^\Uu_{\beta+1})$. 
Note that
$(Q^\Tt\sim B^\Uu_{\eps+1})||((\kappa')^{++})^{Q^\Tt}$,
so by \ref{lem:ult_comm_not_moving_crit_F^Q},
\[ (\Ult(Q^\Tt,F')\sim B^\Uu_{\beta+2})||((\gamma^\Tt)^{++})^{\Ult(Q^\Tt,F')}, 
\]
and so $E^\Uu_{\beta+2}=F(N^{B^\Tt})$
(the extender $E^\Uu_{\beta+2}$ must exist in the first place,
by non-exactness). We leave the remaining details to 
the reader.
\end{scasefive}
\end{casefive}
\end{proof}

We can now complete the proof in the non-exceptional case:

\begin{clmfour} If $B$ is non-exceptional then $m\geq 0$ and $B$ has a good 
core.
\end{clmfour}
\begin{proof}
Suppose $B$ is non-exceptional.
By the previous claim, $m\geq 0$ and we have an iterate $B'$ of $B$ with a good 
core.
But then the 
proof of Claim \ref{clm:pull_back_facts_to_B} 
of 
\ref{thm:no_iterable_sound_bicephalus} shows that $B$ also has a good core.
\end{proof}

We now prove corresponding claims for the exceptional case.

\begin{clmfour}
Suppose $B$ is exceptional. Then 
one 
of
$B^\Tt,B^\Uu$ is a cephalanx with an 
exceptional core.
\end{clmfour}
\begin{proof}
We first observe:

\begin{sclmfour}\label{sclm:terminates_early}
Suppose the comparison terminates early (so $\alpha=\beta+1$
where $\beta$ is unusual). Then $\beta$ is type (ii),
 $B$ is exact, $S^\Uu=\{0\}=S^\Tt$, and the final models of the comparison 
are $M^\Tt=M^\Uu$.
\end{sclmfour}

\begin{proof}
The fact that $\beta$ is type (ii) and $B$ exact, is because $B$ is exceptional
and Claim \ref{clm:T-unusual}.
So
$(\rho^+)^{M^\Tt}=(\rho^+)^{M^\Uu}$
where $\rho=\rho^\Tt$,
 by Claim 
\ref{clm:T-unusual}(\ref{item:T,U-unusual})
\ref{item:type(ii)}\ref{item:double_successor}.
But  $M^\Tt,M^\Uu$ project ${\leq\rho}$,
so $M^\Tt\npins M^\Uu\npins M^\Tt$.
\end{proof} 
 
Now we consider a few cases:
\renewcommand{\thecasesix}{\Roman{casesix}}
\begin{casesix} Either (a) $S^\Tt=\{0\}=S^\Uu$ and $M^\Uu\pins M^\Tt$, or 
(b) $S^\Tt=\{0,1\}$.

This case is covered by the next case and symmetry.\end{casesix}

\begin{casesix} Either (a) $S^\Tt=\{0\}=S^\Uu$ and $M^\Tt\pins M^\Uu$,
or (b) $S^\Uu=\{0,1\}$.

Note that given either (a) or (b), $S^\Tt=\{0\}$ and $M^\Tt\pins 
B^\Uu$. So $M^\Tt$ is sound, so $\alpha\in\curlyB^\Tt$. Let 
$\beta=\movin^\Tt(\alpha)$. Because
$\rho_{m+1}(M^\Tt)\leq\rho^\Tt$ and $M^\Tt\pins B^\Uu$,
$\Uu$ does not use any extender $E$ with $\rho^\Tt<\lh(E)$.
So if
$\beta$ is 
$(\Tt,\Uu)$-unusual (so type (ii))
then $M^\Tt\pins M^\Uu_\beta=M^\Uu_{\beta+1}=B^\Uu$
and $\alpha=\beta+1$,
contradicting Subclaim  \ref{sclm:terminates_early}.
So (let) $E^\Uu_\beta=F'=F(Q^\Tt)$. Similarly, 
$\beta+1$ is not 
$(\Uu,\Tt)$-unusual. 
So by Claim \ref{clm:implies_unusual_ii},
if $\beta+1\in\curlyB^\Uu$ then $\rho^\Tt<\rho^\Uu_{\beta+1}$,
so $B^\Uu_{\beta+1}|\rho^\Tt$ is well-defined.
Let $\kappa=\crit(F')$ and $\eps=\pred^\Uu(\beta+1)$. We split into two 
subcases:

\begin{scasesix}\label{scasesix:B^Uu_beta+1|rho_active} 
$B^\Uu_{\beta+1}|\rho^\Tt$ is active.

Then $\beta+1\notin\curlyB^\Uu$,
$B^\Uu_{\beta+1}$ is type 2, $\rho^\Tt=\OR(B^\Uu_{\beta+1})$,
 $(\kappa^+)^{Q^\Tt}=\OR(B^{*\Uu}_{\beta+1})$, 
$E^\Uu_\eps=F(B^{*\Uu}_{\beta+1})$,
$E^\Uu_{\beta+1}=F(B^\Uu_{\beta+1})$,
and $\alpha=\beta+2$.
Let $R=B^{*\Uu}_{\beta+1}$ and $G=F^R=E^\Uu_\eps$
and $U=\Ult_0(R,G)$. Then
$(\kappa^+)^U=(\kappa^+)^{Q^\Tt}=\OR^R$
and
\[ (U\sim Q^\Tt)||(\kappa^{++})^{Q^\Tt},\]
but $(\kappa^{++})^U>(\kappa^{++})^{Q^\Tt}$
because $B^\Tt$ is exact and $M^\Tt\pins B^\Uu$ and by \ref{lem:ult_comm}.

Let $H\pins U$ and $h\in\{-1\}\un\om$ with 
$\rho_{h+1}^H=(\kappa^+)^H<(\kappa^{++})^H=(\kappa^{++})^{Q^\Tt}\leq\rho_h^H$.
Let $H^*=i^{U}_{F'}(H)$; so
$H^*\pins\Ult_0(U,F')$. Note $i^U_{F'}$ is continuous
at $(\kappa^+)^{Q^\Tt}$.

We claim that $\Ult_h(H,F')\ins 
H^*$.
For if $h=-1$ then by  continuity, in fact 
$\Ult_{-1}(H,F')=H^*$, so suppose $h\geq 0$. Let
\[ \sigma:\Ult_h(H,F')\to H^* \]
be the factor map $\sigma([a,f]^{H,h}_{F'})=i^U_{F'}(f)(a)$. Arguing 
like in \S\ref{sec:prelim},
we get that $H,\sigma\in\Ult_0(U,F')$ and the hypotheses of 
\ref{lem:fully_elem_condensation} hold for $H,\sigma,h,H^*$. By 
\ref{lem:fully_elem_condensation} (and its first order nature), $R\sats$``Lemma 
\ref{lem:fully_elem_condensation} holds for my proper segments''. 
Therefore $\Ult_h(H,F')\ins H^*$, as desired.

So $\Ult_h(H,F')\pins B^\Uu_{\beta+2}$. But
$((\rho^\Tt)^+)^{\Ult_h(H,F')}=((\rho^\Tt)^+)^{M^\Tt}$,
so $h=m$ and $M^\Tt=\Ult_h(H,F')$. It easily follows that $B^\Tt$ has an 
exceptional core, and with 
$X,m'$ as 
in \ref{dfn:exceptional_core},\[H=\cHull_{m'+1}^{M^\Tt}(X\un 
z_{m+1}^{M^\Tt}\un\pvec_m^{M^\Tt}).\]
\end{scasesix}

\begin{scasesix} $B^\Uu_{\beta+1}|\rho^\Tt$ is passive.

Then $\alpha=\beta+1$, so $M^\Tt\pins B^\Uu_{\beta+1}=B^\Uu$.
Let $R=B^{*\Uu}_{\alpha}$. If $\alpha\in\curlyB^\Uu$ then $\kappa<\gamma(R)$, 
so 
$(\kappa^{++})^R$ is well-defined. In any case, $\kappa$ is not the largest 
cardinal of $R$.
We have $(\kappa^+)^R=(\kappa^+)^{Q^\Tt}$ and
$(R\sim Q)||(\kappa^{++})^{Q^\Tt}$.
If $(\kappa^{++})^{R}>(\kappa^{++})^{Q^\Tt}$ then a 
simplification of the argument in the previous subcase works. Suppose then that 
$(\kappa^{++})^R=(\kappa^{++})^{Q^\Tt}$. Because $M^\Tt\pins B^\Uu$, it 
is easy enough to see that $\alpha\notin\curlyB^\Uu$, so $R$ is a premouse.
If $R$ is active type 3, then $(\kappa^+)^R<\nu(F^R)$, because if 
$(\kappa^+)^R=\nu(F^R)$ then 
$\OR(B^\Uu_{\beta+1})=((\rho^\Tt)^+)^{M^\Tt}$, a contradiction. Let 
$d=\deg^\Uu(\beta+1)$. Then $i^{*\Uu}_{\beta+1}$ is discontinuous at 
$(\kappa^{++})^R$, and so 
$(\kappa^+)^R=\rho_d^R$, so $d>0$. Let $r<d$ be such that 
$\rho_{r+1}^R=(\kappa^+)^R<\rho_r^R$. 
Then arguing like in the previous subcase, but using 
\ref{lem:condensation_iterates} 
instead of \ref{lem:fully_elem_condensation},
\[ M^\Tt=\Ult_r(R,F')\pins B^\Uu_{\beta+1} \]
and $B^\Tt$ has an 
exceptional core (with $m=r$).
\end{scasesix}
\end{casesix}

\begin{casesix} $S^\Tt=\{0\}=S^\Uu$ and $M^\Tt=M^\Uu$
but the comparison does not terminate early.

Then $\alpha\in\curlyB^\Tt\Delta\curlyB^\Uu$; assume 
$\alpha\in\curlyB^\Tt\cut\curlyB^\Uu$. Let 
$\beta=\movin^\Tt(\alpha)$.

\begin{sclmfour}\label{sclm:beta_not_T-unusual}
 $\beta$ is not $(\Tt,\Uu)$-unusual.\end{sclmfour}
\begin{proof}
Suppose otherwise, so $\beta$ is type (ii) and
$E^\Tt_\beta=E^\Uu_\beta=\emptyset$.
Since the comparison does 
not terminate 
early and $M^\Tt$ is $\rho^\Tt$-sound, we have 
$E^\Uu_{\beta+1}\neq\emptyset=E^\Tt_{\beta+1}$ and $\alpha=\beta+2$ and
\[ \rho^\Tt=\nu^\Uu_\beta<\lh^\Uu_\beta=((\rho^\Tt)^+)^{M^\Tt}. \]
So $\rho^\Tt$ is not the largest cardinal in $M^\Tt$, so
is also not in
$M^\Uu_\beta$. So 
$\exit^\Uu_\beta\pins 
M^\Uu_\beta$, so
$((\rho^\Tt)^+)^{M^\Tt}<((\rho^\Tt)^+)^{M^\Uu_\beta}$,
contradicting Claim 
\ref{clm:T-unusual}(\ref{item:T,U-unusual})
\ref{item:type(ii)}\ref{item:double_successor}.\end{proof}

So $E^\Uu_\beta=F'=F(Q^\Tt)$.

\begin{sclmfour}
 $\beta+1$ is not $(\Uu,\Tt)$-unusual.
\end{sclmfour}
\begin{proof}
This is like the proof of Subclaim \ref{sclm:beta_not_T-unusual}.
\end{proof}

By the subclaim and Claim \ref{clm:implies_unusual_ii},
(1) if $\beta+1\in\curlyB^\Uu$ then $\rho^\Tt<\rho^\Uu_{\beta+1}$,
and (2) one of the 
following 
holds:
\begin{enumerate}[label=\tu{(}\alph*\tu{)}]
 \item\label{case:alpha=beta+1} $\alpha=\beta+1$.
 \item\label{case:alpha=beta+2_type_2} $\alpha=\beta+2$ and 
$\lh^\Uu_{\beta+1}=\rho^\Tt$ and 
$E^\Uu_{\beta+1}$ is type 2.
 \item\label{case:alpha=beta+2_not_type_2} $\alpha=\beta+2$,
$\lh^\Uu_{\beta+1}=((\rho^\Tt)^+)^{M^\Tt}$ and $E^\Uu_{\beta+1}$ is
(i) type 1 or (ii) type 3.
 \item\label{case:alpha=beta+3} $\alpha=\beta+3$, 
$\lh^\Uu_{\beta+1}=\rho^\Tt$, 
$E^\Uu_{\beta+1}$ is type 2,
$\lh^\Uu_{\beta+2}=((\rho^\Tt)^+)^{M^\Tt}$ and $E^\Uu_{\beta+2}$ is  
(i) type 1 or (ii) 
type 3.
\end{enumerate}
The same general argument works in each case, but the details vary. We just 
discuss cases
\ref{case:alpha=beta+1}, \ref{case:alpha=beta+2_type_2}, 
\ref{case:alpha=beta+2_not_type_2}(i)
and sketch \ref{case:alpha=beta+3}(i). In each case let 
$\eps=\pred^\Uu(\beta+1)$ and 
$R=B^{*\Uu}_{\beta+1}$ and 
$\kappa=\crit(F')$. Note that if $m=-1$ then case \ref{case:alpha=beta+1} 
attains.

Consider case \ref{case:alpha=beta+1}. Since
$S^\Uu=\{0\}$, $R$ is a premouse. Let $d=\deg^\Uu(\alpha)$.
So $R$ is $d$-sound and $M^\Uu=\Ult_d(R,F')$. Clearly $d\geq m$.
We claim that  $d>m$ (so $R$ is $(m+1)$-sound)
and
$\rho_{m+1}^R=(\kappa^+)^R<\rho_m^R$.
For if $\rho_{m+1}^R>(\kappa^+)^R$ then
\[ \rho_{m+1}(M^\Uu)>\rho^\Tt\geq\rho_{m+1}(M^\Tt); \]
and if $\rho_{m+1}^R\leq\kappa$ then $\rho_{m+1}(M^\Uu)<\rho^\Tt$ and $M^\Uu$ 
is 
$\gamma^\Tt$-sound,
so $B^\Tt$ is not exceptional, contradicting \ref{lem:exceptional_equiv}.

Let $U_i=\Ult_i(R,F')$, so $M^\Uu=U_d$.
Note  $\kappa^{++R}=\kappa^{++B^\Tt}$ and
$(\rho^\Tt)^{+U_m}=(\rho^\Tt)^{+M^\Tt}$,
but $\rho_{m+1}^{U_m}=\rho^\Tt$, so $U_m\notin 
M^\Tt=U_d$, 
so arguing like in the proof of \ref{lem:fully_elem_condensation},
$U_m=U_d$
and the factor map $\sigma:U_m\to U_d$ is the identity 
(this does not use condensation). Letting $\pi=i^{R,m}_{F'}$ and $H=R$, then 
$H,\pi$ are as in 
\ref{dfn:exceptional_core}.

Now consider case \ref{case:alpha=beta+2_type_2}. Note that $R=\exit^\Uu_\eps$, 
 is active type 
2 and 
$\OR^R=(\kappa^+)^{B^\Tt}$. Note that $\deg^\Uu(\beta+2)=m$ and 
$\crit(F^R)=\crit^\Uu_{\beta+1}$, so $\pred^\Uu(\beta+2)=\pred^\Uu(\eps+1)$ and 
$B^{*\Uu}_{\eps+1}=B^{*\Uu}_{\beta+2}$ and $\deg^\Uu(\eps+1)=m$. Let 
$H=B^\Uu_{\eps+1}$. Then
$\Ult_m(H,F')=M^\Tt$
and letting $\pi=i^{H,m}_{F'}$, then $H,\pi$ are as in 
\ref{dfn:exceptional_core}.

Now consider case \ref{case:alpha=beta+2_not_type_2}(i).

\begin{sclmfour} In case \ref{case:alpha=beta+2_not_type_2}(i), $E^\Uu_\eps$ is 
the ``preimage'' of 
$E^\Uu_{\beta+1}$ under $i^{*\Uu}_{\beta+1}$ and
$\lh^\Uu_\eps=(\kappa^{++})^{B^\Tt}$.\end{sclmfour}
\begin{proof} We have $\exit^\Uu_\eps\ins R$ and
$(\kappa^+)^{B^\Tt}=(\kappa^+)^{\exit^\Uu_\beta}=(\kappa^+)^R=(\kappa^+)^{
\exit^\Uu_\eps}
<\lh^\Uu_\eps$.
We have $(\kappa^{++})^R\geq(\kappa^{++})^{B^\Tt}$ and
if $\beta+1\in\curlyB^\Uu$ then 
$(\kappa^{++})^R>(\kappa^{++})^{B^\Tt}$;
the latter is because $\exit^\Uu_{\beta+1}\not\ins M^\Tt$ and 
$\exit^\Uu_{\beta+1}$ projects 
$\leq\rho^\Tt$.
Let
$P\ins R$ and $p\in\{-1\}\un\om$ with
\[ \rho_{p+1}^P\leq(\kappa^+)^{B^\Tt}<(\kappa^{++})^{B^\Tt}
=(\kappa^{++})^P\leq\rho_p^P\]
(so $P$ is $p$-sound). By condensation, like before,
$U^{P,p}=\Ult_p(P,F')\ins M^\Uu_{\beta+1}$.
But $((\rho^\Tt)^+)^{U^{P,p}}=((\rho^\Tt)^+)^{M^\Tt}$,
and as $\nu^\Uu_{\beta+1}=\rho^\Tt$, therefore 
$U^{P,p}=\exit^\Uu_{\beta+1}$.

So $P$ 
is type 1, $p=0$, $\OR^P=(\kappa^{++})^{B^\Tt}$, and  
$E^\Uu_\eps=F^P$. Now 
$i^{*\Uu}_{\beta+1}$ 
is 
continuous at $(\kappa^+)^{B^\Tt}$. So if $P\pins R$ then $i^{*\Uu}_{\beta+1}$ 
is continuous at 
$\OR^P$, so 
$i^{*\Uu}_{\beta+1}(P)=\exit^\Uu_{\beta+1}$
(or $\psi_j(P)=\exit^\Uu_{\beta+1}$ where $j=i^{*\Uu}_{\beta+1}$).
If $P=R$ then 
$\Ult_p(P,F')=M^\Uu_{\beta+1}$ (even if $0<\deg^\Uu(\beta+1)$).
\end{proof}

Since $E^\Uu_\eps=F^P$ and $\crit(F^P)=\crit(F')$, $\pred^\Uu(\eps+1)=\eps$ and 
$B^{*\Uu}_{\eps+1}=R$ and $\deg^\Uu(\eps+1)=\deg^\Uu(\beta+1)$. Also, 
$\pred^\Uu(\beta+2)=\beta+1$ 
and $m=\deg^\Uu(\beta+2)=\deg^\Uu(\eps+1)$. Using this, and letting 
$H=B^\Uu_{\eps+1}$, we get
$\Ult_m(H,F')=M^\Tt$
and letting $\pi=i^{H,m}_{F'}$, then $H,\pi$ are as in 
\ref{dfn:exceptional_core}.

Finally consider case \ref{case:alpha=beta+3}(i). For illustration, assume that 
$\beta+2\notin\curlyB^\Uu$. Let $\chi=\pred^\Uu(\beta+2)$ and 
$S=B^{*\Uu}_{\beta+2}$ and 
$c=\deg^\Uu(\beta+2)$. A combination of the preceding arguments gives the 
following: 
\begin{enumerate}[label=--]
\item $\exit^\Uu_\eps$ is the (type 2) preimage of $\exit^\Uu_{\beta+1}$ under 
$i^{*\Uu}_{\beta+1}$,
\item $\pred^\Uu(\eps+1)=\chi$ and $B^{*\Uu}_{\eps+1}=S$ and 
$\deg^\Uu(\eps+1)=c$,
\item $\exit^\Uu_{\eps+1}$ is the (type 1) preimage of $\exit^\Uu_{\beta+2}$
under the map $\sigma$ defined 
below,
\item $\eps=\pred^\Uu(\eps+2)$ and $\deg^\Uu(\eps+2)=0$,
\item $\beta+1=\pred^\Uu(\beta+3)$ and $m=\deg^\Uu(\beta+3)=0$.
\end{enumerate}
Let $J=B^\Uu_{\eps+1}$ and $H=B^\Uu_{\eps+2}$. Then also,
$\Ult_c(J,F')=B^\Uu_{\beta+2}$
and letting $\sigma=i^{J,c}_{F'}$, then 
$\sigma(\exit^\Uu_{\eps+1})=\exit^\Uu_{\beta+2}$ (as mentioned above), 
and
$\Ult_0(H,F') = M^\Tt$,
etc.

Cases \ref{case:alpha=beta+2_not_type_2}(ii) and \ref{case:alpha=beta+3}(ii) 
are fairly similar to the preceding cases.
However, while in the preceding cases
there is always some $\zeta<\rho^\Tt$
such that
\[ M^\Tt=\Hull_{m+1}^{M^\Tt}(\zeta\cup\{\pvec_m^M,z_{m+1}^M\}),\]
there is no such $\zeta$ in \ref{case:alpha=beta+2_not_type_2}(ii) and 
\ref{case:alpha=beta+3}(ii).
\end{casesix}

There is just one case left:

\begin{casesix}
The comparison terminates early
(so by Subclaim \ref{sclm:terminates_early}, $S^\Tt=\{0\}=S^\Uu$ and 
$M^\Tt=M^\Uu$). 

We may assume that $\alpha$ is $(\Tt,\Uu)$-unusual (type (ii)). Let 
$\xi_0<\xi_1<\alpha=\xi_1+1$ 
witness this. So 
$E^\Tt_{\xi_1}=F'=F(Q^\Uu)$. We have  $M^\Uu= 
M^\Tt$. Let $H=M^\Tt_{\xi_0}$. Then
$\Ult_m(H,F')=M^\Tt=M^\Uu$
etc.\qedhere
\end{casesix}
\end{proof}

Since we now have an iterate $B'$ of $B$ with an exceptional core, the next 
claim completes the proof  of the theorem:

\begin{clmfour}\label{clm:exceptional_core_pullback}
Suppose that $B$ is exceptional and let $B'$ be a cephalanx non-dropping 
iterate of $B$. 
If $B'$ has an exceptional core then so does $B$.
\end{clmfour}
\begin{proof}
The proof is similar to \ref{thm:no_iterable_sound_bicephalus}, but with some 
extra 
argument. We 
assume that $m\geq 0$ and leave the other case to the reader (the main 
distinction in that case is 
that even though $m=-1$, all ultrapower embeddings are at least $\rSigma_1$ 
elementary). Fix
$H,\kappa,F,X$ and $\pi:H\to M$ as in \ref{dfn:exceptional_core}. Let 
$B'=(\gamma',\rho',M',Q')$ 
and fix 
$H',\kappa',F',X',\pi'$ as in \ref{dfn:exceptional_core} for $B'$.
Suppose  $B'$ has an exceptional core.
Let $i:M\to M'$ and $j:Q\to Q'$ be the iteration maps.
So $j=i\rest(B||\gamma^{+M})$. 
Note $i(\kappa,\pvec_m^M,z_{m+1}^M)=(\kappa',\pvec_m^{M'},z_{m+1}^{M'})$,
and for $\alpha<\gamma^{+M}$, we have $X\inter\alpha\in B||\gamma^{+M}$ and
\begin{equation}\label{eqn:X_pres} i(X\inter\alpha)=X'\inter i(\alpha).
\end{equation}
From these  facts, and because
$X'=(\gamma')^{+M'}\inter\rg(\pi')$, it is easy to see that
$X=\gamma^{+M}\inter\rg(\pi)$.
It remains to see
$H||\kappa^{++H}=M|\kappa^{++M}$.

Let $Y=\rg(\pi)\inter\gamma^{++M}$.
Let $\sigma=i^M_F$ and
$Z = \rg(\sigma)\inter\gamma^{++M}$.
It suffices to see that $Y=Z$. Let $Y',\sigma',Z'$ be defined analogously from 
$B'$. Because $B'$ 
has an exceptional core, Lemma \ref{lem:exceptional_core_facts} applies,
and $Y'=Z'$ follows. We will use this to deduce that $Y=Z$,
by breaking $Y$ and $Z$ into unions of small pieces, and considering
how they move under the  iteration map $i$.

\begin{sclmfour}\label{sclm:Y,Y'} For any $\alpha<\gamma^{++M}$, we have 
$\alpha\in Y$ iff 
$i(\alpha)\in 
Y'$.\end{sclmfour}
\begin{proof}
If $\alpha\in Y$ then $i(\alpha)\in Y'$ because $i``X\sub X'$ and 
$i(z_{m+1}^M)=z_{m+1}^{M'}$.

Suppose $\alpha\notin Y$. For $\beta<\gamma^{+M}$ and $\delta<\rho_m^M$ let
$Y_{\beta,\delta}$ be the set of all $\xi<\gamma^{++M}$ such that 
\[ \xi\in\Hull_{m+1}^M((X\inter\beta)\un z_{m+1}^M\un\pvec_m^M),\]
as witnessed by some theory below
$\Th_{\rSigma_m}^M(\delta\un\{\vec{p}_m^M\})$.
(See \cite[\S2]{fsit}, in particular, the stratification of $\rSigma_{m+1}$ 
described there, for 
more details. If $m=0$ this needs to be interpreted appropriately; for example, 
if
$M$ is passive and $\OR^M$ is divisible by $\om^2$,
that the $\rSigma_1$ 
fact should hold  in $M|\delta$.)
Then $Y_{\beta,\delta}\in M$.
Define $Y'_{\beta,\delta}$ analogously over $M'$. Let 
$I=\rho_m^M\cross\gamma^{+M}$.
Using line (\ref{eqn:X_pres}), we get
$i(Y_{\beta,\delta})=Y'_{i(\beta),i(\delta)}$,
and note
$Y' =\bigcup_{(\beta,\delta)\in I} i(Y_{\beta,\delta})$.
The fact that $i(\alpha)\notin Y'$ follows easily.
\end{proof}
\begin{sclmfour} For any $\alpha<\gamma^{++M}$, we have $\alpha\in Z$ iff 
$i(\alpha)\in 
Z'$.\end{sclmfour}
\begin{proof}
Let $\alpha<\gamma^{++M}$.
Let $\beta<\kappa^{++M}$ with $\alpha<\sigma(\beta)$.
Fix a surjection $f:\kappa^{+M}\to\beta$ in $M$.  So 
$\sigma(f):\gamma^{+M}\to\sigma(\beta)$ is a surjection in $M$,
and note that $\rg(\sigma)\inter\sigma(\beta)=\sigma(f)``X$.

Now we claim that $i(\sigma(f))=\sigma'(i(f))$.
For let $C$ be the prewellorder of $\kappa^{+M}$ corresponding
to $f$ (so $(\delta,\varepsilon)\in C$ iff $f(\delta)\leq 
f(\varepsilon)$). Then it suffices
to see that $i(\sigma(C))=\sigma'(i(C))$.
But this holds by continuity at $\kappa^{+M}$
and because $i(\sigma(D))=\sigma'(i(D))$ for all $D\in\pow(\kappa)\inter M$.

So let $f'=i(f)$ and $\beta'=i(\beta)=\rg(f')$. Then 
$\sigma'(\beta')=\rg(i(\sigma(f)))$,
so $i(\alpha)<\sigma'(\beta')$.
Therefore $i(\alpha)\in Z'$ iff $i(\alpha)\in\sigma'(f')``X'$
iff $i(\alpha)\in i(\sigma(f))``X'$.

But we have $X'=\bigcup_{\delta<\gamma^{+M}}i(X\inter\delta)$.
So $\alpha\in Z$ iff $\alpha\in\sigma(f)``X$
iff there is $\delta<\gamma^{+M}$ such that 
$\alpha\in\sigma(f)``(X\inter\delta)$ iff there is $\delta'<i(\gamma)^{+M'}$
such that $i(\alpha)\in i(\sigma(f))``(X'\inter\delta')$
iff $i(\alpha)\in i(\sigma(f))``X'$ iff $i(\alpha)\in Z'$, as desired.
\end{proof}

Clearly by the subclaims, we have $Y=Z$, as desired.

This completes the proof of the claim, and hence the theorem.
\end{proof}
\renewcommand{\qedsymbol}{}\end{proof}

\section{Condensation from solidity and normal 
iterability}\label{sec:condensation}

By  $(k+1)$-condensation\footnote{Cf. \cite[pp. 87--88]{fsit} or 
\cite[Theorem 9.3.2]{imlc}.}, if $H,M$ are $(k+1)$-sound premice such that $M$ 
is $(k,\om_1,\om+1)^*$-iterable
and $\pi:H\to M$  a near $k$-embedding
with $\crit(\pi)\geq\rho$ where $\rho=\rho_{k+1}^H$, then ($*$Con) either
$H\ins M$ or [$M|\rho$ is active and $H\pins\Ult(M|\rho,F^{M|\rho})$].

We now prove that 
$(k,\om_1+1)$-iterability suffices for this result.
In our proof, we will replace the phalanx used in 
the 
standard proof with a cephal, and avoid  Dodd-Jensen. 
We will in fact prove a partial analogue of the more refined version
  \cite[Theorem 9.3.2]{imlc} (but for Mitchell-Steel indexing,
with superstrongs). We do not achieve a full analogue here,
because in the case that $H\notin M$ we encounter an obstacle in connection 
with exceptional cephalanxes.
So in this sense we do not quite prove full 
condensation. However, if assume also that $M$ is $(k+1)$-solid, we can deduce 
the full analogous
conclusion.\footnote{It will
in fact be shown in \cite{fsfni} that $M$ \emph{is} $(k+1)$-solid
(from $(k,\om_1+1)$-iterability), so the two papers together
will prove the full result.}

\begin{dfn}
 Let $M$ be a $k$-sound premouse and $\zeta_{k+1}^M\leq\rho\leq\rho_k^M$. 
The 
\dfnemph{$\rho$-solid-core} of $M$ is
$H=\cHull_{k+1}^M(\rho\un z_{k+1}^M\un\pvec_k^M)$,
and the \dfnemph{$\rho$-solid-core map} is the uncollapse map $\pi:H\to M$.
\end{dfn}

The $\rho$-solid-core map is a $k$-embedding, since $H\notin M$ and by
\ref{lem:k-lifting_dichotomy}.

\begin{tm}[Condensation from solidity]\label{thm:condensation}
 Let $M$ be a $k$-sound, $(k,\om_1+1)$-iterable premouse.
 Let $H$ be a $\rho$-sound premouse with
$\rho\in[\rho_{k+1}^H,\rho_k^H)$
   an $H$-cardinal; let 
$\gamma=\card^M(\rho)$.
Let
$\pi:H\to M$ be $k$-lifting with $\crit(\pi)\geq\rho$.
Then\tu{:}
\begin{enumerate}[ref=\tu{(}\arabic*\tu{)}]
\item\label{item:H_notin_M} If $H\notin M$ then\tu{:}
\begin{enumerate}[label=\tu{(}\alph*\tu{)}]
 \item\label{item:zeta_z_pres} $\zeta_{k+1}^H=\zeta_{k+1}^M\leq\rho$ and 
$\pi(z_{k+1}^H)=z_{k+1}^M$,
\item\label{item:H_is_rho-solid-core} $H$ is the $\rho$-solid-core of $M$ and 
$\pi$ is the 
$\rho$-solid-core map,
 \item\label{item:rho_k+1^H_avoids_interval} $\rho_{k+1}^H\notin[\gamma,\rho)$,
 \item\label{item:superstrong_case} if $\rho_{k+1}^H=\rho$ and 
$\rho^{+H}<\rho^{+M}$ then 
$M|\rho$ is active 
with a superstrong extender with critical point $\kappa$ and 
$\rho_{k+1}^M\leq(\kappa^+)^M<\rho$,
\item\label{item:rho_k+1^H_geq_rho_k+1^M} $\rho_{k+1}^H\geq\rho_{k+1}^M$,
\item\label{item:if_M_solid} if $M$ is $(k+1)$-solid then 
$\rho_{k+1}^H=\rho_{k+1}^M$,
 \item\label{item:H_rho-core} if $\rho_{k+1}^H=\rho_{k+1}^M$ then $H$ is the 
$\rho$-core of $M$, 
$\pi$ is the 
$\rho$-core map and $\pi(p_{k+1}^H)=p_{k+1}^M$.
\end{enumerate}
\item\label{item:H_in_M} If $H\in M$ then exactly one of the following 
holds\tu{:}
\begin{enumerate}[label=\tu{(}\alph*\tu{)}]
 \item\label{item:H_in_M_H_pins_M} $H\pins M$, or
 \item\label{item:H_in_M_active} $M|\rho$ is active with extender $F$ and 
$H\pins\Ult(M|\rho,F)$, 
or
 \item\label{item:H_in_M_passive_tp1} $M|\rho$ is passive, $N=M|\rho^{+H}$ is 
active type 1 and $H=\Ult_k(Q,F^N)$, where $Q\pins M$ is such 
that 
$\gamma^{+Q}=\rho$ and $\rho_{k+1}^Q=\gamma<\rho_k^Q$, or
 \item\label{item:H_in_M_active_tp1} $k=0$ and $H,M$ are active type 2 and 
$M|\rho$ is
active with a type 2 extender $F$ and letting $R=\Ult(M|\rho,F)$, 
then $N=R|\rho^{+H}$ is active 
type 1  and $H=\Ult_0(M|\rho,F^N)$.
\end{enumerate}
\end{enumerate}
\end{tm}

\begin{rem}
If we assume further that $H,M$ are 
$(k+1)$-sound, it is now easy to conclude that 
($*$Con) (stated above) holds. In fact, it suffices to 
assume that if $H\notin M$ then $M$ is $\rho$-sound, and if $H\in M$ 
then $H$ is 
$(k+1)$-sound.
\end{rem}

\begin{proof} 
\begin{clmseven}\label{clm:assumptions}We may assume:
\begin{enumerate}[label=--]
 \item $H,M$ have the same type,
 \item if $H,M$ are passive and $k=0$ then $\pi$ is cofinal, and hence
 \item $\pi$ is c-preserving.
\end{enumerate}
\end{clmseven}
\begin{proof}
Suppose that $H,M$ have different types.
Then $k=0$, $H$ is passive and $M$ is active.
Here $M$ might an unsquashed or squashed premouse.
In either case, note that
$\pi:H\to M'=M||\OR^M$ is still $0$-lifting,
and $M'$ is $\om$-sound with $\rho_\om^{M'}=\OR^{M'}$.
If part \ref{item:H_in_M} (of the conclusion of  Theorem \ref{thm:condensation})
 holds for 
$\pi:H\to M'$,
then it also holds for $\pi:H\to M$,
so we are done. So suppose 
part \ref{item:H_notin_M}
holds for $\pi:H\to M'$. Then because $M'$ is $1$-sound, 
by parts \ref{item:H_notin_M}\ref{item:if_M_solid},\ref{item:H_rho-core},
in fact $H=M'$ and $\pi=\id$. But $\rho_1^H<\rho_0^H$,
contradicting that $\rho_\om^{M'}=\OR^{M'}$.

Now suppose that $H,M$ are passive
and $k=0$ but $\pi``\OR^H$ is bounded in $\OR^M$.
Let $M'=M||\sup\pi``\OR^H$. Note that $\pi:H\to M'$ is also $0$-lifting,
and is cofinal in $\OR^{M'}$, and $M'$ is $\om$-sound.
And $M'\pins M$,
since otherwise $M|\OR^{M'}$ is active, but then
$M'\sats\ZFC^-$, which contradicts the fact that 
\[ \rg(\pi)=\Hull_1^{M'}(\rho\cup\{\pi(p_{1}^H)\})\]
is cofinal in $M'$. Again if part \ref{item:H_in_M}
holds for $\pi:H\to M'$, then we are done,
so suppose part \ref{item:H_notin_M} holds.
As $M'$ is $1$-sound, then $H=M'$ and $\pi=\id$,
so $H\pins M$, so we are done.
\end{proof}

From now on we make the assumptions stated in Claim \ref{clm:assumptions}.
Using \ref{lem:k-lifting_dichotomy}, we get:

\begin{clmseven}\label{clm:k-embedding}
If $\rho^{+H}=\rho^{+M}$ or $\rho_{k+1}^H<\gamma$ then $H\notin M$ and $\pi$ 
is a 
$k$-embedding.\end{clmseven}

An easy calculation using the $\rho$-soundness of $H$ gives (cf. 
\cite[2.17]{extmax}):

\begin{clmseven}\label{clm:zeta_z_H}
$\zeta_{k+1}^H\leq\rho$ and $p_{k+1}^H\cut\rho=z_{k+1}^H\cut\rho$.
\end{clmseven}

\begin{clmseven}\label{clm:reduce_H_notin_M}
If $H\notin M$ and 
\ref{item:H_notin_M}\ref{item:zeta_z_pres},\ref{item:rho_k+1^H_avoids_interval} 
hold then so do 
\ref{item:H_notin_M}\ref{item:H_is_rho-solid-core},
\ref{item:rho_k+1^H_geq_rho_k+1^M},
\ref{item:if_M_solid},\ref{item:H_rho-core}.
\end{clmseven}
\begin{proof}
$\pi$ is a $k$-embedding. So
 \ref{item:H_is_rho-solid-core} follows from Claim 
\ref{clm:zeta_z_H} 
and \ref{item:zeta_z_pres}.
Part \ref{item:rho_k+1^H_geq_rho_k+1^M}: Let $\kappa=\rho_{k+1}^H$. If 
$\pow(\kappa)^H=\pow(\kappa)^M$ then \ref{item:rho_k+1^H_geq_rho_k+1^M} is 
clear. If 
$\pow(\kappa)^H\neq\pow(\kappa)^M$ then by \ref{item:rho_k+1^H_avoids_interval}, 
$\kappa=\rho$, so 
because $H\notin M$, \ref{item:rho_k+1^H_geq_rho_k+1^M} holds. Part
\ref{item:if_M_solid}: As
$\rho_{k+1}^M\leq\rho_{k+1}^H$, we get $\rho_{k+1}^M=\rho_{k+1}^H$ 
by
\cite[2.17]{extmax} and \ref{item:zeta_z_pres}. 
Part
\ref{item:H_rho-core}: Suppose $\rho_{k+1}^M=\rho_{k+1}^H$. We have 
$\rho_{k+1}^H\leq\rho$. If $\rho_{k+1}^M=\rho_{k+1}^H=\rho$ then 
as $H\notin M$, and by the solidity of $p_{k+1}^H=p_{k+1}^H\cut\rho$, we 
then have $p_{k+1}^M=\pi(p_{k+1}^H)$. Suppose 
$\rho_{k+1}^M=\rho_{k+1}^H=\kappa<\rho$,
so by \ref{item:rho_k+1^H_avoids_interval}, $\kappa<\gamma$. So 
$\pow(\kappa)^M=\pow(\kappa)^H$, so
$p_{k+1}^M\leq\pi(p_{k+1}^H)$, and so 
using the solidity of $p_{k+1}^H\cut\rho$, we get
$\pi(p_{k+1}^H\cut\rho)=p_{k+1}^M\cut\rho$, and since $\pi\rest\rho=\id$, we 
get 
$\pi(p_{k+1}^H)=p_{k+1}^M$. Now \ref{item:H_rho-core} easily follows.
\end{proof}

There are two main cases overall.

\begin{caseseven}\label{case:rho^+^H=rho^+^M} $\rho^{+H}=\rho^{+M}$.
 
We show \ref{item:H_notin_M}.
It suffices to prove \ref{item:H_notin_M}
\ref{item:zeta_z_pres},\ref{item:rho_k+1^H_avoids_interval},
\ref{item:superstrong_case}, by Claim 
\ref{clm:reduce_H_notin_M}.
 Part \ref{item:superstrong_case} is
trivial by case hypothesis.
By claims above,
$H\notin M$, $\pi$ is a $k$-embedding, and
$\zeta_{k+1}^H\leq\rho$. Using generalized 
solidity witnesses and as $\pow(\rho)^H=\pow(\rho)^M$, 
\ref{item:zeta_z_pres} follows. 
For part \ref{item:rho_k+1^H_avoids_interval},
we show  $\gamma=\rho$.
Suppose $\gamma<\rho$. So $\rho$ is an $H$-cardinal but non-$M$-cardinal,
and $\rho\leq\crit(\pi)$.
So $\rho=\gamma^{+H}$ and $\pi(\rho)=\gamma^{+M}=\rho^{+M}$.
But as $\rho^{+H}=\rho^{+M}$ this 
contradicts condensation for $\om$-sound 
mice (\ref{lem:fully_elem_condensation}).
\end{caseseven}

\begin{caseseven}\label{case:rho^+^H<rho^+^M} $\rho^{+H}<\rho^{+M}$.
 
Let $\eta=\rho^{+H}$. Either (I) $\crit(\pi)=\rho$, or 
(II) $\crit(\pi)=\eta<\rho_0^H$, or (III) $\eta=\rho_k^H=\rho_0^H$ and 
$\crit(\pi)$
does not exist.

Assume (III) holds. Then by Claim \ref{clm:assumptions} and case 
hypothesis, $H,M$ are active 
and $\rho$ is an 
$M$-cardinal. Letting $\mu=\crit(F^H)$, then
$\crit(F^M)=\mu<\rho$ and $(\mu^+)^H=(\mu^+)^M$ and $\pi\rest(\mu^+)^H=\id$.
It follows that $H,M$ are type 3
(consider the amenable predicate $E^H$ coding $F^H$
and its cofinality in $(H|(\mu^+)^H)\cross H||\OR^H$,
and likewise for $M$; we get $E^H\sub E^M$, and a contradiction
to the fact that $\OR^H<\OR^M$). 
So $\eta=\rho_k^H=\rho_0^H=\nu(F^H)<\nu(F^M)$,
and $F^H\rest\eta\sub F^M$. So by the ISC,
if $M|\eta$ is passive then $H\pins M$, 
and if $M|\eta$ is active then $H\pins\Ult(M|\eta,F^{M|\eta})$,
but as $\rho_{k+1}^H<\eta$, the latter is impossible.

From now on we assume  either (I) or (II) holds, so $\crit(\pi)$ exists.
We will produce an 
iterable cephal $C$ and use it to deduce the required facts.
If $M|\rho$ is 
passive then let
$J\pins M$ be least with $\rho_\om^J\leq\rho$
and $\eta\leq\OR^J$.
If $M|\rho$ is active and 
$\eta<\rho^{+U_\rho}$ where
$U_\rho=\Ult(M|\rho,F^{M|\rho})$, let 
$J\pins U_\rho$ be 
least with $\eta\leq\OR^J$ and $\rho_\om^J=\rho$. Otherwise leave $J$ 
undefined.
We may assume $H\neq J$
(otherwise \ref{item:H_in_M} holds). This ensures the cephal $C$ 
defined next is
non-trivial.

If $\rho$ is an $M$-cardinal, let $C = (\rho,H,J)$, a bicephalus. Here the fact 
that 
$H||\eta=J||\eta$, and hence $\eta=(\rho^+)^J$, follows from 
condensation for $\om$-sound 
mice (\ref{lem:fully_elem_condensation}).
If $\rho$ is a non-$M$-cardinal (so $\gamma<\rho<\gamma^{+M}$), 
let
$C=(\gamma,\rho,H,Q)$, where
$Q\pins M$ is least with $\rho\leq\OR^Q$ and $\rho_\om^Q=\gamma$;  here $C$ 
is a cephalanx,  by  \ref{lem:fully_elem_condensation}.

\begin{clmseven}\label{clm:C_iterable} $C$ is a non-trivial, 
$(\om_1+1)$-iterable 
cephal.\end{clmseven}

Assume this claim for now; we will use it to finish the proof.

\begin{clmseven}\label{clm:bicephalus_case} Suppose that either:
\begin{enumerate}[label=\tu{(}\roman*\tu{)}]
 \item\label{item:C_is_bicephalus} $\rho$ is a cardinal of $M$, so 
$C=(\rho,H,J)$ is a bicephalus; 
or
 \item\label{item:C_is_cephalanx} $\rho$ is not a cardinal of $M$ and 
$C=(\gamma,\rho,H,Q)$ is a 
passive cephalanx, $\OR^J=\eta$ and $J$ is type 3.
\end{enumerate}
Then \ref{item:H_notin_M} holds.\end{clmseven}
\begin{proof}
Note $N^C=J$ in case \ref{item:C_is_cephalanx}. 
Using 
\ref{thm:no_iterable_sound_bicephalus}/\ref{thm:nispc}, and as $H\neq J$ and 
$J$ is sound, note 
$\OR^J=\eta$, $J$ is type 1/3,
and letting $F=F^J$ and $\kappa=\crit(F)$, we have $\kappa<\gamma$ (in case 
\ref{item:C_is_bicephalus}, $\gamma=\rho$; in case 
\ref{item:C_is_cephalanx}, $J$ is type 
3), the $\kappa$-core $N$ of $H$  is $\kappa$-sound, and
\begin{equation}\label{eqn:kappa^+_agmt} 
N||(\kappa^+)^N=H||(\kappa^+)^H=M||(\kappa^+)^M
\end{equation}
(so $F$ is weakly amenable to $N$) and $H=\Ult_k(N,F)$.
It follows that 
$\rho_{k+1}^M\leq\rho_{k+1}^H\leq\rho_{k+1}^N\leq\kappa<\gamma$, so $H\notin M$ 
and $\pi$ is a 
$k$-embedding. Now 
$\zeta_{k+1}^N\leq\kappa$ since $N$ is $\kappa$-sound. But then by line 
(\ref{eqn:kappa^+_agmt})
and as $\pi,i^{N,k}_F$ preserve generalized solidity
witnesses, we get $i^{N,k}_F(z_{k+1}^N)=z_{k+1}^H$ and 
$z_{k+1}^M=\pi(z_{k+1}^H)$
and $\zeta_{k+1}^M=\zeta_{k+1}^H=\zeta_{k+1}^N$,
giving \ref{item:zeta_z_pres}.
Since $\rho_{k+1}^H\leq\zeta_{k+1}^H<\gamma$ we have 
\ref{item:rho_k+1^H_avoids_interval}, and  
\ref{item:superstrong_case} is trivial.
\end{proof}

\begin{clmseven} Suppose $\rho$ is a non-$M$-cardinal,
$C=(\gamma,\rho,H,Q)$ is a passive cephalanx \tu{(}so $N^C=J$\tu{)}, and if 
$\OR^J=\eta$ then $J$ 
is non-type 3. Then \ref{item:H_in_M}\ref{item:H_in_M_passive_tp1} 
holds.\end{clmseven}
\begin{proof}
Using \ref{thm:no_iterable_sound_passive_cephalanx}, $\OR^J=\eta$, $J$ is  
type 1,
$\rho_{k+1}^Q=\crit(F^J)=\gamma<\rho_k^Q$
and $H = \Ult_k(Q,F^J)$,
and since $J\ins Q$, therefore $\rho_{k+1}^H=\rho_{k+1}^Q=\gamma<\rho$.
\end{proof}

\begin{clmseven} Suppose  $\rho$ is a non-$M$-cardinal and $C$ is an 
active 
cephalanx. Then 
either \ref{item:H_notin_M} or
\ref{item:H_in_M}\ref{item:H_in_M_active_tp1} holds.\end{clmseven}
\begin{proof}
We have $C=(\gamma,\rho,H,Q)$ where $Q=M|\rho$ is active. Let $F=F^Q$. Apply 
\ref{thm:no_iterable_sound_active_cephalanx} to $C$.
If $C$ is non-exceptional then $C$ has a good core, and as before 
either
 $\rho_{k+1}^H<\gamma$ and \ref{item:H_notin_M} holds, 
or
 $\rho_{k+1}^H=\gamma$ 
and \ref{item:H_in_M}\ref{item:H_in_M_active_tp1} holds.

Now suppose that $C$ is exceptional, so $C$ has an exceptional core. Let
\[ K = \cHull^H_{k+1}(X\un z_{k+1}^H\un\pvec_k^H), \]
where $X$ is defined as in \ref{dfn:exceptional_core}. Let $\kappa=\crit(F)$.
By \ref{lem:exceptional_core_facts}, $K$ is $\kappa^{+K}$-sound, and 
$\rho_{k+1}^K\leq\kappa^{+K}$. 
Since 
$\kappa^{++K}=\kappa^{++M}$, therefore $K\notin M$. Since $Q\in M$ and
\[\Th_{\rSigma_{k+1}}^K(\pvec_{k+1}^K\un\kappa^{+K})
\text{ can be computed from }F^Q\text{ and
}\Th_{\rSigma_{k+1}}^H(\pvec_{k+1}^H\un\rho),\]
it follows that $H\notin M$, so $\pi$ is a $k$-embedding, as is 
$i^{K,k}_F$. So we must verify \ref{item:H_notin_M}.
\begin{sclmseven}\label{sclm:rho_k+1^K_leq_kappa}
If $\rho_{k+1}^K=\kappa^{+K}$ then \ref{item:H_notin_M} holds.
\end{sclmseven}
\begin{proof}
The argument here is similar to that used to illustrate the failure of solidity 
for
long extender premice. By 
\ref{lem:exceptional_core_facts}, we have
$\rho_{k+1}^H=\rho$ and 
$i^{K,k}_F(p_{k+1}^K)=p_{k+1}^H$ and both $K,H$ are $(k+1)$-sound.
Moreover,
\[ p_{k+1}^M\leq\pi(p_{k+1}^H)\conc\left<\rho\right> \]
because $K\notin M$ and by the calcuation above. Since $H$ is $(k+1)$-solid, 
therefore
$p_{k+1}^M\cut\rho=\pi(p_{k+1}^H)$.
But for $\alpha\leq\rho$,
\begin{equation}\label{eqn:which_theories_in}
\alpha<\rho\ \iff\ \Th_{\rSigma_{k+1}}^M(\pi(\pvec_{k+1}^H)\un\alpha)\in M, 
\end{equation}
because (in the case that $\alpha=\rho$) $H\notin M$, and (in the case that 
$\alpha<\rho$) 
$\crit(\pi)=\rho=\rho_{k+1}^H$. But line 
(\ref{eqn:which_theories_in}) gives
$p_{k+1}^M=\pi(p_{k+1}^H)\conc\left<\rho\right>$.

Now $z_{k+1}^H=p_{k+1}^H$ and $\zeta_{k+1}^H=\rho$, and 
\ref{item:H_notin_M}\ref{item:zeta_z_pres},\ref{item:rho_k+1^H_avoids_interval},
\ref
{item:superstrong_case} follow.
\end{proof}

Note that in the above case, $M$ is not $(k+1)$-solid.

\begin{sclmseven} If $\rho_{k+1}^K\leq\kappa<\zeta_{k+1}^K$ then 
\ref{item:H_notin_M} holds.\end{sclmseven}
\begin{proof}
Suppose $\rho_{k+1}^K\leq\kappa$. Then $\zeta_{k+1}^K<\kappa^{+K}$, 
as otherwise,
\[ \Th_{\rSigma_{k+1}}^K(\rho_{k+1}^K\un\pvec_{k+1}^K)\in K, \]
impossible. So 
$\kappa<\zeta_{k+1}^K<\kappa^{+K}$, so $z_{k+1}^H=i^K_F(z_{k+1}^K)$ and 
$\zeta_{k+1}^H=\sup 
i^K_F``\zeta_{k+1}^K$, by \cite[2.20]{extmax}. So $\gamma<\zeta_{k+1}^H<\rho$. 
So to verify \ref{item:zeta_z_pres}, it suffices to see
\[ \Th_{\rSigma_{k+1}}^M(\pi(z_{k+1}^H)\un\zeta_{k+1}^H\un\pvec_k^M)\notin M, \]
so suppose otherwise.
Then because $F^Q\in M$, we get
\[ \Th_{\rSigma_{k+1}}^K(\zeta_{k+1}^K\un z_{k+1}^K\un\pvec_k^K)\in M. \]
But $\pow(\kappa)^K=\pow(\kappa)^M$, so the above theory is in $K$, a 
contradiction.

We also have $\rho_{k+1}^H\leq\rho_{k+1}^K\leq\kappa$, so 
\ref{item:rho_k+1^H_avoids_interval} 
holds 
and \ref{item:superstrong_case} is trivial.\end{proof}

\begin{sclmseven} If $\zeta_{k+1}^K\leq\kappa$ then \ref{item:H_notin_M} 
holds.\end{sclmseven}
\begin{proof} This follows as before since 
$\pow(\kappa)^K=\pow(\kappa)^H=\pow(\kappa)^M$.\end{proof}

This completes the proof of the claim.
\end{proof}
\end{caseseven}

\begin{proof}[Sketch of Proof of Claim \ref{clm:C_iterable}]
The basic approach is to lift iteration trees on $C$ to iteration
trees on $M$. There are some details here that one must be careful with.
For illustration, we assume that
$C=(\gamma,\rho,H,M|\rho)$ is an active cephalanx. The other cases are similar 
(the bicephalus case 
a little 
different, but simpler). Recall that we have already
reduced to the case that $\pi$ is c-preserving.
In order to keep focus on the main 
points, we assume that $\pi$ is in fact c-$\nu$-preserving (see 
\cite{recon_con}). 
This will allow us to inductively maintain that all lifting maps we encounter 
are 
c-$\iota$-preserving, keeping the copying process smooth. (If instead, $\pi$ 
is not $\nu$-preserving, one should just combine the 
copying process to follow with that given in \cite{recon_con}. In the next 
section we do provide 
details of a copying process, with resurrection, which incorporates those extra
details.)

For a tree $\Tt$ on $C$ and $\alpha+1<\lh(\Tt)$, we say $\Tt$
\dfnemph{lift-drops at $\alpha+1$} iff $\alpha+1\in\curlyQ^\Tt$,
$\pred^\Tt(\alpha+1)\in\curlyB^\Tt$
and $[0,\alpha+1]_\Tt$ does not drop in model.

If $\Tt$ lift-drops at 
$\alpha+1$ then $Q$ is type 2, and letting $\beta=\pred^\Tt(\alpha+1)$, 
we have $E^\Tt_\beta=F(Q^\Tt_\beta)$ and 
$\crit(j^\Tt_{\beta,\alpha+1})=\lgcd(Q^\Tt_\beta)$.

Let $\Sigma$ be a $(k,\om_1+1)$-iteration strategy for $M$.
Consider building an iteration tree $\Tt$ on $C$, and
lifting this to a $k$-maximal tree $\Uu$ on
$M$, via $\Sigma$, inductively on $\lh(\Tt)$. Having defined
$(\Tt,\Uu)\rest\lambda+1$,
then for each $\alpha\leq\lambda$, letting $B_\alpha,M_\alpha,Q_\alpha$ be the 
models of
$\Tt$, and $S_\alpha=M^\Uu_\alpha$, and $W_\alpha=i_{0\alpha}^\Uu(Q)$ when 
$[0,\alpha]_\Uu\inter\mathscr{D}^\Uu=\emptyset$, we will have also defined 
embeddings
$\pi_\alpha$ and $\sigma_\alpha$, such that:
\begin{enumerate}
 \item ${<_\Tt}\rest(\lambda+1)={<_\Uu}\rest(\lambda+1)$. The drop structure of 
$\Uu$ matches that of 
$\Tt$, except for the 
following exceptions:
\begin{enumerate}[label=--]
\item If $\alpha\in\curlyB^\Tt$ then $[0,\alpha]_\Uu$ does not drop in model or 
degree (so 
$\deg^\Uu(\alpha)=k$).
\item If $\Tt$ lift-drops at $\alpha$ then $\Uu$ drops in model at $\alpha$.
\end{enumerate}
Moreover, if $\alpha\notin\curlyB^\Tt$ then 
$\deg^\Uu(\alpha)\geq\deg^\Tt(\alpha)$.
 \item\label{item:agmt_within_stage} Suppose $\alpha\in\curlyB^\Tt$. Then:
 \begin{enumerate}[label=--]
\item $\pi_\alpha:\core_0(M_\alpha)\to\core_0(S_\alpha)$
 is c-$\iota$-preserving $k$-lifting, and
\item $\sigma_\alpha:Q_\alpha\to W_\alpha$
 is an $\Sigma_0$-elementary simple embedding.
 \end{enumerate}Moreover,
 $\sigma_\alpha\sub\pi_\alpha$.
 \item Suppose $\alpha\in\curlyM^\Tt$. Then:
 \begin{enumerate}[label=--]
  \item $\pi_\alpha:\core_0(M_\alpha)\to\core_0(S_\alpha)$
is c-$\iota$-preserving $\deg^\Tt(\alpha)$-lifting, and
\item$\sigma_\alpha$ is undefined.
\end{enumerate}
\item Suppose $\alpha\in\curlyQ^\Tt$. Then:
\begin{enumerate}[label=--]
 \item $\pi_\alpha$ is undefined, and
\item $\sigma_\alpha:\core_0(Q_\alpha)\to\core_0(S_\alpha)$
is  c-$\iota$-preserving $\deg^\Tt(\alpha)$-lifting.
\end{enumerate}
\item\label{item:agmt_between_stages} Suppose $\alpha<\lambda$. Let 
$\beta\in(\alpha,\lambda]$. If
$E^\Tt_\alpha\in\es_+(M^\Tt_\alpha)$ let $\psi_\alpha=\psi_{\pi_\alpha}$; 
otherwise
let $\psi_\alpha=\psi_{\sigma_\alpha}$. Let
$\tau\in\{\pi_\beta,\sigma_\beta\}$. Then
\[ \psi_\alpha\rest\lh^\Tt_\alpha\sub\tau\text{ and }
\tau(\iota^\Tt_\alpha)=\psi_\alpha(\iota^\Tt_\alpha)=\nu^\Uu_\alpha.\]
\item 
Suppose $\alpha<\lambda$ and let 
$\delta=\pred^\Tt(\alpha+1)=\pred^\Uu(\alpha+1)$.

\begin{enumerate}
\item Suppose $\Tt$ drops in model at $\alpha+1$. Then so does $\Uu$. If
$\alpha+1\in\curlyM^\Tt$ then
$\psi_\delta(M^{*\Tt}_{\alpha+1})=S^{*\Uu}_{\alpha+1}$ and
\[ \pi_{\alpha+1}\com
i^{*\Tt}_{\alpha+1}=i^{*\Uu}_{\alpha+1}\com\psi_\delta\rest\core_0(M^{*\Tt}_{
\alpha+1}). \]
If $\alpha+1\in\curlyQ^\Tt$ then
$\psi_\delta(Q^{*\Tt}_{\alpha+1})=S^{*\Uu}_{\alpha+1}$ and
\[ \sigma_{\alpha+1}\com
j^{*\Tt}_{\alpha+1}=i^{*\Uu}_{\alpha+1}\com\psi_\delta\rest\core_0(Q^{*\Tt}_{
\alpha+1}). \]
\item Suppose $\Tt$ lift-drops at $\alpha+1$. Then $\Uu$
drops in model at $\alpha+1$ (but $[0,\delta]_\Uu$ does not drop in 
model or 
degree),
$S^{*\Uu}_{\alpha+1}=i^\Uu_{0\delta}(Q)=W_\delta$
and
\[ \sigma_{\alpha+1}\com
j^{\Tt}_{\delta,\alpha+1}=i^{*\Uu}_{\alpha+1}\com\sigma_\delta.\]
\end{enumerate}
\item If $\alpha<\lambda$ and $\alpha<_\Tt\beta\leq\lambda$ and
$(\alpha,\beta]_\Tt$ neither drops in model nor
lift-drops, then:
\begin{enumerate}[label=--]
\item If $M_\beta$ is defined then
$\pi_\beta\com i^\Tt_{\alpha,\beta}= i^\Uu_{\alpha,\beta}\com\pi_\alpha$.
\item If $Q_\beta$ is defined then
$\sigma_\beta\com j^\Tt_{\alpha,\beta}=i^\Uu_{\alpha,\beta}\com\sigma_\alpha$.
\end{enumerate}
\end{enumerate}

This completes the inductive hypotheses.

We now start the construction. We start with $\pi_0=\pi$ and $\sigma_0=\id$. 
Since
$\crit(\pi_0)=\rho$, we have $\sigma_0\sub\pi_0$.

Now let $E_\lambda=E^\Tt_\lambda$ be given. We define $F_\lambda=E^\Uu_\lambda$ 
by copying in the 
usual manner. That is:
\begin{enumerate}[label=(\roman*)]
 \item Suppose $E_\lambda\in\es_+(M_\lambda)$. Then:
 \begin{enumerate}[label=--]
 \item If $E_\lambda=F(M_\lambda)$ then $F_\lambda=F(S_\lambda)$.
 \item If $E_\lambda\neq F(M_\lambda)$ then 
$F_\lambda=\psi_{\pi_\lambda}(E_\lambda)$.
 \end{enumerate}
 \item Suppose $E_\lambda\notin\es_+(M_\lambda)$; so 
$E_\lambda\in\es_+(Q_\lambda)$. Then:
 \begin{enumerate}[label=--]
 \item If $E_\lambda=F(Q_\lambda)$ and $[0,\lambda]_\Tt$ does not drop or 
lift-drop 
then $F_\lambda=F(W_\lambda)$.
 \item If $E_\lambda=F(Q_\lambda)$ and $[0,\lambda]_\Tt$ drops or lift-drops 
then $F_\lambda=F(S_\lambda)$.
 \item If $E_\lambda\neq F(Q_\lambda)$ then 
$F_\lambda=\psi_{\sigma_\lambda}(E_\lambda)$.
 \end{enumerate}
\end{enumerate}
The agreement hypotheses and the fact that $\pi_\lambda$ and $\sigma_\lambda$ 
are 
c-$\iota$-preserving (if defined) ensures that this choice of $F_\lambda$ is 
legitimate.

Let $\beta=\pred^\Tt(\lambda+1)$ and $\kappa=\crit^\Tt_\lambda$.
We consider only the case that
\begin{equation}\label{eqn:the_case} \beta\in\curlyB^\Tt\text{ and 
}\kappa\leq\gamma(B^\Tt_\beta)\text{ and }\Tt\text{ does not drop in model at 
}\lambda+1.\end{equation}
For otherwise it is routine to propagate the inductive 
hypotheses,
except maybe for the $\iota$-preservation of the embeddings.
But we give the details for $\iota$-preservation in the case we consider,
and it is similar in general.
So suppose line (\ref{eqn:the_case}) holds.
We have $\beta=\pred^\Uu(\lambda+1)$ by
property \ref{item:agmt_between_stages}.

\renewcommand{\thecaseeight}{\Roman{caseeight}}
\begin{caseeight} $\lambda+1\in\curlyB^\Tt$.
 
In this case
$[0,\lambda+1]_\Uu$ does not drop in model or degree; this is because 
$\pi_\lambda$ is 
c-preserving and because if $E_\beta=F(Q_\beta)$ then 
$\kappa<\gamma(B^\Tt_\beta)$. By 
\ref{lem:k-lifting_facts} and properties
\ref{item:agmt_within_stage} and \ref{item:agmt_between_stages}, we can apply
(essentially)\footnote{We say \emph{essentially} because if $Q$ is type 3, 
$\sigma_\beta$ 
is a simple embedding, not an embedding between squashed premice.} the Shift 
Lemma to 
$(\pi_\beta,\psi_\lambda\rest \exit^\Tt_\lambda)$ and
$(\sigma_\beta,\psi_\lambda\rest \exit^\Tt_\lambda)$, to produce 
$\pi_{\lambda+1}$ and 
$\sigma_{\lambda+1}$. For the latter, we have
\[ \sigma_\beta:Q_\beta\to W_\beta\pins S_\beta=S^*_{\lambda+1},\]
and we set
\[
\sigma_{\lambda+1}([a,f]^{Q_\beta}_{E_\lambda})=[\psi_\lambda(a),\sigma_\beta(f)
]_{
F_\lambda}^{S_\beta}. \]
It follows easily that $\sigma_{\lambda+1}\sub\pi_{\lambda+1}$.

Now $\iota$-preservation for $\sigma_{\lambda+1}$ is immediate because this 
embedding is simple. We verify that $\pi_{\lambda+1}$ is 
$\iota$-preserving.
This is immediate unless $H$
is type 3, so assume this. So $M_\beta,S_\beta,M_{\lambda+1},S_{\lambda+1}$
are also type 3, so $\iota$-preservation just means $\nu$-preservation here.
Write $\nu^{M_\beta}=\nu(F^{M_\beta})$ and $\nu^{S_\beta}$ likewise.
Write $\psi_\beta=\psi_{\pi_\beta}$ and $\psi_{\lambda+1}$ likewise.
Write $\psi^\Tt_{\beta,\lambda+1}=\psi_{i^\Tt_{\beta,\lambda+1}}$
and $\psi^\Uu_{\beta,\lambda+1}$ likewise.
By induction, $\pi_\beta$ is $\nu$-preserving; that is,
$\psi_{\beta}(\nu^{M_\beta})=\nu^{S_\beta}$.
We must see that $\pi_{\lambda+1}$ is also; that is, that
\[ \psi_{\lambda+1}(\nu^{M_{\lambda+1}})=\nu^{S_{\lambda+1}}.\]
But note that
\[ 
\psi_{\lambda+1}\com\psi^\Tt_{\beta,\lambda+1}=
\psi^\Uu_{\beta,\lambda+1}\com\psi_{\beta}.\]
So if $i^\Tt_{\beta,\lambda+1}$ and $i^\Uu_{\beta,\lambda+1}$ are also 
$\nu$-preserving, then so is $\pi_{\lambda+1}$.
This holds in particular if $k>0$, by elementarity considerations.
So suppose $k=0$. Then
\[ M_{\lambda+1}=\Ult_0(M_\beta,E_\lambda), \]
and note that
\[ 
\Ult(M_{\lambda+1}^\sq,F^{M_{\lambda+1}})=
\Ult(\Ult(M_\beta^\sq,F^{M_\beta}),E_\lambda), \]
and the ultrapower maps commute, and a straightforward
calculation with extenders show that
\[ \psi^\Tt_{\beta,\lambda+1}:\Ult(M_\beta^\sq,F^{M_\beta})\to
\Ult(M_{\lambda+1}^\sq,F^{M_\lambda+1}) \]
(which, recall, is defined via the Shift Lemma) coincides with the resulting 
$E_\lambda$-ultrapower map.
Likewise for $\psi^\Uu_{\beta,\lambda+1}$.

Now let $\mu=\cof^{M_\beta}(\nu^{M_\beta})$. Then 
$\psi_{\beta}(\mu)=\cof^{S_\beta}(\psi_\beta(\nu^{M_\beta}))=
\cof^{S_\beta}(\nu^{S_\beta})$.

By the preceding remarks, if $\kappa\neq\mu$
then $\psi^\Tt_{\beta,\lambda+1}$ and $\psi^\Uu_{\beta,\lambda+1}$
are continuous at $\nu^{M_\beta}$ and $\nu^{S_\beta}$
respectively, so by commutativity, $\pi_{\lambda+1}$ is $\nu$-preserving.
So suppose $\kappa=\mu$, so $\pi_\beta(\kappa)=\cof^{S_\beta}(\nu^{S_\beta})$.
Let $f\in M_\beta$ with $f:\kappa\to\nu^{M_\beta}$ be cofinal.
Write $f^{M_\beta}=f$.
So
\[ f^{S_\beta}=\psi_{\beta}(f^{M_\beta}):
\pi_\beta(\kappa)\to\nu^{S_\beta}\]
is also cofinal. Let 
$f^{M_{\lambda+1}}=\psi^\Tt_{\beta,\lambda+1}(f^{M_\beta})$
and $f^{S_{\lambda+1}}$ be likewise, so commutativity gives
$\psi_{\lambda+1}(f^{M_{\lambda+1}})=f^{S_{\lambda+1}}$.
Note then that
\[ \nu^{M_{\lambda+1}}=\sup 
i_{\beta,\lambda+1}``\nu^{M_\beta}=
\sup f^{M_{\lambda+1}}``\kappa \]
and likewise
\[ \nu^{S_{\lambda+1}}=\sup i^\Uu_{\beta,\lambda+1}``\nu^{S_\beta}
=\sup f^{S_{\lambda+1}}``\pi_\beta(\kappa).\]
Since $\psi_{\lambda+1}(\kappa)=\pi_\beta(\kappa)$,
therefore
\[ \psi_{\lambda+1}(\nu^{M_{\lambda+1}})=
 \sup f^{S_{\lambda+1}}``\psi_{\lambda+1}(\kappa)=
 \sup f^{S_{\lambda+1}}``\pi_\beta(\kappa)=\nu^{S_{\lambda+1}},\]
as desired.

The remaining properties for this case are established as usual.
\end{caseeight}

\begin{caseeight}
$\lambda+1\in\curlyM^\Tt$.

This case is routine, using the fact that 
$E_\beta\in\es_+(M_\beta)$.\end{caseeight}

\begin{caseeight}
$\lambda+1\in\curlyQ^\Tt$.

So $\Tt$ lift-drops at $\lambda+1$, and so $E_\beta=F(Q_\beta)$ and 
$\crit^\Tt_\lambda=\gamma(B_\beta)$. Therefore $F_\beta=F(W_\beta)$ and 
$\crit^\Uu_\lambda=\sigma_\beta(\gamma(B_\beta))$ is the largest cardinal of 
$W_\beta$. Therefore 
\[ S^*_{\lambda+1}=W_\beta\pins S_\beta,\]
and in particular, $\Uu$ drops in model at $\lambda+1$. This is precisely enough 
to 
define $\sigma_{\lambda+1}$. Everything else is routine in this case.
\end{caseeight}

This completes the propagation of the properties to $(\Tt,\Uu)\rest\lambda+2$.

For limit $\lambda$, everything is routine.

This completes the sketch of the proof that $C$ is iterable, and so the proof 
of the theorem.
\end{proof}
\renewcommand{\qedsymbol}{}
\end{proof}

\section{A premouse inner model inheriting strong 
cardinals}\label{sec:ultra-stack}
Let $W\sats\ZFC$ be an iterable transitive model.
In this section we define a proper class premouse $L[\es]^W$ of $W$ which 
inherits all Woodin and strong cardinals from $W$.
(See \S\ref{sec:intro} for some introduction to this
and a comparison with Steel's local $K^c$-construction
of \cite{localKc}.)
The construction allows certain types of partial background extenders. However, 
all 
background extenders will be total in some ultrapower of $W$, and moreover,
assuming enough $\AC$, we 
will be able to lift 
iteration trees on $L[\es]^W$ to (non-dropping) iteration trees on $W$. 
The model $L[\es]^W$ is also outright definable over $W$.

Let us first point out that a fully backgrounded construction
can fail to inherit strong cardinals:

\begin{rem}\label{rem:strong}
Assume $\ZFC$ and suppose $\kappa$ is strong but there is no measurable 
cardinal $\mu>\kappa$.
Let $\left<N_\alpha\right>_{\alpha\leq\OR}$ be a fully backgrounded
$L[\es]$-construction; suppose that this does not break down,
so produces a model $L[\es]=N_\OR$ of height $\OR$.
Then we claim that $\es$ has no extenders with index $\geq\kappa$,
and hence $\kappa$ is certainly not strong in $L[\es]$.
For let $\zeta\in\OR$ be such that $L[\es]|\kappa=N_\zeta$.
Then we claim there is no $\alpha\geq\zeta$ such that $N_\alpha$ is active,
which suffices.
For suppose otherwise and let $\alpha$ be least such.
Since $\OR^{N_\zeta}=\kappa$ is a cardinal,
we have $\alpha>\zeta$. So $\alpha=\beta+1$
where $\beta=\OR^{N_\beta}>\kappa$ and $N_\beta=\J_\beta(N_\zeta)$.
Let $\mu=\crit(F^{N_{\beta+1}})$.
Since the construction is fully backgrounded, $\mu$ is measurable,
so $\mu\leq\kappa$.
Since $N_\beta=\J_\beta(N_\zeta)$ and by coherence,
it easily follows there is $\gamma<\mu$
such that $N_\zeta=\J_\kappa(N_\zeta|\gamma)$.
But then note that $\rho_\om(N_{\beta+1})\leq\gamma<\mu\leq\kappa$,
contradicting the fact that $N_\zeta=L[\es]|\kappa$.

Of course if the background construction does not
make unusual demands on background extenders, then $L[\es]|\kappa$ is 
closed under $\#$'s. In this case,
note that extenders $E$ in $V$ with $\crit(E)=\kappa$
do not cohere $L[\es]$.\end{rem}

Instead of using rank to measure the strength of extenders, we use:

\begin{dfn}
Let $E$ be an extender. The \dfnemph{strength} of $E$, denoted $\strength(E)$, 
is the largest 
$\rho$ such that $H_\rho\sub\Ult(V,E)$.
\end{dfn}

So $\strength(E)$ is always a cardinal.
The backgrounding we use is 
described as follows (in the definition, we imagine we 
are working inside $W$ as mentioned earlier):

\begin{dfn}\label{dfn:ultra_bkgd}Assume $\ZFC$. Let $\lambda\leq\OR+1$. An 
\dfnemph{ultra-backgrounded construction 
(of length $\lambda$)} is
a sequence $\left<S_\alpha\right>_{\alpha<\lambda}$ such that:
\begin{enumerate}
 \item Each $S_\alpha$ is a premouse. 
 \item\label{item:limit} Given a limit $\beta<\lambda$, 
$S_\beta=\liminf_{\alpha<\beta}S_\alpha$.
 \item Given $\beta=\alpha+1<\lambda$, either:
\begin{enumerate}
\item\label{item:construct} For each $n<\om$, $\core_n(S_\alpha)$
is $(n+1)$-universal and $\core_{n+1}(S_\alpha)$ is $(n+1)$-solid, and
$S_{\alpha+1}=\J(\core_\om(S_\alpha))$; or
 \item\label{item:bkgd_ext} $S_\alpha$ is passive and there is $F$ and an
extender $G$ such that $S_{\alpha+1}=(S_\alpha,F)$ and $F\rest\nu(F)\sub G$ and
$\strength(G)\geq\nu(F)$; or
 \item\label{item:ultra} $\alpha$ is a limit, $S_\alpha$ has a largest cardinal 
$\rho$, and there 
is an extender $G$ such that letting $\kappa=\crit(G)$, we have:
\begin{enumerate}
\item $\strength(G)\geq\rho$,
\item $\kappa\leq\rho\leq i_G(\kappa)$,
\item $\rho$ is a cardinal in $i_G(S_\alpha)$,
\item $(S_\alpha\sim i_G(S_\alpha))||\OR(S_\alpha)$,
\item $S_{\alpha+1}\pins i_G(S_\alpha)$,
\item $\rho_\om(S_{\alpha+1})=\rho$,
\item $\OR(S_\alpha)=(\rho^+)^{S_{\alpha+1}}$.\qedhere
\end{enumerate}
\end{enumerate}
 \end{enumerate}
\end{dfn}

\begin{dfn}\label{dfn:pm-ultra_bkgd}
Suppose that $V$ is a premouse (and $\ZFC$ holds). A 
\dfnemph{pm-ultra-\-back\-ground\-ed construction} is
a sequence $\left<S_\alpha\right>_{\alpha<\lambda}$ as in \ref{dfn:ultra_bkgd}, 
except 
that in (\ref{item:bkgd_ext}) and (\ref{item:ultra}) we also require that 
$G\in\es^V$ and $\nu(G)$ is a cardinal.
\end{dfn}

\begin{rem} When we refer to, for example, 
\ref{dfn:pm-ultra_bkgd}(\ref{item:ultra}), we mean the 
analogue of \ref{dfn:ultra_bkgd}(\ref{item:ultra}) for \ref{dfn:pm-ultra_bkgd}. 
We will mostly work 
explicitly with ultra-backgrounded constructions; the adaptation to 
pm-ultra-backgrounded is mostly 
obvious, so we mostly omit it. For all definitions to follow, we either 
implicitly or explicitly 
make the pm-ultra-backgrounded analogue, denoted by the prefix 
\emph{pm-}.\end{rem}

\begin{dfn}
Let $\CC=\left<S_\alpha\right>_{\alpha<\lambda}$ be an ultra-backgrounded
construction. Let $\beta<\lambda$. Then we say that $\beta$, or $S_\beta$, is 
\dfnemph{$\CC$-standard} iff \ref{dfn:ultra_bkgd}(\ref{item:limit}), 
(\ref{item:construct}) or 
(\ref{item:bkgd_ext}) holds (for $\beta$). We say that $\beta$ is 
\dfnemph{$\CC$-strongly 
standard} iff \ref{dfn:ultra_bkgd}(\ref{item:ultra}) does not hold. Given also 
$n\leq\om$, we 
say that $(\beta,n)$ is \dfnemph{$\CC$-relevant} iff either (i) $\beta$ is 
$\CC$-standard, or 
(ii) $\beta=\alpha+1$ and $\rho_n(N_{\alpha+1})=\rho_\om(N_{\alpha+1})$.
\end{dfn}

Clearly $\CC$-strongly standard implies $\CC$-standard.
The next lemma is routine:

\begin{lem}\label{lem:stage_projecting_to_card}
Let $\CC=\left<S_\alpha\right>$ be an ultra-backgrounded construction. Let 
$(\beta,n)$ be 
$\CC$-relevant. Let $\rho$ be a cardinal of
$S_\beta$ such that $\rho\leq\rho_n^{S_\beta}$. Let $P\pins S_\beta$ be such
that $\rho_\om^P=\rho$. Then there is $\alpha<\beta$ such that
$\core_0(P)=\core_\om(S_\alpha)$. 
\end{lem}

\begin{rem}\label{rem:xi}
It follows that if \ref{dfn:ultra_bkgd}(\ref{item:ultra}) holds, there is $\xi$ 
such that 
$S_{\alpha+1}=\core_\om(S_{\xi}^{i_G(\CC)})$,
because $\alpha$ is a limit and $\rho$ 
is a cardinal of $i_G(S_\alpha)$.
\end{rem}

\begin{lem}\label{lem:type_3_standard}
Let $\CC=\left<S_\alpha\right>$ be an ultra-backgrounded construction. Suppose 
that $S_{\alpha+1}$ 
is active type 1 or type 3 and $\rho_\om(S_{\alpha+1})=\nu(F(S_{\alpha+1}))$. 
Then $\alpha+1$ is 
$\CC$-standard, so $F(S_{\alpha+1})$ is backgrounded by a $V$-extender.
\end{lem}
\begin{proof}
 Suppose not and let $\alpha$ be the least counterexample.
 Let $\rho$ be the largest cardinal of $S_\alpha$.
 By \ref{lem:stage_projecting_to_card},
 $S_{\alpha}|\rho=S_\beta$ for some limit $\beta<\alpha$.
Let $G$ be as in \ref{dfn:ultra_bkgd}(\ref{item:ultra}) for $S_{\alpha+1}$.
So $\Ult(V,G)$ satisfies ``the lemma holds for $i^V_G(\CC\rest\beta)$'',
and note that $i^V_G(\rho)>\rho$,
and $i^V_G(S_\alpha)|i^V_G(\rho)=i^V_G(S_\beta)$
and $\rho$ is a cardinal of $i^V_G(S_\beta)$.
So by \ref{lem:stage_projecting_to_card}, 
$S_{\alpha+1}=\core_\om(S_{\gamma}^{i^V_G(\CC\rest\beta)})$
for some $\gamma$.
But because $S_{\alpha+1}$ is type 1 or 3 and
by the ISC, it follows that $S_\gamma^{i^V_G(\CC\rest\beta)}$ is already fully 
sound, so $S_{\alpha+1}=S_\gamma^{i^V_G(\CC\rest\beta)}$.
But then since $\Ult(V,G)$ thinks the lemma holds
for $i^V_G(\CC\rest\beta)$, therefore $\Ult(V,G)\sats$``$\gamma$ is 
$i^V_G(\CC\rest\beta)$-standard''.
So there is $H\in\Ult(V,G)$ such that in $\Ult(V,G)\sats$``$H$ is
a $\rho$-strong extender'' and $F\rest\nu\sub H$,
where $F=F^{S_{\alpha+1}}$ and $\nu=\nu(F)$.
But since $G$ is $\rho$-strong, so is $H$ (in $V$),
and since $\nu\leq\rho$ (as $F$ is type 1 or 3),
$H$ backgrounds $F$ in $V$, so $\alpha+1$ is $\CC$-standard, a contradiction.
\end{proof}

\begin{dfn}\label{dfn:nice_witness}
Let $\CC=\left<S_\alpha\right>$ be an ultra-backgrounded construction. Suppose 
that $\alpha+1$ 
is not $\CC$-standard, and let $\rho=\rho_\om(S_{\alpha+1})$. An extender $G$ is 
a 
\dfnemph{$\CC$-nice 
witness for $\alpha+1$} iff
$G$ witnesses \ref{dfn:ultra_bkgd}(\ref{item:ultra}), $i_G(\crit(G))>\rho$, and 
$S_{\alpha+1}$ is $i_G(\CC)$-strongly standard (in $\Ult(V,G)$).
\end{dfn}

\begin{lem}\label{lem:nice_witness}
Let $\CC=\left<S_\alpha\right>$ be an ultra-backgrounded construction. Suppose 
that $\alpha+1$ 
is not $\CC$-standard and let $\rho=\rho_\om(S_{\alpha+1})$. Then there is a 
$\CC$-nice witness
for 
$\alpha+1$.

Let $G$ be a $\CC$-nice witness for $\alpha+1$. Then:
\begin{enumerate}[label=--]
 \item If $\crit(G)<\rho$ then $\strength(G)$ is the the least cardinal 
$\geq\rho$.
 \item If $\crit(G)=\rho$ then $\strength(G)=\rho^+$.
 \item If condensation for $\om$-sound mice holds for all proper segments of 
$S_\alpha$ then $\rho$ 
is not measurable in $\Ult(V,G)$.
\end{enumerate}\end{lem}
\begin{proof}
Because $V$ is linearly iterable and $\alpha+1$ is 
not $\CC$-standard, there is an extender $H$ witnessing 
\ref{dfn:ultra_bkgd}(\ref{item:ultra}) and 
such that $\Ult(V,H)\sats$``$\xi$ is $i_H(\CC)$-strongly standard'', where 
$\xi$ is defined as in 
\ref{rem:xi}. Letting $G=i_H(H)\com H$, then $G$ is a nice witness 
($S_{\alpha+1}\pins 
i_G(S_\alpha)$ because in $\Ult(V,H)$, $i_H(H)$ coheres $i_H(S_\alpha)$ 
enough). 

Now let $G$ be a nice witness. The facts regarding $\strength(G)$ are easy. 
Suppose $F$ is a 
measure on $\rho$ in $U=\Ult(V,G)$. Then by condensation, $S_{\alpha+1}\pins 
i^U_F(S_{\alpha+1})$,
contradicting the niceness of $G$.
\end{proof}

For pm-ultra-backgrounding, we need to modify the notion of \emph{nice witness}
a little:

\begin{dfn}\label{dfn:pm-nice_witness}
Suppose $V$ is a premouse and let $\CC=\left<S_\alpha\right>$ be a 
pm-ultra-backgrounded 
construction. Suppose that $\alpha+1$ 
is not pm-$\CC$-standard, and let $\rho=\rho_\om(S_{\alpha+1})$. The
\dfnemph{pm-$\CC$-nice 
witness for $\alpha+1$} is the extender $G$ such that, letting $G_1$ be the 
least witness to 
\ref{dfn:pm-ultra_bkgd}(\ref{item:ultra}) (that is, the witness with $\lh(G_1)$ 
minimal), either:
\begin{enumerate}[label=\tu{(}\roman*\tu{)}]
 \item\label{item:single} $S_{\alpha+1}$ is pm-$i_{G_1}(\CC)$-strongly standard 
and $G=G_1$, or
 \item\label{item:composition} $S_{\alpha+1}$ is not pm-$i_{G_1}(\CC)$-strongly 
standard and letting
$G_2$ be the least witness to \ref{dfn:pm-ultra_bkgd}(\ref{item:ultra}) for 
$(i_{G_1}(\CC),S_{\alpha+1})$, then $G=G_2\com G_1$.\qedhere
\end{enumerate}
\end{dfn}

\begin{lem}
Suppose $V$ is a premouse and let $\CC=\left<S_\alpha\right>$ be a 
pm-ultra-backgrounded 
construction. Suppose that $\alpha+1$ 
is not pm-$\CC$-standard, let $\rho=\rho_\om(S_{\alpha+1})$ and let $G$ be the 
pm-$\CC$-nice witness for $\alpha+1$. Suppose that condensation for $\om$-sound 
mice holds for 
all proper segments of $S_\alpha$. Then:
\begin{enumerate}[label=--]
\item $S_{\alpha+1}$ is pm-$i_G(\CC)$-strongly standard.
\item $\rho$ is not measurable in $\Ult(V,G)$, so $i_G(\crit(G))>\rho$.
\item If \ref{dfn:pm-nice_witness}\ref{item:single} attains and $\rho$ is not a 
cardinal then 
$\nu(G)=\rho^+$.
 \item If \ref{dfn:pm-nice_witness}\ref{item:single} attains $\rho$ is a 
cardinal then either 
$\nu(G)=\rho$, or $G$ is type 1 and $\crit(G)=\rho$.
 \item If \ref{dfn:pm-nice_witness}\ref{item:composition} attains then $\rho$ is 
a cardinal 
and letting $G_1,G_2$ be as there,
$\nu(G_1)=\crit(G_2)=\rho$ and $G_2$ is type 1.
\end{enumerate}
\end{lem}
\begin{proof}
By coherence and the ISC, and using condensation as in 
\ref{lem:nice_witness}.\end{proof}

We now introduce what is, at least assuming global choice,
a natural maximal ultra-backgrounded construction:

\begin{dfn}\label{dfn:ultra-stack_con}
The \dfnemph{ultra-stack construction} is the sequence
$\left<R_\alpha\right>_{\alpha\leq\OR}$ such that $R_0=V_\om$,
the sequence is continuous at limits, and for each $\alpha<\OR$ we have the 
following.
Let $\rho=\OR(R_\alpha)$. Then $R_{\alpha+1}$ is the
stack of all sound premice $R$ such that $R_\alpha\pins R$ and
$\rho_\om^R=\rho$ and $R=\core_\om(S^\CC_\gamma)$ for some ultra-backgrounded
construction $\CC$ and $\gamma<\lh(\CC)$, assuming this stack forms a premouse
(if it does not, the construction not well-defined).
\end{dfn}

In order to prove that the ultra-stack construction inherits strong and Woodin 
cardinals, we will 
need to prove that certain pseudo-premice are in fact premice, just like in 
\cite{fsit}. So we make 
one further definition:

\begin{dfn}
Let $\lambda<\OR$. An \dfnemph{ultra-backgrounded pseudo-con\-struct\-ion 
(of length $\lambda+2$)} is
a sequence $\CC=\left<S_\alpha\right>_{\alpha<\lambda+2}$ such that:
\begin{enumerate}[label=--]
 \item $\CC\rest\lambda+1$ is an ultra-backgrounded construction and $S_\lambda$ 
is passive,
 \item For some $F$, $S_{\lambda+1}=(S_\lambda,F)$ is an active pseudo-premouse, 
and 
there is an extender $G$ such that $F\rest\nu(F)\sub G$ and
$\strength(G)\geq\nu(F)$.\qedhere
\end{enumerate}
\end{dfn}

\begin{dfn}
 An \dfnemph{almost normal} iteration tree $\Uu$ on a premouse $P$ is an 
iteration tree as defined 
in \cite{jonsson},\footnote{The only difference between these and normal trees 
is that it is not 
required that $\lh(E^\Tt_\alpha)<\lh(E^\Tt_\beta)$ for $\alpha<\beta$.} such 
that for all 
$\alpha+1<\beta+1<\lh(\Uu)$, we have $\nu(E^\Tt_\alpha)\leq\nu(E^\Tt_\beta)$.
\end{dfn}

\begin{rem} It is easy to see that if $P$ is a normally iterable premouse then 
$P$ is iterable with 
regard to almost normal trees.\end{rem}

We can now state the main theorem of this section:

\begin{tm}\label{thm:stack_well}
Assume $\ZF$. Let $W\sats\ZFC$ be a transitive class, and
suppose there is an $(\om_1+1)$-iteration strategy for $W$
for arbitrary coarse trees. 
Then:
\begin{enumerate}[label=\tu{(}\alph*\tu{)}]
 \item\label{item:ultra_bkgd_iterability}  If $\lambda\in\OR^W$ and 
$\CC=\left<S_\alpha\right>_{\alpha\leq\lambda}\in W\sats$``$\CC$ is an 
ultra-backgrounded 
construction'' and $n<\om$,
then $\core_n(S_\alpha)$
exists and is $(n,\om_1,\om_1+1)^*$-iterable,
and therefore $\core_n(S_\alpha)$ is $(n+1)$-universal
and $\core_{n+1}(S_\alpha)$ is $(n+1)$-solid.
\item\label{item:ultra-stack_well_def} The ultra-stack construction of $W$
is well-defined. Let $L[\es]$ be its final model.
\item\label{item:ultra-stack_iterable} If there is a class wellorder $<^W$ of 
$W$
then $L[\es]$ is  
$(0,\om_1,\om_1+1)^*$-iterable.\footnote{The class
wellorder $<^W$ need not be a class of $W$.
It is only used to allow us to select 
background extenders
canonically when copying iteration trees to $W$.}
 \item\label{item:strong_inherited} $W\sats$``$\kappa$ is strong'' iff 
$L[\es]\sats$``$\kappa$ is strong''.
\item\label{item:Woodin_inherited}  
If $W\sats$``$\delta$
is Woodin'' then $L[\es]\sats$``$\delta$ is Woodin''.
\end{enumerate}
\end{tm}

\begin{tm}\label{thm:pm-stack_well}
Let $W\sats\ZFC$ be a premouse \tu{(}possibly proper class\tu{)} which is 
$(\om,\om_1,\om_1+1)^*$-iterable. 
Then the conclusions of \ref{thm:stack_well} hold, with \emph{ultra} replaced 
by 
\emph{pm-ultra}.\end{tm}

\begin{rem}\label{rem:strong_inherit_also_for_A-strong}
Part \ref{item:strong_inherited} also holds for 
$A$-strong cardinals
$\kappa$, for $A\sub\OR$ such that $A$ is a class of $L[\es]$. (Here $\kappa$ 
is 
\emph{$A$-strong} iff for every $\eta$ there is an $\eta$-strong extender $G$ 
such that
$i_G(A)\inter\eta=A\inter\eta$.)

However, \ref{item:strong_inherited} does not seem to hold for local strength: 
it seems 
that we might have $\kappa$ being $\eta$-strong (some $\eta\in\OR$) but
$L[\es]\sats$``$\kappa$ is not $\eta$-strong''.
\end{rem}

\begin{proof}
Each part will depend on the sufficient iterability of certain structures, 
which we will 
establish in Claim \ref{clm:iterability} below.
We first reduce everything to that iterability.
We write $\left<R_\eta\right>_{\eta\in\OR^W}$ for the ultra-stack construction 
of $W$.

\begin{clmnine}\label{clm:R_eta_well-def} Work in $W$. Let $\eta\in\OR$. Then:
\begin{enumerate}[label=\tu{(}\roman*\tu{)}]
\item\label{item:R_eta_well-def} $R_\eta$ is well-defined.
\item\label{item:R_eta_realised_as_u-b_con} There is an ultra-backgrounded 
construction $\CC=\left<S_\alpha\right>_{\alpha\leq\lambda}$ with 
$S_\lambda=R_\eta$.
\item\label{item:no_proj_across} Let 
$\CC=\left<S_\alpha\right>_{\alpha\leq\lambda}$ be an ultra-backgrounded
construction such that $R_\eta=S_\beta$ for some $\beta\leq\lambda$.
Then for all $\alpha\in[\beta,\lambda]$, we have
$R_\eta=S_\alpha|\rho$ and $\rho_\om(S_\alpha)\geq\rho$, and if $\beta<\alpha$
then $\rho$ is a cardinal of $S_\alpha$.
\item\label{item:cons_merge} Let 
$\CC=\left<S_\alpha\right>_{\alpha\leq\lambda}$ and 
$\CC'=\left<S'_\alpha\right>_{\alpha\leq\lambda'}$ be
ultra-backgrounded constructions such that $R_\eta=S_\beta=S'_{\beta'}$
for some $\beta\leq\lambda$ and $\beta'\leq\lambda'$.
Suppose $\rho_\om(S_\lambda)=\rho$. Suppose there is
$\xi<\lambda'$ such that
$\core_\om(S'_\xi)=\core_\om(S_\lambda)$.
Then
\[ 
\CC\conc(\CC'\rest(\xi',\lambda']) 
\]
is also an ultra-backgrounded construction.
\end{enumerate}
\end{clmnine}
\begin{proof}
 The proof is by induction on $\eta$. When $\eta=0$ it is easy.
 
Suppose
$\eta$ is a limit. Clearly $R_\eta$ is well-defined,
giving part \ref{item:R_eta_well-def}. Part 
\ref{item:R_eta_realised_as_u-b_con}: 
Let 
$\left<\rho_\alpha\right>_{\xi<\gamma}$
enumerate the infinite cardinals of $R_\eta$. Note that by induction,
$\eta=\gamma$ and
$R_\xi=R_\eta|\rho_\xi$ and there is an ultra-backgrounded
construction $\CC_\xi=\left<S_{\xi\alpha}\right>_{\alpha\leq\lambda_\xi}$
with $R_\xi=S_{\xi\lambda_\xi}$. Also by induction
(applying part \ref{item:cons_merge}), we can merge these constructions
into a single ultra-backgrounded construction $\CC$ with last model $R_\eta$.
That is, we set
\[ 
\CC=(\CC_0\rest[0,\lambda_0))\conc(\CC_1\rest(\lambda_0',\lambda_1])\conc\ldots 
, \]
where $S_{1\lambda_0'}=S_{0\lambda_0}=R_\eta|\aleph_1^{R_\eta}$, etc.

For the next two parts, the proof is identical in the limit and successor cases:

Part \ref{item:no_proj_across}: Suppose otherwise
and let $\CC=\left<S_\alpha\right>_{\alpha\leq\lambda}$
be a counterexample of minimal length. Let $k<\om$
be such that $\core_k(S_\alpha)$ exists and $\rho_{k+1}^{S_\alpha}<\rho$.
In Claim \ref{clm:iterability} we will show that
\begin{equation}\label{eqn:core_k(S_alpha)_iterable}\core_k(S_\alpha)\text{
is }(k,\om_1,\om_1+1)^*\text{-iterable in (the background) }V.\end{equation}
It follows
(iterating this) that $\core_\om(S_\alpha)$ exists,
hence is $\om$-sound, and $\rho_\om(S_\alpha)<\rho$.
But then the existence of $\CC$ contradicts the maximality of $S_\eta$
(with respect to mice projecting to $\rho_\om(S_\alpha)$).

Part \ref{item:cons_merge}: This follows easily from the 
definitions
(noting that in \ref{dfn:ultra_bkgd}, we do not require the $V$-extenders
to cohere the \emph{construction} $\CC$ (i.e., the \emph{sequence} of models);
the only kind of coherence required is with respect to individual models 
$S_\alpha$).

Now suppose that $\eta=\xi+1$.

Part \ref{item:R_eta_well-def}:
 Suppose 
not. Then it is easy to see that we have
ultra-backgrounded constructions
\[ 
\CC=\left<S_\alpha\right>_{\alpha\leq\lambda'}\conc\left<S^\CC_\alpha\right>_{
\lambda'<\alpha\leq
\lambda^\CC}\]
and
\[ 
\priCC=\left<S_\alpha\right>_{\alpha\leq\lambda'}\conc\left<S^{\priCC}
_\alpha\right>_{
\lambda'<\alpha\leq\lambda^{\priCC}} \]
and
$\rho\in\OR$ such that letting $M'=S^\CC_{\lambda^\CC}$ and 
$N'=S^{\priCC}_{\lambda^{\priCC}}$:
\begin{enumerate}[label=--]
 \item $M=\core_\om(M')$ and $N=\core_\om(N')$ both exist,
 \item $\rho_\om^M=\rho=\rho_\om^N$,
 \item 
$S_{\lambda'}=M||\rho^{+M}=N||\rho^{+N}=M'||\rho^{+M'}=N'||\rho^{+N'}$, 
but
 \item $M\neq N$.
\end{enumerate}
It follows that $C=(\rho,M,N)$ is a sound, non-trivial bicephalus.
In Claim 
\ref{clm:iterability} 
below, we will show that
\begin{equation}\label{eqn:C_iterable_in_V}C\text{ is }(\om_1+1)\text{-iterable 
in (the background) }V,\end{equation}
contradicting \ref{thm:no_iterable_sound_bicephalus}.

Part \ref{item:R_eta_realised_as_u-b_con}:
This is much as in the limit case, but by merging
constructions which end in mice projecting to $\rho$.
\end{proof}

It easily follows (from Claim \ref{clm:iterability}) that:

\begin{clmnine}
Work in $W$. Then $L[\es]=R_\OR$ is well-defined,
and the cardinal segments of $L[\es]$ are exactly
the models $R_\alpha$ for $\alpha\in\OR$.
\end{clmnine}

So we have reduced \ref{item:ultra-stack_well_def}
 to Claim \ref{clm:iterability}.
We next reduce
\ref{item:strong_inherited} and \ref{item:Woodin_inherited}. The fact that 
every strong cardinal of $L[\es]$ is 
strong in $W$ is by \ref{lem:type_3_standard}. So suppose that either
$\kappa$ 
is strong in $W$ or $\delta$ is Woodin in $W$; we want to 
see that $\kappa$ is strong in $L[\es]$ or $\delta$ Woodin in $L[\es]$ 
respectively.
The key to the strong case is the following claim.

\begin{clmnine}\label{clm:automatic_cohere}Work in $W$. 
Let $\tau$ be a cardinal of $L[\es]$
and $R_\alpha$ be such that $\tau=\OR^{R_\alpha}$.
Then there is $\chi>\tau$ such
that if $F$ is any extender with arbitrary 
critical point and $\strength(F)\geq\chi$ then 
$i_F(R_\alpha)|\rho=R_\alpha$. \end{clmnine}
\begin{proof}
 Let $\chi$
 be such that there is an ultra-backgrounded construction
 $\CC\in\her_\chi$ with last model $R_\alpha$
 (using Claim \ref{clm:R_eta_well-def}),
 and such that $\her_\chi$ includes background extenders
 witnessing the clauses of \ref{dfn:ultra_bkgd} for $\CC$.
 It is straightforward to see that $\chi$ works.
\end{proof}

\begin{rem}\label{rem:claim_fails_for_fully_bkgd}
Note that Claim \ref{clm:automatic_cohere} can fail for fully 
backgrounded $L[\es]$-con\-struct\-ions,
by the last paragraph of \ref{rem:strong}. 
The key difference is that any mice projecting $<\tau$
which are added by $i_F(\CC)$ (when $F$ is strong enough)
are, by definition, added to the ultra-stack construction;
this, however, is not true of fully backgrounded constructions 
(the extenders used in the construction of these projecting mice might be total
in $\Ult(V,F)$, but partial in $V$).
\end{rem}

Using the claim, together with a slight variant of the proof of \cite[Lemma 
11.4]{fsit}, one can show that strength and Woodinness in
$W$ is absorbed by  $L[\es]$, as witnessed by restrictions of extenders
in $W$. The details of the argument relating to the 
uniqueness of the next
extender are somewhat different, so we describe the differences.
We will not 
reproduce all the details or definitions from that text, so the reader should 
have it in hand. 

Let $\tau$ be a cardinal of $L[\es]$,
and $F\in W$ be a $W$-extender with
\[ \crit(F)<\tau\leq i_F(\crit(F))\text{ and 
}i_F(L[\es])|\tau=L[\es]|\tau.\]
We get these as usual from Woodinness,
and by the preceding claim, we also get them
with $\crit(F)=\kappa$ if $\kappa<\tau$
and $\kappa$ is strong in $W$.
We adopt now the notation ``$\rho$'' and ``$G$'' of \cite[Lemma 11.4]{fsit}.

\begin{clmnine}
\tu{\cite[Lemma 11.4]{fsit}} holds for all $\rho<\tau$ such that $G$ is not 
type Z.

Therefore, if $W\sats$``$\kappa$ is strong'' then $L[\es]\sats$``$\kappa$ is 
strong'', and if $W\sats$``$\delta$ is Woodin''
then $L[\es]\sats$``$\delta$ is Woodin'',
and these facts are witnessed by restrictions of extenders in $W$.
\end{clmnine}
\begin{proof}
Recall that the proof is by induction on $\rho$.
Let
\[ \sigma:\Ult(L[\es],G)\to\Ult(L[\es],F)\]
be the natural factor map. Let $\xi=(\rho^+)^{\Ult(L[\es],G)}$. By 
Claim \ref{clm:iterability}, condensation holds for segments of $L[\es]$, 
and so because of the existence of $\sigma$, either:
\begin{enumerate}[label=(\roman*)]
 \item\label{item:rho_passive} $L[\es]|\rho$ is passive and 
$\Ult(L[\es],G)||\xi=L[\es]||\xi$, or
 \item\label{item:rho_active} $L[\es]|\rho$ is active and 
$\Ult(L[\es],G)||\xi=\Ult(L[\es],F^{L[\es]|\rho})||\xi$.
\end{enumerate}

Suppose first that $\rho$ is a cardinal 
of $L[\es]$, and so \ref{item:rho_passive} holds. Then there is an 
ultra-backgrounded construction with last 
model $P=(L[\es]||\xi,G)$. It follows that $\rho_\om^P=\rho$, so 
$P$ is fully sound, and therefore that $P\ins L[\es]$.

Now suppose that $\rho$ is not a cardinal of $L[\es]$. Let 
$\gamma=\card^{L[\es]}(\rho)$.
If $\rho$ is not a generator of $F$ then the previous argument adapts easily. So 
suppose $\rho$ is 
a generator of $F$. So $\crit(\sigma)=\rho=(\gamma^+)^{\Ult(L[\es],G)}$. In this 
case it seems that 
there 
might not be an ultra-backgrounded construction with last model 
$\Ult(L[\es],G)||\xi$. Let $G'$ be 
the trivial completion of $F\rest(\rho+1)$. Let 
$\xi'=(\rho^+)^{\Ult(L[\es],G')}$. Then 
$\Ult(L[\es],G')||\xi'=L[\es]||\xi'$ and $\gamma$ is the largest cardinal of 
$L[\es]||\xi'$.  
So there is an ultra-backgrounded construction with last model $L[\es]||\xi'$. 
Let 
$P=(L[\es]||\xi',G')$. Then there is a pseudo-ultra-backgrounded construction 
with last model $P$. 
By Claim \ref{clm:iterability} below, $P$ is $(0,\om_1,\om_1+1)^*$-iterable in 
$W$.
So by \cite[\S 10]{fsit} (combined with the 
generalization of the latter using the weak Dodd-Jensen property), $P$ is a 
premouse. Therefore 
either $G\in\es$, or $L[\es]|\rho$ is active and 
$G\in\es(\Ult(L[\es]|\rho,F^{L[\es]|\rho}))$, as 
required.
\end{proof}

The following claim completes the proof of the theorem,
as it establishes the iterability we have used above,
and part \ref{item:ultra-stack_iterable}.
Most of the rest of the paper is devoted to its proof;
we focus on one representative case of it:

\begin{clmnine}\label{clm:iterability}We have:
\begin{enumerate}[label=\tu{(}\roman*\tu{)}]
\item\label{item:S_alpha_it} For any $\lambda\in\OR^W$ and ultra-backgrounded 
construction $\CC=\left<S_\alpha\right>_{\alpha\leq\lambda}$
of $W$, and $n<\om$,
$\core_n(S_\lambda^\CC)$ exists and is $(n,\om_1,\om_1+1)^*$-iterable.
\item\label{item:bicephalus_it} The bicephalus $C$ defined in the proof of 
Claim \ref{clm:R_eta_well-def} is 
$(\om_1+1)$-iterable.
\item\label{item:pseudo-pm_it} For any ultra-backgrounded pseudo-construction 
of $W$, with last model $P$, $P$ is 
$(0,\om_1,\om_1+1)^*$-iterable.
\item\label{item:ultra-stack_it} If there is
a class wellorder $<^W$ of $W$ then $L[\es]$
is $(0,\om_1,\om_1+1)^*$-iterable.
\end{enumerate}
\end{clmnine}

\begin{proof}
We focus on the the iterability of $C=(\rho^C,M,N)$ (part 
\ref{item:bicephalus_it}); 
parts \ref{item:S_alpha_it} and \ref{item:pseudo-pm_it} are 
mostly simplifications of this.
At the end we state some adaptations used for
\ref{item:ultra-stack_it}.
The main difference between the present iterability proof and that for a
standard $L[\es]$-construction is in the resurrection process. The details of 
this process will be dealt with in a manner similar to that in \cite{recon_res},
and moreover, the 
resurrection process of \cite{recon_res} will need to be folded into the present 
one. We follow 
the iterability proof of \cite{recon_res} closely. In one regard, the present 
proof is slightly 
simpler: in \cite{recon_res}, arbitrary standard trees were considered, 
whereas here we 
deal with a more restricted class  (roughly, normal) of trees.

In
the pm-ultra-backgrounded setting, i.e. the proof of \ref{thm:pm-stack_well}, 
the natural 
adaptation of the proof to follow lifts a tree on $C$ to an almost normal tree 
$\Uu$ on
$W$. We leave the 
verification of this to the reader. Likewise, its adaptation to stacks of 
normal 
trees on $\core_n(S_\alpha^\CC)$ and $P$ (parts \ref{item:S_alpha_it} and
\ref{item:pseudo-pm_it}) produces stacks of almost normal trees on $W$. This 
ensures that we only use the $(\om,\om_1,\om_1+1)^*$-iterability of $W$ in this 
context. We ensure that this works by our arrangement of the proof to follow,
which increases the work involved a little.
For the adaptations to \ref{thm:pm-stack_well}, one should use background 
extenders $G$ with $\lh(G)$ 
minimal 
(when witnessing \ref{dfn:pm-ultra_bkgd}(\ref{item:bkgd_ext})), and use pm-nice 
witnesses (but when 
the pm-nice witness is as in \ref{dfn:pm-nice_witness}\ref{item:composition}, 
one must use 
the two extenders $G_1$ and $G_2$ in $\Uu$).

Let $\Sigma_W$ be an iteration strategy for $W$. 
We will describe a strategy $\Sigma_C$ for player $\playerII$ in the 
$(\om_1+1)$-iteration game 
on $C$. Let $\Tt$ be an iteration tree on $C$ which is via $\Sigma_C$. 
Then we will have inductively constructed a tree $\Uu$ on 
$W$, via $\Sigma_W$, such that $\Tt$ lifts to $\Uu$ (in a manner to be 
specified), and if $\Tt$ has limit length, we will use $\Sigma_W(\Uu)$ to 
define 
$\Sigma_C(\Tt)$. 

We say that an iteration tree $\Vv$ on $W$ is \dfnemph{neat} 
iff $\Vv$ 
is non-overlapping and
$\strength^{M^{\Vv}_\alpha}(E^{\Vv}_\alpha)\leq
\strength^{M^{\Vv}_\beta}(E^{\Vv}_\beta)$ for $\alpha<\beta$.
The tree $\Uu$ may use padding, but the tree $\Vv$ given by removing all 
padding from $\Uu$ will 
be neat.
(So in the adaptation to the proof of \ref{thm:pm-stack_well}, $\Vv$ would be 
almost normal.)

We will have $\lh(\Uu)\geq\lh(\Tt)$, but $\lh(\Uu)>\lh(\Tt)$ is possible.
For each node $\alpha$ of $\Tt$, $(\alpha,0)$ will be a node of $\Uu$, and the 
model
$M^\Uu_{\alpha 0}$ will 
correspond directly to $B^\Tt_\alpha$. However, there may also be a further 
finite set of nodes 
$(\alpha,i)$ of $\Uu$, and models $M^\Uu_{\alpha i}$ associated to 
initial segments of $M^\Tt_\alpha$ or $N^\Tt_\alpha$.
For indexing, let $\OR^*=\OR\cross\om$; we order $\OR^*$ lexicographically.
We index the nodes of $\Uu$ with elements of a set $\dom(\Uu)\sub\OR^*$, such 
that for some 
sequence $\left<k_\alpha\right>_{\alpha<\lh(\Tt)}$ of integers $k_\alpha\geq 
1$, 
we have
$(\alpha,i)\in\dom(\Uu)$ iff $\alpha<\lh(\Tt)$ and
$i<k_\alpha$.
So if $\lh(\Tt)>1$ then $\dom(\Uu)$ is \emph{not} closed downward 
under 
$<$.

For notational convenience we 
allow $\Uu$ to use padding. If $E=E^\Uu_{\alpha i}=\emptyset$ we 
consider 
$\strength^{M^\Uu_{\alpha i}}(E)=\OR(M^\Uu_{\alpha i})$; we \emph{do} allow 
$\pred^\Uu(\beta,j)=(\alpha,i)$ in this case.

We make some preparations. Let $\alpha<\lh(\Tt)$.
Write $B_\alpha=B^\Tt_\alpha$, 
$M_\alpha=M^\Tt_\alpha$, etc. If 
$\alpha\in\curlyB^\Tt$ let $(m_\alpha,n_\alpha)=\deg^\Tt(\alpha)=(m_0,n_0)$. If 
$\alpha\in\curlyM^\Tt$ let 
$m_\alpha=\deg^\Tt(\alpha)$. If $\alpha\in\curlyN^\Tt$ let 
$n_\alpha=\deg^\Tt(\alpha)$.
If $[0,\alpha]_\Tt$ does not drop in model, then:
\begin{enumerate}[label=--]
\item If $M_\alpha\neq\emptyset$ then
$\vartheta_\alpha$ denotes $\psi_{i^\Tt_{0\alpha}}(\rho^C)$,
(so if $\rho^C<\rho_0^M$, which is the main case of interest here,
then $\vartheta_\alpha=i^\Tt_{0\alpha}(\rho^C)$).
\item If $N_\alpha\neq\emptyset$ then $\widetilde{\vartheta}_\alpha$ denotes 
$\psi_{j^\Tt_{0\alpha}}(\rho^C)$.
\end{enumerate}
Recall here that if $\alpha\in\curlyB^\Tt$
then $\rho(B^\Tt_\alpha)=\sup i^\Tt_{0\alpha}``\rho^C=\sup 
j^\Tt_{0\alpha}``\rho^C$, but these iteration maps
can be discontinuous
at $\rho^C$ and we can have $\rho(B^\Tt_\alpha)<\vartheta_\alpha$
and $\rho(B^\Tt_\alpha)<\widetilde{\vartheta}_\alpha$.
\footnote{Actually, if $\alpha\in\curlyB^\Tt$ 
then $\vartheta_\alpha=\widetilde{\vartheta}_\alpha$
and 
$M_\alpha||(\vartheta_\alpha^+)^{M_\alpha}=N_\alpha||(\vartheta_\alpha^+)^{
N_\alpha} $, but this is not important.}
We say that $\alpha$ is \dfnemph{$\curlyM$-stable}
iff $M_\alpha\neq\emptyset$ and $[0,\alpha]_\Tt$ does 
not drop
and $\crit(i^\Tt_{\beta\alpha})<\vartheta_\beta$ for all 
$\beta\in[0,\alpha)_\Tt$.
We define \dfnemph{$\curlyN$-stable} analogously.
Note that if $\alpha\in\curlyM^\Tt$ is $\curlyM$-stable then 
$\rho^C<\rho_0^M$ and
$i^\Tt_{0\beta}$ is discontinuous
at $\rho^C$, where $\beta=\max([0,\alpha)_\Tt\inter\curlyB^\Tt)$
(since then in fact
$\rho(B^\Tt_\beta)\leq\crit(i^\Tt_{\beta\alpha}
)<\vartheta_\beta$, and in particular $\rho(B^\Tt_\beta)<\rho_0(M_\beta)$). We 
are only interested
in $\vartheta_\alpha$ for  $\curlyM$-stable $\alpha$.

Let $\CC_{\alpha i}=i^\Uu_{00,\alpha i}(\CC)$ and $\Gamma_{\alpha 
i}=\lh(\CC_{\alpha i})$.
Let $\priCC_{\alpha i}=i^\Uu_{00,\alpha i}(\priCC)$ and $\priGamma_{\alpha 
i}=\lh(\priCC_{\alpha 
i})$. When we say, for example, \emph{$\CC_{\alpha i}$-standard} we literally 
mean 
\emph{$\CC_{\alpha i}$-standard in 
$M^\Uu_{\alpha i}$}. Let $\strength_{\alpha i}=\strength^{M^\Uu_{\alpha 
i}}(E^\Uu_{\alpha i})$. We will also associate later an ordinal
$s_{\alpha i}\leq\strength_{\alpha i}$.

We make some arrangements to help us choosing background extenders for $\Uu$
(recall we do not assume $\AC$ in $V$).
We have $W\sats\ZFC$.
Fix $\xi\in\OR^W$  with $\CC,\widetilde{\CC}\in 
V_\xi^W$ and such that 
$V_\xi^W\elem_k W$ for a sufficiently
large $k$. In particular, whenever $W$ has a background extender
for $S_\alpha^\CC$ or $S_\alpha^{\widetilde{\CC}}$,
then $V_\xi^W$ has one too.
Let ${<^*}\in W$ be a wellorder of $V_\xi^W$.
Given $(\beta,j)\in\dom(\Uu)$ we write
${<^*_{\beta j}}=i^\Uu_{00,\beta j}({<^*})$.
We will use $<^*_{\beta j}$
in determining $E^\Uu_{\beta j}$.
Let $\mathscr{E}\sub M^\Uu_{\beta j}$
be a non-empty collection
 of $M^\Uu_{\beta j}$-extenders.
Let $s_0=\min_{E\in\mathscr{E}}\strength^{M^\Uu_{\beta 
j}}(E)$. 
Then the \dfnemph{$*$-least} element of $\mathscr{E}$
is the $<^*_{\beta j}$-least $E\in\mathscr{E}$
such that $\strength^{M^\Uu_{\beta j}}(E)=s_0$
and $E$ is Mitchell-minimal such; that is,
$\mathscr{E}\inter\Ult(M^\Uu_{\beta j},E)=\emptyset$.
Given a property $\varphi$,
the \dfnemph{$*$-least extender $E$ such that $\varphi(E)$} means
$*$-least element of $\{E\bigm|\varphi(E)\}$.

Let ${\uparrow}\notin\OR$ (here ``$\uparrow$'' means
\emph{undefined}). Let 
$\xivec,\zetavec\in(\OR\cup\{\uparrow\})^2\cut\{({\uparrow},{\uparrow})\}$
and let 
$\xivec=(\xi,\prixi)$ and $\zetavec=(\zeta,\prizeta)$. For $\gamma\in\OR$ let 
$\max(\gamma,\uparrow)=\max(\uparrow,\gamma)=\gamma$. We write 
$\xivec<\zetavec$ 
iff ${\uparrow}\in\xivec$ and $\max(\xivec)\leq\max(\zetavec)$ and
 if ${\uparrow}\in\zetavec$ then $\max(\xivec)<\max(\zetavec)$.
Clearly this order is wellfounded.

\begin{dfn}\label{dfn:dropdown} Let $D$ be a premouse and $\gamma\leq\OR^D$. 
The 
\textbf{$\gamma$-dropdown 
sequence} of $D$ is the sequence $\sigma=\left<(D_i,\delta_i)\right>_{i<n}$ of 
maximum 
length such if $\gamma=\OR^D$ then $\sigma=\emptyset$, and if $\gamma<\OR^D$ 
then $D_0=M|\gamma$, 
and for each $i<n$, $\delta_i=\rho_\om(D_i)$, and if $i+1<n$ then 
$D_{i+1}$ is the least $A$ such that $D_i\pins A\pins D$ and 
$\rho_\om^A<\delta_i$.

Suppose $M_\alpha\neq\emptyset$ and let $\gamma\leq\OR(M_\alpha)$. The 
\dfnemph{$(\Tt,\alpha,\gamma)$-dropdown sequence $\tau$ of $M_\alpha$} is 
defined 
as follows. Let $\sigma$ be the $\gamma$-dropdown sequence of $M_\alpha$. Then:
\begin{enumerate}[label=\tu{(}\alph*\tu{)}]
\item\label{item:stable_and_high_case} if $\alpha$ is $\curlyM$-stable
and $\vartheta_\alpha<\OR^{M_\alpha}$ and $(\vartheta_\alpha^+)^{M_\alpha}
\leq\gamma$
\footnote{So  $\vartheta_\alpha<(\vartheta_\alpha^+)^{M_\alpha}\leq\varrho_i$
for each $i<\lh(\sigma)$.}
then
\[ \tau = 
\sigma\conc\left<(M_\alpha,\vartheta_\alpha),(M_\alpha,0)\right>;\]
\item otherwise $\tau=\sigma\conc\left<(M_\alpha,0)\right>$.
\end{enumerate}
If $N_\alpha\neq\emptyset$ then for $\gamma\leq\OR(N_\alpha)$, we define the 
\dfnemph{$(\Tt,\alpha,\gamma)$-dropdown sequence of $N_\alpha$} analogously.
\end{dfn}

\begin{rem}\label{rem:dropdown}Suppose $\alpha+1<\lh(\Tt)$ and 
$E^\Tt_\alpha\in\es_+(M_\alpha)$.
Let $\tau^*=\left<(M_i,\varrho_i)\right>_{i\leq n'}$ be the reverse of the 
$(\Tt,\alpha,\lh(E^\Tt_\alpha))$-dropdown sequence $\tau$ 
(the same sequence but in reversed order). Note that 
$M_0=M_\alpha$, $M_{n'}=\exit^\Tt_\alpha$, $\varrho'_0=0$ and
$\varrho'_i<\varrho'_{i+1}$ for all $i+1\leq n'$.
Note that if $\beta+1<\lh(\Tt)$ and $\alpha=\pred^\Tt(\beta+1)$
then $M^{*\Tt}_{\beta+1}=M_i$ for some $i\leq n'$.\end{rem}

\begin{rem}\label{rem:resurrection_sketch}
We now give a sketch of the resurrection process
and the meaning of $\Uu\rest[(\delta,0),(\delta,k_\delta-1)]$.
Figures \ref{fgr:one} and \ref{fgr:two} depict
various features discussed here, under certain further simplifying assumptions, 
but they incorporate
more details which will be explained later. In Figure \ref{fgr:one}
it happens that we do not need to take ultrapowers at even stages of the 
resurrection (see below), so there the diagram looks more like standard
resurrection, although there can be ultrapowers taken at odd stages (see below).

For simplicity, we assume for the duration of the sketch that 
$\delta\in\curlyM^\Tt$
is non-$\curlyM$-stable, but in other cases things are similar.
Then we will have
$k_\delta=2u_\delta+1$ where the reversed 
$(\Tt,\delta,\lh(E_\delta))$-dropdown sequence
is $\left<(M_{\delta i},\varrho_{\delta i})\right>_{i\leq u_\delta}$.
If $u_\delta>0$, then the extenders
\[ E^\Uu_{\delta0},E^\Uu_{\delta1},\ldots,E^\Uu_{\delta,2u_\delta-1} \]
will facilitate the resurrection process
used to find a background extender $E^*$
into which we can embed $E_\delta$,
and then we will set
\[ E^\Uu_{\delta,2u_\delta}=E^*\neq\emptyset.\]
It is possible that $E^\Uu_{\delta j}=\emptyset$ for
$j<2u_\delta$. The resurrection will yield for each $i<u_\delta$, a sound
model $\core_\om(R_{\delta,i+1})$, constructed
(by an ultrabackgrounded construction)
in $M^\Uu_{\delta,2i+1}$, and a fully elementary
\[ \pi_{\delta,i+1}:\core_0(M_{i+1})\to\core_\om(R_{\delta,i+1}), \]
and will also yield $Q_{\delta,i+1}$, constructed at a standard stage in 
$M^\Uu_{\delta,2i+2}$, with 
$\core_\om(Q_{\delta,i+1})=\core_\om(R_{\delta,i+1})$,
and thus, we also embed $M_{i+1}$ into $Q_{\delta,i+1}$.
If $i<u_\delta$ and $E^\Uu_{\delta,2i}\neq\emptyset$, then 
$M_i$ is type 3 and 
$\varrho_{i+1}=\nu(M_i)$, and it happens that the standard
resurrection process fails to 
yield an appropriate
model $R_{\delta,i+1}$.
That is, we will have already found $Q_{\delta i}\in M^\Uu_{\delta,2i}$ and 
$\sigma:\core_0(M_i)\to\core_0(Q_{\delta i})$.
So $M_{i+1}\notin M_i^\sq=\dom(\sigma)$, so $\sigma$ does not act 
directly
on $M_{i+1}$. If $\sigma$ is non-$\nu$-high then
 $\psi_\sigma(M_{i+1})\pins Q_{\delta i}$,
in which case we can set $E^\Uu_{\delta,2i}=\emptyset$ and 
$\core_\om(R_{\delta,i+1})=\psi_\sigma(M_{i+1})$
 (this is basically the standard resurrection process).
But if $\sigma$ is $\nu$-high, then $\psi_\sigma(M_{i+1})$
is not a stage of the construction in $M^\Uu_{\alpha,2i}$;
in this case we set 
$E^\Uu_{\delta,2i}$
to be a background extender for $F(Q_{\delta i})$,
and this will actually ensure that an appropriate $R_{\delta,i+1}$
appears in $M^\Uu_{\delta,2i+1}$ (with $\core_\om(R_{\delta,i+1})$
either
$=\psi_\sigma(M_{i+1})$
or some variant thereof). We will write 
$s_{\delta,2i}=\nu(F(Q_{\delta 
i}))$ here; if $E^\Uu_{\delta,2i}=\emptyset$
then we set $s_{\delta,2i}=s_{\delta,2i+1}$.
 If $i<u_\delta$ and $E^\Uu_{\delta,2i+1}\neq\emptyset$,
then $R_{\delta,i+1}$ was constructed by a non-standard
stage in $M^\Uu_{\delta,2i+1}$, and $E^\Uu_{\delta,2i+1}$
is then selected witnessing \ref{dfn:ultra_bkgd}(\ref{item:ultra}) for 
$R_{\delta,i+1}$,
yielding an appropriate $Q_{\delta,2i+1}$ 
in $M^\Uu_{\delta,2i+2}$,
and we proceed with the next step of resurrection (uncoring) there.
In this case we set $s_{\delta,2i+1}=\rho_\om(R_{\delta,i+1})$.
This eventually leads 
(including in the case that $u_\delta=0$) to a
model $Q^*_\delta=Q_{\delta,2u_\delta}\in M^\Uu_{\delta,2u_\delta}$
and $0$-lifting embedding
\[ \pi^*_\delta:\core_0(\exit^\Tt_\delta)\to\core_0(Q^*_\delta). \]
If $\pi^*_\delta$ is non-$\nu$-low then we set
$E^\Uu_{\delta,2u_\delta}=E^*$ to be an appropriate
background for $F(Q^*_\delta)$, and $s_{\delta,2u_\delta}=\nu(Q^*_\delta)$, 
whereas if $\pi^*_\delta$
is $\nu$-low and $\nu'=\psi_{\pi^*_\delta}(\nu^\Tt_\delta)$
then we do likewise but with $Q'\pins Q^*_\delta$
instead of $Q^*_\delta$,
where $Q'$ is chosen as in \ref{rem:Q'_from_nu-low}
with respect to $\pi^*_\delta$.
This completes the sketch of resurrection.
\end{rem}

\begin{figure}
\centering
\begin{tikzpicture}
 [mymatrix/.style={
    matrix of  nodes,
    row sep=0.15cm,
    column sep=0.8cm}]
   \matrix(m)[mymatrix]{
{}&{}&{}&{}&{}&{}&{}&{}&{}&{}&{}\\
{}&{}&{}&{}&{}&{}&{}&{}&{}&{}&{}\\
{}&{}&{}&{}&{}&{}&{}&{}&{}&{}&{}\\
{}&{}&{}&{}&{}&{}&{}&{}&{}&{}&{}\\
{}&{}&{}&{}&{}&{}&{}&{}&{}&{}&{}\\
{}&{}&{}&{}&{}&{}&{}&{}&{}&{}&{}\\
{}&{}&{}&{}&{}&{}&{}&{}&{}&{}&{}\\
{}&{}&{}&{}&{}&{}&{}&{}&{}&{}&{}\\
{}&{}&{}&{}&{}&{}&{}&{}&{}&{}&{}\\
{}&{}&{}&{}&{}&{}&{}&{}&{}&{}&{}\\
{}&{}&{}&{}&{}&{}&{}&{}&{}&{}&{}\\
{}&{}&{}&{}&{}&{}&{}&{}&{}&{}&{}\\
{}&{}&{}&{}&{}&{}&{}&{}&{}&{}&{}\\
{}&{}&{}&{}&{}&{}&{}&{}&{}&{}&{}\\
{}&{}&{}&{}&{}&{}&{}&{}&{}&{}&{}\\
{}&{}&{}&{}&{}&{}&{}&{}&{}&{}&{}\\
{}&{}&{}&{}&{}&{}&{}&{}&{}&{}&{}\\
{}&{}&{}&{}&{}&{}&{}&{}&{}&{}&{}\\};
\draw[fill=black]
(m-7-1) circle (0.25ex)
(m-8-1) circle (0.25ex)
(m-11-1) circle (0.25ex)
(m-13-1) circle (0.15ex)
(m-15-1) circle (0.15ex)
(m-17-1) circle (0.15ex);
\draw[fill=black]
(m-5-3) circle (0.25ex)
(m-16-3) circle (0.15ex);
\draw[fill=black]
(m-3-5) circle (0.25ex)
(m-5-5) circle (0.25ex)
(m-11-5) circle (0.2ex);
\draw[fill=black]
(m-3-7) circle (0.25ex)
(m-5-7) circle (0.25ex)
(m-7-7) circle (0.2ex);
\draw[fill=black]
(m-3-9) circle (0.25ex)
(m-5-9) circle (0.2ex);
\draw[fill=black]
(m-1-11) circle (0.25ex);

\draw (m-18-1)--(m-7-1);
\path[->,font=\scriptsize]
(m-18-11) edge[-] node {} (m-1-11)
(m-18-1) edge[-] node [left,pos=(2.8/11)] {$\varrho_{\delta 1}$} (m-7-1) 
(m-18-1) edge[-] node [left,pos=(5/11)] {$\varrho_{\delta 2}$} (m-7-1) %
(m-18-1) edge[-] node [left,pos=(7.2/11)] {$M_{\delta 2}$} (m-7-1) %
(m-18-1) edge[-] node [left,pos=(9.4/11)] {$M_{\delta 1}$} (m-7-1) %
(m-18-1) edge[-] node [left,pos=(11.6/11)] {$M_{\delta}$} (m-7-1) %
(m-18-3) edge[-] node [left,pos=(13.8/13)] {$\core_{m_\delta}(Q_{\delta 0})$} 
(m-5-3) 
(m-18-5) edge[-] node [right=-2mm,pos=(6.4/15)] {$\begin{array}{l}s_{\delta 
0}=s_{\delta 1}\\
=\rho_\om(R_{\delta 1})\end{array}$} (m-3-5) 
(m-18-5) edge[-] node [left=1mm,pos=(13.3/15)] {$\core_\om(R_{\delta 1})$} 
(m-3-5) %
(m-18-5) edge[-] node [left,pos=(16/15)] {$Q_{\delta 0}$} (m-3-5) %
(m-18-7) edge[-] node [right=-2mm,pos=(10.6/15)] {$\begin{array}{l}s_{\delta 
2}=s_{\delta 3}\\
=\rho_\om(R_{\delta 2})\end{array}$} (m-3-7) 
(m-18-7) edge[-] node [left,pos=(16/15)] {$Q_{\delta 1}$} (m-3-7) %
(m-18-7) edge[-] node [left=1mm,pos=(13.3/15)] {$\core_\om(R_{\delta 2})$} 
(m-3-7) %

(m-18-9) edge[-] node {} (m-3-9) 
(m-18-9) edge[-] node [left,pos=(16/15)] {$Q_{\delta 2}$} (m-3-9) %
(m-18-9) edge[-] node [right=-2mm,pos=12.7/15] 
{$\begin{array}{l}s_{\delta4}=\\\nu(Q_{\delta 2})\geq s_{\delta 3}
\end{array}$} (m-3-9)
(m-18-11) edge[-] node [left,pos=17.5/17] 
{$\core_{m_{\delta+1}}(Q_{\delta+1,0})$} (m-1-11)

(m-7-1) edge[->] node[below=1mm]{$\pi_{\delta 0}$} (m-5-3)
(m-5-3) edge[->] node[above,pos=0.3]{$\tau^{m_\delta 0}_{\delta 0}$} (m-3-5) 
(m-9-1) edge[->] node[below=1mm,pos=0.75]{$\pi_{\delta 1}$} (m-5-5)
(m-5-5) edge[->] node[above,pos=0.3]{$\tau^{\om 0}_{\delta 1}$} (m-3-7)
(m-11-1) edge[->] node[below=1mm,pos=.5]{$\pi_{\delta 2}$} (m-5-7)
(m-5-7) edge[->] node[above,pos=0.3]{$\tau^{\om 0}_{\delta 2}$} (m-3-9)
(m-13-1) edge[dashed] node {} (m-7-7)
(m-15-1) edge[dashed] node {} (m-11-5)
(m-9-1) edge[bend left,dashed] node {} (m-15-1)
(m-11-1) edge[bend left,dashed] node {} (m-13-1)
(m-5-5) edge[bend left,dashed] node {} (m-11-5)
(m-5-7) edge[bend left,dashed] node {} (m-7-7)
;
\path[-{Straight Barb[left]}]
(m-17-1) edge[dotted] node {} (m-16-2)
(m-16-3) edge[dotted] node {} (m-15-4)
(m-11-5) edge[dotted] node {} (m-10-6)
(m-7-7) edge[dotted] node {} (m-6-8);
\path[-] (m-11-5)++(13mm,0mm) edge[dotted] node {} (m-11-11);
\path[-] (m-7-7)++(13mm,0mm) edge[dotted] node {} (m-7-11);
\path[-] (m-5-9)++(9mm,0mm) edge[dotted] node {} (m-5-11);
\end{tikzpicture}
\caption{Dropdown resurrection and associated objects,
in the case that $\delta\in\curlyM^\Tt$ is non-$\curlyM$-stable,
$u_\delta=2$, and all models are non-type 3. 
Vertical lines indicate ordinals, with height roughly corresponding
to ordinal rank.
Solid arrows denote embeddings ($\pi_{\delta 0}$ etc),
with 
domains and codomains denoted by 
large bullets (except that the domain is literally the squash).
Small bullets indicate the positions of labelled ordinals
($\varrho_{\delta 1}$ etc). Dashed arrows indicate trajectories under 
embeddings.
The bases of the dotted half-arrows indicate approximately
the critical points of $\pi_{\delta 0}$, $\tau^{m_\delta 0}_{\delta 0}$,
$\tau^{\om 0}_{\delta 1}$, $\tau^{\om 0}_{\delta 2}$ respectively;
the main point here is that $\crit(\tau^{\om 0}_{\delta,i+1})\geq 
s_{\delta,2i+1}$.
Curved dashed arrows indicate positions of $\om$th projecta
$\rho_\om^N$ of models $N$.
Dotted horizontal lines indicate agreement
between models $N_1,N_2$ strictly below that ordinal $\alpha$;
that is, $N_1||\alpha=N_2||\alpha$.
(If $\alpha<s_{\delta 4}$ then in fact,
$\alpha$ is also a cardinal of both models and so $N_1|\alpha=N_2|\alpha$.)
Note 
$\core_\om(R_{\delta,i+1})=\core_\om(Q_{\delta,i+1})$.
Note  $s_{\delta 1}<s_{\delta 2}$ and $s_{\delta3}\leq s_{\delta4}$,
but $s_{\delta 3}=s_{\delta 4}$ is possible.
Note that because $\exit^\Tt_\delta$ is non-type 3,
we have $s_{\delta4}=\nu(Q_{\delta2})$.
}
\label{fgr:one}
\end{figure}

\begin{figure}
\centering
\begin{tikzpicture}
 [mymatrix/.style={
    matrix of  nodes,
    row sep=0.4cm,
    column sep=0.65cm}]
   \matrix(m)[mymatrix]{
{}&{}&{}&{}&{}&{}&{}&{}&{}&{}&{}&{}\\
{}&{}&{}&{}&{}&{}&{}&{}&{}&{}&{}&{}\\
{}&{}&{}&{}&{}&{}&{}&{}&{}&{}&{}&{}\\
{}&{}&{}&{}&{}&{}&{}&{}&{}&{}&{}&{}\\
{}&{}&{}&{}&{}&{}&{}&{}&{}&{}&{}&{}\\
{}&{}&{}&{}&{}&{}&{}&{}&{}&{}&{}&{}\\
{}&{}&{}&{}&{}&{}&{}&{}&{}&{}&{}&{}\\
{}&{}&{}&{}&{}&{}&{}&{}&{}&{}&{}&{}\\
{}&{}&{}&{}&{}&{}&{}&{}&{}&{}&{}&{}\\
{}&{}&{}&{}&{}&{}&{}&{}&{}&{}&{}&{}\\
{}&{}&{}&{}&{}&{}&{}&{}&{}&{}&{}&{}\\
{}&{}&{}&{}&{}&{}&{}&{}&{}&{}&{}&{}\\
{}&{}&{}&{}&{}&{}&{}&{}&{}&{}&{}&{}\\
{}&{}&{}&{}&{}&{}&{}&{}&{}&{}&{}&{}\\};
\draw[fill=black]
(m-9-1) circle (0.25ex)
(m-10-1) circle (0.25ex)
(m-11-1) circle (0.25ex)
(m-12-1) circle (0.2ex);
\draw
(m-13-1) circle (0.3ex);
\draw
(m-11-3) circle (0.3ex);
\draw[fill=black]
(m-5-5) circle (0.25ex)
(m-6-5) circle (0.25ex)
(m-8-5) circle (0.2ex);
\draw
(m-7-5) circle (0.3ex);
\draw
(m-6-7) circle (0.3ex);
\draw[fill=black]
(m-2-9) circle (0.25ex)
(m-3-9) circle (0.2ex);
\draw[fill=black]
(m-1-11) circle (0.25ex)
(m-2-11) circle (0.25ex);

\path[->,font=\scriptsize]

(m-14-1) edge[-] node [right,pos=(0.7/5)] {$M_\delta^\sq$} (m-9-1) 
(m-14-1) edge[-] node [left,pos=(1/5)] {$\varrho_{\delta 1}$} (m-9-1)
(m-14-1) edge[-] node [left,pos=(2.1/5)] {$\varrho_{\delta 2}$} (m-9-1) %
(m-14-1) edge[-] node [left,pos=(3.3/5)] {$M_{\delta 2}$} (m-9-1) %
(m-14-1) edge[-] node [left,pos=(4.3/5)] {$M_{\delta 1}$} (m-9-1) %
(m-14-1) edge[-] node [left,pos=(5.3/5)] {$M_{\delta}$} (m-9-1) %

(m-14-3) edge[-] node [right=-1mm,pos=(14.5/15)] 
{$\begin{array}{l}s_{\delta0}\\=\nu(Q_{\delta0})\end{array}$} (m-11-3) %
(m-14-3) edge[-] node [right,pos=(3.7/3)] {$Q_{\delta0}^\sq$} (m-11-3) %

(m-14-5) edge[-] node [left,pos=(16/15)] {$\bar{R}=\core_\om(R_{\delta 
1})^\unsq$} 
(m-5-5) 
(m-9-5) edge[-] node [right=-1mm,pos=(0.6/4)] {$\begin{array}{l}s_{\delta 
1}\\
=\rho_\om(\bar{R})\end{array}$} 
(m-5-5) 
(m-9-5) edge[-] node [left=0mm,pos=(1.7/4)] 
{$\bar{R}^\sq$} (m-5-5) 

(m-14-7) edge[-] node [right=-1mm,pos=(10/10)] {$\begin{array}{l}s_{\delta 
2}\\=\nu(Q_{\delta 1})\end{array}$} (m-6-7) %
(m-14-7) edge[-] node [right,pos=10.9/10] 
{$Q_{\delta 1}^\sq$} (m-6-7)

(m-14-9) edge node [left=0mm,pos=10.4/10] 
{$\core_\om(R_{\delta 2})$} (m-2-9)
(m-14-9) edge node [right=-1mm,pos=10.9/12] 
{$\begin{array}{l}s_{\delta 3}\\
=\rho_\om^{R_{\delta2}}\end{array}$} (m-2-9)

(m-14-11) edge[-] node [left,pos=(13.5/13)] {$Q_{\delta 2}$} (m-1-11)
(m-14-11) edge[-] node [right=-1mm,pos=(12/13)]
{$\begin{array}{l}s_{\delta4}=\\\nu(Q_{\delta 2})\end{array}$} (m-1-11) %

(m-13-1) edge[->] node[below]
{$\sigma$} (m-11-3)
(m-10-1) edge[->] node[above=0.5mm]
{$\pi_{\delta 1}$} (m-5-5)
(m-11-1) edge[dashed,->] node {} (m-6-5)
(m-13-1) edge[dashed,->] node {} (m-8-5)

(m-7-5) edge[->] node[below,pos=0.5]{$\tau_{\delta 1}^{\om0}$} (m-6-7)
(m-7-5) edge[dashed,->] node {} (m-3-9)
(m-6-5) edge[->] node[above]{$\hat{\psi}$} (m-2-9)
(m-2-9) edge[->] node[below,pos=.5]{$\tau^{\om0}_{\delta 2}$} (m-1-11)

(m-11-1) edge[bend left,dashed] node {} (m-12-1)
(m-10-1) edge[bend left,dashed] node {} (m-13-1)
(m-5-5) edge[bend left,dashed] node {} (m-8-5)
(m-6-5) edge[bend left,dashed] node {} (m-7-5)
(m-2-9) edge[bend left,dashed] node {} (m-3-9)
;
\path[-] (m-11-3)++(8mm,0mm) edge[dotted] node {} (m-11-12)++(-18mm,0mm);
\path[-] (m-8-5)++(8mm,0mm) edge[dotted] node {} (m-8-12)++(-18mm,0mm);
\path[-] (m-6-7)++(8mm,0mm) edge[dotted] node {} (m-6-12)++(-18mm,0mm);
\path[-] (m-3-9)++(8mm,0mm) edge[dotted] node {} (m-3-12)++(-18mm,0mm);
\path[-{Straight Barb[left]}]
(m-3-9) edge[dotted] node {} +(10mm,3mm)
(m-8-5) edge[dotted] node {} +(10mm,3mm);
\end{tikzpicture}
\caption{Dropdown resurrection and associated objects,
in the case that $\delta\in\curlyM^\Tt$ is non-$\curlyM$-stable,
$u_\delta=2$,  
$E^\Uu_{\delta 
j}\neq\emptyset$ for 
each $j<4$, and $M_{\delta 2}$ is non-type 3; hence, $M_{\delta i}$ is 
type 3
and $\nu(M_{\delta i})=\varrho_{\delta,i+1}$ for $i=0,1$.
Notation is as in Figure \ref{fgr:one};
also, an open circle denotes the height of a squashed premouse,
which is also the domain or codomain of $\sigma$ or $\tau^{\om 0}_{\delta 1}$.
Note $\sigma=\tau^{m_\delta0}_{\delta 0}\com\pi_{\delta 0}$
and $\hat{\psi}$ is defined as in Subsubcase \ref{sscase:upsilon_nu-high}
in the inductive construction of $\Uu$ to follow
(so $\pi_{\delta 2}=\hat{\psi}\com\psi_{\pi_{\delta 1}}\rest\core_0(M_{\delta 
2})$).
Note that $s_{\delta0}<s_{\delta1}<s_{\delta2}<s_{\delta3}<s_{\delta4}$
(because $E^\Uu_{\delta j}\neq\emptyset$ for each $j<4$).
Note that the models $\core_{m_\delta}(Q_{\delta 0})$
and $\core_{m_{\delta+1}}(Q_{\delta+1,0})$ have been omitted
from this figure, though they are present in Figure \ref{fgr:one}.
Note that we again have $s_{\delta4}=\nu(Q_{\delta2})$
because $M_{\delta 2}=\exit^\Tt_\delta$ is non-type 3.
}\label{fgr:two}
\end{figure}

We now proceed with the details.
We first state some intentions, introduce more notation,
and state hypotheses 
\ref{item:if_alpha_M-stable}--\ref{item:gens_at_stage_alpha}, to be maintained
by induction on initial 
segments of 
$(\Tt,\Uu)$.
Let $\alpha<\lh(\Tt)$.
If $M_\alpha\neq\emptyset$ we will define:
\begin{enumerate}[label=--]
 \item $\DD_{\alpha 0}\in\{\CC_{\alpha 0},\priCC_{\alpha 0}\}$,
 \item $\Delta_{\alpha 0}=\lh(\DD_{\alpha 0})$,
 \item $\xi_{\alpha 0}<\Delta_{\alpha 0}$,
 \item $(M_{\alpha 0},m_{\alpha 0})=(M_\alpha,m_\alpha)$ and $Q_{\alpha 
0}=S^{\DD_{\alpha 0}}_{\xi_{\alpha 0}}$,
 \item and a c-preserving $m_{\alpha 0}$-lifting $\pi_{\alpha 
0}:\core_0(M_{\alpha 
0})\to\core_{m_{\alpha 0}}(Q_{\alpha 0})$
\end{enumerate}
such that:
\begin{enumerate}[label=\tu{(}L\arabic*\tu{)}]
 \item\label{item:if_alpha_M-stable} If $\alpha$ is $\curlyM$-stable  then 
$\DD_{\alpha 
0}=\CC_{\alpha 0}$ and $\xi_{\alpha 0}=i^\Uu_{00,\alpha 0}(\lambda^\CC)$
and
\[ \pi_{\alpha 0}\com i^\Tt_{0\alpha}=i^\Uu_{00,\alpha0}\rest\core_0(M),\]
so if $\rho^C<\rho_0^M$ then $\pi_{\alpha 
0}(\vartheta_\alpha)=i^\Uu_{00,\alpha0}(\rho^C)$.
 \item\label{item:if_alpha_non-M-stable} If $\alpha$
 is non-$\curlyM$-stable then 
$\xi_{\alpha 0}$ 
is $\DD_{\alpha 0}$-standard.
\end{enumerate}

If $N_\alpha\neq\emptyset$ we will define $\priDD_{\alpha 0}$, 
$\priDelta_{\alpha 0}$, 
$\prixi_{\alpha 0}$, $\priQ_{\alpha 0}$ and $\pripi_{\alpha 0}$ analogously
(and maintain analogous properties).

Now suppose $\alpha+1<\lh(\Tt)$.

Suppose that $E_\alpha\in\es_+(M_\alpha)$. Let 
$\left<(M_{\alpha 
i},\varrho_{\alpha i})\right>_{i\leq u_\alpha}$ be the reverse of the 
$(\Tt,\alpha,\lh(E_\alpha))$-dropdown sequence of $M_\alpha$.
Then
$k_\alpha=2u_\alpha+1$.
Fix $i\leq u_\alpha$. Let $m_{\alpha i}=m_\alpha$ if $M_{\alpha i}=M_\alpha$ 
and 
$m_{\alpha i}=\om$ otherwise (in fact we already defined
$M_{\alpha0}$ and $m_{\alpha0}$ above).
If $i>0$ then for each $j\in\{2i-1,2i\}$ we will define:
\begin{enumerate}[label=--]
 \item $\DD_{\alpha j}\in\{\CC_{\alpha j},\priCC_{\alpha j}\}$,
 \item $\Delta_{\alpha j}=\lh(\DD_{\alpha j})$,
 \item $\xi_{\alpha j}<\Delta_{\alpha j}$,
 \item $R_{\alpha i}=S_{\xi_{\alpha,2i-1}}^{\DD_{\alpha,2i-1}}$ and $Q_{\alpha 
i}=S_{\xi_{\alpha,2i}}^{\DD_{\alpha,2i}}$, and
 \item a c-preserving $m_{\alpha i}$-lifting embedding $\pi_{\alpha 
i}:\core_0(M_{\alpha 
i})\to\core_{m_{\alpha i}}(Q_{\alpha i})$
\end{enumerate}
such that:
\begin{enumerate}[label=\tu{(}L\arabic*\tu{)}]
\setcounter{enumi}{2}
\item\label{item:x_alpha,2i_standard} if $0<i\leq u_\alpha$ then 
$\xi_{\alpha,2i}$ is 
$\DD_{\alpha,2i}$-standard and either:
\begin{enumerate}[label=--]\item $\core_{m_{\alpha 
i}}(R_{\alpha 
i})=\core_{m_{\alpha i}}(Q_{\alpha i})$, or
\item $M$ is non-standard (in $W$), the hypothesis of Definition 
\ref{dfn:dropdown}\ref{item:stable_and_high_case}
holds
for the $(\Tt,\alpha,\lh(E_\alpha))$-dropdown 
(that is, $\alpha$ is 
$\curlyM$-stable
and $(\vartheta_\alpha^+)^{M_\alpha}
\leq\lh(E_\alpha)$, so
$M_{\alpha 1}=M_{\alpha 0}=M_\alpha$),
$i=1$,
$\DD_{\alpha 0}=\CC_{\alpha 
0}=\CC_{\alpha 1}=\DD_{\alpha 1}$,
$m_{\alpha 1}=m_{\alpha 0}=m_\alpha=m_0$ and $\core_{m_0+1}(Q_{\alpha 
1})=\core_0(R_{\alpha 1})=\core_0(Q_{\alpha 0})$
is fully sound.
\end{enumerate}
\end{enumerate}
For $m\leq n\leq m_{\alpha i}$ let
\[ \tau^{n m}_{\alpha i}:\core_n(Q_{\alpha i})\to\core_m(Q_{\alpha i}) \]
be the core embedding. Let $Q^*_\alpha=Q_{\alpha u_\alpha}$ and
$m=m_{\alpha  u_\alpha}$ and
\[ \pi^*_\alpha:\core_0(\exit^\Tt_\alpha)\to\core_0(Q^*_\alpha), \]
\[ \pi^*_\alpha=\tau^{m 0}_{\alpha u_\alpha}\com\pi_{\alpha u_\alpha}.\]
If $\pi^*_\alpha$ is $\nu$-low then let $Q'_\alpha$
be derived from $\pi^*_\alpha$ as in \ref{rem:Q'_from_nu-low},
and otherwise let $Q'_\alpha=Q^*_\alpha$.
Let 
$c_\alpha$ be the set of infinite $\exit^\Tt_\alpha$-cardinals
$\kappa<\nu(E)$. 
Fix $\kappa\in 
c_{\alpha}$. Let $i_{\alpha\kappa}$ be the largest $i$ such that 
$\varrho_{\alpha i}\leq\kappa$. Let 
$i=i_{\alpha\kappa}$. Let $m_{\alpha\kappa}$ be the least $m$ such that either
\begin{enumerate}[label=--]
 \item $M_{\alpha i}=M_\alpha$ and $m=m_\alpha$, or
\item $\rho_{m+1}(M_{\alpha i})\leq\kappa$.
\end{enumerate}
Let $M_{\alpha\kappa}=M_{\alpha i}$. We define the c-preserving
$m_{\alpha \kappa}$-lifting embedding
\[ 
\pi_{\alpha \kappa}:\core_0(M_{\alpha\kappa})\to\core_{m_{\alpha \kappa}
}(Q_{\alpha i})\]
by $\pi_{\alpha \kappa}=\tau^{n m}_{\alpha i}\com\pi_{\alpha i}$,
where $n=m_{\alpha i}$ and $m=m_{\alpha \kappa}$.
If $\alpha\in\curlyB^\Tt$ and $\kappa<\rho_\alpha$ and
$(\kappa^+)^{B_\alpha}<\lh(E)$ then we also define $N_{\alpha\kappa}=N_\alpha$, 
$n_{\alpha\kappa}=n_\alpha$, and 
$\pripi_{\alpha\kappa}=\pripi_{\alpha 0}$.

Now suppose instead that $E\in\es_+(N_\alpha)\cut\es_+(M_\alpha)$. Then we make 
symmetric 
definitions by analogy to the preceding ones. (So for example, we let $\sigma$ 
be the 
$(\Tt,\alpha,\lh(E))$-dropdown sequence of $N_\alpha$, and set 
$u_\alpha+1=\lh(\sigma)$, and for 
$i\leq u_\alpha$ we define $N_{\alpha i}$ and $n_{\alpha i}$, and also define 
$\prixi_{\alpha 
i}$, $\priQ_{\alpha i}$, etc.)

Let $\omega^*_\alpha=\pi^*_\alpha$ or 
$\omega^*_\alpha=\pripi^*_\alpha$, whichever is defined.
Let $\xivec_{\alpha i}=(\xi_{\alpha i},\prixi_{\alpha i})$,
where if $\prixi_{\alpha i}$ is undefined
then $\xivec_{\alpha i}=(\xi_{\alpha i},{\uparrow})$, etc.

We now list the remaining inductive hypotheses:
\begin{enumerate}[label=\tu{(}L\arabic*\tu{)}]
 \setcounter{enumi}{3}
\item\label{item:pred^U_biceph_not_biceph} Let $\beta+1<\lh(\Tt)$ and 
$\alpha=\pred^\Tt(\beta+1)$. Then:
\begin{enumerate}
 \item\label{item:pred^U_biceph} If $\beta+1$ is $\curlyM$-stable
 or $\curlyN$-stable then 
$\pred^\Uu(\beta+1,0)=(\alpha,0)$.

\item\label{item:pred^U_not_biceph}  Suppose $\beta+1\in\curlyM^\Tt$ is
non-$\curlyM$-stable and let 
$i\leq 
u_\alpha$ be such that $M^*_{\beta+1}=M_{\alpha i}$ and if 
$\alpha$ is $\curlyM$-stable then $i\geq 
1$.\footnote{Note that if $\xi\in\curlyM^\Tt$ is
is non-$\curlyM$-stable but all $\alpha<_\Tt\xi$ are $\curlyM$-stable,
then $\xi=\beta+1$ for some $\beta$, 
and letting $\alpha=\pred^\Tt(\beta+1)$,
then $\alpha$ is $\curlyM$-stable
(so $[0,\alpha]_\Tt$ does not drop in model),
and if $\Tt$ does not drop in model at $\beta+1$
then $\vartheta_\alpha\leq\crit(E_\beta)$ and $E_\beta$
is $M_\alpha$-total, so $(\vartheta_\alpha^+)^{M_\alpha}\leq\lh(E_\alpha)$,
so $M_{\alpha 1}=M_{\alpha 0}=M_\alpha=M^{*}_{\beta+1}$.}
Then
\[ \pred^\Uu(\beta+1,0)=(\alpha,2i).\] 
\item Likewise if $\beta+1\in\curlyN^\Tt$ is non-$\curlyN$-stable.
\end{enumerate}
\item\label{item:drops} Let $\alpha\leq\beta<\lh(\Tt)$ and $j<k_\alpha$ and 
$k<k_\beta$ with  $\beta$ or $k$ a successor, and 
$(\alpha,j)=\pred^\Uu(\beta,k)$.
Then:
\begin{enumerate}
\item\label{item:xi_mon_dec} $\xivec_{\beta k}\leq i^\Uu_{\alpha j,\beta 
k}(\xivec_{\alpha j})$.
\item\label{item:xi_when_k_odd} If $k$ is odd then $\xivec_{\beta k}< 
i^\Uu_{\alpha j,\beta 
k}(\xivec_{\alpha j})$.
\item\label{item:xi_when_k_pos_even} If $k$ is even and $k>0$ then either:
\begin{enumerate}
 \item $(\alpha,j)=(\beta,k-1)$, so $j$ is odd and $\xivec_{\beta 
k}<i^\Uu_{\gamma\ell,\beta k}(\xivec_{\gamma\ell})$ where
 $(\gamma,\ell)=\pred^\Uu(\alpha,j)$, or
 \item  $\xivec_{\beta k}< i^\Uu_{\alpha j,\beta 
k}(\xivec_{\alpha j})$.
\end{enumerate}
\item\label{item:xi_when_k=0} Suppose $k=0$; so $\alpha=\pred^\Tt(\beta)$ and 
$j=2i$ is even
(by part \ref{item:pred^U_biceph_not_biceph} above).
Then
\[ (\DD_{\beta 0},\priDD_{\beta 0},\xivec_{\beta 0})= i^\Uu_{\alpha j,\beta 
0}(\DD_{\alpha 
j},\priDD_{\alpha j},\xivec_{\alpha j})\]
and if $M_\beta\neq\emptyset$ then
\[ \pi_{\beta 0}
\com i^{*\Tt}_{\beta}=
i^\Uu_{\alpha j,\beta 0}\com\tau_{\alpha i}^{m_{\alpha i} 
m_\beta}\com\pi_{\alpha i}
=\tau^{m_{\alpha i}m_\beta}_{\beta 0}\com i^\Uu_{\alpha j,\beta0}\com
\pi_{\alpha i},\]
and if $N_\beta\neq\emptyset$ then $\pripi_{\beta 0}\com j^{*\Tt}_\beta$ is 
likewise.
\end{enumerate}
\item\label{item:limits} Let $\alpha<_\Tt\lambda<\lh(\Tt)$ with
$\lambda$ a limit and such that:
\begin{enumerate}[label=--]
 \item $(\alpha,\lambda]_\Tt$ does not drop in model,
\item  if $\alpha\in\curlyB^\Tt$ then $\lambda\in\curlyB^\Tt$,
\item if $\alpha$ is 
$\curlyM$-stable then $\lambda$ is $\curlyM$-stable,
\item  if $\alpha$ is 
$\curlyN$-stable
then $\lambda$ is $\curlyN$-stable. 
\end{enumerate}
Then:
\begin{enumerate}[label=--]
 \item 
$(\alpha,0)<_\Uu(\lambda,0)$,
\item $(\alpha,0)\leq_\Uu(\beta,i)\leq_\Uu(\lambda,0)\text{ iff }
[i=0\text{ and }\alpha\leq_\Tt\beta\leq_\Tt\lambda]$,
\item $i^\Uu_{\alpha 0,\lambda 0}(\DD_{\alpha 0},\priDD_{\alpha 
0},\xivec_{\alpha 0})=(\DD_{\lambda 
0},\priDD_{\lambda 0},\xivec_{\lambda 0})$,
\item
 if $M_\alpha\neq\emptyset$ then letting 
$m=m_\alpha$ and 
$n=m_\lambda$,
\[ \pi_{\lambda 0}\com 
i^\Tt_{\alpha\lambda}=i^\Uu_{\alpha 0,\lambda 0}\com\tau^{m n}_{\alpha 
0}\com\pi_{\alpha 0}=
\tau^{mn}_{\lambda 0}\com i^\Uu_{\alpha0,\lambda0}\com\pi_{\alpha0},\]
\item  likewise if $N_\alpha\neq\emptyset$.
\end{enumerate}

\item\label{item:agreement_within} If $\alpha\in\curlyB^\Tt$
 then
$\pi_{\alpha 0}\rest\rho_\alpha = \pripi_{\alpha 0}\rest\rho_\alpha$.
\footnote{Note then that if $\alpha+1<\lh(\Tt)$ then for all $\kappa\in 
c_\alpha\inter\rho_\alpha$ such that 
$(\kappa^+)^{B_\alpha}<\lh^\Tt_\alpha$, we have $\pi_{\alpha\kappa}=\pi_{\alpha 
0}$ and $\pripi_{\alpha\kappa}=\pripi_{\alpha 0}$, so
$\pi_{\alpha\kappa}\rest(\kappa^+)^{B_\alpha}=\pripi_{\alpha\kappa}
\rest(\kappa^+)^{B_\alpha}$.}

\item\label{item:agreement_across} Let $\alpha<\beta<\lh(\Tt)$ and $\alpha<
\beta'<\lh(\Tt)$ and $\kappa\in c_\alpha$. Then:
\begin{enumerate}[label=--]
\item If $E_\alpha\in\es_+^{M_\alpha}$ and 
$0\leq i<u_\alpha$
then 
$\tau^{m_{\alpha i}0}_{\alpha i}\com\pi_{\alpha
i}\rest\varrho_{\alpha,i+1}\sub\om^*_\alpha$;
likewise otherwise.

\item If $\pi_{\alpha \kappa}$ is defined then
$\pi_{\alpha \kappa}\rest(\kappa^+)^{M_{\alpha\kappa}}\sub\omega^*_\alpha$.
\item If $\pripi_{\alpha\kappa}$ is defined then
$\pripi_{\alpha \kappa}\rest(\kappa^+)^{N_{\alpha\kappa}}\sub\omega^*_\alpha$.
 \item If $\pi_{\beta 0}$ is defined then $\omega^*_\alpha\sub\pi_{\beta 0}$.
 \item If $\pripi_{\beta 0}$ is defined then $\omega^*_\alpha\sub\pripi_{\beta 
0}$.
 \item If
$\pi_{\beta0}$ and $\pripi_{\beta'0}$ are both defined then they agree over
$\exit^\Tt_\alpha$ (not just over $(\exit^\Tt_\alpha)^\sq=\dom(\om^*_\alpha)$).
\end{enumerate}
We write $\omega_\alpha$ for the 
restriction of $\pi_{\alpha+1,0}$ or $\pripi_{\alpha+1,0}$ to 
$\exit^\Tt_\alpha$ (not just $(\exit^\Tt_\alpha)^\sq$), whichever 
is defined. Then moreover:
\begin{enumerate}[label=--]
\item $\omega_\infty=\bigcup_{\alpha+1<\lh(\Tt)}\omega_\alpha$ is a function.
\end{enumerate}
\end{enumerate}
\begin{enumerate}[resume*]
\item\label{item:Uu_normal} The tree given by removing padding from 
$\Uu$ is neat. Moreover, given $\alpha+1,\beta+1<\lh(\Tt)$ and 
$j\leq 2u_\alpha$ 
and $k\leq 2u_\beta$, and letting $(\alpha,j)+1$
be the successor of $(\alpha,j)$ in $\dom(\Uu)$, we have
\begin{enumerate}
\item if $(\alpha,j)<(\beta,k)$ then [$s_{\alpha j}\leq s_{\beta k}$
and if $E^\Uu_{\alpha j}\neq\emptyset$
then $s_{\alpha j}<s_{\beta k}$],
\item if $E^\Uu_{\alpha j}\neq\emptyset$
then either:
\begin{enumerate}[label=--]
\item (i) $\crit(E^\Uu_{\alpha 
j})<s_{\alpha j}$, (ii) $\strength_{\alpha j}$ is the least
$M^\Uu_{\alpha j}$-cardinal $\geq s_{\alpha j}$,
 and (iii) $\pred^\Uu((\alpha,j)+1)$ is the least 
$(\gamma,\ell)$
such that $\crit(E^\Uu_{\alpha j})<s_{\gamma\ell}$, or
\item (i) $\crit(E^\Uu_{\alpha j})=s_{\alpha j}$,
(ii) $\strength_{\alpha j}=(s_{\alpha j}^+)^{M^\Uu_{\alpha j}}$,
(iii) $s_{\alpha j}$ is not measurable in $M^\Uu_{(\alpha,j)+1}$,
(iv) $j$ is odd.
\end{enumerate}
\end{enumerate}
\item\label{item:internal_coherence} Let $\alpha+1<\lh(\Tt)$ and $i<u_\alpha$. 
Then:
 \begin{enumerate}[label=--] \item If $E^\Uu_{\alpha,
2i}\neq\emptyset$
then [$Q_{\alpha i}$ is type 3, $\nu(Q_{\alpha i})$
is a limit cardinal of $Q_{\alpha i}$, and $s_{\alpha,2i}=\nu(Q_{\alpha 
i})<\OR^{Q_{\alpha i}}<s_{\alpha,2i+1}$],
\item if $E^\Uu_{\alpha,2i}=\emptyset$ then $s_{\alpha,2i}=s_{\alpha,2i+1}$, 
\item  $s_{\alpha,2i+1}=\rho_\om(R_{\alpha,i+1})=\rho_\om(Q_{\alpha,i+1})$,
\item $s_{\alpha,2i},s_{\alpha,2i+1}$ are cardinals of $Q_{\alpha,i}$,
$\core_\om(R_{\alpha,i+1})$, $\core_\om(Q_{\alpha,i+1})$, and $Q_{\alpha,i+1}$,
\item 
$Q_{\alpha 
i}|\rho=\core_\om(R_{\alpha,i+1})|\rho=\core_\om(Q_{\alpha,i+1})|\rho
=Q_{\alpha,i+1}|\rho$ where $\rho=s_{\alpha,2i+1}$.
\end{enumerate}

\item\label{item:s_alpha,2u_alpha_is_nu} Let $\alpha<\beta<\lh(\Tt)$ and $i\leq 
u_\beta$. Then:
\begin{enumerate}[label=--]
\item $s_{\alpha,2u_\alpha}=\nu(Q'_\alpha)$; let $\nu=\nu(Q'_\alpha)$,
\footnote{
Recall that if $\om^*_\alpha$ is non-$\nu$-low then $Q'_\alpha=Q^*_\alpha$
and if $\om^*_\alpha$ is $\nu$-low
and $\nu'=\psi_{\om^*_\alpha}(\nu^\Tt_\alpha)$ 
then [$Q'_\alpha\pins Q^*_\alpha$ and 
$F^{Q'_\alpha}\rest\nu'=F^{Q^*_\alpha}\rest\nu'$
and $\nu'=\nu(Q'_\alpha)$].} 
\item $Q'_\alpha||\nu=Q^*_{\alpha}||\nu=\core_{m_{\beta i}}(Q_{\beta 
i})||\nu$,\footnote{
If $\nu=(\gamma^+)^{Q^*_\alpha}$ then it seems possible
that $N=\core_{m_{\alpha+1}}(Q_{\alpha+1,0})|\nu$
is active with an extender $G$, in which case of course
$Q^*_\alpha|\nu\neq N$,
but $Q^*_\alpha||\OR(Q^*_\alpha)=\Ult(N,G)||\OR(Q^*_\alpha)$.}
\item $\nu<\rho_{m_{\beta i}}(Q_{\beta i})$,
\item $Q^*_\alpha$, $Q'_\alpha$ and $\core_{m_{\beta i}}(Q_{\beta i})$
agree about cardinals $<\nu$.
\end{enumerate}

\item\label{item:gens_at_stage_alpha} Let $\alpha+1<\lh(\Tt)$ and suppose that 
$E_\alpha\in\es_+(M_\alpha)$.
Then $E^\Uu_{\alpha,2u_\alpha}\neq\emptyset$ and
\[ 
\sup\om_\infty``\nu^\Tt_\alpha\leq s_{\alpha,2u_\alpha}=\nu(Q'_\alpha)
\leq\omega_\infty(\nu^\Tt_\alpha).\]
Now suppose also that $u_\alpha>0$. Let 
$\bar{\nu}=\sup_{\beta<\alpha}\nu^\Tt_\beta$. Note that
\[ \bar{\nu}\leq\varrho_{\alpha 1}<\varrho_{\alpha 2}<\ldots<\varrho_{\alpha 
u_\alpha}\leq\nu^\Tt_\alpha. 
\]
(Here if $u_\alpha\leq 1$ and $\alpha$ is a limit we could have 
$\bar{\nu}=\nu^\Tt_\alpha$.) Let 
$i<u_\alpha$ and let $j\in\{2i,2i+1\}$. If $E^\Uu_{\alpha j}\neq\emptyset$ then
\[ \sup\omega_\infty``\varrho_{\alpha,i+1}\leq s_{\alpha j}\leq
\omega_\infty(\varrho_{\alpha,i+1}).\]
\end{enumerate}

We now begin the construction.
Recall that $C=(\rho^C,M,N)$ was defined in the proof of Claim
\ref{clm:R_eta_well-def} of \ref{thm:stack_well},
as were $\CC$, $\lambda^\CC$, $M'=S_{\lambda^\CC}^\CC$, $M=\core_\om(M')$,
and likewise $\widetilde{\CC}$ etc.
We have
$\deg^\Tt(0)=(m_0,n_0)=$ the degree of $C$;
that is, either
\begin{enumerate}[label=--]
 \item $m_0=0$ and $M$ is type 3
with $\nu(M)=\rho^C$, or
\item $m_0\geq 0$ and 
$\rho_{m_0+1}^M\leq\rho^C<\rho_{m_0}^M$,
\end{enumerate}
and likewise for $n_0,N$.
We have $M_0=M$ and $N_0=N$.
Let 
$\DD_{00}=\CC$ and
$\xi_{00}=\lambda^\CC$ and
\[ \pi_{00}:\core_0(M_0)\to\core_{m_0}(Q_{00}) \]
be the core embedding. (Recall that
$M,N$ are fully sound, but $M$ may or may not be $\CC$-standard,
and $Q_{00}=S_{\lambda^\CC}^\CC$ may or may not be sound,
and likewise $N,\DD$; in the 
notation that assumes 
$1<\lh(\Tt)$, the core embedding is $\tau^{\om m_0}_{00}$.) We define 
$\widetilde{\DD}_{00}$, $\prixi_{00}$, $\pripi_{00}$ analogously.
Then $\pi_{00}\rest\rho_0=\id=\pripi_{00}\rest\rho_0$,
which gives inductive hypothesis \ref{item:agreement_within}
(for $(\Tt\rest 1,\Uu\rest(0,1))$),
and the others are trivial as $\lh(\Tt\rest 1)=1=\lh(\Uu\rest(0,1))$.

Now let $\lambda$ be a limit ordinal and suppose that the inductive hypotheses 
hold 
of $\Tt\rest\lambda$ and 
$\Uu\rest(\lambda,0)$; we will define $\Uu\rest(\lambda,1)$ and 
$\Tt\rest(\lambda+1)$ and verify 
that the hypotheses still hold.
Note that $\Uu\rest(\lambda,0)$ has limit length and is cofinally non-padded. 
Let $c=\Sigma_V(\Uu\rest(\lambda,0))$.

We can define
$\Sigma_M(\Tt\rest\lambda)$ as
 the unique branch $b$ such 
that 
for 
eventually all $\alpha\in b$, we have $(\alpha,0)\in c$.
For by inductive hypotheses 
\ref{item:xi_mon_dec}--\ref{item:xi_when_k_pos_even},
there are only boundedly
many $(\beta,k)\in c$ with $k>0$,
so by hypothesis \ref{item:pred^U_biceph_not_biceph},
this branch is well-defined and $\Tt$-cofinal.
Similarly, there are only finitely many 
drops in model along $b$, and there are unique choices for 
$\pi_{\lambda 0}$, etc, maintaining the inductive hypotheses.

We now move to the successor case. Suppose that the inductive hypotheses 
hold for 
$\Tt\rest\delta+1$ and $\Uu\rest(\delta,1)$. We will define 
$\Uu\rest(\delta+1,1)$ and show that 
they hold for $\Tt\rest\delta+2$ and 
$\Uu\rest(\delta+1,1)$. Note that this step involves
the use of just one extender in $\Tt$, and finitely many in $\Uu$.

\begin{casenine}\label{case:u_delta=0_not_biceph} $u_\delta=0$ and 
$M_\delta\neq\emptyset$ and $E_\delta\in\es_+^{M_\delta}$
(so $E_\delta=F(M_\delta)$).

\begin{sclmnine}\label{sclm:x_delta,0_is_std}
$\xi_{\delta 0}$ is 
$\DD_{\delta 0}$-standard.\end{sclmnine}
\begin{proof}Suppose otherwise.
Then by induction with \ref{item:if_alpha_non-M-stable},
$\delta$ is $\curlyM$-stable.

Now $\delta\in\curlyB^\Tt$.
For suppose $\delta\in\curlyM^\Tt$ and 
let 
$\beta=\max(\curlyB^\Tt\inter[0,\delta]_\Tt)$.
Then
\[ \rho(B_\beta)\leq\crit(i^\Tt_{\beta\delta})<\rho_0(M_\beta),\]
so $\rho^C<\rho_0^M$.  But then 
$\vartheta_\delta=i_{0\delta}^\Tt(\rho^C)<\rho_0(M_\delta)$,
and
\[ (\vartheta_\delta^+)^{M_\delta}\leq\OR(M_\delta)=\lh(E_\delta),\]
so $u_\delta>0$, by \ref{dfn:dropdown}, contradiction.

So $\delta\in\curlyB^\Tt$.
Note then that by \ref{dfn:dropdown} and since $u_\delta=0$, 
we have $\vartheta_\delta\geq\OR^{M_\delta}$, so $M$ is type 3
with $\rho^C=\nu(M)$. But then by \ref{lem:type_3_standard},
$\lambda^\CC$ is $\CC$-standard, so  by \ref{item:if_alpha_M-stable},
$\xi_{\delta0}$ is $\DD_{\delta0}$-standard.
This completes the proof of the subclaim.
\end{proof}

Let $\lambda=\sup_{\beta<\delta}\lh^\Tt_\beta$. Then 
$\lambda\leq\rho_{m_\delta}(M_\delta)$, so $\pi_{\delta 
0}``\lambda\sub\rho_{m_\delta}(Q_{\delta0})$, but $\tau^{m_\delta 0}_{\delta 
0}\rest\rho_{m_\delta}(Q_{\delta0})=\id$, so
\begin{equation}\label{eqn:agmt_pi*delta_pi_delta_0} 
\omega^*_\delta\rest\lambda=\pi_{\delta 
0}\rest\lambda. \end{equation}

If $\om^*_\delta$ is non-$\nu$-low
then set $E^{\Uu}_{\delta 0}=$ the $*$-least extender $E^*\in M^\Uu_{\delta 0}$ 
witnessing
\ref{dfn:ultra_bkgd}(\ref{item:bkgd_ext}) for $(\DD_{\delta 0},Q_{\delta 0})$;
and as required by \ref{item:s_alpha,2u_alpha_is_nu}, set
$s_{\delta 0}=\nu(Q_{\delta 0})$  in this case. 
Suppose now that $\om^*_\delta$ is $\nu$-low.
In particular, $E_\delta$ is type 3.
Let $\nu'=\psi_{\om^*_\delta}(\nu^\Tt_\delta)$
and $Q'\pins Q_{\delta 0}$ be as in \ref{rem:Q'_from_nu-low}.
By \ref{lem:stage_projecting_to_card} and \ref{lem:type_3_standard},
and as $\nu'$ is a $Q_{\delta 0}$-cardinal, we get
$Q'=S_\gamma^{\DD_{\delta 0}}$ 
for some $\DD_{\delta 0}$-standard 
$\gamma$. Let $E^\Uu_{\delta 0}$ be
the $*$-least $E^*\in 
M^\Uu_{\delta 0}$ 
witnessing \ref{dfn:ultra_bkgd}(\ref{item:bkgd_ext}) for $Q'$,
and $s_{\delta 0}=\nu^{Q'}$.
Note that in both cases, by $*$-minimality,
$\strength_{\delta 0}$ is the least $M^\Uu_{\delta 0}$-cardinal
$\rho\geq s_{\delta 0}$.

Let $\kappa=\crit^\Tt_\delta$ and 
$\alpha=\pred^\Tt(\delta+1)$ 
and $i=i_{\alpha\kappa}$. Note that 
$M^*_{\delta+1}=M_{\alpha\kappa}$ and $m_{\alpha \kappa}=m_{\delta+1}$ and 
$N^*_{\delta+1}=N_{\alpha\kappa}$ and $n_{\alpha\kappa}=n_{\delta+1}$ (with each 
of these 
equalities, it is included that the object on the left is defined iff the one on 
the right is). We 
can and do set $\pred^\Uu(\delta+1,0)=(\alpha,2i)$, by properties 
\ref{item:Uu_normal}--\ref{item:gens_at_stage_alpha}.
The identities of 
$\DD_{\delta+1,0},\priDD_{\delta+1,0},\xi_{\delta+1,0},\prixi_{\delta+1,0}$ are 
determined by property \ref{item:xi_when_k=0}.
We define $\pi_{\delta+1,0}$ and/or $\pripi_{\delta+1,0}$ as usual.

It is 
straightforward and mostly routine to show that the 
inductive hypotheses are maintained, and we leave this mostly
to the reader, just making a 
couple of remarks.
Consider the verification of \ref{item:if_alpha_non-M-stable} at $\delta+1$
in the case that $M_{\delta+1}\neq\emptyset$ and $[0,\delta+1]_\Tt$
does not drop, and $\alpha$ is $\curlyM$-stable,
but $\delta+1$ is not.
So $\vartheta_\alpha\leq\crit(E_\delta)$ and  note that 
$E_\alpha\in\es_+(M_\alpha)$,
$\vartheta_\alpha<\nu(E_\alpha)$, and
$(\vartheta_\alpha^+)^{M_\alpha}\leq\lh(E_\alpha)$.
So by \ref{dfn:dropdown}, $u_\alpha>0$,
and since $\vartheta_\alpha\leq\crit(E_\delta)$
but $E_\delta$ is $M_\alpha$-total,
\ref{item:pred^U_not_biceph} implies
$\pred^\Uu(\delta+1,0)=(\alpha,2)$, and by \ref{item:x_alpha,2i_standard},
$\xi_{\alpha2}$ is $\DD_{\alpha2}$-standard.
Considering the definition of 
$(\DD_{\delta+1,0},\xi_{\delta+1,0})$
given by \ref{item:xi_when_k=0},
this yields \ref{item:if_alpha_non-M-stable} (in the case
under consideration).
Hypothesis \ref{item:s_alpha,2u_alpha_is_nu}
is pretty standard, but we remark
that if $\exit^\Tt_\delta$ is type 1 or 3
(hence $Q'_\delta$ has the same type) 
and $\nu'=\nu(F(Q'_\delta))=(\gamma^+)^{Q'_\delta}$ (some $\gamma$), then we 
can only expect
$Q'_\delta||\nu'=Q_{\delta+1,0}||\nu'$,
because \ref{dfn:ultra_bkgd}(\ref{item:bkgd_ext}) only
gives 
that $F(Q'_\delta)\rest\nu'\sub E^\Uu_\delta\rest\nu'$,
so while $Q'_\delta|\nu'$ is passive,
it might be that
$Q_{\delta+1,0}|\nu'$ is active with an extender $F'$.
\footnote{If $F'=\emptyset$, 
then standard arguments with condensation
show that $Q_{\delta+1,0}||\lambda'=Q'_\delta||\lambda'$,
where $\lambda'=\OR(Q'_\delta)$;
if $F'\neq\emptyset$, standard arguments
show that 
$Q_{\delta+1,0}||\lambda'=\Ult(Q_{\delta+1,0}||\lambda',F')||\lambda'$.
Analogously (and irrespective of the type of $\exit^\Tt_\delta$),
if $Q'_{\delta+1,0}|\nu'$ is passive,
it seems that $Q'_{\delta+1,0}|\lambda'$ 
might be active.}
Regarding \ref{item:agreement_across}, by line 
(\ref{eqn:agmt_pi*delta_pi_delta_0}), 
and because $\omega^*_\delta\sub\pi_{\delta+1,0}$ and/or 
$\om^*_\delta\sub\pripi_{\delta+1,0}$, we 
have maintained the well-definedness of $\om_\infty$.
And regarding \ref{item:gens_at_stage_alpha}, the fact, 
for example 
if $\pi_{\delta+1,0}$ is defined, that
\[ \om_\infty(\nu^\Tt_\delta)=\pi_{\delta+1,0}(\nu^\Tt_\delta)\geq s_{\delta 
0} \]
follows from our choice of $E^*$ (this is why we
use $Q'$ when $\om^*_\delta$ is $\nu$-low).
\end{casenine}

\begin{casenine} $u_\delta=0$ and $E_\delta\notin\es_+^{M_\delta}$
(so $E_\delta=F(N_\delta)$).
 
By symmetry with the previous case.
\end{casenine}

\begin{casenine}\label{case:u_delta>0_B_delta=M_delta} 
$u_\delta>0$ and $E_\delta\in\es_+^{M_\delta}$ 
and if $\delta$ is $\curlyM$-stable
and $\vartheta_\delta<\OR^{M_\delta}$ 
then $\lh(E_\delta)<(\vartheta_\delta^+)^{M_\delta}$.

Let $\varrho=\varrho_{\delta 1}$; 
then $\varrho$ is a cardinal of $M_\delta$, so $\varrho\leq\rho_0(M_\delta)$
and if $\delta$ is $\curlyM$-stable then $\varrho\leq\vartheta_\delta$.
We will first determine $\Uu\rest((\delta,2)+1)$ and the associated objects
(that is, through $M^\Uu_{\delta 2}$, $\DD_{\delta 2}$, $\xi_{\delta 2}$,
$Q_{\delta 1}$, $\pi_{\delta 1}$, etc); 
this splits into various subcases.

\begin{scasenine}\label{scase:rho<rho_0} $\varrho<\rho_0(M_\delta)$.

Set $E^\Uu_{\delta 0}=\emptyset$; so 
$\pred^\Uu(\delta,1)=(\delta,0)$ and
$M^\Uu_{\delta 1}=M^\Uu_{\delta 0}$ and $i^\Uu_{\delta 0,\delta 1}=\id$. Set 
$\DD_{\delta 
1}=\DD_{\delta 0}$.
Let
$\varphi:\core_0(M_\delta)\to\core_0(Q_{\delta 0})$
be 
$\varphi=\tau^{m_\delta 0}_{\delta 0}\com\pi_{\delta 0}$. Let
$R=\varphi(M_{\delta 1})$. Set  $s_{\delta 0}=s_{\delta 
1}=\varphi(\varrho)=\rho_\om^R$.
Note $\varphi(\varrho)$ is 
a cardinal of 
$Q_{\delta 0}$ (as $\pi_{\delta 0}$ and $\tau^{m_\delta 0}_{\delta 0}$
are c-preserving).
Moreover, if $\xi_{\delta 0}$ is 
non-$\DD_{\delta 0}$-standard then 
$\varphi(\varrho)\leq\rho_\om(Q_{\delta0})$;
for by non-standardness,
 \ref{item:if_alpha_M-stable}
and
\ref{item:if_alpha_non-M-stable},
$\delta$ is $\curlyM$-stable and
 $\varphi=\pi_{\delta0}$,
and since $\varrho\leq\vartheta_\delta$,
therefore 
$\varphi(\varrho)\leq\pi_{\delta0}(\vartheta_\delta)=\rho_\om^{Q_{\delta0}}$.
So we can fix
$\xi=\xi_{\delta 1}$ with
$\core_0(R)=\core_\om(S_\xi^{\DD_{\delta 1}})$.
So $R_{\delta 1}=S_{\xi_{\delta 1}}^{\DD_{\delta 1}}=S_\xi^{\DD_{\delta 0}}$.
We set
\[ \pi_{\delta 1}=\varphi\rest\core_0(M_{\delta 1}):\core_0(M_{\delta 
1})\to\core_0(R)=\core_\om(R_{\delta 1}). \]

We  now define $E^\Uu_{\delta 1}$,
$M^\Uu_{\delta 2}$, 
$\DD_{\delta 2}$, $\xi_{\delta 2}$, and hence $Q_{\delta 1}=S_{\xi_{\delta 
2}}^{\DD_{\delta 2}}$.
Recall we want $\xi_{\delta 2}$ to be $\DD_{\delta 2}$-standard and
$\core_\om(Q_{\delta 1})=\core_\om(R_{\delta 1})$.

Now if $\xi_{\delta 1}$ is $\DD_{\delta 1}$-standard we set 
$E^\Uu_{\delta 
1}=\emptyset$
(so $\pred^\Uu(\delta,2)=(\delta,1)$ etc),
$\DD_{\delta 
2}=\DD_{\delta 1}$ and $\xi_{\delta 2}=\xi_{\delta 1}$.

Suppose that $\xi_{\delta 1}$ is non-$\DD_{\delta 1}$-standard. So 
$R_{\delta 1}$ is sound. Let $E^\Uu_{\delta 1}$ 
be the $*$-least $G\in M_{\delta 1}$ which is a $\DD_{\delta 
1}$-nice witness for 
$R_{\delta1}$.
Let $\mu=\crit(G)$.
(Note we can) let 
$(\alpha,j)=\pred^\Uu(\delta,2)$ be least $\leq(\delta,1)$ such that  
$\mu<s_{\alpha j}$, if such exists; otherwise
$\mu=s_{\delta 1}$ and $(\alpha,j)=(\delta,2)$.
 If 
$E_\alpha\in\es_+(M_\alpha)$ then let $\FF=\DD_{\alpha j}$ and 
$\zeta=\xi_{\alpha j}$, and 
if $2i=j$ then let $P=Q_{\alpha i}$, and if $2i-1=j$ then $P=R_{\alpha i}$. 
Otherwise let 
$\FF=\priDD_{\alpha j}$, $\zeta=\prixi_{\alpha j}$ and $P=\priQ_{\alpha i}$ or 
$P=\priR_{\alpha 
i}$. 
Let $f=i^\Uu_{\alpha j,\delta 2}$.

By properties \ref{item:internal_coherence},\ref{item:s_alpha,2u_alpha_is_nu}, 
$(P\sim R)|\mu$, and note that
$P|\mu=S_\gamma^{\FF}$ for some $\FF$-standard $\gamma<\zeta$. Since $G$ is a 
nice witness, 
$f(\mu)>\varphi(\varrho)$,  so
$R_{\delta 1}\pins f(P|\mu)$,
and note that $\varphi(\varrho)$ is a cardinal of $f(P|\mu)$.
We set $\DD_{\delta 2}=f(\FF)$, and let $\xi_{\delta 2}$ be the $\xi<f(\zeta)$ 
such that
$\core_0(R)=\core_\om(S^{\DD_{\delta 2}}_\xi)$.
Because $G$ is a nice witness, the agreement between $M^\Uu_{\delta 2}$ and 
$\Ult(M^\Uu_{\delta 
1},G)$ implies that $\xi_{\delta 2}$ is strongly $\DD_{\delta 2}$-standard 
(note that this only depends on $S_{\xi_{\delta 2}}^{\DD_{\delta 2}}$
in $M^\Uu_{\delta 2}$ or in $\Ult(M^\Uu_{\delta 1},G)$,
not on the relevant constructions themselves).
\end{scasenine}

\begin{scasenine}\label{scase:varrho=rho_0(M_delta)} $\varrho=\rho_0(M_\delta)$.

So $M_\delta$ is active type 3. Let $E=F^{M_\delta}$
and $Q=Q_{\delta 0}$ and $F=F^Q$ and $\nu=\nu(F)$. Let
$\upsilon:\core_0(M_\delta)\to\core_0(Q)$ be 
$\upsilon=\tau^{m_\delta 0}_{\delta 0}\com\pi_{\delta 0}$.
Let $\psi=\psi_\upsilon$.

\begin{sscasenine} $\upsilon$ is non-$\nu$-high; that is, 
$\psi(\varrho)\leq\nu$.

Proceed as in Subcase \ref{scase:rho<rho_0}, but using $\varphi=\psi$ instead.
\end{sscasenine}

\begin{sscasenine}\label{sscase:upsilon_nu-high} $\upsilon$ is $\nu$-high; that 
is, $\psi(\varrho)>\nu$.

Here we proceed basically as in 
\cite{recon_res}.
Let $E^\Uu_{\delta 0}$ be the $*$-least witness 
to
\ref{dfn:ultra_bkgd}(\ref{item:bkgd_ext}) for 
$(\DD_{\delta 0},Q)$ and set $s_{\delta 0}=\nu$. Let 
$\Tt'$ be the putative iteration tree on $C$ 
of the form $(\Tt\rest\delta+1)\conc E$. Then 
$M_{\delta 1}\pins\core_0(M^{\Tt'}_{\delta+1})$. Let 
$\alpha=\pred^{\Tt'}(\delta+1)$ and 
$\kappa=\crit(E)$ and $i=i_{\alpha\kappa}$. We set 
$\pred^\Uu(\delta,1)=(\alpha,2i)$; as 
in Case 
\ref{case:u_delta=0_not_biceph} this works. Let $\FF,\zeta,P,f$ be defined from 
 $(\alpha,2i)$ as in 
Subcase \ref{scase:rho<rho_0}. For notational specificity,
let us assume that $E_\alpha\in\es_+(M_\alpha)$;
the other case is likewise,
but instead of $P=Q_{\alpha i}$, we have $P=\widetilde{Q}_{\alpha i}$, etc.
So
$P=Q_{\alpha i}=S_\zeta^{\DD_{\alpha,2i}}$
(as $j=2i$). 
We have $M_{\alpha\kappa}=M_{\alpha i}$ and
$\kappa<\rho_{m_{\alpha\kappa}}(M_{\alpha\kappa})$ and
$\kappa<\rho_0(M_\delta)$,
\[ \pi_{\alpha\kappa}:\core_0(M_{\alpha\kappa})\to
\core_{m_{\alpha\kappa}}(P)\text{ and } 
\upsilon:\core_0(M_\delta)\to\core_0(Q).\]
Let
$\chi=(\kappa^+)^{M_{\delta}}=(\kappa^+)^{M_{\alpha\kappa}}$.
We have (and let)
$\bar{M}=M_{\alpha\kappa}||\chi=M_{\delta}|\chi$.
We have
$\pi_{\alpha\kappa}\rest\chi=\upsilon\rest\chi=\omega_\infty\rest\chi$.
Let $\kappa'=\pi_{\alpha\kappa}(\kappa)=\upsilon(\kappa)=\crit(F)$ and
\[ 
\chi'=\sup\pi_{\alpha\kappa}``\chi=\sup\upsilon``\chi=\sup\om_\infty``\chi.\]
Then $\kappa'<\rho_{m_{\alpha\kappa}}(P)$ is a cardinal of $P$ and
$Q$
and (letting)
$\bar{P}=P||\chi'=Q||\chi'$,
then $\bar{P}$ has largest cardinal $\kappa'$, 
and there is $\gamma\leq\zeta$
which is $\DD_{\alpha,2i}$-standard with
$\bar{P}=S_\gamma^{\DD_{\alpha,2i}}$.
Let
$\bar{\pi}:\bar{M}\to\bar{P}$
be $\bar{\pi}=\pi_{\alpha\kappa}\rest\bar{M}$. So $\bar{\pi}$
is cofinal $\Sigma_1$-elementary.
Let
\[ \bar{U}=\Ult_0(\bar{M},E)=\Ult_0(\bar{M},E\rest\nu(E)) \]
and
$\bar{\psi}:\bar{U}\to
\Ult_0(\bar{P},F\rest\nu)$
be induced from $\bar{\pi}$ and $\upsilon$ via the Shift Lemma.
\footnote{Note here that if $\chi'<((\kappa')^+)^{Q_{\delta 0}}$
then $F\rest\nu$ is not weakly amenable
to $\bar{P}$, and the ultrapower would be different
if we used the full $F$ (with generators through $\OR^Q$), and 
actually if $F$ is of supertstrong type,
it is not clear how it should even be defined.}
So $\bar{\psi}$ is cofinal $\Sigma_1$-elementary
and $\bar{\psi}\com i^{\bar{M}}_{E}=i^{\bar{P}}_{F\rest\nu}\com\bar{\pi}$
and $\upsilon\sub\bar{\psi}$ (and note $\rg(\upsilon)\sub Q|\nu$).
We have $\bar{P}\in M^\Uu_{\alpha,2i}$.
Let
$\psi':\Ult_0(\bar{P},F\rest\nu)
 \to f(\bar{P})$
be the natural factor map; that is, for $g\in\bar{P}$ and generators 
$a\in\nu^{<\om}$,
\[ \psi'([a,g]^{\bar{P}}_{F\rest\nu})=
[a,g]^{M^\Uu_{\alpha j}}_{E^\Uu_{\delta 0}} \]
(recalling that $F\rest\nu\sub E^\Uu_{\delta 
0}$). So $\psi'$ is $\Sigma_0$-elementary and c-preserving
and $\psi'\com i^{\bar{P}}_{F\rest\nu}
=f\rest\bar{P}$ and $\crit(\psi')\geq\nu$. 
 Let
 $\psi_1=\psi'\com\bar{\psi}$. So $\psi_1:\bar{U}\to f(\bar{P})$,
and $\psi_1$ is also $\Sigma_0$-elementary c-preserving,
commutes, and $\upsilon\sub\psi_1$, and of course
$M_\delta||\OR^{M_\delta}\ins\bar{U}$
and $\varrho=\nu(E)$ is a cardinal of $\bar{U}$.

 \begin{sclmnine}\label{sclm:psi_1(varrho)>nu}
$\psi_1(\varrho)>\nu=\nu(F)=s_{\delta0}$.\end{sclmnine}
\begin{proof}We have $\nu=\nu(F)=s_{\delta 0}$ by definition.
If 
$\chi'=((\kappa')^+)^{Q}$ then in fact 
$\bar{\psi}(\varrho)>\nu$,
because in fact $\bar{\psi}\sub\psi=\psi_\upsilon$, and by subsubcase 
hypothesis,
$\psi(\varrho)>\nu$. So suppose  $\chi'<((\kappa')^+)^{Q}$.
Let $\bar{P}^+=Q|((\kappa')^+)^Q$,
so $\bar{P}=(\bar{P}^+)||\chi'$, so noting that
$\bar{P}^+\in M^\Uu_{\alpha,2i}$, we have $f(\bar{P})=(f(\bar{P}^+))||f(\chi')$.
Let $\psi_1^+:\bar{U}\to f(\bar{P}^+)$
be $\psi_1^+=\mathrm{inc}\com\psi_1$, where $\mathrm{inc}$ denotes
inclusion. Then $\psi_1^+$ is also $\Sigma_0$-elementary
c-preserving, and $\psi_1,\psi_1^+$ have the same graph.
But $\psi_1^+=\sigma\com
(\psi\rest\bar{U})$,
where
\[ \sigma:\Ult_0(\bar{P}^+,F)\to f(\bar{P}^+) \]
is the natural factor map.
Since $\psi(\varrho)>\nu$ by subsubcase
hypothesis,
 therefore $\psi_1^+(\varrho)>\nu$, as desired.
\end{proof}

Let $\DD_{\delta 1}=f(\FF)$ and $R=\psi_1(M_{\delta 1})$. Then  
$R\pins 
f(\bar{P})$ and $R=S_\xi^{\DD_{\delta 1}}$ for some $\xi<f(\gamma)$; 
let $\xi_{\delta 1}$ be this $\xi$. Let $\pi_{\delta 
1}=\psi_1\rest\core_0(M_{\delta 1})$ (a fully elementary map).

Now if $\xi$ is $\DD_{\delta 1}$-standard, we set $E^\Uu_{\delta 1}=\emptyset$, 
etc. Otherwise, 
proceed as in Subcase \ref{scase:rho<rho_0} to define $E^\Uu_{\delta 1}$ etc. 
In either case,
$s_{\delta 0}=\nu<\psi_1(\varrho)=s_{\delta 1}$,
by Subclaim \ref{sclm:psi_1(varrho)>nu}.
\end{sscasenine}\end{scasenine}

This completes the definition of $\Uu\rest(\delta,2)$ in all subcases.
We now proceed in general for Case \ref{case:u_delta>0_B_delta=M_delta},
to determine $\Uu\rest(\delta+1,1)$ etc.
If $u_\delta=1$ then we have already determined
\[ 
\om^*_\delta=\pi^*_\delta=\tau^{\om0}_{\delta1}\com\pi_{\delta1}
:\exit^\Tt_\delta\to Q^*_\delta=Q_{\delta1}.\]
Here $\om^*_\delta$ is $\Sigma_1$-elementary and so non-$\nu$-low,
by \ref{rem:Q'_from_nu-low}. So 
$Q'_\delta=Q^*_\delta$ and (following \ref{item:s_alpha,2u_alpha_is_nu})
we set
$E^\Uu_{\delta 2}$  be the $*$-least
background for $Q^*_\delta$, and $s_{\delta 2}=\nu(F(Q^*_\delta))$.

Now $s_{\delta 
1}=\rho_\om(R_{\delta 1})=\rho_\om(Q_{\delta 1})\leq s_{\delta 2}$.
We claim that if 
$E^\Uu_{\delta 
1}\neq\emptyset$ then $s_{\delta 1}<s_{\delta 2}$. For otherwise
$s_{\delta 
1}=s_{\delta 2}$, which implies $R_{\delta 1}=Q_{\delta 1}$ is type 1 or type 
3, 
but then by \ref{lem:type_3_standard}, $\xi_{\delta 1}$ is $\DD_{\delta 
1}$-standard, so 
$E^\Uu_{\delta 1}=\emptyset$, contradiction. (Also if $E^\Uu_{\delta 
0}\neq\emptyset=E^\Uu_{\delta 
1}$, then $s_{\delta 0}<\psi_1(\varrho)=s_{\delta1}\leq s_{\delta 2}$,
by Subclaim \ref{sclm:psi_1(varrho)>nu}.)
The remaining definitions and propagation of inductive hypotheses
are like in 
Case \ref{case:u_delta=0_not_biceph}.

If $u_\delta>1$ then we next repeat
the preceding subcases, working with $M_{\delta 2}$, $\pi_{\delta 1}$, etc, in 
place of $M_{\delta 1}$, 
$\pi_{\delta 0}$, etc. We iterate this until producing $\om^*_\delta$, 
$Q^*_\delta$ and $E^\Uu_{\delta,2u_\delta}$
(as above, $\om^*_\delta$ is non-$\nu$-low). This 
completes the definition of $\Uu\rest(\delta+1,1)$.
\end{casenine}

\begin{casenine} $u_\delta>0$ and $E_\delta\notin\es_+^{M_\delta}$  
and if $\delta$ is $\curlyN$-stable
and $\vartheta_\delta<\OR^{N_\delta}$
then $\lh(E_\delta)<(\vartheta_\delta^+)^{N_\delta}$.

By symmetry.
\end{casenine}

\begin{casenine}
$u_\delta>0$ and $E_\delta\in\es_+^{M_\delta}$
and $\delta$ is $\curlyM$-stable
and $\vartheta_\delta<\OR^{M_\delta}$ and 
$(\vartheta_\delta^+)^{M_\delta}\leq\lh(E_\delta)$.
 
This case proceeds mostly like the preceding cases, but the 
first step is a little different. Recall that here 
the
reversed 
$(\Tt,\delta,\lh(E_\delta))$-dropdown
sequence begins with $(M_\delta,0),(M_\delta,\vartheta_\delta)$,
and since $\delta$ is $\curlyM$-stable, 
 $(\DD_{\delta0},\xi_{\delta0})=i^\Uu_{00,\delta0}(\CC,\lambda^\CC)$,
and recall that $m_0<\om$ and $M$ is fully sound
and either
\begin{enumerate}[label=--]
 \item 
$m_0=0$ and $M=S_{\lambda^\CC}^\CC$ is 
type 3 with $\nu(M)=\rho^C$, or 
\item $M=\core_{m_0+1}(S_{\lambda^\CC}^\CC)$ is fully sound
with $\rho_{m_0+1}^M=\rho^C<\rho_{m_0}^M$.
\end{enumerate}

We set $E^\Uu_{\delta 
0}=\emptyset$ and $\DD_{\delta 
1}=\DD_{\delta 0}$, $\xi_{\delta 1}=\xi_{\delta 0}$, 
$m_{\delta1}=m_{\delta0}=m_\delta=m_0$, etc,
so $R_{\delta1}=Q_{\delta0}$. (We also set $\widetilde{\xi}_{\delta 
1}={\uparrow}$.)
We will set 
\[ s_{\delta0}=s_{\delta1}=\rho_\om^{R_{\delta1}}=i^\Uu_{00,\delta0}(\rho^C)
=i^\Uu_{00,\delta1}(\rho^C)=\psi_{\pi_{\delta0}}(\vartheta_\delta).\]
If $\xi_{\delta 1}$ is $\DD_{\delta 
1}$-standard then we also set $E^\Uu_{\delta 1}=\emptyset$, etc. Suppose 
otherwise. So $\lambda^\CC$ is non-$\CC$-standard,
$R_{\delta 1}=Q_{\delta 0}=i^\Uu_{00,\delta 1}(M)$, and these models are
$\om$-sound. So $\rho^C<\rho_0^M$ (by \ref{lem:type_3_standard}).
We set 
$E^\Uu_{\delta 1}=$ 
the $*$-least
$\DD_{\delta 1}$-nice 
witness $G$ for $R_{\delta 1}$, set
$\pred^\Uu(\delta,2),\FF,f$ like usual, set $\DD_{\delta 
2}=f(\FF)$, and set $\xi_{\delta 2}$ to be the $\xi$ such that $R_{\delta 
1}=\core_\om(S_\xi^{\DD_{\delta 
2}})$. So either way, $\xi_{\delta 2}$ is $\DD_{\delta 2}$-standard. After this 
we proceed as 
before.
\end{casenine}

\begin{casenine} Otherwise (equivalently,
 $u_\delta>0$ and 
$E_\delta\notin\es_+^{M_\delta}$
and $\delta$ is $\curlyN$-stable
and $\vartheta_\delta<\OR^{N_\delta}$
and $(\vartheta_\delta^+)^{N_\delta}\leq\lh(E_\delta)$).

By symmetry.
\end{casenine}

This completes the proof of part \ref{item:bicephalus_it} of Claim 
\ref{clm:iterability}.
In the proof of part \ref{item:ultra-stack_it},
we need substitutes $\CC'_{\alpha j}$ for $\CC_{\alpha j}$
and $<'_{\alpha j}$ for 
$<^*_{\alpha j}$ (there is, however, no $\widetilde{\CC}_{\alpha j}$ to 
consider,
and we do not need the notation $\DD_{\alpha j}$ or $\widetilde{\DD}_{\alpha 
j}$).
We have the class wellorder $<^W$ of $W$
(which, however, need not be a class of $W$ itself).
Given $\gamma\in\OR^W$, let $\CC^{\gamma W}$
be the $<^W$-least ultra-backgrounded construction
of $W$ with last model $R_\gamma^W$. We only
define $\xi_{\alpha 0}$ in the case that $[0,\alpha]_\Tt$ drops in model
(but we always define $\xi_{\alpha j}$ when $j>0$).

Suppose $[0,\alpha]_\Tt$ does not drop.
In this case we determine $\CC_{\alpha0}$ only after
selecting $E^\Tt_\alpha$.
We will have $Q_{\alpha 
0}=L[\es]^{M^\Uu_{\alpha 0}}$
and $\pi_{\alpha 0}:M_\alpha\to Q_{\alpha 0}$ is elementary and
$\pi_{\alpha 0}\com i^\Tt_{0\alpha}=i^\Uu_{00,\alpha0}$.
Let $\gamma$ be least such that 
$\pi_{\alpha0}(\lh(E^\Tt_\alpha))<i^\Uu_{00,\alpha0}(\OR(R_\gamma))$.
Then we set $\CC_{\alpha0}=i^\Uu_{00,\alpha0}(\CC^{\gamma W})$.
So in $M^\Uu_{\alpha0}$, $\CC_{\alpha 0}$ has last model
$L[\es]|\zeta$ for some $L[\es]$-cardinal $\zeta$,
and $\pi_{\alpha0}(\lh(E^\Tt_\alpha))<\zeta$.

If $[0,\alpha]_\Tt$ drops or $j>0$, then
$\CC_{\alpha j}$ is defined basically as before
(though it is more standard, because there is no $\widetilde{\CC}_{\alpha j}$ 
etc).

Now consider $<^*_{\alpha j}$.
Let $f=i^\Uu_{00,\alpha j}$.
(Note that $f({<^W})$ need not make sense,
since ${<^W}$ need not be a $W$-class.)
Given $x,y\in M^\Uu_{\alpha j}$, set
$x<_{\alpha j}'y$
iff either (i) $\rank(x)<\rank(y)$,
or (ii) [$\rank(x)=\rank(y)$ and letting $\beta<\OR^W$
be least such that $x,y\in f(V_\beta^W)$,
and letting $<_0$ be the $<^W$-least wellorder
of $V_\beta^W$, then $(x,y)\in f({<_0})$].
Clearly $<_{\alpha j}'$ is a wellorder of $M^\Uu_{\alpha j}$;
we  use this in place of $<^*_{\alpha j}$ when selecting $E^\Uu_{\alpha j}$.

The rest is straightforward.
This completes the proof of the theorem.
\end{proof}\renewcommand{\qedsymbol}{}
\end{proof}

\begin{rem}
Suppose $W\sats\ZFC$ is an iterable premouse. Let 
$L[\es]^W$ be the output of the 
pm-ultra stack construction of $W$.
Then we have the usual partial converse to 
the fact that $L[\es]^W$ 
inherits Woodins. That is, let $\delta$ be Woodin in 
$L[\es]$. Then $W|\delta$ is generic for the 
extender algebra of $L[\es]$ 
at $\delta$, and $\delta$ is Woodin in $L[\es][W|\delta]$.

The natural analogue of \cite[Theorem 3.2]{maxcore} also holds
for  ultra-backgrounded constructions, and hence for the ultra-stack 
construction, assuming that extenders cohere the relevant iteration strategy:
\begin{tm}\label{tm:stationarity_of_LE}
Assume $\ZFC$ and let $M$ be a countable, $k$-sound, $(k,\OR)$-iterable
premouse, and $\Sigma$ be a $(k,\OR)$-iteration strategy for $M$. 
Suppose
that
$i_E(\Sigma)=\Sigma\rest\Ult(V,E$
 for every short $V$-extender $E$.
 
 Let $\CC=\left<S_\alpha\right>_{\alpha\leq\lambda}$ be an ultra-backgrounded
construction. Then there is $\xi\leq\lambda$ 
and $\left<\Tt_\alpha\right>_{\alpha\leq\xi}$ such that
(i) $\Tt_\alpha$ is a successor length tree via $\Sigma$
 with $S_\alpha\ins M^{\Tt_\alpha}_\infty$,
(ii) if $\alpha<\xi$ and $S_\alpha=M^{\Tt_\alpha}_\infty$
 then $b^{\Tt_\alpha}$ drops in model,
(iii) if $\xi<\lambda$ then $b^{\Tt_\xi}$ does not drop in model.

Moreover, either:
\begin{enumerate}[label=\tu{(}\alph*\tu{)}]
 \item\label{item:non-drop_construction} there is some ultra-backgrounded 
construction $\CC$,
with last model $S_\lambda$, such that $\Tt^\CC_\lambda$
exists and $b^{\Tt^\CC_\lambda}$
does not drop in model, or
\item\label{item:R_OR_is_iterate} the ultra-stack construction 
$\left<R_\alpha\right>_{\alpha\in\OR}$
is well-defined, and there is a length $\OR$ tree $\Uu$ on $M$,
via $\Sigma$, such that $R_\OR=M(\Uu)$.
\end{enumerate}
\end{tm}
\begin{proof}
 The proof is like that of 
  \cite[Theorem 3.2]{maxcore};
  there are also variants of this argument in \cite[\S5]{mim}
  and elsewhere. Let $\CC$ be an ultra-backgrounded construction.
  One constructs $\Tt_\alpha$ by induction on $\alpha\leq\lambda$,
  until reaching either $\lambda$ or some appropriate $\xi<\lambda$.
  The induction is straightforward except for when either
  $S_\alpha$ is active or $\alpha$ is non-standard,
  so we consider these cases.
  
 \begin{caseten} $S_\alpha$ is active and $\alpha$ is standard.
  
 So $S_\alpha=(S,F)$ where $S=S_\beta$. Let $F^*$ be a background for $F$ (as 
in \ref{dfn:ultra_bkgd}). Let $k:\Ult_0(S,F)\to j(S)$ be
the natural factor map; so $\nu(F)\leq\crit(k)$.
Let $j=i^V_{F^*}$ and $\kappa=\crit(j)$. We have
$\Tt=\Tt_\beta$
  with $S\ins 
M^\Tt_\infty$;
since $\kappa$ is measurable and $M$ countable, $\lh(\Tt)\geq\kappa+1$.
We may assume that $\lh(G)\leq\OR^S$ for all extenders
$G$ used in $\Tt$.
Then
$j(S)\ins M^{j(\Tt)}_\infty$,
  and by assumption, $j(\Tt)$ is via $\Sigma$.
Note $\Tt\rest(\kappa+1)\ins j(\Tt)$
and $j([0,\kappa]_\Tt)=[0,j(\kappa)]_{j(\Tt)}$,
 and $[0,\kappa)_\Tt\sub [0,j(\kappa)]_{j(\Tt)}$,
 so $\kappa<_{j(\Tt)} j(\kappa)$.  We have
$i_{\kappa,j(\kappa)}^{j(\Tt)}\rest\pow(\kappa)\sub j$, by the proof
 of termination of comparison;
in particular, $\kappa=\crit(i^{j(\Tt)}_{\kappa,j(\kappa)})$.
 Note $[0,j(\kappa)]_{j(\Tt)}$ has no drops $\geq\kappa$, 
 so $(\kappa^+)^{M^{j(\Tt)}_\infty}=(\kappa^+)^{M^\Tt_\kappa}$.
 Clearly 
$\theta=(\kappa^+)^S\leq(\kappa^+)^{M^\Tt_\kappa}$ and
and $S|\theta=M^{j(\Tt)}_\infty||\theta$.
 Let $\gamma+1=\mathrm{succ}^{j(\Tt)}(\kappa,j(\kappa))$
 and $E=E^{j(\Tt)}_\gamma$. Then $\crit(E)=\kappa$
and $E\rest\nu(E)$ is derived from $j$.
Let $\nu=\min(\nu(E),\nu(F))$.
Then
\[ E\inter((M^\Tt_\kappa||\theta)\cross[\nu]^{<\om})=
F\rest\nu.\]

\begin{scaseten} $(\kappa^+)^{M^\Tt_\kappa}=\theta=(\kappa^+)^S$.

So $E,F$ are compatible; that is,
$E\rest\nu=F\rest\nu$.
By the ISC for $(S,F)$, $\Tt$ does not use any extender of index
${<\OR^S}$
which is compatible with $F$.

\begin{sscaseten} $j(S)||\OR^S=S$.

Then $\Tt,j(\Tt)$ use the same extenders with index ${<\OR^S}$.
So by the previous paragraph, $\OR^S\leq\lh(E)$.
If $\lh(E)=\OR^S$,
then $E^{j(\Tt)}_\gamma=E=F$,
so $(S,F)\ins M^{j(\Tt)}_\gamma$,
so $\Tt_\alpha=j(\Tt)\rest(\gamma+1)$ is as desired.
If $\lh(E)>\OR^S$,
then by the ISC applied to $\exit^{j(\Tt)}_\gamma$,
and since $S=\exit^{j(\Tt)}_\gamma||\OR^S$,
we get $(S,F)\pins\exit^{j(\Tt)}_\gamma$,
so we can set $\Tt_\alpha=j(\Tt)\rest(\gamma'+1)$ with some $\gamma'$.
\end{sscaseten}

\begin{sscaseten}
$j(S)||\OR^S\neq S$.

Then $F$ is type 1 or 3.
Let $\nu'=\nu(F)$. By condensation arguments using $k$,
$j(S)|\nu'$ is active with an extender $G$
and $S=\Ult(j(S)|\nu',G)||\lambda$.
It follows that there is $\gamma'$
such that $\Tt\rest(\gamma'+1)\ins j(\Tt)$ and
$E^\Tt_{\gamma'}=G$, but $\nu'<\lh(E^{j(\Tt)}_{\gamma'})$.
Like before, $\lh(E)\geq\nu'$, so $\lh(E)>\nu'$
(as $\lh(G)=\nu'$).
So by the ISC applied to $\exit^{j(\Tt)}_{\gamma}$,
with respect to $\nu'$,
and since $\lh(G)=\nu'$,
 $(S,F)\ins M^{\Tt}_{\gamma'+1}$, which suffices.
\end{sscaseten}
\end{scaseten}

\begin{scaseten} $(\kappa^+)^{M^\Tt_\kappa}>\theta=(\kappa^+)^S$. 

Then $\crit(k)=\theta$, so $F$ is type 1.
We now argue with subsubcases much as before, but using the (proofs of)
\cite[Theorems 4.11, 4.12, 4.15]{mim} in place of the ISC. (In \cite{mim},
premice are always assumed to be below a superstrong.
But the proofs adapt routinely to the superstrong setting.)
\end{scaseten}

This completes the construction of $\Tt_\alpha$ in this case.
\end{caseten}

\begin{caseten}
Now suppose instead that $\alpha$ is non-$\CC$-standard.
So $\alpha=\beta+1$, and letting $\rho=\rho_\om(S_\alpha)$,
and $\Tt=\Tt_\beta$, we have $S_\beta=S_\alpha||(\rho^+)^{S_\alpha}\ins 
M^\Tt_\infty$.
Let $G$ be a nice witness for $S_\alpha$ and $j=i^V_G$.
Then $S_\alpha\pins j(S_\alpha|\rho)$,
so $S_\alpha\pins M^{j(\Tt)}_\infty$,
and since $j(\Tt)$ is via $\Sigma$, this suffices to yield $\Tt_\alpha$.
\end{caseten}

This completes the inductive construction of the trees $\Tt_\alpha$.

Now suppose there is no ultra-stack construction as in 
part \ref{item:non-drop_construction}
of the the ``moreover'' clause of the theorem.
Then for every ultra-backgrounded construction $\CC$
and every $\alpha<\lh(\CC)$, $\Tt^\CC_\alpha$ exists
and if $S_\alpha^\CC=M^{\Tt^\CC_\alpha}_\infty$
then $b^{\Tt^\CC_\alpha}$ drops in model.
But then note that no ultra-backgrounded construction can break down;
that is, for each $n<\om$, $\core_n(S_\alpha^\CC)$
is $(n+1)$-universal and $\core_{n+1}(S_\alpha^\CC)$ is $(n+1)$-solid.

Let $\left<R_\alpha\right>_{\alpha\in\OR}$ be the ultra-stack construction.
We show by induction on $\alpha$ that $R_\alpha$ is well-defined
and sound, and for each 
$\alpha\in\OR$,
there is a tree $\Uu_\alpha$ via $\Sigma$ such that
$R_\alpha\pins M^{\Uu_\alpha}_\infty$,
and by taking $\Uu_\alpha$ of minimal length,
then $\Uu_\alpha\ins\Uu_\beta$ for $\alpha<\beta$.
This suffices, because then $\Uu=\bigcup_{\alpha\in\OR}\Uu_\alpha$ is as 
desired.

So suppose $R_\alpha$ is defined and we have $\Uu_\alpha$
with $R_\alpha\ins M^{\Uu_\alpha}_\infty$.
Let $\CC$ be an ultra-backgrounded construction
with $S_\lambda^\CC=R_\alpha$.

Now $R_\alpha\pins M^{\Uu_\alpha}_\infty$.
For suppose $R_\alpha=M^{\Uu_\alpha}_\infty$.
Then because \ref{item:non-drop_construction} fails,
$b^{\Uu_\alpha}$ drops in model,
so $R_\alpha$ is not sound.
Let $R=\core_\om(M^{\Uu_\alpha}_\infty)$
(this exists and in fact $R\pins M^{\Uu_\alpha}_\xi$
for some $\xi$).
So $R\not\ins R_\alpha$.
But $R_\alpha=S_\lambda^\CC$ is produced by ultra-backgrounded
construction, so $R$ is also, so by maximality of the ultra-stack
construction,  $R\ins R_\alpha$,
a contradiction.

So $R_\alpha\pins M^{\Uu_\alpha}_\infty$.
Now consider the sound premice $R$ which project
to $\OR^{R_\alpha}$ and form the stack
$R_{\alpha+1}$ above $R_\alpha$.
These $R$ are produced by ultra-backgrounded construction,
and $R_\alpha\pins R$, so we get $\Tt_R$ such that $R\ins M^{\Tt_R}_\infty$,
and note $\Uu_\alpha\pins\Tt_R$, and it easily follows
that $R\ins M^{\Uu_\alpha}_\infty$, giving well-definedness
of $R_{\alpha+1}$. And note we get either $\Uu_{\alpha+1}=\Uu_\alpha$,
or $\Uu_{\alpha+1}=\Uu_\alpha\conc\left<E\right>$ where 
$\lh(E)=\OR(R_{\alpha+1})$.

Limit stages are easy. This completes the proof.\end{proof}
\end{rem}

\section{Questions}\label{sec:questions}

Since condensation follows from solidity and normal iterability, we  ask:

\begin{enumerate}[label=--]
 \item Let $m<\om$ and let $M$ be an $m$-sound, $(m,\om_1+1)$-iterable premouse. Is $M$
$(m+1)$-universal? Is $M$ $(m+1)$-solid?
 \item Let $M$ be a $1$-sound $(0,\om_1+1)$-iterable premouse. Is $M$ 
Dodd-solid?
\end{enumerate}

We conjecture that the answer in each case is ``yes'',\footnote{The author
has since confirmed this conjecture, including superstrongs;
see \cite{fsfni}.} at least if $M$ has no superstrong initial 
segments. However, it appears less clear how to prove these things than it is condensation; if one 
attempts an approach similar to the proof of condensation (from normal iterability) then, at least 
na\"ively, structures arise similar to bicephali $B$, but the premice involved may fail to be 
$\rho(B)$-sound. Such generalizations of cephalanxes also arise. This lack of soundness makes the 
analysis of these structures less clear than those considered in this paper.

One also uses $(0,\om_1,\om_1+1)^*$-iterability of pseudo-premice to prove that 
they satisfy the ISC. 
It seems that one might get around this by avoiding pseudo-premice entirely (in 
the proof of 
\ref{thm:stack_well}), using bicephali 
and cephalanxes instead. Extra difficulties also seem to arise here with 
superstrong premice.\footnote{Actually, there is an easy direct proof
that normal iterability suffices for this result, which does not 
need any bicephalus-style comparisons. This will also appear in 
\cite{fsfni}.}

\bibliographystyle{plain}
\bibliography{../../bibliography/bibliography}

\end{document}